\documentclass[reqno]{amsart}
\usepackage{cite}
\usepackage{mathrsfs}
\usepackage{color}
\usepackage{amsmath}
\usepackage{amsfonts}
\usepackage{amssymb}
\usepackage{graphicx}
\usepackage{hyperref}
\usepackage{multirow}
\usepackage{caption}
\usepackage{float}
\usepackage[all,pdf]{xy}
\usepackage{tikz}
\usepackage{tikz-cd}
\usepackage{mathabx}
\usepackage[numbers,sort&compress]{natbib} 
\usepackage[left=4.5cm,right=4.5cm,bottom=2.5cm]{geometry}

\newtheorem{Theorem}{Theorem}[section]
\newtheorem{Corollary}[Theorem]{Corollary}
\newtheorem{Lemma}[Theorem]{Lemma}

\newtheorem{Definition}[Theorem]{Definition}

\newtheorem{Remark}[Theorem]{Remark}
\newtheorem{Example}[Theorem]{Example}
\newtheorem{Construction}[Theorem]{Construction}
\newtheorem{Claim}[Theorem]{Claim}
\numberwithin{equation}{section}

\title{Codimension 2 transfer of signatures in $L$ theory}
\author{Yuetong Luo}
\begin{document}
	\maketitle
	\begin{abstract}
		The signature of a closed manifold is an important geometric topology. In \cite{Codimension2}, Higson, Xie and Schick proved an invariance theorem in codimension 2 for the $K$-theoretic signature. They asked for the $L$-theoretic counterpart of their result. In this note, we will answer their question and moreover, construct a tranfer map between the symmetric $L$-groups of the fundamental groups of $M$ and $N$, which carries the signature of $M$ to that of $N$ up to a torsion of order at most $4$. 
	\end{abstract}
	\tableofcontents
	\section{Introduction}
	An important invariant in geometric topology for manifolds is the signature. Let $M$ be a closed manifold of dimension $m$. There are two ways of constructing the signature, namely the $K$-theoretic way and the $L$-theoretic way.

(1) $K$-theoretic construction

Let $\widetilde{M}$ be the universal cover of $M$. Consider the signature operator on $M$, twist it with the Mishchenko bundle $\nu_M=\widetilde{M} \times_{\pi_1(M)} C^*_{max}\pi_1(M)$ and take its index. This gives the $K$-theoretic signature $Sgn^K(M)$, which is an element of the abelian group $K_m(C^*_{max}\pi_1(M))$. 

(2) $L$-theoretic construction

Consider the chain complex $C_*(\widetilde{M})$ of $\widetilde{M}$. Since $M$ is a Poincare space, there is a Poincare symmetric structure $\phi_M$ on the chain complex. The cobordism class of the Poincare symmetric complex $(C_*(\widetilde{M}),\phi_M)$ defines the $L$-theoretic signature $Sgn^L(M)$, which is an element of the abelian group $L^m(\mathbb{Z}\pi_1(M))$.

Starting from the well-known work of Gromov and Lawson \cite{GroLaw}, Hanke,Pape and Schick \cite{HBScodim2transfer} obtained a codimension-two vanishing theorem for the index of the Dirac operator on spin manifolds. In \cite{Codimension2}, Higson, Xie and Schick proved the counterpart for signature class of \cite{HBScodim2transfer}. In \cite{Kubotatransfer}, Kubota and Schick further gives a transfer map between the $K$-groups of the group $C^*$-algebra of the corresponding manifolds:

\begin{Theorem}[Theorem 1.3 in \cite{Kubotatransfer}]
	\label{K theory transfer}
	\
	
	Let $M$ be a closed, connected and oriented manifold of dimension $m$ and $N \subset M$ be a connected submanifold of codimension 2 with trivial normal bundle. Assume that the induced map $\pi_1(N) \longrightarrow \pi_1(M)$ is injective and $\pi_2(N) \longrightarrow \pi_2(M)$ is surjective. Then there is a homomorphism $\rho_{M,N}: K_*(C^*_{max}\pi_1(M)) \longrightarrow K_{*-2}(C^*_{max}\pi_1(N))$, called the transfer map, such that:
	
	Let $f:M' \longrightarrow M$ be a map between
	closed oriented manifolds of degree 1. Assume $N$ is transversal to $f$ and set $N'=f^{-1}(N)$. Then
	\begin{equation*}
		\rho_{M,N}(Sgn^K(M_1;f^*\nu_M))=2Sgn^K(N_1;f^*\nu_N) \in K_{m-2}(C^*_{max}\pi_1(N))
	\end{equation*}
	In particular, if $f$ is a homotopy equivalence, then we get:
	\begin{equation*}
		2\big(Sgn^K(N)-f_*Sgn^K(N')\big)=0 \in K_{m-2}(C^*_{max}\pi_1(M))
	\end{equation*}
\end{Theorem}

In this note, we will prove the following counterpart of the above theorem in $L$-theory, which answers the question raised in \cite{Codimension2}:

\begin{Theorem}
	\label{L theory transfer}
	\
	
	Let $M$ with submanifold $N$ be as in Theorem \ref{K theory transfer}. Then there is a homomorphism $\rho_{M,N}: L_*(\mathbb{Z}\pi_1(M)) \longrightarrow L_{*-2}^{\textlangle -1 \textrangle}(\mathbb{Z}\pi_1(N))$, called the transfer map, such that:
	
	Let $(f,b):M' \longrightarrow M$ be a normal map between
	closed oriented manifolds. Assume $N$ is transversal to $f$ and set $N'=f^{-1}(N)$. For any $k \in \mathbb{N}$, denote $\sigma^{ \textlangle -k \textrangle}(f,b)$ to be the image of the surgery obstruction of $(f,b)$ in the $(-k)$-decorated L-group. Then:
	\begin{equation*}
		\rho_{M,N}(\sigma(f,b))=\sigma^{\textlangle -1 \textrangle}(f|_{N'},b|_{N'}) \in L_{m-2}^{\textlangle -1 \textrangle}(\mathbb{Z}\pi_1(N))
	\end{equation*}
	
	In particular, if $f$ is a homotopy equivalence, then we get:
	\begin{equation*}
		4\big(Sgn^L(N)-f_*Sgn^L(N')\big)=0 \in L^{m-2}(\mathbb{Z}\pi_1(N))
	\end{equation*}
	obtaining the $L$-theoretic counterpart of Theorem 1.1 in \cite{Codimension2}.
\end{Theorem}

Here is a brief outline of the paper. In Section 2, we shall introduce the general geometric setup that we work on. In Section 3, assuming some result for $L$ group of the suspension ring (Theorem \ref{suspensionL}), we will construct the transfer map stated in Theorem \ref{L theory transfer}. In Section 4, we construct a category called the locally finite $\mathbb{N}$-graded category at infinty and analyze its $L$-group. In Section 5, using the results in Section 4, we will prove Theorem \ref{suspensionL} and therefore justify our construction in Section 3. In Section 6, we will prove Theorem \ref{L theory transfer}, which is the main theorem of this paper. The last section is devoted to some general computations that will be used in Section 6.

\textbf{Acknowledgements} \quad The author would like to thank his advisor Thomas Schick for the useful discussions of the research question and for carefully reading the first draft. The author would like thank Wolfgang Lueck, Tibor Macko, Steve Ferry for their useful help in the research. The note uses OpenAI GPT OSS 120B for grammar checking. The author is supported by DAAD Graduate School Scholarship Programme (57650678).

	\section{Geometric setting}
	\label{geometric setting}
In this section, we introduce the general geometric setup that will be used in the construction of the transfer map.
\begin{Construction}[Geometric setting]
\

$(1)$ $M$ is a closed, oriented $n$-dimensional manifold, $N \subset M$ is a connected submanifold of codimensional 2 with trivial normal bundle.

$(2)$ Denote $\pi_1(M)=\Gamma,\pi_1(N)=\pi$. Suppose that the inclusion map induces an injection $\pi \hookrightarrow \Gamma$ on the fundamental group and a surjection on $\pi_2$. Denote $\Pi=\pi_1(N \times S^1)=\pi \times \mathbb{Z}$ and let $t$ be the generator of the subgroup $\mathbb{Z}$ of $\Pi$.

$(3)$ Let $\tilde{p}:\widetilde{M} \longrightarrow M$ be the universal cover of $M$. Define $\overline{M}=\pi \backslash \widetilde{M}$. Denote the corresponding covering maps by $\underline{p}:\widetilde{M} \longrightarrow \overline{M}$, $p:\overline{M} \longrightarrow M$.

$(4)$ Choose a tubular neighborhood embedding $e:N \times \mathbb{R}^2 \longrightarrow M$ and lift it to an embedding $\bar{e}:N \times \mathbb{R}^2 \longrightarrow \overline{M}$. Let $W=M \backslash e(N \times  \mathring{D}^2)$, $\overline{W}=p^{-1}(W)$ and $W_{\infty}=\overline{M} \backslash \bar{e}(N \times \mathring{D}^2)$. These spaces are path connected and we have $W_{\infty}=\overline{W} \cup (p^{-1}(e(N \times D^2)) \backslash \bar{e}(N \times \mathring{D}^2))$. If we denote $\pi_1(W)=G$, $\pi_1(\overline{W})=H$, then we have the following commutative diagram of maps between fundamental groups induced by inclusions and covering projections:
$$
\begin{tikzcd}
	\pi_1(N \times S^1)=\Pi \dar{\bar{e}_*} \rar{id} & \pi_1(N \times S^1)=\Pi \dar{\bar{e}_*} \rar& \pi_1(N \times D^2)=\pi \dar \\
	\pi_1(\overline{W})=H \rar{i_*} \dar{p_*} & \pi_1(W_{\infty}) \rar{i'_*} & \pi_1(\overline{M})=\pi \dar{p_*} \\
	\pi_1(W)=G \ar{rr}{j_*} & & \pi_1(M)=\Gamma 
\end{tikzcd}
$$

Moreover, the horizontal maps in the diagram are surjective by the Van-Kampen Theorem. The maps $p_*:H \longrightarrow G$ and $p_*:\pi \longrightarrow \Gamma$ are injective.

$(5)$ By Theorem 4.3 in \cite{HBScodim2transfer}, the inclusion map $\bar{e}: N \times S^1 \longrightarrow W_{\infty}$ induces a split injection on the fundamental group. That is, there is a group homomorphism $r_0:\pi_1(W_{\infty}) \longrightarrow \Pi$, such that $r_0\bar{e}_*=id$ and $p_{\pi}r_0=i'_*$ ($p_{\pi}:\Pi \longrightarrow \pi$ is the projection map). By composing with $i_*$, there is a group homomorphism $r_{\Pi}:H \longrightarrow \Pi$, such that $r_{\Pi}\bar{e}_*=id$.

$(6)$ Let $\widetilde{N}=\underline{p}^{-1}(\bar{e}(N \times \{0\}))$, $\widetilde{W}=\underline{p}^{-1}(\overline{W})$ and $\widetilde{W}_{\infty}=\underline{p}^{-1}(W_{\infty})$. These spaces are path connected and we have $\pi_1(\widetilde{N})=\{e\},\pi_1(\widetilde{W}_{\infty})=\ker i'_*$. Then $\tilde{p}|_{\widetilde{N}}: \widetilde{N} \longrightarrow N$ is a universal cover for $N$. We can lift $\bar{e}$ uniquely to an embedding $\tilde{e}: \widetilde{N} \times \mathbb{R}^2 \longrightarrow \widetilde{M}$ such that $\tilde{e}|_{\widetilde{N} \times \{0\}}=id$.

$(7)$ Denote $\widehat{W}_{\infty}$ to be the covering of $W_{\infty}$ with respect to $\ker r_0$ and let $\hat{p}:\widehat{W}_{\infty} \longrightarrow W_{\infty}$ be the corresponding covering map. Since $\ker r_0$ is normal in $\pi_1(W_{\infty})$ and $\pi_1(W_{\infty})/\ker r_0 \cong \Pi$, we have that $\Pi$ acts on $\widehat{W}_{\infty}$ and $\Pi \backslash \widehat{W}_{\infty}=W_{\infty}$. Denote $\widehat{W}=\hat{p}^{-1}(\overline{W})$, this is a path-connected covering of $W$ and $\pi_1(\widehat{W})=\ker r_0i_*$.

$(8)$ We have $\ker r_0 \vartriangleright \ker i'_* $ and $ker i'_*/ \ker r_0 \cong \mathbb{Z}$ by (5), so $\widehat{W}_{\infty}$ is a covering of $\widetilde{W}_{\infty}$ and $\mathbb{Z} \backslash \widehat{W}_{\infty}=\widetilde{W}_{\infty}$. Denote the covering map to be $\acute{p}:\widehat{W}_{\infty} \longrightarrow \widetilde{W}_{\infty}$. Let $\overrightarrow{W}$ be the universal cover of $W$, it covers $\widehat{W}$ and we denote the covering map to be $\vec{p}:\overrightarrow{W} \longrightarrow \widehat{W}$. 

$(9)$ Since $\widetilde{N} \times \mathbb{R}^1$ is the universal cover of $\widetilde{N} \times S^1$, we can lift $\widetilde{e}|_{\widetilde{N} \times S^1}: \widetilde{N} \times S^1 \longrightarrow \widetilde{W}$ to a map $\vec{e}:\widetilde{N} \times \mathbb{R}^1 \longrightarrow \overrightarrow{W}$. By (5), we have $r_{\Pi}\bar{e}_*=id$ on $\pi_1$, so $\bar{e}_*$ is injective on $\pi_1$. By the covering property, $\tilde{e}_*$ is injective on $\pi_1$. Therefore, $\vec{e}$ is an embedding.

In summary, we have the following commutative diagram of maps:

\begin{tikzcd}
	\widetilde{N} \times \mathbb{R}^1 \ar{dd}\rar[hook]{\vec{e}} & \overrightarrow{W} \dar{\vec{p}} \\
	& \widehat{W} \dar{\acute{p}} \rar{\subset} & \widehat{W}_{\infty} \dar{\acute{p}} \\
	\widetilde{N} \times S^1 \ar{dd}\rar{\tilde{e}} & \widetilde{W} \rar{\subset}\dar{\underline{p}} & \widetilde{W}_{\infty} \rar{\subset}\dar{\underline{p}} & \widetilde{M} \dar{\underline{p}} & \widetilde{N} \times D^2 \ar{dd} \lar[hook]{\tilde{e}} \\
	& \overline{W} \rar{\subset} \dar{p} & W_{\infty} \rar{\subset} & \overline{M} \dar{p} \\
	N \times S^1 \rar[hook]{e} \urar[hook]{\bar{e}} & W \ar{rr}{\subset} & & M & N \times D^2 \lar[hook]{e} \ular[hook]{\bar{e}} \\
\end{tikzcd}

with $\tilde{p}=p\underline{p}:\widetilde{M} \longrightarrow M,\hat{p}=\underline{p}\acute{p}:\widehat{W}_{\infty} \longrightarrow W_{\infty}$.

Moreover, we have the following equivariance of covering maps:

$(a)$ $\vec{p}(hx)=r_{\Pi}(h)\vec{p}(x)$ for all $h \in H$, $x \in \overrightarrow{W}$.

$(b)$ $\acute{p}\vec{p}(gx)=j_*(g)\acute{p}\vec{p}(x)$ for all $g \in G$, $x \in \overrightarrow{W}$.

$(c)$ $\acute{p}(wx)=p_{\pi}(\omega)\acute{p}(x)$ for all $\omega \in \Pi,x \in \widehat{W}_{\infty}$.

$(10)$ By definition in $(4)$, we have $W_{\infty}=\overline{W} \cup (p^{-1}e(N \times D^2) \backslash \bar{e}(N \times \mathring{D}^2))$, so $\widetilde{W}_{\infty}=\widetilde{W} \cup (\mathop{\cup}\limits_{g\pi \neq \pi \in \Gamma/\pi}g\tilde{e}(\widetilde{N} \times D^2))$. Then for every $g\pi \neq \pi \in \Gamma/\pi$, since $\widetilde{N} \times D^2$ is simply connected, we have $\acute{p}^{-1}(g\tilde{e}(\widetilde{N} \times D^2)) \cong \widetilde{N} \times D^2 \times \mathbb{Z}$, with the $\mathbb{Z}$ action given by translations on the $\mathbb{Z}$ component.

\end{Construction}

\begin{Remark}
	In the following sections till the end of the article, for simplicity, we will make an identification of subsets of $N \times \mathbb{R}^2$ with their images under $e$.
\end{Remark}

	\section{Transfer map}
	We will construct the transfer map $\rho_{M,N}$ stated in Theorem \ref{L theory transfer} in this section.

In order to give a description of the map, we recall a basic definition:

\begin{Definition}[suspension ring]
	\
	
	Let $R$ be a ring with involution. Denote $M_{\infty}(R)$ to be the ring of infinite matrices $(a[i,j])_{i,j \in \mathbb{N}}$ with finitely many nonzero entries in each row and column. Denote $M_{\infty}^{fin}(R) \subset M_{\infty}(R)$ to be the ideal consisting of infinite matrices with finitely many nonzero entries.
	
	Then the suspension ring of $R$, denoted by $\Sigma R$, is defined by $\Sigma R=M_{\infty}(R)/M_{\infty}^{fin}(R)$. There is a natural involution on $\Sigma R$ induced by the involution of $R$.
\end{Definition}

The $L$-theory of the $\Sigma R$ is related to the $L$-theory of $R$ by the following theorem:

\begin{Theorem}
	\label{suspensionL}
	For any unital ring $R$ with involution and $m \in \mathbb{Z}$, we have $L_m^h(\Sigma R) \cong L_{m-1}^p(R)$.
\end{Theorem}

The proof of Theorem \ref{suspensionL} will be given in Section \ref{proofsuspension}.

Now we begin the construction of the transfer map $\rho_{M,N}$. We will construct a homomorphism $\rho: \mathbb{Z}\Gamma \longrightarrow \Sigma\mathbb{Z}\Pi$ below. Then $\rho_{M,N}$ is given by $\rho_{M,N}:L_m^h(\mathbb{Z}\Gamma) \stackrel{\rho_*}{\longrightarrow} L_m^h(\Sigma\mathbb{Z}\Pi) \cong L_{m-1}^p(\mathbb{Z}\Pi) \stackrel{S}{\longrightarrow} L_{m-2}^{\textlangle -1 \textrangle}(\mathbb{Z}\pi)$, where $S$ is the splitting map given by Theorem 17.2 in \cite{ranickilowerltheory}.

\begin{Construction}[Construction of $\rho$]
	\label{transfer map}
	\
	
	By $(5)$ in the geometric settings \ref{geometric setting}, the map $r_{\Pi}:H \longrightarrow \Pi$ induces a left action of $H$ on $\mathbb{Z}\Pi$ by $h \cdot x=r_{\Pi}(h)x$ and we can consider its induction to $G$, which acts on $\mathop{\oplus}\limits_{gH \in G/H} \mathbb{Z}\Pi$. The construction is divided into two cases:
	
	$(1)$ If $G/H$ is an infinite set, then the induced action gives us a group homomorphism $\rho_G:G \longrightarrow M_{\infty}(\mathbb{Z}\Pi)$, projecting onto $\Sigma \mathbb{Z}\Pi$ gives us a group homomorphism: $\bar{\rho}:G \longrightarrow \Sigma \mathbb{Z}\Pi$. Now we prove that this homomorphism can descend to a homomorphism on $\Gamma$.
	
	Let $\Lambda$ be the normal subgroup in $G$ generated by $e_*(t)$. By the Van-Kampen theorem, we have $\Lambda=ker (G \longrightarrow \Gamma)$. By the properties of the covering space, we have $ker(G \longrightarrow \Gamma)=ker(H \longrightarrow \pi)$. Thus for any $gH \in G/H$, as $g^{-1}e_*(t)g \in \Lambda \subset H$, we have $e_*(t)gH=gH$ and $\bar{\rho}(e_*(t))$ is the left multiplication of $r(g^{-1}e_*(t)g)$ on each component with index $gH$.
	
	For $g \in G\backslash H$, let $\gamma$ be a loop in $W$ representing $g$. Lift $\gamma$ to a path $\bar{\gamma}:[0,1] \longrightarrow \overline{W}$, such that $\bar{\gamma}(0) \in \bar{e}(N \times S^1) \subset \bar{e}(N \times D^2)$. Since $g \notin H$, $\bar{\gamma}(1)$ is in a different component $A$ of $p^{-1}(N \times D^2)$. Now we can lift $t$ in $A$, denoted by $\bar{t}$. Then $g^{-1}tg$ is given by the concatenation of curves $\bar{\gamma}$,$\bar{t}$ and $\bar{\gamma}^{-1}$. We can also lift the null homotopy of $t$ to a null homotopy of $\bar{t}$ in $A$. Since $A \subset W_{\infty}$, we have $i_*(g^{-1}tg)=e$ and then $r_{\Pi}(g^{-1}tg)=e$.
	
	Thus we get that $\bar{\rho}(e_*(t))$ is identity on all components except that with index $eH$. Then $\bar{\rho}(e_*(t))=[Id]$ and thus $\bar{\rho}(\Lambda)=[Id]$, which means that $\bar{\rho}$ can descend to a group homomorphism $\rho: \Gamma \longrightarrow \Sigma \mathbb{Z}\Pi$. The map can be further extended linearly to a ring homomorphism $\rho:\mathbb{Z}\Gamma \longrightarrow \Sigma \mathbb{Z}\Pi$.

	$(2)$ If $G/H$ is a finite set, then we define $\rho$ to be the trivial map: $\rho(x)=0$ for all $x \in \mathbb{Z}\Gamma$.
\end{Construction}

	\section{Locally finite $\mathbb{N}$-graded category at infinity}
	We will construct a bridge to the proof of Theorem \ref{suspensionL} in this section. The bridge is the locally finite $\mathbb{N}$-graded category at infinity $\mathbb{F}_{\mathbb{N},b}(\mathbb{A})$ assigned to any additive category $\mathbb{A}$ with involution. We will show that $L_*(\mathbb{F}_{\mathbb{N},b}(\mathbb{A})) \cong L_{*-1}^p(\mathbb{A}) := L_{*-1}(\mathbb{P}_0(\mathbb{A}))$ in this section.

For any ring $R$ with involution, let $M^h(R)$ be the additive category of finitely generated free right $R$ modules. We will observe that, by definition (below), the elements of its suspension ring can be viewed as certain morphisms in $\mathbb{F}_{\mathbb{N},b}(M^h(R))$. We will explore further about their relation and the relation of their $L$-groups in the next section.

\subsection{Construction of the category $\mathbb{F}_{\mathbb{N},b}(\mathbb{A})$}
\

We begin with the construction of the category. We recall some basic definitions from Ranicki's book \cite{ranickilowerltheory} first:

\begin{Definition}[Additive category with involution]
	\
	
	Let $\mathbb{A}$ be an additive category. An involution on an additive category $\mathbb{A}$ is a contravariant additive functor $*: \mathbb{A} \longrightarrow \mathbb{A}; M \mapsto M^*$, together with a natural equivalence
	$e : 1 \longrightarrow *^2: \mathbb{A} \longrightarrow \mathbb{A} ; M \longrightarrow (e(M) : M \longrightarrow M
	^{**})$.
\end{Definition}

\begin{Definition}
	\
	
	Let $\mathbb{A}$ be an additive category. $\mathbb{P}_0(\mathbb{A})$ is the additive category with objects $(P,p)$, where $P$ is an object in $\mathbb{A}$ and $p^2=p$. Morphisms $f:(P,p) \longrightarrow (Q,q)$ are morphisms $f \in Hom_{\mathbb{A}}(P,Q)$ that satisfy $f=qfp$.
	
	If there is an involution on $\mathbb{A}$, then there is a natural involution on $\mathbb{P}_0(\mathbb{A})$ given by $(P,p) \mapsto (P^*,p^*), f \mapsto f^*$.
\end{Definition}

\begin{Definition}
	\
	
	Let $R$ be a ring with involution $-$, denote $M^h(R)$ to be the small additive category of based finitely generated free right $R$ modules, that is, the objects are $R^n$ with $n \in \mathbb{N}$ and morphisms are $R$ module homomorphisms. There is an involution on this category given by: $*: P \longrightarrow P^*=Hom_R(P,R)$, with the right $R$ moudle structure on $P^*$ given by $(f \cdot r)(p)= \bar{r} \cdot f(p)$.
\end{Definition}

\begin{Definition}[Locally finite $\mathbb{N}$-graded category]
	\
	
	Let $\mathbb{A}$ be an additive category with involution. The locally finite $\mathbb{N}$-graded category $\mathbb{F}_{\mathbb{N}}(\mathbb{A})$ is defined to be the following additive category with involution:
	
	$(1)$ An object of $\mathbb{F}_{\mathbb{N}}(\mathbb{A})$ is a collection $\{M(j) \ | \ j \in \mathbb{N}\}$ of objects in $\mathbb{A}$ indexed by the natural numbers $\mathbb{N}$, written as $M=\sum\limits_{i \geq 0} M(i)$.
	
	$(2)$ A morphism $f:M \longrightarrow N$ between two objects is a collection $\{f(j,i):M(i) \longrightarrow N(j) \ | \ i,j \in \mathbb{N}\}$ of morphisms in $\mathbb{A}$, such that, the sets $\{j\ | \ f(j,i) \neq 0\}$ and $\{j\ | \ f(i,j) \neq 0\}$ are finite for any $i \in \mathbb{N}$.
	
	$(3)$ The composition is given by $(f \circ g)(i,j)=\sum\limits_{k \geq 0} f(i,k)g(k,j)$ for all $i,j \in \mathbb{N}$.
	
	$(4)$ The involution is given by taking dual pointwise: $ M=\sum\limits_{i \geq 0} M(i) \mapsto M^*=\sum\limits_{i \geq 0} M(i)^*, f={f(i,j)}_{i,j \in \mathbb{N}} \mapsto f^*={f(j,i)^*}_{i,j \in \mathbb{N}}$.
\end{Definition}

The category $\mathbb{F}_{\mathbb{N},b}(\mathbb{A})$ is the quotient category $\mathbb{F}_{\mathbb{N}}(\mathbb{A})$ by finite morphisms:

\begin{Definition}[Locally finite $\mathbb{N}$-graded category at infinity]
	\
	
	Let $\mathbb{A}$ be an additive category with involution. Define the locally finite $\mathbb{N}$-graded category at infinity $\mathbb{F}_{\mathbb{N},b}(\mathbb{A})$ to be the following additive category with involution:
	
	$(1)$ The objects of $\mathbb{F}_{\mathbb{N},b}(\mathbb{A})$ are the same as $\mathbb{F}_{\mathbb{N}}(\mathbb{A})$.
	
	$(2)$ The morphisms are equivalence classes of morphisms in $\mathbb{F}_{\mathbb{N}}(\mathbb{A})$ by the following relation:
	
	$f \sim g \Leftrightarrow \text{the set } \{(i,j) \in \mathbb{N} \times \mathbb{N} \ | (f-g)(i,j) \neq 0\} \text{ is finite}$.
	
	$(3)$ The involution is given by pointwise taking dual: $ M \mapsto M^*, [f] \mapsto [f^*]$.
\end{Definition}

\begin{Remark}
	As we have $f_1+f_2 \sim g_1+g_2$ and $f_1 \circ f_2 \sim g_1 \circ g_2$ for any $f_1 \sim f_2, g_1 \sim g_2$, and $f \sim g$ implies $f^* \sim g^*$, the category $\mathbb{F}_{\mathbb{N},b}(\mathbb{A})$ is an additive category with involution.
\end{Remark}

\begin{Remark}
	Let $R$ be a ring with involution, $\mathbb{A}=M^h(R)$, the additive category of finite generated free right $R$ modules, $M$ be the object in $\mathbb{F}_{\mathbb{N},b}(\mathbb{A})$ with $M(i)=R$ for all $i \in \mathbb{N}$. Then any $[A=(a[i,j])_{i,j \in \mathbb{N}}] \in \Sigma R$ can be viewed as an endomorphism $f_A$ of the object $M$: $f_A(i,j)(a)=a[i,j]a$ for $a \in M(j)=R$.
\end{Remark}

The two categories above are naturally related by a functor, and the following notation is introduced for simplicity:

\begin{Definition}
	\
	
	There is a natural functor $\mathbb{F}_{\mathbb{N}}(\mathbb{A}) \longrightarrow \mathbb{F}_{\mathbb{N},b}(\mathbb{A})$ given by $M \mapsto M$ on objects and $f \mapsto [f]$ on morphisms. Given any quadratic chain complex $(C,\psi)$ in $\mathbb{F}_{\mathbb{N}}(\mathbb{A})$, denote $[(C,\psi)]$ to be the image of it under the natural functor. $[(C,\psi)]$ is a quadratic chain complex in $\mathbb{F}_{\mathbb{N},b}(\mathbb{A})$.
\end{Definition}

\subsection{The $L$-theory of the category $\mathbb{F}_{\mathbb{N},b}(\mathbb{A})$}
\label{computationLgroup}
\

Before analyzing its $L$-theory, we recall a definition in Chapter 5 of \cite{ranickilowerltheory} that is useful for computations :

\begin{Definition}[natural flasque structure]
	\label{flasquedef}
	\
	
	Let $\mathbb{A}$ be an additive category, a natural flasque structure is a triple $(\Sigma,\sigma,\phi)$ that consists of:
	
	$(1)$ An additive functor $\Sigma: \mathbb{A} \longrightarrow \mathbb{A}$.
	
	$(2)$ A natural isomorphism $\sigma: Id \oplus \Sigma \longrightarrow \Sigma$.
	
	$(3)$ An isomorphism $\phi_{M,N}: \Sigma (M \oplus N) \longrightarrow \Sigma M \oplus \Sigma N$ for every pair of objects $M,N$ in $\mathbb{A}$, such that
	$$\sigma_{M \oplus N}= \phi_{M,N}^{-1}(\sigma_M \oplus \sigma_N)(Id_{M \oplus N} \oplus \phi_{M,N}): M \oplus N \oplus \Sigma (M \oplus N) \longrightarrow \Sigma(M \oplus N)$$
\end{Definition}

The definition above implies a vanishing result of certain $K$ and $L$ groups:

\begin{Lemma}
	\label{flasquelem}
	\
	
	$(1)$ Let $\mathbb{A}$ be an additive category with a natural flasque structure. Suppose that we give all the additive categories the split exact structure, then $K_0(\mathbb{A})=K_0(\mathbb{P}_0(\mathbb{A}))=0$.
	
	$(2)$ Furthermore, if we assume that there is an involution $*$ on $\mathbb{A}$ that is compatible with the natural flasque structure, i.e., $(\Sigma M)^*=\Sigma M^*, (\Sigma f)^*=\Sigma f^*$ for any object $M$ and any morphism $f$, then $L_n(\mathbb{A})=0$ for all $n \in \mathbb{Z}$.
\end{Lemma}

\begin{proof}
	\
	
	$(1)$ By the definition of natural flasque structure, we have $M \oplus \Sigma M \cong \Sigma M$ for any object $M$ in $\mathbb{A}$. Therefore, we have $[M]=[M \oplus \Sigma M]-[\Sigma M]=0$, showing that $K_0(\mathbb{A})=0$.
	
	For any object $(M,p)$ in $\mathbb{P}_0(\mathbb{A})$, we have $(\Sigma p)^2=\Sigma p^2=\Sigma p$. Therefore, we have $[(M,p)]=[(M \oplus \Sigma M, p \oplus \Sigma p)]-[(\Sigma M,\Sigma p)]=0$, showing that $K_0(\mathbb{P}_0(\mathbb{A}))=0$.
	
	$(2)$ Let $(C,\psi)$ be any $n$-dimensional quadratic chain complex in $\mathbb{A}$. Since the natural flasque structure is compatible with the involution, we have that $(\Sigma C,\Sigma \psi)$ is also an $n$-dimensional quadratic chain complex in $\mathbb{A}$. Since $(C \oplus \Sigma C,\psi \oplus \Sigma\psi) \cong (\Sigma C,\Sigma \psi)$, we can deduce that $(C,\psi)$ is null-cobordant and thus $L_n(\mathbb{A})=0$.
\end{proof}

Then we begin to analyze the $L$-theory of locally finite $\mathbb{N}$-graded category at infinity $\mathbb{F}_{\mathbb{N},b}(\mathbb{A})$, the main result is:

\begin{Theorem}
	\label{exactseq}
	Let $\mathbb{A}$ be any additive category and define $J=ker(\widetilde{K}_0(\mathbb{P}_0(\mathbb{A})) \longrightarrow \widetilde{K}_0(\mathbb{P}_0(\mathbb{F}_{\mathbb{N}}(\mathbb{A}))))$. Then there is an exact sequence:
	
	\begin{tikzcd}
		... \rar & L_n^J(\mathbb{P}_0(\mathbb{A})) \rar \ar[d,phantom,""{coordinate, name=Z}] & L_n(\mathbb{F}_{\mathbb{N}}(\mathbb{A})) \rar  & L_n(\mathbb{F}_{\mathbb{N},b}(\mathbb{A})) \ar[dll,
		"\partial",
		rounded corners,
		to path={ -- ([xshift=2ex]\tikztostart.east)
			|- (Z) [near end]\tikztonodes
			-| ([xshift=-2ex]\tikztotarget.west)
			-- (\tikztotarget)}] \\
		& L_{n-1}^J(\mathbb{P}_0(\mathbb{A})) \rar & ...
	\end{tikzcd}
\end{Theorem}

\begin{Lemma}
	\label{speflasque}
	There is a natural flasque structure on $\mathbb{F}_{\mathbb{N}}(\mathbb{A})$ that is compatible with the involution.
\end{Lemma}

Combining the two results above with Lemma \ref{flasquelem} gives:

\begin{Corollary}
	\label{Ltheorycatinfinty}
	For any additive category $\mathbb{A}$, we have $L_n(\mathbb{F}_{\mathbb{N},b}(\mathbb{A})) \stackrel{\partial}{\cong} L_{n-1}(\mathbb{P}_0(\mathbb{A}))$.
\end{Corollary}

The rest of this subsection is devoted to the proof of Theorem \ref{exactseq} and Lemma \ref{speflasque}. We start by proving a more straightforward result, namely Lemma \ref{speflasque}:

\begin{proof}[Proof of Lemma \ref{speflasque}]
	\
	
	Let $T:\mathbb{F}_{\mathbb{N}}(\mathbb{A}) \longrightarrow \mathbb{F}_{\mathbb{N}}(\mathbb{A})$ be the right shift functor, defined by $TM(0)=0$, $TM(i)=M(i-1)$ for $i \geq 1$ and $Tf(i,j)=\left\{\begin{aligned} & f(i-1,j-1)  &\ i \geq 1 \ and \ j \geq 1 \\  & \ \ \ \ \ \ \ \ \ 0 &otherwise \ \ \ \ 
	\end{aligned} \right.$ for any morphism $f:M \longrightarrow N$, then we can define a natural flasque structure $(\Sigma,\sigma,\rho)$:
	\begin{equation*}
		\begin{aligned}
			\Sigma M=\mathop{\oplus}\limits_{i=1}^{\infty} T^iM \text{ and } \Sigma f=\mathop{\oplus}\limits_{i=1}^{\infty}T^if \qquad \qquad \qquad \ & \\
			\sigma_M: M \oplus \Sigma M \longrightarrow \Sigma M, \ (a_0,(a_1,a_2,...)) \mapsto (Ta_0,Ta_1,Ta_2,...) & \\
			\phi_{M,N}: \Sigma(M \oplus N) \longrightarrow \Sigma M \oplus \Sigma N, \ (a,b) \mapsto (a,b) \qquad \ \ \ & \\
		\end{aligned}
	\end{equation*}
	
	Examination of the conditions in Definition \ref{flasquedef}:
	
	(a) $\Sigma$ is an additive functor:
	
	As $T$ is an additive functor, we only need to verify that $\Sigma M$ and $\Sigma f$ is an object and a morphism in $\mathbb{F}_{\mathbb{N}}(\mathbb{A})$,  respectively.
	
	By definition, we have $\Sigma M(0)=0$ and $\Sigma M(i)=\mathop{\oplus}\limits_{k=0}^{i-1} M(k)$ for $i \geq 1$. These are objects in $\mathbb{A}$, so $\Sigma M $ is an object in $\mathbb{F}_{\mathbb{N}}(\mathbb{A})$. 
	
	For any morphism $f=\{f(i,j)\}:M \longrightarrow\ N$, $\Sigma f (i,j)=\mathop{\oplus}\limits_{l=1}^{\min(i,j)} f(i-l,j-l): \Sigma M(j)=\mathop{\oplus}\limits_{l=1}^{j} M(j-l) \longrightarrow \Sigma N(i)=\mathop{\oplus}\limits_{l=1}^{i} N(i-l)$. Fixing $i$, we have $\{j \in \mathbb{N}| \ \Sigma f(i,j) \neq 0\} \subset \mathop{\cup}\limits_{l=1}^{\min(i,j)} \{j \in \mathbb{N} | f(i-l,j-l) \neq 0\}$ and $\{j \in \mathbb{N}| \ \Sigma f(j,i) \neq 0\} \subset \mathop{\cup}\limits_{l=1}^{\min(i,j)} \{j \in \mathbb{N} | f(j-l,i-l) \neq 0\}$. Since $f$ is a morphism in $\mathbb{F}_{\mathbb{N}}(\mathbb{A})$, the set on the right side is of finite order, and thus we conclude that $\Sigma f$ is a morphism in $\mathbb{F}_{\mathbb{N}}(\mathbb{A})$.
	
	(b) $\sigma$ is a natural isomorphism:
	
	For any map $f:M \longrightarrow N$, since $T$ is a functor, we have:
	\begin{equation*}
		\begin{aligned}
			\Sigma f \circ \sigma_M(a_0,(a_1,a_2,...))&=\Sigma f(Ta_0,Ta_1,Ta_2,...) \\
			&=(Tf(Ta_0),T^2f(Ta_1),T^3f(Ta_2),...) \\
			&=(T(fa_0),T(Tfa_1),T(T^2f(a_2)),...) \\
			&=\sigma_N (f \oplus \Sigma f)(a_0,(a_1,a_2,...)) \\
		\end{aligned}
	\end{equation*} 

	Therefore, $\sigma$ is a natural transformation.
	
	Writing $\sigma_M$ in the components form: $\sigma_M(i+1,i)=Id: M(i) \oplus \Sigma M(i)=\mathop{\oplus}\limits_{k=0}^i M(k) \longrightarrow \Sigma M(i+1)=\mathop{\oplus}\limits_{k=0}^i M(k)$ and $0$ otherwise. It is clear that $\sigma_M$ is an isomorphism in $\mathbb{F}_{\mathbb{N}}(\mathbb{A})$.
	
	(c) Examination of the equality $\sigma_{M \oplus N}= \phi_{M,N}^{-1}(\sigma_M \oplus \sigma_N)(Id_{M \oplus N} \oplus \phi_{M,N})$:
	
	$
	\begin{aligned}
			& \ \ \ \ \phi_{M,N}^{-1}(\sigma_M \oplus \sigma_N)(Id_{M \oplus N} \oplus \phi_{M,N})(a_0,b_0,((a_1,b_1),(a_2,b_2),...))\\&=\phi_{M,N}^{-1}(\sigma_M \oplus \sigma_N)(a_0,b_0,(a_1,a_2,...),(b_1,b_2,...))\\&=\phi_{M,N}^{-1}((Ta_0,Ta_1,...),(Tb_0,Tb_1,...))\\&=((Ta_0,Tb_0),(Ta_1,Tb_1),...) \\&=\sigma_{M \oplus N}(a_0,b_0,((a_1,b_1),(a_2,b_2),...))
	\end{aligned}
	$
	
	Thus, we have $\sigma_{M \oplus N}= \phi_{M,N}^{-1}(\sigma_M \oplus \sigma_N)(Id_{M \oplus N} \oplus \phi_{M,N})$.

\end{proof}

The proof of Theorem \ref{exactseq} is in analogue with the proof of Theorem 14.2 in \cite{ranickilowerltheory}, beginning with the following lifting lemma:

\begin{Lemma}
	\label{lift}
	\
	
	$(1)$ For any $n$-dimensional quadratic chain complex $(C,\psi)$ in $\mathbb{F}_{\mathbb{N},b}(\mathbb{A})$, there is a $n$-dimensional quadratic chain complex $(C',\psi')$ in $\mathbb{F}_{\mathbb{N}}(\mathbb{A})$, such that $[(C',\psi')]=(C,\psi)$.
	
	$(2)$ For any $n+1$-dimensional quadratic pair $(f:C \longrightarrow D, (\delta\psi,\psi))$ in $\mathbb{F}_{\mathbb{N},b}(\mathbb{A})$, given any $n$-dimensional quadratic chain complex $(C',\psi')$ in $\mathbb{F}_{\mathbb{N}}(\mathbb{A})$ such that $[(C',\psi')]=(C,\psi)$, there is a $n+1$-dimensional quadratic pair $(f':C' \longrightarrow D', (\delta\psi',\psi'))$ in $\mathbb{F}_{\mathbb{N}}(\mathbb{A})$ extending $(C',\psi')$, such that, it maps to $(f:C \longrightarrow D,(\delta\psi,\psi))$ under the natural functor.
\end{Lemma}

\begin{proof}
	\
	
	(1)
	For every object $M$ in $\mathbb{F}_{\mathbb{N}}(\mathbb{A})$ and every subset $I \subset [0,+\infty)$, denote $M\{I\}$ to be the object with $M\{I\}(i)=M(i)$ for $i \in I$ and $M\{I\}(i)=0$ for $i \notin I$. For any two objects $M,N$ in $\mathbb{F}_{\mathbb{N}}(\mathbb{A})$ and a morphism $f$ between them, we denote $f(M) \subset N\{I\}$ if $f(j,i)=0$ for every $i \in \mathbb{N}$ and $j \notin I$.
	
	Now given any $n$-dimensional quadratic chain complex $(C,\psi)$ in $\mathbb{F}_{\mathbb{N},b}(\mathbb{A})$, there exist $k \in \mathbb{Z}$ and $p \in \mathbb{N}$ such that $C_r=0$ if $r<k$ or $r>k+p$. Then, we have $\psi_s^r=0: C_r^* \longrightarrow C_{n-r-s}$ for every $s>n-2k$ and $r \in \mathbb{Z}$, as one of the objects in the morphism must be 0.
	
	Choose any morphisms in $\mathbb{F}_{\mathbb{N}}(\mathbb{A})$ that lifts the differential of $C$ and $\psi$ respectively and denote them by $d_C$ and $\overline{\psi}$. Then, by definition of $\mathbb{F}_{\mathbb{N},b}(\mathbb{A})$, there exist natural numbers $b_0,b_1,...,b_p$, such that:
	
	(a1) $d_C^2(C_{k+r}) \subset C_{k+r-2}\{[0,b_{r-2}]\}$ for all $2 \leq r \leq p$.
	
	(a2) $d_C(C_{k+r}[0,b_r]) \subset C_{k+r-1}\{[0,b_{r-1}]\}$ for all $1 \leq r \leq p$.
	
	(a3) $\partial\overline{\psi}_s(C_{k+r}^*) \subset C_{n-k-r-s-1}\{[0,b_{n-2k-r-s-1}]\}$ for all $0 \leq r \leq p$ and $0 \leq s \leq n-2k-r-1$.
	
	We can define $b_l=-1$ for $l<0$ or $l>p$, then the properties above can be extended to hold for all $r$ and $s$.
	
	Choose the decomposition $C_{k+r}=C_{k+r}\{[0,b_r]\} \oplus C_{k+r}\{[b_r+1, \infty)\}$ and denote $d_C$ and $\overline{\psi}_s$ to be $\begin{bmatrix} d_C^{00} & d_C^{01}\\ d_C^{10} & d_C^{11} \end{bmatrix}$ and $\begin{bmatrix} \overline{\psi}_s^{00} & \overline{\psi}_s^{01}\\ \overline{\psi}_s^{10} & \overline{\psi}_s^{11} \end{bmatrix}$ with respect to this decomposition. Now we define $C'$ to be the chain complex with $C'_r=C_r$ and
	$d_{C'}=\begin{bmatrix} 0 & 0\\ 0 & d_C^{11} \end{bmatrix}$.  Let $\psi'_s=\begin{bmatrix}
	0 & 0\\
	0 & \overline{\psi}_s^{11}
	\end{bmatrix}$, we claim that $(C',\psi')$ is a $n$-dimensional quadratic chain complex in $\mathbb{F}_{\mathbb{N}}(\mathbb{A})$ and $[(C',\psi')]=(C,\psi)$. The proof is divided into three steps:
	
	(a) $C'$ is a chain complex:
	
		Conditions (a1) and (a2) can be rephrased as $d_C^2=\begin{bmatrix} * & * \\ 0 & 0 \end{bmatrix}$ and $d_C^{10}=0$, thus $d_C^{11} \circ d_C^{11}=0$ , and then $d_{C'}^2=0$.
		
	(b) $(C',\psi')$ is quadratic:
		
		Denote $r_0=n-k-r-s-1$, condition (a3) can be rephrased as follows:
		\begin{align*}
			\begin{bmatrix} * & * \\ 0 & 0 \end{bmatrix}&=d_C\overline{\psi}_s-(-1)^{r_0}\overline{\psi}_sd_C^*+(-1)^{(r_0+1)(n+s)}\overline{\psi}_{s+1}^*+(-1)^{s}\overline{\psi}_{s+1} \\
			&=\begin{bmatrix} d_C^{00} & d_C^{01} \\ 0 & d_C^{11} \end{bmatrix} \begin{bmatrix} \overline{\psi}_s^{00} & \overline{\psi}_s^{01} \\ \overline{\psi}_s^{10} & \overline{\psi}_s^{11} \end{bmatrix}-(-1)^{r_0}\begin{bmatrix} \overline{\psi}_s^{00} & \overline{\psi}_s^{01} \\ \overline{\psi}_s^{10} & \overline{\psi}_s^{11} \end{bmatrix} \begin{bmatrix} {d_C^{00}}^* & 0 \\ {d_C^{01}}^* & {d_C^{11}}^* \end{bmatrix} \\
			& \quad+(-1)^{(r_0+1)(n+s)}\begin{bmatrix} {\overline{\psi}_{s+1}^{00}}^* & {\overline{\psi}_{s+1}^{10}}^* \\ {\overline{\psi}_{s+1}^{01}}^* & {\overline{\psi}_{s+1}^{11}}^* \end{bmatrix}+(-1)^{s}\begin{bmatrix} \overline{\psi}_{s+1}^{00} & \overline{\psi}_{s+1}^{01} \\ \overline{\psi}_{s+1}^{10} & \overline{\psi}_{s+1}^{11} \end{bmatrix}
		\end{align*}

		Thus we get $d_{C'}\psi'_s-(-1)^{r_0}\psi'_sd_{C'}^*+(-1)^{(r_0+1)(n+s)}{\psi'}^*_{s+1}+(-1)^{s}\psi'_{s+1}=0$, i.e.,  $\partial \psi'=0$. Then, we can conclude that $(C',\psi')$ is a quadratic chain complex.
	
	(c) $[(C',\psi')]=(C,\psi)$:
	
		By definition, $(d_{C'}-d_C)(i,j)=0:C_{k+r} \longrightarrow C_{k+r-1}$ for $i>b_r,j>b_{r-1}$. Consequently, we have:
		\begin{equation*}
			\begin{aligned}
				& \mathop{\cup}\limits_{i=0}^{b_r} \{(i,j) \ |(d_{C'}-d_C)(i,j) \neq 0\} \\
				\{(i,j) \in \mathbb{N} \times \mathbb{N}\ |(d_{C'}-d_C)(i,j) \neq 0\} \subset \ & \qquad \qquad \qquad \ \ \cup \\
				& \mathop{\cup}\limits_{j=0}^{b_{r-1}} \{(i,j) \  |(d_{C'}-d_C)(i,j) \neq 0\} \\
			\end{aligned}
		\end{equation*}
		
		The sets on the right hand side are finite, as $d_{C'}-d_C$ is a morphism in $\mathbb{F}_{\mathbb{N}}(\mathbb{A})$. Thus, we have $d_{C'} \sim d_C$ and similarly $\psi'_s \sim \overline{\psi}_s$. Then, we have $[(C',\psi')]=(C,\psi)$, completing the proof of (1).
		
	(2)
		Similar to the proof of (1), there exist $k \in \mathbb{Z}$ and $p \in \mathbb{N}$, such that $C_r=D_r=0$ if $r<k$ or $r>k+p$.
		
		Choose any morphisms in $\mathbb{F}_{\mathbb{N}}(\mathbb{A})$ that lift the differential of $D$, the morphism $f$ and $\delta\psi$ respectively and denote them by $d_D$, $\bar{f}$ and $\overline{\delta\psi}$. Set $b_l=-1$ for $l<0$ and $l>p$, similar to the proof of (1), there exists natural numbers $b_0,b_1,...,b_p$, such that:
		
		(a1) $d_D^2(D_{k+r}) \subset D_{k+r-2}\{[0,b_{r-2}]\}$ for all $r$.
		
		(a2) $d_D(D_{k+r}\{[0,b_r]\}) \subset D_{k+r-1}\{[0,b_{r-1}]\}$ for all $r$.
		
		(a3) $(d_D\bar{f}-\bar{f}d_{C'})(C'_{k+r}) \subset D_{k+r-1}\{[0,b_{r-1}]\}$ for all $r$.
		
		(a4) $(\partial\overline{\delta\psi}_s-\bar{f}_{\%}\psi'_s)(D_{k+r}^*) \subset D_{n-k-r-s}\{[0,b_{n-2k-r-s}]\}$ for all $r$ and $s$.
		
		Choose the decomposition $D_{k+r}=D_{k+r}\{[0,b_r]\} \oplus D_{k+r}\{[b_r+1, \infty)\}$ and denote the maps $d_D$, $\bar{f}$ and $\overline{\delta\psi}$ to be $\begin{bmatrix} d_D^{00} & d_D^{01}\\ d_D^{10} & d_D^{11} \end{bmatrix}$, $\begin{bmatrix} \bar{f}^{00} & \bar{f}^{01}\\ \bar{f}^{10} & \bar{f}^{11} \end{bmatrix}$ and $\begin{bmatrix} \overline{\delta\psi}^{00} & \overline{\delta\psi}^{01}\\ \overline{\delta\psi}^{10} & \overline{\delta\psi}^{11} \end{bmatrix}$with respect to this decompositon. Now we define $D'$ to be the chain complex with $D_r'=D_r$ and $d_{D'}=\begin{bmatrix} 0 & 0\\ 0 & d_D^{11} \end{bmatrix}$. Let $f'=\begin{bmatrix} 0 & 0\\ \bar{f}^{10} & \bar{f}^{11} \end{bmatrix}$ $\delta\psi'_s=\begin{bmatrix}
			0 & 0\\
			0 & \overline{\delta\psi}_s^{11}
		\end{bmatrix}$ with respect to the same decomposition above, we claim that $(f':C' \longrightarrow D',(\delta\psi',\psi'))$ is a quadratic pair that maps to $(f:C \longrightarrow D,(\delta\psi,\psi))$ under the natural functor. The proof is divided into three steps:
	
		(a) $D'$ is a chain complex and $f'$ is a chain map.
		
		By condition (a2), we have $d_D^{10}=0$. Condition (a1) can be repharsed as $d_D^2=\begin{bmatrix} * & *\\ 0 & 0 \end{bmatrix}$, thus $d_D^{11} \circ d_D^{11}=0$, showing that $D'$ is a chain complex.
		
		Condition (a3) can be written as:
		\begin{equation*}
			\begin{bmatrix} d_D^{00} & d_D^{01}\\ 0 & d_D^{11} \end{bmatrix} \begin{bmatrix} \bar{f}^{00} & \bar{f}^{01}\\ \bar{f}^{10} & \bar{f}^{11} \end{bmatrix}-\begin{bmatrix} \bar{f}^{00} & \bar{f}^{01}\\ \bar{f}^{10} & \bar{f}^{11} \end{bmatrix} \begin{bmatrix} d_{C'}^{00} & d_{C'}^{01}\\ d_{C'}^{10} & d_{C'}^{11} \end{bmatrix}=\begin{bmatrix} * & *\\ 0 & 0 \end{bmatrix}
		\end{equation*}
	
		 Comparing the entries in the matrix, we can get $d_{D'}f'=f'd_{C'}$.
		
		(b) $(f':C' \longrightarrow D',(\delta\psi',\psi'))$ is a quadratic pair.
		
		Denote $r_1=n-k-r-s$, condition (a4) can be written as 
		\[
			\begin{bmatrix} * & *\\ 0 & 0 \end{bmatrix}=d_D\overline{\delta\psi}_s-(-1)^{r_1}\overline{\delta\psi}_sd_D^*+(-1)^{(r_1+1)(n+1+s)}\overline{\delta\psi}_{s+1}^*+(-1)^s\overline{\delta\psi}_{s+1}+\bar{f}\psi_s'\bar{f}^*
		\]
		
		Comparing the entries in the matrix, we get $d_{D'}\delta\psi_s'-(-1)^{r_1}\delta\psi_s'd_{D'}^*+(-1)^{(r_1+1)(n+1+s)}{\delta\psi'_{s+1}}^*+(-1)^s\delta\psi_{s+1}'+f'\psi_s'{f'}^*=0$, i.e.,  $(f':C' \longrightarrow D',(\delta\psi',\psi'))$ is a quadratic pair.
		
		(c) $(f':C' \longrightarrow D',(\delta\psi',\psi'))$ maps to $(f:C \longrightarrow D,(\delta\psi,\psi))$ under the natural functor.
		
		By definition, we have $(d_{D'}-d_D)(i,j)=0:D_{k+r} \longrightarrow D_{k+r-1}$ for $i>b_r,j>b_{r-1}$. Consequently, we have:
		\begin{equation*}
			\begin{aligned}
				& \mathop{\cup}\limits_{i=0}^{b_r} \{(i,j) \ |(d_{D'}-d_D)(i,j) \neq 0\} \\
				\{(i,j) \in \mathbb{N} \times \mathbb{N}\ |(d_{D'}-d_D)(i,j) \neq 0\} \subset \ & \qquad \qquad \qquad \ \ \cup \\
				& \mathop{\cup}\limits_{j=0}^{b_{r-1}} \{(i,j) \  |(d_{D'}-d_D)(i,j) \neq 0\} \\
			\end{aligned}
		\end{equation*}
		
	The sets on the right hand side are finite, as $d_{D'}-d_D$ is a morphism in $\mathbb{F}_{\mathbb{N}}(\mathbb{A})$. Thus, we have $d_{D'} \sim d_D$ and similarly $f' \sim \bar{f}$ and $\delta\psi'_s \sim \overline{\delta\psi}_s$. Then, we can conclude that the statement of (c) holds, completing the proof of (2).
\end{proof}

From the above lemma, we get:
\begin{Corollary}
	$L_n(\mathbb{F}_{\mathbb{N},b}(\mathbb{A}))$ is naturally isomorphic to the cobordism group of $n$-dimensional quadratic $\mathbb{F}_{\mathbb{N},b}(\mathbb{A})$-Poincare complexes in $\mathbb{F}_{\mathbb{N}}(\mathbb{A})$.
\end{Corollary}

Another ingredient for the proof of Theorem \ref{exactseq} is the following lemma:

\begin{Lemma}
	\label{finitedomin}
	\
	
	Let $C$ be a finite chain complex in $\mathbb{F}_{\mathbb{N}}(\mathbb{A})$. If it is contractible chain complex in $\mathbb{F}_{\mathbb{N},b}(\mathbb{A})$, then $C$ is $(\mathbb{F}_{\mathbb{N}}(\mathbb{A}),\mathbb{A})$-finite dominated. That is, there exists a finite chain complex $D$ in $\mathbb{A}$ and chain maps $f:C \longrightarrow D,g:D \longrightarrow C$ in $\mathbb{F}_{\mathbb{N}}(\mathbb{A})$, such that $gf \simeq Id$.
	
	In particular, $C$ is homotopic to a finite chain complex in $\mathbb{P}_0(\mathbb{A})$.
\end{Lemma}

\begin{proof}
	Let $[T_r] \in Hom_{\mathbb{F}_{\mathbb{N},b}(\mathbb{A})}(C_r,C_{r+1})$ be the chain homotopy $0 \simeq Id$ in $\mathbb{F}_{\mathbb{N},b}(\mathbb{A})$. For every $r \in \mathbb{Z}$, choose a morphism $T_r$ in $\mathbb{F}_{\mathbb{N}}(\mathbb{A})$ that represents $[T_r]$.

	Since $(C,\psi)$ is finite in $\mathbb{F}_{\mathbb{N},b}(\mathbb{A})$, there exist $k \in \mathbb{Z}$ and $p \in \mathbb{N}$, such that $C_r=0$ if $r<k$ or $r>k+p$. Since $0 \stackrel{[T_r]}{\simeq} Id$, we have $[d_C][T_r]+[T_r][d_C]=Id$. By the definition of morphisms in $\mathbb{F}_{\mathbb{N},b}(\mathbb{A})$, we can choose natural numbers $b_0,b_1,...,b_p$, such that:
	
	(a1) $d_C^2(C_{k+r}) \subset C_{k+r-2}\{[0,b_{r-2}]\}$ for all $2 \leq r \leq p$.
	
	(a2) $d_C(C_{k+r}[0,b_r]) \subset C_{k+r-1}\{[0,b_{r-1}]\}$ for all $1 \leq r \leq p$.
	
	(a3) $(Id-d_{C}T_r-T_{r-1}d_{C})(C_r) \subset C_{r}\{[0,b_{r}]\}$ for all $0 \leq r \leq p$.
	
	Let $D$ be the chain complex in $\mathbb{A}$ given by $D_r=C_r\{[0,b_r]\}$ and $d_D=d_C|_{D_r}$. By condition (a2), the definition of $D$ gives a chain complex. Then the required maps $f:C \longrightarrow D,g: D \longrightarrow C$ and homotopy $h:gf \simeq Id$ are given as follows:
	\begin{equation*}
		f=Id-d_{C}T_r-T_{r-1}:C_r \longrightarrow D_r
	\end{equation*}
	\begin{equation*}
		g=\text{Inclusion}:D_r \longrightarrow C_r
	\end{equation*}
	\begin{equation*}
		h=T_r:C_r \longrightarrow C_{r+1}
	\end{equation*}
	
	In conclusion, we have finished our proof.
\end{proof}

\begin{proof}[Proof of Theorem \ref{exactseq}]
	We only need to prove that there is an isomorphism:
	\begin{equation*}
		L_n(\mathbb{F}_{\mathbb{N},b}(\mathbb{A})) \cong L_n^J(\mathbb{P}_0(\mathbb{A}) \longrightarrow \mathbb{F}_{\mathbb{N}}(\mathbb{A}))
	\end{equation*}
	
	The relative $L$ group $L_n^J(\mathbb{P}_0(\mathbb{A}) \longrightarrow \mathbb{F}_{\mathbb{N}}(\mathbb{A}))$ is the cobordism group of $n$-dimensional quadratic Poincare pairs $(f: C \longrightarrow D, (\delta\psi, \psi))$ in $\mathbb{P}_0(\mathbb{F}_{\mathbb{N}}(\mathbb{A}))$, such that $(C,\psi)$ is defined in $\mathbb{P}_0(\mathbb{A})$ and $D$ is defined in $\mathbb{F}_{\mathbb{N}}(\mathbb{A})$. The algebraic Thom construction on such a pair gives an $n$-dimensional quadratic complex $(C(f),\delta\psi/\psi)$ in $\mathbb{P}_0(\mathbb{F}_{\mathbb{N}}(\mathbb{A}))$. Since the inclusion of $\mathbb{A}$ into $\mathbb{F}_{\mathbb{N}}(\mathbb{A})$ is given by:
	\begin{equation*}
		M \mapsto M(k)=\begin{cases} M & \text{If } k=0 \\ \ 0 & \text{else} \end{cases}; \ f \mapsto f(j,i)=\begin{cases} f & \text{If } i=j=0 \\ 0 & \text{else} \end{cases}
	\end{equation*}
	
	It is straightforward to verify from the definitions that $[(C(f),\delta\psi/\psi)]$ is $n$-dimensional Poincare in $\mathbb{P}_0(\mathbb{F}_{\mathbb{N},b}(\mathbb{A}))$. Moreover, the reduced projective class of $C(f)$ is given by $[C(f)] = [C] \in \widetilde{K}_0(\mathbb{F}_{\mathbb{N}}(\mathbb{A}))$. So it will map to $0$ in $\widetilde{K}_0(\mathbb{F}_{\mathbb{N},b}(\mathbb{A}))$, thus $C(f)$ is homotopic as a chain complex in $\mathbb{P}_0(\mathbb{F}_{\mathbb{N},b}(\mathbb{A}))$ to a chain complex in $\mathbb{F}_{\mathbb{N},b}(\mathbb{A})$. 
	
	In summary, the algebraic Thom construction gives a map:
	\begin{equation*}
		L_n^J(\mathbb{P}_0(\mathbb{A}) \longrightarrow \mathbb{F}_{\mathbb{N}}(\mathbb{A})) \longrightarrow L_n(\mathbb{F}_{\mathbb{N},b}(\mathbb{A}))  
	\end{equation*}
	\begin{equation*}
		(f: C \longrightarrow D, (\delta\psi,\psi)) \mapsto [(C(f),\delta\psi/\psi)]
	\end{equation*}

	Conversely, $L_n(\mathbb{F}_{\mathbb{N},b}(\mathbb{A}))$ is naturally isomorphic to the cobordism group of $n$-dimensional quadratic $\mathbb{F}_{\mathbb{N},b}(\mathbb{A})$ Poincare complexes $(C,\psi)$ in $\mathbb{F}_{\mathbb{N}}(\mathbb{A})$. Consider the algebraic thickening $(\partial C \longrightarrow C^{n-*},(0,\partial \psi))$, it is a $n$-dimensional quadratic Poincare pair in $\mathbb{F}_{\mathbb{N}}(\mathbb{A})$. Since $\partial C$ is contractible in $\mathbb{F}_{\mathbb{N},b}(\mathbb{A})$, by Lemma \ref{finitedomin}, it is homotopic to a chain complex in $\mathbb{P}_0(\mathbb{A})$. Moreover, the image of the reduce projective class of $\partial C$ in $\widetilde{K}_0(\mathbb{F}_{\mathbb{N}}(\mathbb{A}))$ is $[C_{*+1}]+[C^{n-*}]$, which is $0$ as $C$ is in $\mathbb{F}_{\mathbb{N}}(\mathbb{A})$. Therefore, we can view $(\partial C \longrightarrow C^{n-*},(0,\partial \psi))$ as a $n$-dimensional quadratic Poincare pair $(f: C' \longrightarrow D', (\delta\psi', \psi'))$ in $\mathbb{P}_0(\mathbb{F}_{\mathbb{N}}(\mathbb{A}))$ such that $(C',\psi')$ is defined in $\mathbb{P}_0(\mathbb{A})$ with reduce projective class in $J$ and $D'$ is defined in $\mathbb{F}_{\mathbb{N}}(\mathbb{A})$.
	
	In summary, the algebraic Thom thickening gives a map:
	\begin{equation*}
		L_n(\mathbb{F}_{\mathbb{N},b}(\mathbb{A})) \longrightarrow L_n^J(\mathbb{P}_0(\mathbb{A}) \longrightarrow \mathbb{F}_{\mathbb{N}}(\mathbb{A}))
	\end{equation*}
	\begin{equation*}
		(C,\psi) \mapsto (\partial C \longrightarrow C^{n-*},(0,\partial\psi))
	\end{equation*}

	Since algebaric Thom construction and algebraic thickening gives reverse isomorphism, we have that $	L_n(\mathbb{F}_{\mathbb{N},b}(\mathbb{A})) \cong L_n^J(\mathbb{P}_0(\mathbb{A}) \longrightarrow \mathbb{F}_{\mathbb{N}}(\mathbb{A}))$ and thus the Theorem holds.
\end{proof}

\begin{Remark}
	\label{Rempartial}
	It is easily seen from the proof that the partial map:
	\begin{equation*}
		\partial: L_n(\mathbb{F}_{\mathbb{N},b}(\mathbb{A})) \longrightarrow L_{n-1}(\mathbb{P}_{0}(\mathbb{A}))
	\end{equation*}

	is given as follows:
	
	Choose any Poincare quadratic chain complex $(C,\psi)$ in $\mathbb{F}_{\mathbb{N},b}(\mathbb{A})$ representing an element $x \in L_n(\mathbb{F}_{\mathbb{N},b}(\mathbb{A}))$, by lemma \ref{lift}, there is a quadratic chain complex $(C',\psi')$, such that $[(C',\psi')]=(C,\psi)$. Consider $(\partial C',\partial\psi')=\partial(C',\psi')$, $\partial C'$ is contractible in $\mathbb{F}_{\mathbb{N},b}(\mathbb{A})$ and thus by lemma \ref{finitedomin}, it is homotopic to a chain complex in $\mathbb{P}_0(\mathbb{A})$. Then $\partial x$ is the element represented by some Poincare quadratic chain complex in $\mathbb{P}_0(\mathbb{A})$ that is homotopic to $(\partial C',\partial\psi')$.
\end{Remark}

	\section{Proof of Theorem \ref{suspensionL}}
	\label{proofsuspension}

In this section we will prove Theorem \ref{suspensionL}, we will construct an explict functor from $M^h(\Sigma R)$ to $\mathbb{F}_{\mathbb{N},b}(M^h(R))$ and prove that the functor induces an isomorphism in L-theory. Combining with Corollary \ref{Ltheorycatinfinty} we get Theorem \ref{suspensionL}. 

In order to give the functor explictly, we need to first describe the morphisms in $M^h(\Sigma R)$, i.e. the matrix ring of $\Sigma R$.

\subsection{Matrix ring of $\Sigma R$}
\

The main goal of this subsection is to prove the following Lemma:

\begin{Lemma}
	\label{matrixsuspension}
	\
	
	Let $r,s \in \mathbb{N}$ and $R$ be a unital ring with involution, denote $M_{r,s}(R)$ to be the $r \times s$ matrix ring of $R$. Then there is an isomorphism $\theta_{r,s}:M_{r,s}(\Sigma R) \cong \Sigma M_{r,s}(R)$, such that
	
	$(1)$ $\theta_{r,t}(xy)=\theta_{r,s}(x)\theta_{s,t}(y)$ for all $r,s,t \in \mathbb{N}$ and $x \in M_{r,s}(\Sigma R)$, $y \in M_{s,t}(\Sigma R)$.
	
	$(2)$ $\theta_{r,r}(I_r)=I_r$, where $I_r$ is the unit on both side.
	
	$(3)$ $\theta_{r,s}$ commutes with the natural involution.
\end{Lemma}

Recall that $\Sigma R$ is defined to be the quotient of $M_{\infty}(R)$ by $M_{\infty}^{fin}(R)$, where $M_{\infty}(R)$ is the ring of infinity matrix with finite many nonzero entries in each row and column and $M_{\infty}^{fin}(R)$ is the ideal consisting of infinity matrix with finite many nonzero entries. Thus we should prove the following analogus result for $M_{\infty}(M_{r,s}(R))$ first.

\begin{Lemma}
	\label{infinitymatrix}
	\
	
	Let $r,s \in \mathbb{N}$. There is an isomorphism $\tilde{\theta}_{r,s}:M_{r,s}(M_{\infty}(R)) \cong M_{\infty}(M_{r,s}(R))$, such that
	
	$(1)$ $\tilde{\theta}_{r,t}(xy)=\tilde{\theta}_{r,s}(x)\tilde{\theta}_{s,t}(y)$ for all $r,s,t \in \mathbb{N}$ and $x \in M_{r,s}(M_{\infty}(R))$, $y \in M_{s,t}(M_{\infty}(R))$.
	
	$(2)$ $\tilde{\theta}_{r,r}(I_r)=I_r$, where $I_r$ is the unit on both side.
	
	$(3)$ $\tilde{\theta}_{r,s}$ commutes with the natural involution.
\end{Lemma}

\begin{proof}
	For any element $X \in M_{r,s}(M_{\infty}(R))$, denote $X_{kl} \in M_{\infty}(R)$ to be the element in the $k$th row and $l$th column of $X$ for $1 \leq k \leq r$, $1 \leq l \leq s$. We define the map $\tilde{\theta}_{r,s}$ by reordering indices: $\tilde{\theta}_{r,s}(X)[i,j]:=(X_{kl}[i,j])_{1 \leq k \leq r,1 \leq l \leq s} \in M_{r,s}(R)$ for any $i,j \in \mathbb{N}$. It is easy to check that the properties holds.
\end{proof}

\begin{Lemma}
	The isomorphism $\tilde{\theta}_{r,s}$ constructed in Lemma \ref{infinitymatrix} takes $M_{r,s}(M_{\infty}^{fin}(R))$ to $M_{\infty}^{fin}(M_{r,s}(R))$.
\end{Lemma}

\begin{proof}
	Since $\tilde{\theta}_{r,s}$ is defined by reordering the indices and $M_{\infty}^{fin}(R)$ is defined to be the matrix with finitely many nonzero entries, the result is obvious.
\end{proof}

Combining the two lemmas above gives the result of lemma \ref{matrixsuspension}.

\subsection{The functor from $M^h(\Sigma R)$ to $\mathbb{F}_{\mathbb{N},b}(M^h(R))$}
\

We will construct the functor and prove it induces an isomorphism on L-theory in this subsection.

We first give a description of morphisms between some special type of objects in the category. These special objects are:

\begin{Definition}
	Let $r \in \mathbb{N}$, denote $\underline{R^r}$ to the object in $\mathbb{F}_{\mathbb{N},b}(M^h(R))$, such that $\underline{R^r}(i)=R^r$ for all $i \in \mathbb{N}$.
\end{Definition}

Morphisms between these objects are closely related to the suspension ring, as shown by the following lemma:

\begin{Lemma}
	\
	\label{identifysuspension}
	
	Let $r,s \in \mathbb{N}$ and $R$ be a unital ring with involution, then $Hom_{\mathbb{F}_{\mathbb{N},b}(M^h(R))}(\underline{R^r},\underline{R^s})$ can be naturally identified with $\Sigma M_{r,s}(R)$, i.e. there is an isomorphism of abelian groups $\mathbb{F}_{r,s}:Hom_{\mathbb{F}_{\mathbb{N},b}(M^h(R))}(\underline{R^r},\underline{R^s}) \longrightarrow \Sigma M_{s,r}(R)$, such that:
	
	$(1)$ $\mathbb{F}_{r,t}(g \circ f)=\mathbb{F}_{s,t}(g)\mathbb{F}_{r,s}(f)$ for any $r,s,t \in \mathbb{N}$ and $f \in Hom_{\mathbb{F}_{\mathbb{N},b}(M^h(R))}(\underline{R^r},\underline{R^s})$, $g \in Hom_{\mathbb{F}_{\mathbb{N},b}(M^h(R))}(\underline{R^s},\underline{R^t})$
	
	$(2)$ $\mathbb{F}_{r,r}(Id_{\underline{R^r}})=I_r$ for any $r \in \mathbb{N}$, where $I_r$ is the unit in $\Sigma M_{r,r}(R)$ 
	
	$(3)$ $\mathbb{F}_{r,s}$ commutes with the natural involution.
\end{Lemma}

\begin{proof}
	We make a sketch of the proof here. By definition of the additive category $\mathbb{F}_{\mathbb{N},b}(M^h(R))$, the abelian group $Hom_{\mathbb{F}_{\mathbb{N},b}(M^h(R))}(\underline{R^r},\underline{R^s})$ is the quotient of $Hom_{\mathbb{F}_{\mathbb{N}}(M^h(R))}(\underline{R^r},\underline{R^s})$. By definition, we have that $Hom_{\mathbb{F}_{\mathbb{N}}(M^h(R))}(\underline{R^r},\underline{R^s})$ is a collection of morphisms $\{f(j,i): \underline{R^r}(i)=R^r \longrightarrow \underline{R^s}(j)=R^s\}$. Since $R$ is unital, it is the same with a collection of $s \times r$ matrices in $R$. Then we have an idenfication:
	\begin{equation*}
		Hom_{\mathbb{F}_{\mathbb{N}}(M^h(R))}(\underline{R^r},\underline{R^s}) \cong M_{\infty}(M_{s,r}(R)), \ \{f(j,i)\}_{i,j \in \mathbb{N}} \mapsto F
	\end{equation*}
	
	with $F[j,i]=f(j,i)$ for all $i,j \in \mathbb{N}$. Furthermore, under this identification, it is straightforward to verify that $Hom_{\mathbb{F}_{\mathbb{N},b}(M^h(R))}(\underline{R^r},\underline{R^s})$ is the quotient of $ M_{\infty}(M_{s,r}(R))$ by  $M_{\infty}^{fin}(M_{s,r}(R))$, which is $\Sigma M_{s,r}(R)$. Therefore, we have an isomorphism of abelian groups, denoted by $\mathbb{F}_{r,s}$, such that $\mathbb{F}_{r,s}:Hom_{\mathbb{F}_{\mathbb{N},b}(M^h(R))}(\underline{R^r},\underline{R^s}) \stackrel{\cong}{\longrightarrow} \Sigma M_{s,r}(R)$. The three properties in the lemma can be easily shown by direct computations.
\end{proof}

Now we can construct the functor stated at the beginning of the section:

\begin{Definition}
	\
	\label{suspensionfunctor}
	
	Let $R$ be any unital ring with involution, we define $\Theta:M^h(\Sigma R) \longrightarrow \mathbb{F}_{\mathbb{N},b}(M^h(R))$ to the functor given by the followings:
	
	Object: For any $s \in \mathbb{N}$, define $\Theta((\Sigma R)^s)=\underline{R^s}$.
	
	Morphism: Let $r,s \in \mathbb{N}$, $f \in Hom_{M^h(\Sigma R)}((\Sigma R)^r,(\Sigma R)^{s})$, we can represent it by a matrix $M_f \in M_{s,r}(\Sigma R)$. Define $\Theta(f)=\mathbb{F}_{r,s}^{-1}\theta_{s,r}(M_f)$ for all $i,j \in \mathbb{N}$, where $\mathbb{F}_{r,s}$ is the isomorphism given in lemma \ref{identifysuspension}.
	
	Since for any $r,s,t \in \mathbb{N}$ and $f \in Mor((\Sigma R)^r,(\Sigma R)^{s}), g \in Mor((\Sigma R)^s,(\Sigma R)^{t})$
	
	$$
	\begin{aligned}
		\Theta(g \circ f)=\mathbb{F}_{r,t}^{-1}\theta_{t,r}(M_{g \circ f})&=\mathbb{F}_{r,t}^{-1}\theta_{t,r}(M_gM_f) \\ &=\mathbb{F}_{r,t}^{-1}(\theta_{t,s}(M_g)\theta_{s,r}(M_f)) \ (by \ Lemma \ \ref{matrixsuspension}) \\ &=\mathbb{F}_{s,t}^{-1}(\theta_{t,s}(M_g))\circ \mathbb{F}_{r,s}^{-1}(\theta_{s,r}(M_f)) \ (by \ Lemma \ \ref{identifysuspension}) \\ &=\Theta(g) \circ \Theta(f)
	\end{aligned}
	$$
	
	and $\Theta(id_{(\Sigma R)^s})=\mathbb{F}_{s,s}^{-1}(\theta_{s,s}(M_{id}))=\mathbb{F}_{s,s}^{-1}(\theta_{s,s}(I))=\mathbb{F}_{r,r}^{-1}(I)=id_{\underline{R^s}}$, the above definition gives a functor. Furthermore, it is easy to check that it is additive and it commutes with the involution by lemma \ref{matrixsuspension} and \ref{identifysuspension}, thus it is an additive functor between additive categories with involution.
\end{Definition}

The main property of $\Theta$ is that it is almost an equivalence of categories, as shown by the following lemma:

\begin{Lemma}
	\
	\label{thmTheta}
	
	$(1)$ Let $M=\sum\limits_{i \geq 0}M(i)$ be an object in $\mathbb{F}_{\mathbb{N},b}(M^h(R))$, then:
	
	$(i)$ If $S_M:=\{i \in \mathbb{N} \ | \ M(i) \neq 0\}$ is infinite, then there is an object $M' \in M^h(\Sigma R)$, such that $M$ is isomorphic to $\Theta(M')$.
	
	$(ii)$ If $S_M:=\{i \in \mathbb{N} \ | \ M(i) \neq 0\}$ is finite, denote $\iota_{\infty}:M^h(R) \longrightarrow \mathbb{F}_{\mathbb{N}}(M^h(R))$ to be the inclusion functor, then there is an object $M' \in M^h(R)$, such that $M \cong \iota_{\infty}(M')$ as objects in $\mathbb{F}_{\mathbb{N}}(M^h(R))$. 
	
	In particular, $M \oplus \underline{R}$ is always in the essential image of $\Theta$ for any object $M$.
	
	$(2)$ $\Theta$ is a faithful and full functor.
\end{Lemma}

\begin{proof}
	(1) We will give proof by dividing into two cases, depending on whether $S_M$ is finite set or not.
	
	(i) If $S_M$ is an infinite set:
	
	The idea of the proof is to reorder the terms in the sum to make the rank of $M(i)$ equal. More precisely, suppose $M(i)=R^{n(i)}$ for $i \in S_M$ with $n(i) \neq 0$. Denote $N(i)=\sum\limits_{\substack{k<i \\ k \in S_M}}n(k)$. Denote $\mathcal{P}_s^i:R^{n(i)} \longrightarrow R$ to be the projection map onto the $s$th component and $\mathcal{I}_s^i:R \longrightarrow R^{n(i)}$ be the inclusion map into the $s$th component.
	
	Let $T:M \longrightarrow \underline{R}$ be the morphism given by:
	\begin{equation*}
		\begin{aligned}
			& \qquad \qquad \qquad \qquad \qquad \quad \ \ T(j,i): M(i) \longrightarrow R \\
			& T(j,i)(x)=\begin{cases} \mathcal{P}_{j-N(i)+1}^ix & \text{If} \ i \in S_M \ \text{and} \ N(i) \leq j \leq N(i)+n(i)-1\\ \qquad 0 \ \ & \text{else} \end{cases} \\
		\end{aligned}
	\end{equation*}
	
	It is an isomorphism with the inverse given by:
	\begin{equation*}
		\begin{aligned}
			& \qquad \qquad \qquad \qquad \qquad \quad \ \ T^{-1}(j,i): R \longrightarrow M(j) \\
			& T^{-1}(j,i)(x)=\begin{cases} \mathcal{I}_{i-N(j)+1}^jx & \text{If} \ j \in S_M \ \text{and} \ N(j) \leq i \leq N(j)+n(j)-1\\ \qquad 0 & \text{else} \\ \end{cases} \\
		\end{aligned}
	\end{equation*}
	
	(ii) If $S_M$ is an finite set:
	
	Let $k=\max \{i \ | \ i \in S_M\}$ and denote $OM=\mathop{\oplus}\limits_{1 \leq i \leq k}M(i)$. Denote $\mathfrak{I}_i$ and $\mathfrak{P}_i$ to be the inclusion $M(i) \longrightarrow OM$ and projection $OM \longrightarrow M(i)$. Define $M'$ to be the object $OM$, $T:M \longrightarrow \iota_{\infty}(M')$ to be the morphism given by: 
	\begin{equation*}
		T(j,i)(x,y)=\begin{cases} \mathfrak{I}_{i}x & \text{If} \ j=0 \ \text{and} \ 1 \leq i \leq k \\ \ \, 0 & \text{else} \\ \end{cases}: M(i) \longrightarrow \iota_{\infty}(M')(j)
	\end{equation*}

	It is an isomorphism with the inverse given by:
	\begin{equation*}
		T^{-1}(j,i)(z)=\begin{cases} \mathfrak{P}_j(z) & \text{If} \ i=0 \ \text{and} \ 1 \leq j \leq k \\ \ \ 0 & \text{else} \\ \end{cases}:\iota_{\infty}(M')(i) \longrightarrow M(j)
	\end{equation*}
	
	(2) It is obvious from lemma \ref{matrixsuspension} and \ref{identifysuspension}.
\end{proof}

To prove that $\Theta$ induces an isomorphism in L-theory, we need the following Lemma, which essentially implies that any quadratic chain complex in $\mathbb{F}_{\mathbb{N}}(M^h(R))$ is cobordant to one in the essential image of $\Theta$:
\begin{Lemma}
	\label{addnull}
	For any $r,n\in \mathbb{Z}$, there is a $n$-dimensional quadratic chain complex $(C,\psi)$ in $M^h(\Sigma R)$, such that $C_r \neq 0$ and $\partial C_r \neq 0$.
\end{Lemma}
\begin{proof}
	Note that first that direct sum of $n$-dimensional quadratic chain complex is still a $n$-dimensional quadratic chain complex, also $C_{r+1} \neq 0$ implies $\partial C_r \neq 0$. Therefore, we only need to construct a $n$-dimensional quadraric chain complex $(C,\psi)$ with $C_r \neq 0$.
	
	Let $C$ be the chain complex with $C_r=\Sigma R$ and $C_k=0$ for all $k \neq r$. We can choose the quadratic structure $\psi$ to be $0$. Then $(C,\psi)$ is a $n$-dimensional quadratic chain complex in $M^h(\Sigma R)$ with $C_r \neq 0$, completing the proof of the lemma.
\end{proof}

\begin{Remark}
	\label{remaddnull}
	Note that by definition of $\Theta$, $\Theta(C)_r$ and $\Theta(\partial C_r)$ contain $\underline{R}$ as its subsummand.
\end{Remark}

\begin{Lemma}
	The functor $\Theta$ induces an isomorphism in $L$-theory.
\end{Lemma}

\begin{proof}
	We only need to prove that $\Theta_*$ is injective and surjective. Let us briefly explain the idea of the proof first, by Lemma \ref{addnull}, Remark \ref{remaddnull} and (1) in Lemma \ref{thmTheta}, any quadratic chain complex in $\mathbb{F}_{\mathbb{N}}(M^h(R))$ is cobordant to one in the essential image of $\Theta$. Since $\Theta$ is a faithful and full functor, we can then conclude that $\Theta$ induces an isomorphism in $L$-theory. We will present the details in the following paragraphs.
	
	Injectivity: Choose any $x \in L_n(M^h(\Sigma R))$ and represent it by a $n$-dimensional quadratic Poincare chain complex $(C,\psi)$. Suppose that $\Theta_*(x)=0$, we need to prove that $x=0$.
	
	By definition, $\Theta_*(x)=0$ implies that there is a $(n+1)$-dimensional Poincare quadratic pair $\big(f: \Theta(C) \longrightarrow D, (\delta\psi_D,\Theta(\psi))\big)$ in $\mathbb{F}_{\mathbb{N}}(M^h(R))$. Since $D$ is a finite chain complex, let $k \in \mathbb{N}$ be the smallest number such that $D_r=0$ for all $r>|k|$. By Lemma \ref{addnull} and Remark \ref{remaddnull}, there is a  $(n+1)$-dimensional quadratic chain complex $(E,\theta_E)$ in $M^h(\Sigma R)$, such that $D_r \oplus \Theta(E_r)$ contains $\underline{R}$ as its subsummand for all $|r| \leq k$. By (1) in Lemma \ref{thmTheta}, there is a finite chain complex $E'$ in $M^h(\Sigma R)$, such that $\Theta(E') \cong D \oplus \Theta(E)$. Denote $(i_E:\partial E \longrightarrow E,(0,\partial \theta_E))$ be the algebraic thickening of $E$, it is a $(n+1)$-dimensional Poincare quadratic pair in $M^h(\Sigma R)$. By (2) in Lemma \ref{thmTheta}, we have that $\big(\Theta(i_E):\Theta(\partial E) \longrightarrow \Theta(E),(0,\Theta(\partial \theta_E))\big)$ is a $(n+1)$-dimensional Poincare quadratic pair in $\mathbb{F}_{\mathbb{N}}(M^h(R))$. Then $\big(f \oplus \Theta(i_E): \Theta(C) \oplus \Theta(\partial E) \longrightarrow D \oplus \Theta(E), (\delta\psi_D \oplus 0, \Theta(\psi) \oplus \Theta(\partial\theta_E))\big)$  is a Poincare quadratic pair in $\mathbb{F}_{\mathbb{N}}(M^h(R))$. Since $D \oplus \Theta(E) \cong \Theta(E')$, by (2) in Lemma \ref{thmTheta}, there exist a chain map $f':C \oplus \partial E \longrightarrow E'$ and morphisms $\delta\psi_u:(E')^{n+1-u-*} \longrightarrow E_*'$ with $u \in \mathbb{N}$, such that $\big(f':C \oplus \partial E \longrightarrow E',(\delta\psi,\psi \oplus \partial \theta_E)\big)$ is a $(n+1)$-dimensional Poincare quadratic pair in $M^h(\Sigma R)$. By definition, this implies that $(C \oplus \partial E,\psi \oplus \partial \theta_E)$ is null-cobordant. Since $(\partial E,\partial \theta_E)$ is null-cobordant, we get that $(C,\psi)$ is null-cobordant and thus $x=0$.
	
	Surjectivity: Choose any $y \in L_n(\mathbb{F}_{\mathbb{N}}(M^h(R)))$, we need to prove that there is $x \in L_n(M^h(\Sigma R))$, such that $\Theta_*(x)=y$.
	
	Let $(C,\psi)$ be the $n$-dimensional Poincare quadratic chain complex in \\
	$\mathbb{F}_{\mathbb{N}}(M^h(R))$ representing $y$. Since $C$ is a finite chain complex, let $k \in \mathbb{N}$ be the smallest number such that $C_r=0$ for all $r>|k|$. By Lemma \ref{addnull} and Remark \ref{remaddnull}, there is a $(n+1)$-dimensional quadratic chain complex $(D,\varphi_D)$ in $M^h(\Sigma R)$, such that $C_r \oplus \Theta(\partial D)_r$ contains $\underline{R}$ as its subsummand for all $|r| \leq k$. By (1) in Lemma \ref{thmTheta}, there is a chain complex $C'$ in $M^h(\Sigma R)$, such that $\Theta(C') \cong C \oplus \Theta(\partial D)$. By (2) in Lemma \ref{thmTheta}, $\Theta$ is a faithful and full functor. Then there are morphisms $\psi'_u:(C')^{n-u-*} \longrightarrow C_*'$, such that $(C',\psi')$ is a $n$-dimensional Poincare quadratic chain complex in $M^h(\Sigma R)$ and $\Theta_*(C',\psi') \cong (C \oplus \Theta(\partial D),\psi \oplus \Theta(\partial\varphi_D))$. Now since $\Theta$ is a faithful and full functor, we have that $(\Theta(\partial D),\Theta(\partial\varphi_D))$ is null-cobordant. Therefore, we have that $(C \oplus \Theta(\partial D),\psi \oplus \Theta(\partial\varphi_D))$ represents $y$. Denote $x$ to be the element represented by $(C',\psi')$, we have $\Theta_*(x)=y$, proving that $\Theta$ is surjective.
	
\end{proof}


	\section{Proof of Theorem \ref{L theory transfer}}
	We will prove the main theorem of the article, Theorem \ref{L theory transfer}, in this section.

Before getting to the proof, we briefly recall the definition of the transfer map $\rho_{M,N}$ in Theorem \ref{L theory transfer}:

By (4) and (5) in the geometric setup, the map $r_{\Pi}:H \longrightarrow \Pi$ induces a left action of $H$ on $\mathbb{Z}\Pi$ by $h \cdot x=r_{\Pi}(h)x$. The induced action of  $G$ acts on $\mathop{\oplus}\limits_{gH \in G/H} \mathbb{Z}\Pi$.

Let $S$ be the splitting map of Shanneson. The homomorphism $\rho:\mathbb{Z}\Gamma \longrightarrow \Sigma\mathbb{Z}\Pi$ is given by (See construction \ref{transfer map}): 

(1) If $G/H$ is an infinite set, then the induced action gives rise to a group homomorphism: $G \longrightarrow M_{\infty}(\mathbb{Z}\Pi)$. Projecting the map onto $\Sigma \mathbb{Z}\Pi$ gives a group homomorphism: $\bar{\rho}:G \longrightarrow \Sigma \mathbb{Z}\Pi$. Then $\bar{\rho}$ can be descended to a homomorphism $\rho:\Gamma \longrightarrow \Sigma \mathbb{Z}\Pi$ and we extend linearly to get the homomorphism $\rho:\mathbb{Z}\Gamma \longrightarrow \Sigma\mathbb{Z}\Pi$.

(2) If $G/H$ is a finite set, then we define $\rho$ to be the trivial map: $\rho(x)=0$ for all $x \in \mathbb{Z}\Gamma$.

Then $\rho_{M,N}$ is given by $\rho_{M,N}:L_m(\mathbb{Z}\Gamma) \stackrel{\rho_*}{\longrightarrow} L_m(\Sigma\mathbb{Z}\Pi) \cong L_{m-1}^p(\mathbb{Z}\Pi) \stackrel{S}{\longrightarrow} L_{m-2}^{\textlangle -1 \textrangle}(\mathbb{Z}\pi)$ with $S$ being the algebraic splitting map in Theorem 17.2 of \cite{ranickilowerltheory}.

By the last paragraph of page 352 in \cite{Ranickisplit}, the algebraic splitting map agrees with the geometric splitting map. Therefore, we have: 

\begin{Theorem}
	\label{split}
	$S\big(\sigma^{\textlangle 0 \textrangle}(f|_{N' \times S^1},b|_{N' \times S^1})\big)=\sigma^{\textlangle -1 \textrangle}(f|_{N'},b|_{N'})$.
\end{Theorem}

It remains to prove that $\rho_*(\sigma(f,b))=\sigma^{\textlangle 0 \textrangle}(f|_{N' \times S^1},b|_{N' \times S^1})$. We will use Ranicki's description of surgery obstruction by chain complexes. We recall some definitions and theorems in Ranicki's book \cite{ranickiltheory} first:

\subsection{Simiplical descriptions of $L$-theory}
\begin{Definition}
	Let $\mathbb{A}$ be an additive category, and let $C,D$ be finite chain complexes in $\mathbb{A}$. Denote $Hom_{\mathbb{A}}(C,D)$ to be the following chain complex $(r \in \mathbb{Z})$ :
	\begin{equation*}
		Hom_{\mathbb{A}}(C,D)_r=\mathop{\oplus}_{q \in \mathbb{Z}}Hom_{\mathbb{A}}(C_q,D_{r+q})
	\end{equation*}
	\begin{equation*}
		d_{Hom_{\mathbb{A}}(C,D)}:Hom_{\mathbb{A}}(C,D)_r \longrightarrow Hom_{\mathbb{A}}(C,D)_{r-1}
	\end{equation*}
	\begin{equation*}
		f \in Hom_{\mathbb{A}}(C_q,D_{r+q}) \mapsto d_Df+(-1)^{r+q}fd_C
	\end{equation*}
\end{Definition}

\begin{Definition}
	\
	\label{chain functor}
	
	Let $\mathbb{A}$ be an additive category,  and denote $\mathbb{B}(\mathbb{A})$ to be the additive category of finite chain complex in $\mathbb{A}$ and chain maps.
	
	$(1)$ Denote $\iota:\mathbb{A} \longrightarrow \mathbb{B}(\mathbb{A})$ to be the canonical embedding: $\iota(A)_q=\begin{cases} A \ \text{if } q=0 \\ 0 \ \text{if } q \neq 0
	\end{cases}$

	$(2)$ For any contravariant additive functor $T:\mathbb{A} \longrightarrow \mathbb{B}(\mathbb{A})$, we can define an extension of it: $T:\mathbb{B}(\mathbb{A}) \longrightarrow \mathbb{B}(\mathbb{A})$. It is given by:
	\begin{gather*}
		(TC)_q =\mathop{\oplus}\limits_{x+y=q}	T(C_{-x})_{y} \\  
		d_{T(C)} = \mathop{\oplus}\limits_{x+y=q}\big(d_{T(C_{-x})} + (-1)^yT(d_C)\big): T(C)_q \longrightarrow T(C)_{q-1}
	\end{gather*}

	$(3)$ Denote $C^r=(TC)_{-r}$ with $d_{C^{-*}}=(-1)^rd_{TC}:C^r=TC_{-r} \longrightarrow C^{r+1}=TC_{-r-1}$.
\end{Definition}

\begin{Remark}
	For a ring $R$ with involution and $\mathbb{A}=M^h(R)$, we can define $T(P)=\iota(P^*)$ for $P \in M^h(R)$. Then $C^r=C_r^*$ and $d_{C^{-*}}=(-1)^rd_C^*:C^r \longrightarrow C^{r+1}$. Unless otherwise stated, we will take this chain complex to be the dual of a chain complex $C$ in $M^h(R)$ in the remaining parts of this article.
\end{Remark}

The following lemma gives a detailed description of how we extend the additive functor $T:\mathbb{A} \longrightarrow \mathbb{B}(\mathbb{A})$ in the definition above:

\begin{Lemma}
	\label{dualfun}
	The extension of $T$ can be chosen in a way such that the following morphism is a chain map for all  finite chain complexes $C$ and $D$ in $\mathbb{A}$:
	\begin{equation*}
		T:Hom_{\mathbb{A}}(C,D)_* \longrightarrow Hom_{\mathbb{A}}(TD,TC)_*
	\end{equation*}
\end{Lemma}

\begin{proof}
	Choose $r \in \mathbb{Z}$ and $f \in Hom_{\mathbb{A}}(C,D)_r$. For every $q \in \mathbb{Z}$,  let $f_q \in Hom_{\mathbb{A}}(C_q,D_{q+r})$ denote the corresponding component of $f$. We will define $Tf \in Hom_{\mathbb{A}}(TD,TC)_r$ below.
	
	Since
	\begin{equation*}
		Hom_{\mathbb{A}}(TD,TC)_r=\mathop{\oplus}\limits_{q' \in \mathbb{Z}}\mathop{\oplus}
		\limits_{s,s' \in \mathbb{Z}}Hom_{\mathbb{A}}((TD_{s})_{q'+s},(TC_{s'})_{q'+r+s'})
	\end{equation*}
	
	Then $Tf$ is given by specifying the components $Tf_{s,q'+s}^{s',q'+r+s'}:(TD_{s})_{q'+s} \longrightarrow (TC_{s'})_{q'+r+s'}$, which is given as follows:
	\begin{equation}
		\label{EqT}
		Tf_{s,q'+s}^{s',q'+r+s'}=\begin{cases} (-1)^{q's'}T(f_{s'})_{q'+s} & \text{if } s=s'+r \\ \ \ \ \ \ \ \ \ 0 & \text{else} \end{cases}
	\end{equation}
\end{proof}

\begin{Definition}[Chain Duality, Definition 1.1 in \cite{ranickiltheory}]
	
	A chain duality $(T,\mathfrak{D})$ on an additive category $\mathbb{A}$ is a contravariant
	additive functor $T:\mathbb{A} \longrightarrow \mathbb{B}(\mathbb{A})$ together with a natural transformation: 
	$$\mathfrak{D}: T^2 \longrightarrow \iota : \mathbb{A} \longrightarrow \mathbb{B}(\mathbb{A})$$
	
	such that for each object $A$ in $\mathbb{A}$, we have:
	
	$(1)$ $\mathfrak{D}(T(A)) \circ T(\mathfrak{D}(A)) = id : T(A) \longrightarrow T^3(A) \longrightarrow T(A)$.
	
	$(2)$ $\mathfrak{D}(A): T^2(A) \longrightarrow \iota(A)$ is a chain equivalence.
\end{Definition}

The chain duality gives a $\mathbb{Z}_2$-action on $Hom_{\mathbb{A}}(TC,C)$. The following lemma gives a detailed description of this action:

\begin{Lemma}
	\label{ExpressT}
	Let $\mathbb{A}$ be an additive category with chain duality $(T,\mathfrak{D})$ and $C,D$ be chain complexes in $\mathbb{A}$. Let $l \in \mathbb{Z}$ and $\psi \in Hom_{\mathbb{A}}(TC,D)_l$. Denote $T\psi$ to be the image of $\psi$ given by the following composition of maps:
	\begin{equation*}
		Hom_{\mathbb{A}}(TC,D) \stackrel{T}{\longrightarrow} Hom_{\mathbb{A}}(TD,T^2C) \stackrel{\mathfrak{D}(C)}{\longrightarrow} Hom_{\mathbb{A}}(TD,C)
	\end{equation*}
	
	Since $Hom_{\mathbb{A}}(TC,D)_l=\mathop{\oplus}\limits_{ \substack{q,r,s \in \mathbb{Z} \\ q+r+s=l}}Hom_{\mathbb{A}}((TC_q)_{-r},D_s)$, write $\psi_{q,r}^s$ for the corresponding component, similarly for $T\psi$. Then $T\psi_{q,r}^s:(TD_q)_{-r} \longrightarrow C_s$ is given by $(-1)^{(r+q)(r+s)}$ times the composition of the following maps:
	\begin{equation*}
		\begin{tikzcd}
			(TD_q)_{-r} \ar{rr}{T(\psi_{s,r}^q)_{-r}} & & T(T(C_s)_{-r})_{-r} \rar[hook]{\subset} & T^2(C_s)_0 \ar{rr}{\mathfrak{D}(C_s)_0} & & C_s
		\end{tikzcd}
	\end{equation*}
\end{Lemma}

\begin{proof}
	It follows from the equation \ref{EqT} by choosing $C=TC,D=D,r=l,s=q,q'=-r-q,s'=-s-r$ and considering the inclusion relations.
\end{proof}

\begin{Definition}[quasi quadratic, quasi symmetric]
	\
	
	Let $W$ be the following $\mathbb{Z}[\mathbb{Z}_2]$ chain complex:
	\begin{equation*}
		... \longrightarrow \mathbb{Z}[\mathbb{Z}_2] \stackrel{1-T}{\longrightarrow} \mathbb{Z}[\mathbb{Z}_2] \stackrel{1+T}{\longrightarrow} \mathbb{Z}[\mathbb{Z}_2] \stackrel{1-T}{\longrightarrow} \mathbb{Z}[\mathbb{Z}_2]
	\end{equation*}
	
	Let $\mathbb{A}$ be an additive category with involution $*$. A quasi quadratic (resp. symmetric) complex of dimension $n$ in $\mathbb{A}$ is a pair $(C,\psi)$, where $C$ is a finite chain complex in $\mathbb{A}$ and $\psi$ is an element of $(W \otimes_{\mathbb{Z}[\mathbb{Z}_2]}Hom_{\mathbb{A}}(TC,C))_n$ (resp. $Hom(W,Hom_{\mathbb{A}}(TC,C))_n)$.
\end{Definition}

\begin{Remark}
	We will denote $W_{\%}(C)$ to be $W \otimes_{\mathbb{Z}[\mathbb{Z}_2]}Hom_{\mathbb{A}}(TC,C) $ and $W^{\%}(C)$ to be $Hom(W,Hom_{\mathbb{A}}(TC,C))$.
\end{Remark}

Then we begin to recall some definitions of categories over complexes in Ranicki's book \cite{ranickiltheory}, which will be our main tools for the proof of Theorem \ref{L theory transfer}.

\begin{Definition}[Definiton 4.1 in \cite{ranickiltheory}]
	\
	
	Let $\mathbb{A}$ be an additive category and $K$ be an ordered simplicial complex.
	
	$(1)$ An object $M$ in $\mathbb{A}$ is $K$-based if it is expressed as a direct sum
	$M= \sum\limits_{\sigma \in K}M(\sigma)$ of objects $M(\sigma)$ in $\mathbb{A}$, s.t. $\{\sigma \in K|\ M(\sigma) \neq 0\}$ is finite. A morphism $f:M \longrightarrow N$ of $K$-based objects is a collection of morphisms in $\mathbb{A}$:
	\begin{equation*}
		f=\{f(\tau,\sigma):M(\sigma) \longrightarrow N(\tau)| \ \sigma,\tau \in K\}
	\end{equation*}
		
	$(2)$ Denote $\mathbb{A}_*(K)$ to be the additive category of $K$-based objects $M$ in $\mathbb{A}$, with
	morphisms $f: M \longrightarrow N$, such that $f(\tau,\sigma)=0$ unless $\sigma \leq \tau$.
		
	$(3)$ Denote $\mathbb{A}^*(K)$ to be the additive category of $K$-based objects $M$ in $\mathbb{A}$, with
	morphisms $f: M \longrightarrow N$, such that $f(\tau,\sigma)=0$ unless $\sigma \geq \tau$.
\end{Definition}

Regard a simplicial complex $K$ as a category with one object for each simplex $\sigma \in K$ and one morphism for each face inclusion $\sigma \leq \tau$, we have the following definition:

\begin{Definition}
	 Let $\mathbb{A}$ be an additive category. Denote $\mathbb{A}_*[K]$ (resp. $\mathbb{A}^*[K]$) to be the additive category with objects the covariant (resp. contravariant) functors
	 \begin{equation*}
	 	M:K \longrightarrow \mathbb{A}; \sigma \mapsto M[\sigma]
	 \end{equation*}
	such that $\{\sigma \in K| \ M[\sigma] \neq 0\}$ is finite. The morphisms are the natural transformations of such functors.
\end{Definition}

The category $\mathbb{A}_*(K)$ $(resp. \ \mathbb{A}^*(K))$ and $\mathbb{A}_*[K]$ $(resp. \ \mathbb{A}^*[K])$ are related by the following covariant functor:
\begin{equation*}
	\begin{aligned}
		& \mathbb{A}_*(K) \longrightarrow \mathbb{A}_*[K]; M \mapsto [M], [M][\sigma]=\mathop{\oplus}\limits_{\tau \geq \sigma}M(\tau) \\
		& f \in Hom(M,N) \mapsto [f] \in Hom([M],[N]) \\
		& [f][\sigma]=\sum\limits_{\tau' \geq \sigma} \sum \limits_{\tau \geq \sigma} f(\tau',\tau):[M][\sigma] \longrightarrow [N][\sigma] \\
		& \mathbb{A}^*(K) \longrightarrow \mathbb{A}^*[K]; M \mapsto [M], [M][\sigma]=\mathop{\oplus}\limits_{\tau \leq \sigma}M(\tau) \\
		& f \in Hom(M,N) \mapsto [f] \in Hom([M],[N]) \\
		& [f][\sigma]=\sum\limits_{\tau' \geq \sigma} \sum \limits_{\tau \geq \sigma} f(\tau',\tau):[M][\sigma] \longrightarrow [N][\sigma] \\
	\end{aligned}
\end{equation*}

\begin{Definition}
	Let $K$ be a locally finite and ordered simplicial complex. Then for each simplex $\sigma \in K$, the set
	$K^*(\sigma)=\{\tau \in K| \ \tau >\sigma, |\tau|=|\sigma|+1\}$ is finite. For every $\tau \in K^*(\sigma)$, denote $n_{\sigma}^{\tau} \in \mathbb{N}$ to be the unique number such that $\partial_{n_{\sigma}^{\tau}} \ \tau=\sigma$.
	
	For every simplex $\sigma \in K$, denote $K_*(\sigma)$ to be the set $\{\tau \in K| \ \tau <\sigma, |\tau|=|\sigma|-1\}$.
\end{Definition}

In order to make clear of the morphisms used later, we introduce the following notation:

\begin{Definition}
	\
	
	Let $R$ be any ring and $I$ be a set. Let $\{A_i\}_{i \in I}$ be a family of $R$ modules and $B$ be a $R$ module. We denote $\mathop{\oplus}\limits_{i \in I}x_i \in \mathop{\oplus}\limits_{i \in I}A_i$ to be the element with $x_i$ on the $A_i$ component, then $x_i \neq 0$ for at most finitely many $i$. For every $i \in I$, let $f_i:A_i \longrightarrow B$, $g_i: B \longrightarrow A_i$ be morphisms of $R$ modules, such that for every $x \in B$, $g_i(x) \neq 0$ for at most finitely many $i$. Then we denote the following morphisms of $R$ modules by $\mathop{\boxplus}\limits_{i \in I}f_i$ and $\mathop{\oplus}\limits_{i \in I}g_i$ respectively:
	\begin{equation*}
		\begin{aligned}
			& \mathop{\boxplus}\limits_{i \in I}f_i: \mathop{\oplus}\limits_{i \in I}A_i \longrightarrow B, \ \mathop{\oplus}\limits_{i \in I}x_i \mapsto \sum\limits_{i \in I}f_i(x_i) \\
			& \quad \mathop{\oplus}\limits_{i \in I}g_i: B \longrightarrow \mathop{\oplus}\limits_{i \in I}A_i, \ x \mapsto \mathop{\oplus}\limits_{i \in I}g_i(x) \\
		\end{aligned}
	\end{equation*}
\end{Definition}

If there is an involution $*$ on the additive category $\mathbb{A}$, then for any locally finite ordered simplicial complex $K$, there is a chain duality on the additive category $\mathbb{A}_*(K)$ by Proposition 5.1 in Ranicki's book \cite{ranickiltheory}:

\begin{Theorem}[Proposition 5.1 in \cite{ranickiltheory}]
	\
	\label{chaindualityonAK1}
	
	Let $K$ be a locally finite ordered simpicial complex and $\mathbb{A}$ be an additive category with involution. Then there is a chain duality $(T,\mathfrak{D})$ on the additive category $\mathbb{A}_*(K)$ given by:
	
	$(1)$ For any object $M \in \mathbb{A}_*(K)$, $TM$ is the following chain complex: 
	\begin{equation*}
		TM_r(\sigma)=\begin{cases} [M][\sigma]^*=(\mathop{\oplus}\limits_{\kappa \geq \sigma}M(\kappa))^* & \text{If} \ r=-|\sigma| \\ \qquad \qquad \ 0 & \text{else} \\ \end{cases}
	\end{equation*}
	\begin{equation*}
		d_{TM}(\tau,\sigma)=\begin{cases} (-1)^{n_{\sigma}^{\tau}} i_{\tau\sigma}^* & \text{If} \ \tau \in K^*(\sigma)  \ \text{and} \ r=-|\sigma|\\ \quad \quad 0 & \text{else} \\ \end{cases} :TM_r(\sigma) \longrightarrow TM_{r-1}(\tau)
	\end{equation*}

	where $i_{\tau\sigma}:\mathop{\oplus}\limits_{\kappa \geq \tau}M(\kappa) \longrightarrow \mathop{\oplus}\limits_{\kappa \geq \sigma}M(\kappa)$ is the natural inclusion map for $\tau \geq \sigma$.
	
	$(2)$ For any morphism $f:M \longrightarrow N$, $Tf$ is the following chain map:
	\begin{equation*}
		\begin{aligned}
			Tf_r:TN_r \longrightarrow TM_r \qquad \qquad \qquad \qquad \qquad \qquad & \\
			Tf_r(\tau,\sigma)=\begin{cases}	\mathop{\boxplus}\limits_{\kappa \geq \sigma}\mathop{\oplus}\limits_{\kappa' \geq \tau}f(\kappa,\kappa')^*	& \text{If }\tau=\sigma \text{ and } r=-|\sigma| \\ \qquad \quad 0 & \text{else} \end{cases}: TN_r(\sigma) \longrightarrow TM_r(\tau) & \\
		\end{aligned}
	\end{equation*}
	
	Moreover, $\mathfrak{D}:T^2 \longrightarrow \iota$ is given as follows:
	
	For every object $M$, in order to define $\mathfrak{D}(M):T^2M \longrightarrow \iota(M)$, we only need to give the morphism on the $0$ degree: $\mathfrak{D}(M)_0:(T^2M)_0 \longrightarrow M$.
	
	Since for every $\sigma \in K$, we have:
	\begin{equation*}
		\begin{aligned}
			(T^2M)_0(\sigma)=\mathop{\oplus}\limits_{x \in \mathbb{Z}}T(TC_{x})_x(\sigma)=(\mathop{\oplus}\limits_{\kappa \geq \sigma}TM_{-|\sigma|}(\kappa))^*=(TM_{-|\sigma|}(\sigma))^*=\mathop{\oplus}\limits_{\kappa \geq \sigma}M(\kappa)
		\end{aligned}
	\end{equation*}
	
	We define $\mathfrak{D}(M)_0:(T^2M)_0 \longrightarrow M$ to be the following morphism in $\mathbb{A}_*(K)$:
	\begin{equation*}
		\begin{aligned}
			\mathfrak{D}(M)_0(\tau,\sigma):(T^2M)_0(\sigma)=\mathop{\oplus}\limits_{\kappa \geq \sigma}M(\kappa) \longrightarrow M(\tau) \\
			\mathfrak{D}(M)_0(\tau,\sigma)=\begin{cases} (-1)^{|\sigma|}p_{\sigma,\tau} & \text{If} \ \sigma \leq \tau \\ \quad \ \ \ 0 & \text{else} \end{cases}
		\end{aligned}
	\end{equation*}

	Where $p_{\sigma,\tau}$ is the projection map for $\sigma \leq \tau$.
\end{Theorem}

\begin{proof}
	The basic proofs are contained in \cite{ranickiltheory}. It only leaves us to check the expression of $\mathfrak{D}$.
	
	By the proof of Proposition 5.1 in \cite{ranickiltheory}, for every object $M$ in $\mathbb{A}_*(K)$, the morphism $\mathfrak{D}(M):T^2M \longrightarrow \iota(M)$ is given by $Id \in Hom_{\mathbb{A}_*(K)}(TM,TM)_0$ under the isomorphism $T_0$ given by the following commutative diagram:
	
	\begin{small}
		\begin{tikzcd}
			Hom_{\mathbb{A}_*(K)}(TM,TM)_0 \rar{T_0}\dar{\cong} & 	Hom_{\mathbb{A}_*(K)}(T^2M,\iota(M))_0 \dar{\cong} \\
			\mathop{\oplus}\limits_{\sigma \in K}\mathop{\oplus}\limits_{\kappa,\mu \geq \sigma}Hom_{\mathbb{A}}(M(\kappa)^*,TM_{-|\sigma|}(\mu)) \rar{(-1)^{|\sigma|}*}\dar{=} & \mathop{\oplus}\limits_{\sigma \in K}\mathop{\oplus}\limits_{\kappa,\mu \geq \sigma}Hom_{\mathbb{A}}(TM_{-|\sigma|}(\mu)^*,M(\kappa)) \dar{=} \\
			\mathop{\oplus}\limits_{\sigma \in K}\mathop{\oplus}\limits_{\kappa \geq \sigma}Hom_{\mathbb{A}}(M(\kappa)^*,TM_{-|\sigma|}(\sigma)) \rar{(-1)^{|\sigma|}*} & \mathop{\oplus}\limits_{\sigma \in K}\mathop{\oplus}\limits_{\kappa \geq \sigma}Hom_{\mathbb{A}}(TM_{-|\sigma|}(\sigma)^*,M(\kappa)) \\
		\end{tikzcd}
	\end{small}
	
	
	Where the upper vertical isomorphisms are given by looking at the morphism componentwise.
	
	For every $\sigma \in K,\kappa \in K$, if $\kappa \geq \sigma$, denote $i_{\kappa,\sigma}:M(\kappa)^* \longrightarrow TM_{-|\sigma|}(\sigma)=(\mathop{\oplus}\limits_{\kappa \geq \sigma}M(\kappa))^*$ to be the inclusion map.
	
	For every $\sigma \in K,\mu \in K$, $Id$ gives a morphism $TM_{-|\sigma|}(\sigma)=(\mathop{\oplus}\limits_{\kappa \geq \sigma}M(\kappa))^* \longrightarrow TM_{-|\sigma|}(\mu)$, which is $id$ if $\mu=\sigma$ and 0 else. Thus, the identity corresponds to $\mathop{\oplus}\limits_{\sigma \in K}\mathop{\oplus}\limits_{\kappa \geq \sigma}i_{\kappa,\sigma} \in \mathop{\oplus}\limits_{\sigma \in K}\mathop{\oplus}\limits_{\kappa \geq \sigma}Hom_{\mathbb{A}}(M(\kappa)^*,TM_{-|\sigma|}(\sigma))$. Since $i_{\kappa,\sigma}^*=p_{\sigma,\kappa}$, we have that $(-1)^{|\sigma|}$ times the dual of $\mathop{\oplus}\limits_{\sigma \in K}\mathop{\oplus}\limits_{\kappa \geq \sigma}i_{\kappa,\sigma} \in \mathop{\oplus}\limits_{\sigma \in K}\mathop{\oplus}\limits_{\kappa \geq \sigma}Hom_{\mathbb{A}}(M(\kappa)^*,TM_{-|\sigma|}(\sigma))$ is $\mathop{\oplus}\limits_{\sigma \in K}(-1)^{|\sigma|}(\mathop{\oplus}\limits_{\kappa \geq \sigma}p_{\sigma,\kappa}) \in \mathop{\oplus}\limits_{\sigma \in K}\mathop{\oplus}\limits_{\kappa \geq \sigma} Hom_{\mathbb{A}}(TM_{-|\sigma|}(\sigma)^*,M(\kappa))$. Reinterpreting it as a morphism in $\mathbb{A}_*(K)$, it is the morphism $\mathfrak{D}(M)_0$ stated in the theorem.
\end{proof}

\begin{Remark}
	\label{chaindualityonAK2}
	
	Similarly, there is a chain duality $(T,\mathfrak{D})$ on $\mathbb{A}^*(K)$ given by:
	
	$(1)$ For any object $M \in \mathbb{A}^*(K)$, $TM$ is the following chain complex: 
	
	$$TM_r(\sigma)=\begin{cases}
		[M][\sigma]^*=(\mathop{\oplus}\limits_{\kappa \leq \sigma}M(\kappa))^* & \text{If} \ r=|\sigma| \\ \qquad \qquad \  0 & \text{else} \\
	\end{cases}$$
	
	$$d_{TM}(\tau,\sigma)=\begin{cases}
		(-1)^{n_{\tau}^{\sigma}} i_{\tau\sigma}^* & \text{If} \ \sigma \in K^*(\tau)  \ \text{and} \ r=|\sigma|\\
		\quad \quad 0 & \text{else} \\
	\end{cases} :TM_r(\sigma) \longrightarrow TM_{r-1}(\tau)$$
	where $i_{\tau\sigma}:\mathop{\oplus}\limits_{\kappa \leq \tau}M(\kappa) \longrightarrow \mathop{\oplus}\limits_{\kappa \leq \sigma}M(\kappa)$ is the natural inclusion map for $\tau \leq \sigma$.
	
	$(2)$ For any morphism $f:M \longrightarrow N$, $Tf$ is the following chain map:
	\begin{equation*}
		\begin{aligned}
			& \qquad \qquad \qquad \qquad \qquad \qquad Tf_r:TN_r \longrightarrow TM_r \\
			& Tf_r(\tau,\sigma)=\begin{cases}	\mathop{\boxplus}\limits_{\kappa \geq \sigma}\mathop{\oplus}\limits_{\kappa' \geq \tau}f(\kappa,\kappa')^*	& \text{If }\tau=\sigma \text{ and } r=|\sigma| \\  \qquad \quad 0 & else \end{cases}: TN_r(\sigma) \longrightarrow TM_r(\tau)
		\end{aligned}
	\end{equation*}
	
	Moreover, the morphism $\mathfrak{D}(M)_0:(T^2M)_0 \longrightarrow M$ is given by:
	
	\begin{equation*}
		\begin{aligned}
			\mathfrak{D}(M)_0(\tau,\sigma):(T^2M)_0(\sigma)=\mathop{\oplus}\limits_{\kappa \leq \sigma}M(\kappa) \longrightarrow M(\tau) \\
			\mathfrak{D}(M)_0(\tau,\sigma)=\begin{cases} (-1)^{|\sigma|}p_{\sigma,\tau} & \text{If} \ \tau \leq \sigma \\ \quad \ \ \ 0 & \text{else} \end{cases}
		\end{aligned}
	\end{equation*}
	
	Where $p_{\sigma,\tau}$ is the projection map for $\tau \leq \sigma$.
\end{Remark}

\begin{Remark}
	Let $C$ be a finite chain complex in $\mathbb{A}_*(K)$. Fix $n \in \mathbb{Z}$, by Definition \ref{chain functor}, $C^{n-*}$ is a finite chain complex in $\mathbb{A}_*(K)$ given by:
	
	$$C^{n-r}(\sigma)=TC_{r-n}(\sigma)=T(C_{n-r-|\sigma|})_{-|\sigma|}=\mathop{\oplus}\limits_{\kappa \geq \sigma}C_{n-r-|\sigma|}(\kappa)^*$$
	
\end{Remark}

If we choose $\mathbb{A}=M^h(R)$ for some ring $R$ with involution, then we have the following definition of the assembly functor given by Ranicki \cite{ranickiltheory}:
\begin{Definition}[Assembly functor]
	\
	\label{assemblyfun}
	
	Let $K$ be a locally finite ordered simpicial complex. Let $\widetilde{K}$ be a Galois covering of $K$ with transformation group $G$ and denote $\mathbf{p}$ to be the covering map. The assembly functor is the following functor:
	\begin{equation*}
		A_{\mathbf{p}}:M^h(R)_*(K) \longrightarrow M^h(RG), \ M \mapsto M(\widetilde{K}):=\mathop{\oplus}\limits_{\tilde{\sigma} \in \widetilde{K}} M(\mathbf{p}\tilde{\sigma})
	\end{equation*}
	\begin{equation*}
		f \in Mor(M,N) \mapsto f(\widetilde{K}):=\mathop{\boxplus}\limits_{\tilde{\sigma} \in \widetilde{K}}\mathop{\oplus}\limits_{\tilde{\tau} \in \widetilde{K}}f(\tilde{\tau},\tilde{\sigma}): M(\widetilde{K}) \longrightarrow N(\widetilde{K})
	\end{equation*}
	
	Where $f(\tilde{\tau},\tilde{\sigma})$ is defined by:
	\begin{equation*}
		f(\tilde{\tau},\tilde{\sigma})=\begin{cases}	f(\mathbf{p}\tilde{\tau},\mathbf{p}\tilde{\sigma}) & \text{If } \tilde{\sigma} \leq \tilde{\tau} \\ \ \ \ \ \ 0 & \text{else} \\ \end{cases}: M(\mathbf{p}\tilde{\sigma}) \longrightarrow N(\mathbf{p}\tilde{\tau})
	\end{equation*}
	
	It can be extended to the corresponding category of chain complexes by applying the functor to every degree of the chain complex.
	
	If $\widetilde{K}$ is taken to be the universal covering of $K$, then the assembly functor is called the universal assembly functor.
\end{Definition}

The assembly map maps Poincare quadratic chain complexes in $M^h(R)_*(K)$ to Poincare quadratic chain complexes in $M^h(RG)$, as shown by the following theorem:

\begin{Theorem}[Page 94 in \cite{ranickiltheory}]
	\
	\label{assemblydual}
	
	Let $K,\widetilde{K},\mathbf{p}$ be as above. For every vertex $\tilde{x} \in \widetilde{K}$ and simplex $\tilde{\sigma} \in \widetilde{K}$ such that $\tilde{x} \leq \tilde{\sigma}$, let $\Upsilon_{\tilde{x},\tilde{\sigma}}^r:C_r(\mathbf{p}\tilde{\sigma})^* \longrightarrow C^r(\mathbf{p}\tilde{x})=\mathop{\oplus}\limits_{\kappa \geq \mathbf{p}\tilde{x}}C_r(\kappa)^*$ be the inclusion map. Then the chain map $\Upsilon:\mathop{\oplus}\limits_{\tilde{\sigma} \in \widetilde{K}} C_r(\mathbf{p}\tilde{\sigma})^*=C_r(\widetilde{K})^* \longrightarrow C^r(\widetilde{K})$ given by
	\begin{equation*}
		z_{\tilde{\sigma}} \in C_r(\mathbf{p}\tilde{\sigma})^* \mapsto \mathop{\oplus} \limits_{\tilde{x} \leq \tilde{\sigma}}\Upsilon_{\tilde{x},\tilde{\sigma}}^rz_{\tilde{\sigma}} \in \mathop{\oplus} \limits_{\tilde{x} \leq \tilde{\sigma}}C^r(\mathbf{p}\tilde{x}) \subset C^r(\widetilde{K})
	\end{equation*}

	is a chain homotopy equivalence of $RG$ modules.
\end{Theorem}

\begin{Remark}
	Under the same assumptions as above, after some computations, it can be shown that if $\psi$ is a Poincare quadratic structure of a chain complex $C$ in $M^h(R)_*(K)$, then $\psi(\widetilde{K})\Upsilon$ is a quadratic structure of $C(\widetilde{K})$ (See page 100 in Ranicki's book \cite{ranickiltheory}). The above theorem guarantees that this quadratic structure is Poincare and thus the assembly gives a homomorphism in $L$ theory:
	\begin{equation*}
		A_{\mathbf{p}}:L_*(M^h(R)_*(K)) \longrightarrow L_*(RG)
	\end{equation*}
\end{Remark}

It is also possible to relate the term $L_*(M^h(R)_*(K))$ with some generalized homology theories and the relation is contained in Proposition 13.7 of \cite{ranickiltheory}. We will set up the basic prepartions first and then make a statement about the Proposition.

We first extend the categories $M^h(R)_*(K),M^h(R)^*(K)$ to the case of simplicial pairs:
\begin{Definition}
	Let $K$ be a locally finite ordered simplicial complex and $L$ be its subcomplex, then:
	
	$(1)$ Define $M^h(R)_*(K,L)$ to be the full subcategory of $M^h(R)_*(K)$ with objects $M \in M^h(R)_*(K)$ such that $M(\sigma)=0$ for all $\sigma \in L$.
	
	$(2)$ Define $M^h(R)^*(K,L)$ to be the full subcategory of $M^h(R)^*(K)$ with objects $M \in M^h(R)^*(K)$ such that $M(\sigma)=0$ for all $\sigma \in L$.
\end{Definition}

Let $K_c$ be a finite ordered simplicial complex. Let $l \in \mathbb{Z}$ be a sufficiently large number, such that we can embed $K_c$ simplically and order-preservingly in $\partial\Delta^{l+1}$. Denote $\Sigma^l$ to be the simplicial complex with one $k$-simplex $\sigma^*$ for each $(l-k)$-simplex $\sigma$ in $\partial\Delta^{l+1}$, with $\sigma^* \leq \tau^*$ if and only if $\sigma \geq \tau$ in $\partial\Delta^{l+1}$. It inherits an order from the following simplicial isomorphism:
\begin{equation*}
	\Sigma^l \longrightarrow \partial\Delta^{l+1}; \ \sigma^* \mapsto \{0,1,2,...,l+1\} \backslash \sigma
\end{equation*}
 
Then $\Sigma^l$ can be viewed as the dual cell decomposition of the barycentric subdivision of $\partial\Delta^{l+1}$. For any subcomplex $V \subset \partial\Delta^{l+1}$, denote $\underline{V}=\mathop{\cup}\limits_{\sigma \notin V}\sigma^*$. It is a subcomplex of $\Sigma^l$. 

For any $\sigma \in \partial \Delta^{l+1}$, denote $J_{\sigma}:\{0,1,2,...,l-|\sigma|\} \longrightarrow \{0,1,2,...,l\}$ to be the map that maps $i \in \{0,1,2,...,l-|\sigma|\}$ to the $i+1$ th element of $\{0,1,2,...,l+1\} \backslash \sigma$. Let $J_{\sigma}^{all}=\sum\limits_{i=0}^{l-|\sigma|}J_{\sigma}(i)$, then we have the following lemma describing the order on $\Sigma^l$:
\begin{Lemma}
	\label{Dualincidnum}
	For any $\tau \in (\partial\Delta^{l+1})^*(\sigma)$, we have $n_{\sigma}^{\tau}+n_{\tau^*}^{\sigma^*}=J_{\sigma}^{all}-J_{\tau}^{all}$.
\end{Lemma}

\begin{proof}
	Let $S_{\sigma}=\{0,1,2,...,l+1\} \backslash \sigma,S_{\tau}=\{0,1,2,...,\} \backslash \tau$ and suppose that $n_{\tau^*}^{\sigma^*}=i$. By definition, we have $S_{\sigma} \backslash \{J_{\sigma}(i)\}=S_{\tau}$. Then the unique vertex that is in $\tau$ but not in $\sigma$ is $J_{\sigma}(i)$. Since the set of vertices of $\tau$ is $\{0,1,2,...,l+1\} \backslash S_{\tau}=\{0,1,2,...,l+1\} \backslash S_{\sigma} \amalg \{J_{\sigma}(i)\}$, we can deduce by definition that $n_{\sigma}^{\tau}=J_{\sigma}(i)-i$. Thus, we have $n_{\sigma}^{\tau}+n_{\tau^*}^{\sigma^*}=J_{\sigma}(i)$. Notice that since $S_{\sigma} \backslash \{J_{\sigma}(i)\}=S_{\tau}$, we have that $J_{\sigma}^{all}-J_{\tau}^{all}=J_{\sigma}(i)$. Then we can conclude that the equation stated in the Lemma holds.
\end{proof}

Let $R$ be a ring with involution. Denote $\underline{L}(R)$ to be the $\Omega$-spectrum of the category $M^h(R)$, as given by the $\Delta$ sets $\mathbb{L}_k(M^h(R))$ in Definition 13.2 of \cite{ranickiltheory}. Then we can describe the generalized homology $H_k(K_c,\underline{L}(R))$ simplicially:

\begin{Theorem}[Proposition 12.4 in \cite{ranickiltheory}]
	\
	\label{Lhomologyandco}
	
	For every $k \in \mathbb{Z}$ and $L_c \subset K_c$ subcomplex, there is an identification:
	\begin{equation*}
		H_k(K_c,L_c;\underline{L}(R))=[\underline{L_c},\underline{K_c};\mathbb{L}_{k-l}(M^h(R)),\emptyset]=H^{l-k}(\underline{L_c},\underline{K_c};\underline{L}(R))
	\end{equation*}

	In particular,
	\begin{equation*}
		H_k(K_c;\underline{L}(R))=[\Sigma^l,\underline{K_c};\mathbb{L}_{k-l}(R),\emptyset]=H^{l-k}(\Sigma^l,\underline{K_c};\underline{L}(R))
	\end{equation*}
	
	Moreover, for every complex triple $J_c \subset L_c \subset K_c$, the following diagram commutes:
	\begin{equation*}
		\begin{tikzcd}
			H_k(K_c,L_c;\underline{L}(R)) \dar{\partial}\rar{\cong} & H^{l-k}(\underline{L_c},\underline{K_c};\underline{L}(R))\dar{\delta} \\
			H_{k-1}(L_c,J_c;\underline{L}(R)) \rar{\cong} & H^{l-k+1}(\underline{J_c},\underline{L_c};\underline{L}(R)) \\
		\end{tikzcd}
	\end{equation*}
\end{Theorem}

We can also identify the above cohomology theory with the $L$-group of certain categories, as shown by the following theorem:

\begin{Theorem}[Proposition 13.7 in \cite{ranickiltheory}]
	\
	\label{Lhomology}
	
	Let $K$ be any locally finite simplicial complex and $L$ be its subcomplex. Then for every $k \in \mathbb{Z}$, there is an identification:
	\begin{equation*}
		H^{-k}(K,L;\underline{L}(R))=L_{k}(M^h(R)^*(K,L))
	\end{equation*}

	It is given as follows:
	
	Let $g:(K,L) \longrightarrow (\mathbb{L}_{k}(M^h(R)),*)$ be any $\Delta$-map. For any simplex $\sigma$, let $s_{\sigma}:\sigma \longrightarrow \Delta^{|\sigma|}$ be the unique simplicial, order-preserving isomorphism. Denote $g(\sigma)=(C_{\sigma},\psi_{\sigma})$. Then the image of $g$ under identification is the following Poincare quadratic chain complex $(D,\theta)$:
	\begin{equation*}
		D(\sigma)=C_{\sigma}(\Delta^{|\sigma|})
	\end{equation*} 
	
	For any simplices $\tau \leq \sigma$, the inclusion map $i_{\tau,\sigma}$ induces a homomorphism $C_{\tau} \longrightarrow C_{\sigma}$ from a chain complex in $M^h(R)^*(\Delta^{|\tau|})$ to a chain complex in $M^h(R)^*(\Delta^{|\sigma|})$ via the inclusion map $s_{\sigma}i_{\tau,\sigma}s_{\tau}^{-1}$. Therefore, it is possible to identify $C_{\tau}(\Delta^{|\tau|})$ with $C_{\sigma}(s_{\sigma}(\tau))$. The boundary map of the chain complex and the quadratic structure are then given by:
	\begin{equation*}
		d_D(\tau,\sigma)=d_{C_{\sigma}}(s_{\sigma}(\tau),\Delta^{|\sigma|}):D_*(\sigma)=(C_{\sigma})_*(\Delta^{|\sigma|}) \longrightarrow D_{*-1}(\tau)=(C_{\sigma})_{*-1}(s_{\sigma}(\tau))
	\end{equation*}
	\begin{equation*}
		\psi_s(\tau,\sigma):D^{k-l-s-*}(\sigma)=C_{\sigma}^{k-l-s-*}(\Delta^{|\sigma|}) \longrightarrow D_{*}(\tau)=(C_{\sigma})_{*}(s_{\sigma}(\tau))
	\end{equation*}
	\begin{equation*}
		\psi_s(\tau,\sigma)=\psi_{\sigma}^s(s_{\sigma}(\tau),\Delta^{|\sigma|})
	\end{equation*}

	Conversely, given any $k$-dimensional Poincare quadratic chain complex $(D,\theta)$, restricted on $\sigma$ gives a $k$-dimensional Poincare quadratic chain complex $(D_{\sigma},\theta_{\sigma})$ in $M^h(R)^*(\sigma)$. Then we can define $g(\sigma)=(s_{\sigma})_*(D_{\sigma},\theta_{\sigma})$.
\end{Theorem}

\begin{Theorem}
	\
	\label{localdual}
	
	There is a one-to-one correspondance between $(k-l)$-dimensional Poincare quadratic chain complexes in  $M^h(R)^*(\underline{L_c},\underline{K_c})$ and $k$-dimensional Poincare quadratic chain complexes in $M^h(R)_*(K_c)$ with the chain complex in $M^h(R)_*(K_c,L_c)$. Similarly for pairs. We will call this correspondence local dual.
	
	In particular, when taking $L_c=\emptyset$, we get an identification:
	\begin{equation*}
		L_k(M^h(R)_*(K_c)) = L_{k-l}(M^h(R)^*(\Sigma^l,\underline{K_c}))=H_k(K_c;\underline{L}(R))
	\end{equation*}
\end{Theorem}

\begin{proof}
	Let $(C,\psi)$ be any $k$-dimensional Poincare quadratic chain complex in $M^h(R)_*(K_c)$, such that $C$ is a chain complex in $M^h(R)_*(K_c,L_c)$. We define its local dual $(\widecheck{C},\widecheck{\psi})$ as follows:
	
	For all $u \in \mathbb{N},r \in \mathbb{Z}$:
	\begin{equation}
		\label{Eqlocaldualchaincpx}
		\widecheck{C}_r(\sigma^*)=C_r(\sigma) \text{ for } \sigma \in K_c \backslash L_c, \	\widecheck{C}_r(\sigma^*)=0 \text{ for } \sigma \notin K_c
	\end{equation}
	\begin{equation}
		\label{Eqlocaldualboundarymap}
		d_{\widecheck{C},r}(\tau^*,\sigma^*)=\begin{cases} d_{C,r}(\tau,\sigma) & \text{If } \sigma,\tau \in K_c \backslash L_c \\ \quad \ \ 0 & \text{else}	\end{cases}:\widecheck{C}_r(\sigma^*) \longrightarrow \widecheck{C}_{r-1}(\tau^*)
	\end{equation}
	\begin{equation*}
		\widecheck{\psi}_u^r(\tau^*,\sigma^*): \widecheck{C}^{n-l-u-r}(\sigma^*) \longrightarrow \widecheck{C}_r(\tau^*)
	\end{equation*}
	\begin{equation}
		\label{Eqlocaldualquadraticstr}
		\widecheck{\psi}_u^r(\tau^*,\sigma^*)=\begin{cases} (-1)^{k+J_{\sigma}^{all}+l|\sigma|+lr+\frac{l(l-1)}{2}}\psi_u^r(\tau,\sigma) & \text{If } \sigma,\tau \in K_c \backslash L_c \\ \quad \ \ 0 & \text{else}	\end{cases}
	\end{equation}

	It is clear from the definition that the above construction gives a one-to-one correspondence. Therefore, we only have to check that the definition above is well-defined, that is, $(\widecheck{C},\widecheck{\psi})$ is a $(k-l)$-dimensional Poincare quadratic chain complex. The proof is divided into three steps:
	
	(1) $\widecheck{C}$ is a chain complex in $M^h(R)^*(\underline{L_c},\underline{K_c})$ and $\psi_u^r$ are morphisms in $M^h(R)^*(\underline{L_c},\underline{K_c})$ for all $u \in \mathbb{N},r \in \mathbb{Z}$.
	
	Since $\sigma \leq \tau$ is equivalent to $\sigma^* \geq \tau^*$, the statement in (1) follows from the facts that $C$ is a chain complex in $M^h(R)_*(K_c,L_c)$ and $\psi_u^r$ is a morphism in $M^h(R)_*(K_c)$.
	
	(2) $(\widecheck{C},\widecheck{\psi})$ is a $(k-l)$-dimensional quadratic chain complex in $M^h(R)^*(\underline{L_c},\underline{K_c})$.
	
	We have to check that for all $\sigma,\tau \notin L_c$ and for all $u \in \mathbb{N},r \in \mathbb{Z}$, the following equation holds:
	\begin{equation}
		\label{Eq6.1}
		\begin{aligned}
			0= \ &(d_{\widecheck{C},r+1}\widecheck{\psi}_u^{r+1})(\tau^*,\sigma^*)-(-1)^{k-l-u}(\widecheck{\psi}_u^rd_{\widecheck{C}^{-*},k-l-u-r-1})(\tau^*,\sigma^*) \\
			&+(-1)^{k-l-u-1}\widecheck{\psi}_{u+1}^r(\tau^*,\sigma^*)+(-1)^{k-l}T\widecheck{\psi}_{u+1}^{r}(\tau^*,\sigma^*) \\
		\end{aligned}
	\end{equation}

	To start with, since $\widecheck{C}_r(\sigma^*)=0$ for all $r \in \mathbb{Z},\sigma \notin K_c$, we only have to check the above equation under the case $\sigma,\tau \in K_c \backslash L_c$. We will compute each term in the equation separately.
	
	Denote $n_0=k+\frac{l(l-1)}{2}$ and $n_1=k-l-u-r$. For the first term, by definition we have:
	\begin{equation*}
		\begin{aligned}
			(d_{\widecheck{C},r+1}\widecheck{\psi}_u^{r+1})(\tau^*,\sigma^*)&=\sum\limits_{\kappa \notin L_c}d_{\widecheck{C},r+1}(\tau^*,\kappa^*)\widecheck{\psi}_u^{r+1}(\kappa^*,\sigma^*) \\
			& \quad (\text{By definition }  \ref{Eqlocaldualboundarymap} \text{ and } \ref{Eqlocaldualquadraticstr}) \\
			&=(-1)^{n_0+l(r+1)+J_{\sigma}^{all}+l|\sigma|}\sum\limits_{\kappa \in K_c \backslash L_c}d_{C,r+1}(\tau,\kappa)\psi_u^{r+1}(\kappa,\sigma) \\
		\end{aligned}
	\end{equation*}
	
	Since $C$ is a chain complex in $M^h(R)_*(K_c,L_c)$, we can get that $\psi_u^{r+1}(\kappa,\sigma) = 0$ unless $\kappa \in K_c \backslash L_c$. Thus we have:
	\begin{equation}
		\label{Eq6.2}
		\begin{aligned}
			(d_{\widecheck{C},r+1}\widecheck{\psi}_u^{r+1})(\tau^*,\sigma^*)&=(-1)^{n_0+l(r+1)+J_{\sigma}^{all}+l|\sigma|}\sum\limits_{\kappa \in K_c}d_{C,r+1}(\tau,\kappa)\psi_u^{r+1}(\kappa,\sigma) \\
			&=(-1)^{n_0+l(r+1)+J_{\sigma}^{all}+l|\sigma|}(d_{C,r+1}\psi_u^{r+1})(\tau,\sigma) \\
		\end{aligned}
	\end{equation}

	For the second term, we have:
	\begin{equation}
		\label{Eq6.3}
		(\widecheck{\psi}_u^rd_{\widecheck{C}^{-*},k-l-u-r-1})(\tau^*,\sigma^*)=\sum\limits_{\kappa \notin L_c}\widecheck{\psi}_u^r(\tau^*,\kappa^*)d_{\widecheck{C}^{-*},n_1-1}(\kappa^*,\sigma^*)
	\end{equation}

	Note that for any simplex $\eta^* \in \underline{L_c}$, we have:
	\begin{equation}
		\label{Identify6.1}
		\begin{aligned}
			\widecheck{C}^{n_1-1}(\eta^*)&=(T\widecheck{C})_{-n_1+1}(\eta^*) \\
			&=(T\widecheck{C}_{n_1-1+|\eta^*|})_{|\eta^*|}(\eta^*) \\
			&=\mathop{\oplus}\limits_{\eta_0^* \leq \eta^*}\widecheck{C}_{n_1-1+|\eta^*|}(\eta_0^*)^* \\
			& \quad (\text{By definition } \ref{Eqlocaldualchaincpx} \text{ and } C \text{ is in }M^h(R)_*(K_c,L_c)) \\
			&=\mathop{\oplus}\limits_{\substack{\eta_0 \geq \eta \\ \eta_0 \in K_c}}C_{k-u-r-1-|\eta|}(\eta_0)^*=C^{k-u-r-1}(\eta) \\
		\end{aligned}
	\end{equation}

	Therefore, we have:
	\begin{equation}
		\label{Eq6.4}
		\begin{aligned}
			&d_{\widecheck{C}^{-*},n_1-1}(\kappa^*,\sigma^*) \\
			&= \begin{cases}(-1)^{n_1-1+|\sigma^*|} \mathop{\boxplus}\limits_{\sigma_0^* \leq \sigma^*}\mathop{\oplus}\limits_{\kappa_0^* \leq \kappa^*}d_{\widecheck{C},n_1+|\sigma^*|}(\sigma_0^*,\kappa_0^*)^* & \text{If } \kappa^*=\sigma^* \\ \qquad \qquad \qquad \quad (-1)^{n_1-1+n_{\kappa^*}^{\sigma^*}}i_{\kappa^*\sigma^*}^* & \text{If } |\kappa^*|=|\sigma^*|-1 \\
			& \text{and } \kappa^* \leq \sigma^* \\ \qquad \qquad \qquad \qquad \qquad \quad 0 & \text{else}
			\end{cases} \\
		\end{aligned}
	\end{equation}
	
	If $\kappa^*=\sigma^*$, then $\kappa=\sigma$. By equation \ref{Eqlocaldualboundarymap} we have:
	\begin{equation}
		\label{Eq6.5}
		\begin{aligned}
			&\quad \; (-1)^{n_1-1+|\sigma^*|}\mathop{\boxplus}\limits_{\sigma_0^* \leq \sigma^*}\mathop{\oplus}\limits_{\kappa_0^* \leq \kappa^*}d_{\widecheck{C},n_1+|\sigma^*|}(\sigma_0^*,\kappa_0^*)^* \\
			&=(-1)^{k-u-r-1-|\sigma|}\mathop{\boxplus}\limits_{\sigma_0 \geq \sigma}\mathop{\oplus}\limits_{\kappa_0 \geq \kappa}d_{C,k-u-r-|\sigma|}(\sigma_0,\kappa_0)^* \\
			&=d_{C^{-*},k-u-r-1}(\kappa,\sigma) \\
		\end{aligned}
	\end{equation}
	
	If $|\kappa^*|=|\sigma^*|-1$, then $|\kappa|=|\sigma|+1$. By Lemma \ref{Dualincidnum} and equation \ref{Identify6.1}, we have:
	\begin{equation}
		\label{Eq6.6}
		\begin{aligned}
			(-1)^{n_1-1+n_{\kappa^*}^{\sigma^*}}i_{\kappa^*\sigma^*}^*&=(-1)^{k-l-u-r-1+n_{\sigma}^{\kappa}+J_{\sigma}^{all}-J_{\kappa}^{all}}i_{\kappa\sigma}^* \\
			&=(-1)^{-l+J_{\sigma}^{all}-J_{\kappa}^{all}} d_{C^{-*},k-u-r-1}(\kappa,\sigma)
		\end{aligned}
	\end{equation}
	
	Substituting equation \ref{Eqlocaldualquadraticstr},\ref{Eq6.4}, \ref{Eq6.5}, \ref{Eq6.6} into equation \ref{Eq6.3}, we get:
	\begin{equation*}
		\begin{aligned}
			& \quad (\widecheck{\psi}_u^rd_{\widecheck{C}^{-*},k-l-u-r-1})(\tau^*,\sigma^*) \\
			&=\widecheck{\psi}_u^r(\tau^*,\sigma^*)d_{\widecheck{C}^{-*},n_1-1}(\sigma^*,\sigma^*)+\sum\limits_{\substack{\kappa^* \leq \sigma^* \\ |\kappa^*|=|\sigma^*|-1}}\widecheck{\psi}_u^r(\tau^*,\kappa^*)d_{\widecheck{C}^{-*},n_1-1}(\kappa^*,\sigma^*) \\
			&=(-1)^{n_0+lr+J_{\sigma}^{all}+l|\sigma|}\psi_u^r(\tau,\sigma)d_{C^{-*},k-u-r-1}(\sigma,\sigma) \\
			& \quad +\sum\limits_{\substack{\kappa \geq \sigma \\ |\kappa|=|\sigma|+1}}(-1)^{n_0+lr+J_{\kappa}^{all}+l|\kappa|}\psi_u^r(\tau,\kappa)(-1)^{-l+J_{\sigma}^{all}-J_{\kappa}^{all}}d_{C^{-*},k-u-r-1}(\kappa,\sigma) \\
			&=(-1)^{n_0+lr+J_{\sigma}^{all}+l|\sigma|}\psi_u^r(\tau,\sigma)d_{C^{-*},k-u-r-1}(\sigma,\sigma) \\
			& \quad +\sum\limits_{\substack{\kappa \geq \sigma \\ |\kappa|=|\sigma|+1}}(-1)^{n_0+lr+J_{\sigma}^{all}+l|\sigma|}\psi_u^r(\tau,\kappa)d_{C^{-*},k-u-r-1}(\kappa,\sigma) \\
		\end{aligned}
	\end{equation*}

	Similar to equation \ref{Eq6.4}, we have that $d_{C^{-*},k-u-r-1}(\kappa,\sigma)=0$ unless $\kappa \geq \sigma$ and $|\kappa|-|\sigma| \leq 1$. Therefore, we can get:
	\begin{equation}
		\label{Eq6.7}
		(\widecheck{\psi}_u^rd_{\widecheck{C}^{-*},k-l-u-r-1})(\tau^*,\sigma^*)=(-1)^{n_0+lr+J_{\sigma}^{all}+l|\sigma|}(\psi_u^rd_{C^{-*},k-u-r-1})(\tau,\sigma)
	\end{equation}

	For the third term, by equation \ref{Eqlocaldualquadraticstr} we have:
	\begin{equation}
		\label{Eq6.8}
		\widecheck{\psi}_{u+1}^r(\tau^*,\sigma^*)=(-1)^{n_0+lr+J_{\sigma}^{all}+l|\sigma|}\psi_{u+1}^r(\tau,\sigma)
	\end{equation}

	For the last term, by Lemma \ref{ExpressT}, Theorem \ref{chaindualityonAK1} and Remark \ref{chaindualityonAK2}, we have the following commutative diagrams:
	\begin{small}
		\begin{equation*}
			\begin{tikzcd}
				\widecheck{C}^{n_1-1}(\sigma^*) \ar{rrr}{(-1)^{(n_1-1)(r-|\sigma^*|)}T\widecheck{\psi}_{u+1}^r(\tau^*,\sigma^*)}\dar{=} & & & \widecheck{C}_r(\tau^*) \\ 
				(T\widecheck{C}_{n_1-1+|\sigma^*|})_{|\sigma^*|}(\sigma^*) \ar[bend left=8]{rr}{T\big(\widecheck{\psi}_{u+1}^{n_1-1+|\sigma^*|}(\sigma^*,\sigma^*)\big)_{|\sigma^*|}}\dar{=} & & T(T(\widecheck{C}_{r})_{|\sigma^*|})_{|\sigma^*|}(\sigma^*) \rar{\subset}\dar{=} & (T^2\widecheck{C}_{r})_0(\sigma^*) \uar{\mathfrak{D}(\widecheck{C}_r)_0}\dar{=} \\
				\mathop{\oplus}\limits_{\kappa^* \leq \sigma^*}\widecheck{C}_{n_1-1+|\sigma^*|}(\kappa^*)^* \ar[bend left=8]{rr}{\mathop{\boxplus}\limits_{\kappa^* \leq \sigma^*}\widecheck{\psi}_{u+1}^{n_1-1+|\sigma^*|}(\kappa^*,\sigma^*)^*} & & \mathop{\oplus}\limits_{\kappa^* \leq \sigma^*}\widecheck{C}_r(\kappa^*) \rar{=} & \mathop{\oplus}\limits_{\kappa^* \leq \sigma^*}\widecheck{C}_r(\kappa^*) \ar[bend right=70,swap]{uu}{\substack{(-1)^{|\sigma^*|} \\ p_{\sigma^*,\tau^*}}}
			\end{tikzcd}
		\end{equation*}
	\end{small}
	\begin{small}
		\begin{equation*}
			\begin{tikzcd}
				C^{k-u-r-1}(\sigma) \ar{rrr}{(-1)^{(k-u-r-1)(r+|\sigma|)}T\psi_{u+1}^r(\tau,\sigma)}\dar{=} & & & C_r(\tau) \\ 
				(TC_{k-u-r-1-|\sigma|})_{-|\sigma|}(\sigma) \ar[bend left=8]{rr}{T\big(\psi_{u+1}^{k-u-r-1-|\sigma|}(\sigma,\sigma)\big)_{-|\sigma|}}\dar{=} & & T(T(C_{r})_{-|\sigma|})_{-|\sigma|}(\sigma) \rar{\subset}\dar{=} & (T^2C_{r})_0(\sigma) \uar{\mathfrak{D}(C_r)_0}\dar{=} \\
				\mathop{\oplus}\limits_{\kappa \geq \sigma}C_{k-u-r-1-|\sigma|}(\kappa) \ar[bend left=8]{rr}{\psi_{u+1}^{k-u-r-1-|\sigma|}(\kappa,\sigma)^*} & & \mathop{\oplus}\limits_{\kappa \geq \sigma}C_r(\kappa) \rar{=} & \mathop{\oplus}\limits_{\kappa \geq \sigma}C_r(\kappa) \ar[bend right=80,swap]{uu}{\substack{(-1)^{|\sigma|} \\ p_{\sigma,\tau}}}
			\end{tikzcd}
		\end{equation*}
	\end{small}
	
	By \ref{Identify6.1}, we have that the modules in  corresponding positions of the two commutative diagrams agree. By \ref{Eqlocaldualquadraticstr}, we have:
	\begin{equation}
		\label{Eq6.9}
		\begin{aligned}
			T\widecheck{\psi}_{u+1}^r(\tau^*,\sigma^*)&=(-1)^{(n_1-1)(r-|\sigma^*|)+|\sigma^*|}(-1)^{(k-u-r-1)(r+|\sigma|)+|\sigma|} \\ & \quad \  (-1)^{n_0+l(k-u-r-1-|\sigma|)+J_{\sigma}^{all}+l|\sigma|}T\psi_{u+1}^r(\tau,\sigma) \\
			&=(-1)^{n_0+lr+J_{\sigma}^{all}+l|\sigma|}T\psi_{u+1}^r(\tau,\sigma) \\
		\end{aligned}
	\end{equation} 

	By equation \ref{Eq6.2},\ref{Eq6.7},\ref{Eq6.8},\ref{Eq6.9}, we have that equation \ref{Eq6.1} is equivalent to the following equation:
	\begin{equation*}
		\begin{aligned}
			0= \ &(d_{C,r+1}\psi_u^{r+1})(\tau,\sigma)-(-1)^{k-u}(\psi_u^rd_{C^{-*},k-u-r-1})(\tau,\sigma) \\ &+(-1)^{k-u-1}\psi_{u+1}^r(\tau,\sigma)+(-1)^{k}T\psi_{u+1}^{r}(\tau,\sigma) \\
		\end{aligned}
	\end{equation*}

	Since $(C,\psi)$ is a $k$-dimensional quadratic chain complex in $M^h(R)_*(K_c)$, the above equation holds. Therefore, $(\widecheck{C},\widecheck{\psi})$ is a $(k-l)$-dimensional quadratic chain complex in $M^h(R)^*(\underline{L_c},\underline{K_c})$.

	(3) $(\widecheck{C},\widecheck{\psi})$ is Poincare.

	We have to check that the following morphism is a chain homotopy equivalence:
	\begin{equation*}
		(1+T)\widecheck{\psi}_0^r: \widecheck{C}^{k-l-r} \longrightarrow \widecheck{C}_r
	\end{equation*}

	By Proposition 4.7 in \cite{ranickiltheory}, it is equivalent to check that for every $\sigma \notin L_c$, the following morphism is a chain homotopy equivalence:
	\begin{equation}
		\label{Eq6.10}
		(1+T)\widecheck{\psi}_0^r(\sigma^*,\sigma^*): \widecheck{C}^{k-l-r}(\sigma^*) \longrightarrow \widecheck{C}_r(\sigma^*)
	\end{equation}

	If $\sigma \notin K_c$, by definition we have that the chain complex on both sides is the zero chain complex. The statement holds.

	If $\sigma \in K_c$, by \ref{Identify6.1},\ref{Eq6.4},\ref{Eq6.5},\ref{Eq6.6}, we have that the following morphism is a chain isomorphism:
	\begin{equation}
		\label{Eq6.11}
		\mathop{\oplus}\limits_{\sigma \in K_c \backslash L_c}(-1)^{J_{\sigma}^{all}+l|\sigma|}:\widecheck{C}^{k-l-r}=\mathop{\oplus}\limits_{\sigma \in K_c \backslash L_c}C^{l-r}(\sigma) \longrightarrow C^{l-r}=\mathop{\oplus}\limits_{\sigma \in K_c \backslash L_c}C^{l-r}(\sigma)
	\end{equation}

	By equation \ref{Eqlocaldualquadraticstr} and \ref{Eq6.9}, we have:
	\begin{equation*}
		(1+T)\widecheck{\psi}_0^r(\sigma^*,\sigma^*)=(-1)^{n_0+lr+J_{\sigma}^{all}+l|\sigma|}(1+T)\psi_{0}^r(\sigma,\sigma)
	\end{equation*}

	Therefore, under the chain isomorphism \ref{Eq6.11}, the chain map given in \ref{Eq6.10} is equivalent to the following chain map:
	\begin{equation*}
		(-1)^{n_0+lr}(1+T)\psi_{0}(\sigma,\sigma):C^{k-l-r}(\sigma) \longrightarrow C_r(\sigma)
	\end{equation*}

	Since $(1+T)\psi_{0}$ is a chain homotopy equivalence in $M^h(R)_*(K_c)$, by Proposition 4.7 in \cite{ranickiltheory}, we have that the chain map above is a chain homotopy equivalence. Therefore, $(\widecheck{C},\widecheck{\psi})$ is Poincare. 
\end{proof}

\begin{Remark}
	In general, the category $M^h(R)_*(K,L)$ is not fixed under the dual functor of $M^h(R)_*(K)$. Therefore, in the general case where $L \neq \emptyset$, the category is not the ideal one to take, but we still keep the definition for simplicity.
\end{Remark}

We end this subsection with some important computations that will be used later in the proof. We introduce some definitions first.


To start with, we introduce a generalization of the category $\mathbb{A}_*(K)$:

\begin{Definition}
	\
	
	Let $\mathbb{A}$ be an additive category and $K$ be a simplicial complex. Denote $\mathbb{A}^{lf}(K)$ to be the following additive category:
	 
	The objects are formal direct sums $M=\sum\limits_{\sigma \in K}M(\sigma)$ of objects $M(\sigma)$ in $\mathbb{A}$. 
	 
	A morphism $f:M \longrightarrow N$ is a collection of morphisms in $\mathbb{A}$:
	\begin{equation*}
		f=\{f(\tau,\sigma):M(\sigma) \longrightarrow N(\tau)| \ \sigma,\tau \in K\}
	\end{equation*}

	Denote $\mathbb{A}_*^{lf}(K)$ to be the subcategory of $\mathbb{A}^{lf}(K)$ with the same objects and with morphisms such that $f(\tau,\sigma) \neq 0$ unless $\sigma \leq \tau$.
	
	Denote $\mathbb{A}^*_{lf}(K)$ to be the subcategory of $\mathbb{A}^{lf}(K)$ with the same objects and with morphisms such that $f(\tau,\sigma) \neq 0$ unless $\sigma \geq \tau$.
\end{Definition}

\begin{Remark}
	If $K$ is finite, then $\mathbb{A}_*(K)=\mathbb{A}_*^{lf}(K),\mathbb{A}^*(K)=\mathbb{A}^*_{lf}(K)$.
\end{Remark}

If $K$ is locally finite, finite-dimensional and ordered, suppose that there is an involution on $\mathbb{A}$, then there are chain dualities on $\mathbb{A}_*^{lf}(K)$ and $\mathbb{A}^*_{lf}(K)$, which are given analogously to the construction in theorem \ref{chaindualityonAK1} and remark \ref{chaindualityonAK2}.

From now on until the end of the subsection we will assume $K $ to be a locally finite, finite-dimensional and ordered simplicial complex.

The following definition is a version of how to describe subcomplexes of the dual cell complex of a simplicial complex:

\begin{Definition}[upper closed]
	\
	
	Let $K'$ be a simplicial complex, a collection $S$ of simplices in $K'$ is called upper closed in $K'$, if for any $\sigma \in S$ and $\tau \in K'$, $\sigma \leq \tau$ implies $\tau \in S$.
\end{Definition}

\begin{Remark}
	It is straightforward to see that $S$ is upper closed if and only if its complement $S^c$ is a subcomplex of $K'$.
\end{Remark}

\begin{Definition}
	\
	
	Let $K'$ be a simplicial complex and $S$ be an upper closed set in $K'$.
	
	$(1)$ The closure of $S$, denoted by $\overline{S}$, is the smallest subcomplex in $K'$ that contains $S$.
	
	$(2)$ The boundary of $S$, denoted by $\partial S$, is defined by $\partial S=\overline{S} \backslash S$. It is a subcomplex of $K'$.
	
	$(3)$ The simplicial interior of $S$, denoted by $S^{-}$, is defined by $S^{-}=\{\sigma \in S | \ \sigma \cap \partial S =\emptyset \}$. It is a subcomplex of $K'$.
\end{Definition}

Now for an object $M \in M^h(R)_*^{lf}(K) \ (resp. \ M^h(R)^*_{lf}(K))$, we can also assemble it over subsets of $K$. It will have some nice property when the subset is good, as shown by the following:

\begin{Definition}[Partial assembly over a subset]
	\
	
	Let $M,N$ be two objects in $M \in M^h(R)_*^{lf}(K) \ (resp. \ M^h(R)^*_{lf}(K))$ and $f:M \longrightarrow N$ be a morphism. Let $S \subset K$ be a subset of simplices in $K$. Let $L$ be a subcomplex containing $S$ and $\mathbf{p}: \widetilde{L} \longrightarrow L$ be a Galois covering with transformation group $G_0$. Denote $\widetilde{S}=\mathbf{p}^{-1}(S)$, we define the following notations:
	
	$(1)$ $M(\widetilde{S}):=\mathop{\oplus}\limits_{\tilde{\sigma} \in \widetilde{S}}M(\mathbf{p}\tilde{\sigma})$. It is a $RG_0$ module.
	
	$(2)$ $f(\widetilde{S}):=\mathop{\boxplus}\limits_{\substack{\tilde{\sigma} \in \widetilde{L} \\ \mathbf{p}(\tilde{\sigma}) \in S}}\mathop{\oplus}\limits_{\substack{\tilde{\tau} \in \widetilde{L} \\ \mathbf{p}(\tilde{\tau}) \in S}}f(\tilde{\tau},\tilde{\sigma}):M(\widetilde{S}) \longrightarrow N(\widetilde{S})$. It is a $RG_0$ module homomorphism.
\end{Definition}


\begin{Lemma}
	\label{partialassemblyfun}
	Let $R$ be a ring and denote $M^{f}(RG_0)$ to be the category of free modules of $RG_0$. If $S$ is the intersection of an upper closed set in $K$ with a subcomplex in $K$, then the construction above is a functor from $M^h(R)_*^{lf}(K) \ (resp. \ M^h(R)^*_{lf}(K))$ to the category $M^{f}(RG_0)$. We call this functor the partial assembly over $S$ with respect to $\mathbf{p}$.
\end{Lemma}

\begin{Remark}
	We can also define in a similar way the partial assembly over a subset for objects in $M^h(R)^{lf}(K)$. However, in general it is not a functor.
\end{Remark}

If $R$ is a ring with involution, in general the functor above is not commutative with the "dual functors" when composing with $\Upsilon$. In order to explain the relation explicitly, we introduce the following definition:

\begin{Definition}
	\
	
	Let $R$ be a ring and $G_0$ be a group. An assemble structure $(K',S,\mathbf{p},O)$ on an object $M \in M^{f}(RG_0)$ consists of the following data:
	
	$(1)$ A locally finite, finite dimensional and ordered simplicial complex $K'$.
	
	$(2)$ A subset $S$ of simplices in $K'$.
	
	$(3)$ A Galois covering $\mathbf{p}:S^{cover} \longrightarrow \overline{S}$ with transformation group $G_0$. Denote $\widetilde{S}=\mathbf{p}^{-1}(S)$.
	
	$(4)$ An object $O \in M^h(R)^{lf}(K')$, such that $M=O(\widetilde{S})$.
	
	An object $M$ in $M^f(RG_0)$ is called assembled if there is an assemble structure. Define its compact supported dual (with respect to the assemble structure) to be $M^{cd}=O^*(\widetilde{S})$, where $O^*$ is the object in $M^h(R)^{lf}(K')$ with $O^*(\sigma)=O(\sigma)^*$ for all simplex $\sigma \in K'$.
	
	Let $h:M_1 \longrightarrow M_2$ be a $RG_0$ morphism between objects in $M^f(RG_0)$ with assemble structures $(K_i,S_i,\mathbf{p}_i,O_i)_{i=1,2}$. We call $h$ assembled (with respect to the assemble structures), if:
	
	$(1)$ $K_1=K_2$
	
	$(2)$ There is a subcomplex $L$ containing $S_1,S_2$ and a Galois covering $\mathbf{p}:\widetilde{L} \longrightarrow L$ with transformation group $G_0$, such that $\mathbf{p}|_{S_i}=\mathbf{p}_i$ $(i=1,2)$.
	
	$(3)$ There is a collection of morphisms $\{h(\tau,\sigma):O_1(\sigma) \longrightarrow O_2(\tau) \ | \ \sigma \in S_1,\tau \in S_2\}$ of $R$ modules, such that $h=\mathop{\boxplus}\limits_{\tilde{\sigma} \in \widetilde{S}_1}\mathop{\oplus}\limits_{\tilde{\tau} \in \widetilde{S}_2}h(\mathbf{p}\tilde{\tau},\mathbf{p}\tilde{\sigma})$.
	
	If $h$ is assembled, we call the (unique) collection of morphisms $\{h(\tau,\sigma):O_1(\sigma) \longrightarrow O_2(\tau) \ | \ \sigma \in S_1,\tau \in S_2\}$ to be the assemble structure of $h$. Define the compact supported dual of $h$, denoted by $h^{cd}$, to be as follows:
	\begin{equation*}
		h^{cd}=\mathop{\oplus}\limits_{\tilde{\sigma} \in \widetilde{S}_1}\mathop{\boxplus}\limits_{\tilde{\tau} \in \widetilde{S}_2}h(\mathbf{p}\tilde{\tau},\mathbf{p}\tilde{\sigma})^*:M_2^{cd}=\mathop{\oplus}\limits_{\tilde{\tau} \in \widetilde{S}_2}O_2(\mathbf{p}\tilde{\tau})^* \longrightarrow \mathop{\oplus}\limits_{\tilde{\sigma} \in \widetilde{S}_1}O_1(\mathbf{p}\tilde{\sigma})^*=M_1^{cd}
	\end{equation*}
\end{Definition}

\begin{Remark}
	For all the objects in $M^f(RG_0)$ that we will consider in the article, there will be obvious assemble structures on them and all the morphisms we consider are also assembled. Therefore, we will omit the step of pointing out the assemble structures of objects and morphisms mentioned in the article.
\end{Remark}

Then we have the following lemma:

\begin{Lemma}
	\
	\label{Dualexchange}
	
	Let $K$ be a locally finite, finite dimensional, ordered simplicial complex and $R$ be a ring with involution. Let $S$ be any upper closed set of $K$ and $\mathbf{p}: S^{cover} \longrightarrow \overline{S}$ be a Galois covering with tranformation group $G_0$. We endow $S^{cover}$ with a simplicial complex structure by the covering map $\mathbf{p}$ and denote $\widetilde{S}=\mathbf{p}^{-1}(S)$.
	
	Let $C$ be any chain complex in $M^h(R)_*^{lf}(K)$ and $f \in Hom_{\mathbb{A}}(TC,C)_{l}$. For any $r \in \mathbb{Z}$, denote $f^r \in Hom_{\mathbb{A}}(TC_{r-l},C_r)$ and $Tf^r \in Hom_{\mathbb{A}}(TC_{r-l},C_r)$ to be the corresponding components of $f$ and $Tf$. Let $\Upsilon^r[\widetilde{S}^{-}]$ be the map given by:
	\begin{equation*}
		\begin{aligned}
				\Upsilon^r[\widetilde{S}^{-}]: \mathop{\oplus}\limits_{\substack{\tilde{\sigma} \in \widetilde{S} \\ \mathbf{p}(\tilde{\sigma}) \in S^{-}}}C_r(\mathbf{p}\tilde{\sigma})^* \longrightarrow \mathop{\oplus}\limits_{\substack{\tilde{\tau} \in \widetilde{S} \\ \mathbf{p}(\tilde{\tau}) \in S}}C^r(\mathbf{p}\tilde{\tau}) \\
				\Upsilon^r[\widetilde{S}^{-}]=\mathop{\oplus} \limits_{\substack{\tilde{x} \leq \tilde{\sigma} \\ |\tilde{x}|=0}}\Upsilon_{\tilde{x},\tilde{\sigma}}^r \text{ on } C_r(\mathbf{p}\tilde{\sigma})^* \text{ with } \mathbf{p}(\tilde{\sigma}) \in S^{-} \\
		\end{aligned}
	\end{equation*}
	
	Then we have the following commutative diagram:
	\begin{equation*}
		\begin{tikzcd}
			\mathop{\oplus}\limits_{\substack{\tilde{\sigma} \in \widetilde{K} \\ \mathbf{p}(\tilde{\sigma}) \in S^{-}}}C_r(\mathbf{p}\tilde{\sigma})^*\ar{rrrr}{(-1)^{r(l-r)}Tf^{l-r}(\widetilde{S}) \circ \Upsilon^r[\widetilde{S}^{-}]} \dar{i^S} & & & & \mathop{\oplus}\limits_{\substack{\tilde{\sigma} \in \widetilde{K} \\ \mathbf{p}(\tilde{\sigma}) \in S}}C_{l-r}(\mathbf{p}\tilde{\sigma}) \dar{p_S} \\
			\mathop{\oplus}\limits_{\substack{\tilde{\sigma} \in \widetilde{K} \\ \mathbf{p}(\tilde{\sigma}) \in S}}C_r(\mathbf{p}\tilde{\sigma})^* \ar{rrrr}{ \Upsilon^{l-r}[\widetilde{S}^{-}]^{cd} \circ f^r(\widetilde{S})^{cd} } & & & & \mathop{\oplus}\limits_{\substack{\tilde{\sigma} \in \widetilde{K} \\ \mathbf{p}(\tilde{\sigma}) \in S^{-}}}C_{l-r}(\mathbf{p}\tilde{\sigma}) \\ 
		\end{tikzcd}
	\end{equation*}
	
	Where $i^S,p_S$ are the inclusion map and the projection map, respectively.
\end{Lemma}

\begin{proof}
	For any $r \in \mathbb{Z}$, we have $(TC)_{r-l}=\mathop{\oplus}\limits_{x \in \mathbb{Z}}T(C_{x})_{x+r-l}$. For any $x \in \mathbb{Z}$, let $f_{x}^r:T(C_{x})_{x+r-l} \longrightarrow C_r$ be the restriction of $f^r$ on $T(C_{x})_{x+r-l}$.
	
	For any $\sigma,\tau \in K$ with $\sigma \leq \tau$, let $x'=l-x-|\sigma|$. By Lemma \ref{ExpressT} and Theorem \ref{chaindualityonAK1}, we have the following commutative diagram:
	\begin{equation}
		\label{Tfdiagram}
		\begin{tikzcd}
			T(C)_{x-l}(\sigma) \ar{rrr}{(-1)^{(x-|\sigma|)(x'-|\sigma|)}Tf^x_{x'}(\tau,\sigma)}\dar{=} & & & C_x(\tau) \\ 
			T(C_{x'})_{-|\sigma|}(\sigma) \ar{rr}{T(f_{x}^{x'})_{-|\sigma|}(\sigma,\sigma)}\dar{=} & & T(T(C_{x})_{-|\sigma|})_{-|\sigma|}(\sigma) \rar{\subset}\dar{=} & (T^2C_{x})_0(\sigma) \uar{D(C_x)_0}\dar{=} \\
			\mathop{\oplus}\limits_{\kappa \geq \sigma}C_{x'}(\kappa)^* \ar{rr}{\mathop{\boxplus}\limits_{\kappa \geq \sigma}f_{x}^{x'}(\kappa,\sigma)^*} & & \mathop{\oplus}\limits_{\kappa \geq \sigma}C_x(\kappa) \rar{=} & \mathop{\oplus}\limits_{\kappa \geq \sigma}C_x(\kappa) \ar[bend right=80,swap]{uu}{(-1)^{|\sigma|}p_{\sigma,\tau}}
		\end{tikzcd}
	\end{equation}
	
	For any vertex $z \in K$, simplices $\sigma,\tau \in K$ and any $x \in \mathbb{Z}$, denote $f_x^{z,\sigma,\tau}:C_{l-x}^*(\sigma) \longrightarrow C_x(\tau)$ to be the following morphism if $z \in \sigma \cap \tau$:
	\begin{equation*}
		C_{l-x}^*(\sigma) \subset TC_{x-l}(z)=	\mathop{\oplus}\limits_{\kappa \geq z}C_{l-x}(\kappa)^* \stackrel{f_{l-x}^{x}(\tau,z)}{\longrightarrow} C_x(\tau)
	\end{equation*}
	
	Otherwise, define $f_{x}^{z,\sigma,\tau}$ to be $0$.
	
	Then we can have a description of $(-1)^{r(l-r)}Tf^{l-r}(\tau,z)$ for any $r \in \mathbb{Z}$. To start with, by our definition,
	\begin{equation}
		\label{fexpress}
		f^{r}_{l-r}(\kappa,z)=\mathop{\boxplus}\limits_{\kappa' \geq \sigma}f_{r}^{z,\kappa',\kappa}: T(C_{l-r})_{0}(z)=\mathop{\oplus}\limits_{\kappa' \geq z}C_{l-r}(\kappa')^* \longrightarrow C_{r}(\kappa)
	\end{equation}
	
	Therefore, by the commutative diagram \ref{Tfdiagram}, we can express the morphism $(-1)^{r(l-r)}Tf^{l-r}(\tau,z)$ restricted on $C_{r}(\sigma)^*$ as follows:
	\begin{equation}
		\label{Tfexpress}
		C_r(\sigma)^* \stackrel{(f_{r}^{z,\tau,\sigma})^*}{\longrightarrow} C_{l-r}(\tau)
	\end{equation}
	
	Notice that to prove the Lemma, it suffices to prove that for any simplex $\tilde{\sigma}$ with $\mathbf{p}(\tilde{\sigma}) \in S^{-}$ and $z_{\tilde{\sigma}} \in C_r(\mathbf{p}\tilde{\sigma})^*$, the following equality holds:
	\begin{equation*}
		(-1)^{r(l-r)}p_S \circ Tf^{l-r}(\widetilde{S}) \circ \Upsilon^r[\widetilde{S}^{-}](z_{\tilde{\sigma}})=\Upsilon^{l-r}[\widetilde{S}^{-}]^{cd} \circ f^r(\widetilde{S})^{cd} \circ i^S(z_{\tilde{\sigma}})
	\end{equation*}

	Equivalently, we can verify the following equality for all $\tilde{\tau}$ with $\mathbf{p}(\tilde{\tau}) \in S^-$ and $w_{\tilde{\tau}} \in C_{n-l}(\mathbf{p}\tilde{\tau})^*$:
	\begin{equation*}
		(-1)^{r(l-r)}\textlangle p_S \circ Tf^{l-r}(\widetilde{S}) \circ \Upsilon^r[\widetilde{S}^{-}](z_{\tilde{\sigma}}),w_{\tilde{\tau}} \textrangle=\textlangle i^S(z_{\tilde{\sigma}}),f^r(\widetilde{S}) \circ \Upsilon^{l-r}[\widetilde{S}^{-}](w_{\tilde{\tau}}) \textrangle
	\end{equation*}
	
	(Here $\textlangle,\textrangle$ is the pairing of a module with its dual)
	
	We compute both sides separately. Let $\tilde{x}_1,\tilde{x}_2,...,\tilde{x}_k$ be all the vertices of $\tilde{\sigma}$ and suppose that $\tilde{x}_1,\tilde{x}_2,...,\tilde{x}_j$ are all the vertices of $\tilde{\sigma} \cap \tilde{\tau}$, then:
	\begin{equation*}
		\begin{aligned}
		(-1)^{r(l-r)}p_S \circ Tf^{l-r}(\widetilde{S}) \circ \Upsilon^r[\widetilde{S}^{-}](z_{\tilde{\sigma}})&=(-1)^{r(l-r)}p_S \circ Tf^{l-r}(\widetilde{S})(\mathop{\oplus}\limits_{i=1}^k\Upsilon_{\tilde{x}_i,\tilde{\sigma}}^r(z_{\tilde{\sigma}})) \\
		&=(-1)^{r(l-r)}\sum\limits_{i=1}^kp_STf^{l-r}(\widetilde{S})\Upsilon_{\tilde{x}_i,\tilde{\sigma}}^rz_{\tilde{\sigma}}
		\end{aligned}
	\end{equation*}

	Notice that for every $1 \leq i \leq k$, we have $\Upsilon_{\tilde{x}_i,\tilde{\sigma}}^rz_{\tilde{\sigma}} \in C^r(\tilde{x}_i)$. By definition we have that $Tf^{l-r}(\tilde{\tau}',\tilde{x}_i) \neq 0$ implies $\tilde{x}_i \leq \tilde{\tau}'$. Since $S^-$ is a subcomplex, we have $\mathbf{p}(\tilde{x}_i) \in S^- \subset S$ and thus $\mathbf{p}(\tilde{\tau}') \in S$ by the upper closed property of $S$. Thus we have:
	\begin{equation*}
		\begin{aligned}
			(-1)^{r(l-r)}Tf^{l-r}(\widetilde{S})\Upsilon_{\tilde{x}_i,\tilde{\sigma}}^rz_{\tilde{\sigma}}&=	\mathop{\oplus}\limits_{\tilde{\tau}' \geq \tilde{x}_i}(-1)^{r(l-r)}Tf^{l-r}(\tilde{\tau}',\tilde{x}_i)\Upsilon_{\tilde{x}_i,\tilde{\sigma}}^rz_{\tilde{\sigma}} \\
			&=\mathop{\oplus}\limits_{\tilde{\tau}' \geq \tilde{x}_i}(-1)^{r(l-r)}Tf^{l-r}(\mathbf{p}\tilde{\tau}',\mathbf{p}\tilde{x}_i)\Upsilon_{\tilde{x}_i,\tilde{\sigma}}^rz_{\tilde{\sigma}} \\
			& \quad (\text{By } \ref{Tfexpress} \text{ and definition of } \Upsilon_{\tilde{x}_i,\tilde{\sigma}}^r) \\
			&=\mathop{\oplus}\limits_{\tilde{\tau}' \geq \tilde{x}_i}(f_r^{\mathbf{p}\tilde{x}_i,\mathbf{p}\tilde{\tau}',\mathbf{p}\tilde{\sigma}})^*z_{\tilde{\sigma}} \\
		\end{aligned}
	\end{equation*}
	
	Then we have:
	\begin{equation}
		\label{Eq6.12}
		\begin{aligned}
			& \quad \; (-1)^{r(l-r)}\textlangle p_S \circ Tf^{l-r}(\widetilde{S}) \circ \Upsilon^r[\widetilde{S}^{-}](z_{\tilde{\sigma}}),w_{\tilde{\tau}} \textrangle \\
			&=\textlangle (-1)^{r(l-r)}\sum\limits_{i=1}^kp_STf^{l-r}(\widetilde{S})\Upsilon_{\tilde{x}_i,\tilde{\sigma}}^rz_{\tilde{\sigma}},w_{\tilde{\tau}}\textrangle \\
			&=\sum\limits_{i=1}^j\textlangle(f_{r}^{\mathbf{p}\tilde{x}_i,\mathbf{p}\tilde{\tau},\mathbf{p}\tilde{\sigma}})^*z_{\tilde{\sigma}},w_{\tilde{\tau}}\textrangle \\
			&=\sum\limits_{i=1}^j\textlangle z_{\tilde{\sigma}},f_{r}^{\mathbf{p}\tilde{x}_i,\mathbf{p}\tilde{\tau},\mathbf{p}\tilde{\sigma}}w_{\tilde{\tau}}\textrangle \\
		\end{aligned}
	\end{equation}
	For the right-hand side, we can compute similarly:
	\begin{equation}
		\label{Eq6.13}
		\begin{aligned}
			\textlangle i^S(z_{\tilde{\sigma}}),f^r(\widetilde{S}) \circ \Upsilon^{l-r}[\widetilde{S}^{-}](w_{\tilde{\tau}}) \textrangle&=\textlangle z_{\tilde{\sigma}},f^r(\widetilde{S}) \circ \Upsilon^{l-r}[\widetilde{S}^{-}](w_{\tilde{\tau}}) \textrangle \\
			&=\textlangle z_{\tilde{\sigma}},f^r(\widetilde{S})( \mathop{\oplus}\limits_{\substack{\tilde{x} \leq \tilde{\tau} \\ |\tilde{x}|=0}}\Upsilon_{\tilde{x},\tilde{\tau}}^{l-r}w_{\tilde{\tau}}) \textrangle\\
			&=\textlangle z_{\tilde{\sigma}},\sum\limits_{\substack{\tilde{x} \leq \tilde{\tau} \\ |\tilde{x}|=0}}\mathop{\oplus}\limits_{\tilde{\sigma}' \in \widetilde{S}}f^r(\tilde{\sigma}',\tilde{x}) \Upsilon_{\tilde{x},\tilde{\tau}}^{l-r}w_{\tilde{\tau}} \textrangle\\
			&=\textlangle z_{\tilde{\sigma}},\sum\limits_{\substack{\tilde{x} \leq \tilde{\tau} \\ |\tilde{x}|=0}}f^r(\tilde{\sigma},\tilde{x}) \Upsilon_{\tilde{x},\tilde{\tau}}^{l-r}w_{\tilde{\tau}} \textrangle\\
			&=\textlangle z_{\tilde{\sigma}},\sum\limits_{\substack{\tilde{x} \leq \tilde{\tau} \cap \tilde{\sigma} \\ |\tilde{x}|=0}}f^r(\mathbf{p}\tilde{\sigma},\mathbf{p}\tilde{x}) \Upsilon_{\tilde{x},\tilde{\tau}}^{l-r}w_{\tilde{\tau}} \textrangle\\
			&=\textlangle z_{\tilde{\sigma}},\sum\limits_{i=1}^jf^r(\mathbf{p}\tilde{\sigma},\mathbf{p}\tilde{x}_i) \Upsilon_{\tilde{x}_i,\tilde{\tau}}^{l-r}w_{\tilde{\tau}} \textrangle\\
			&=\textlangle z_{\tilde{\sigma}},\sum\limits_{i=1}^jf^r(\mathbf{p}\tilde{\sigma},\mathbf{p}\tilde{x}_i) \Upsilon_{\tilde{x}_i,\tilde{\tau}}^{l-r}w_{\tilde{\tau}} \textrangle\\
			&=\textlangle z_{\tilde{\sigma}},\sum\limits_{i=1}^jf_r^{\mathbf{p}\tilde{x}_i,\mathbf{p}\tilde{\tau},\mathbf{p}\tilde{\sigma}}w_{\tilde{\tau}} \textrangle\\
		\end{aligned}
	\end{equation}

	Comparing the two equations \ref{Eq6.12},\ref{Eq6.13}, we arrive at the conclusion that the lemma holds.
\end{proof}

\begin{Remark}
	\label{Dualexchangeupper}
	Let $\Upsilon^r[\widetilde{S}]$ be the map given by:
	\begin{equation*}
		\begin{aligned}
			& \qquad \quad \Upsilon^r[\widetilde{S}]: \mathop{\oplus}\limits_{\tilde{\sigma} \in \widetilde{S}}C_r(\mathbf{p}\tilde{\sigma})^* \longrightarrow \mathop{\oplus}\limits_{\tilde{\tau} \in \widetilde{S}}C^r(\mathbf{p}\tilde{\tau}) \\
			& \Upsilon^r[\widetilde{S}]=\mathop{\oplus} \limits_{\substack{\tilde{x} \leq \tilde{\sigma} \\ |\tilde{x}|=0,\mathbf{p}\tilde{x} \in S}}\Upsilon_{\tilde{x},\tilde{\sigma}}^r \text{ on } C_r(\mathbf{p}\tilde{\sigma})^* \text{ with } \mathbf{p}(\tilde{\sigma}) \in S \\
		\end{aligned}
	\end{equation*}

	Similar computations show that the following diagram commutes:
	
	\begin{equation*}
		\begin{tikzcd}
			\mathop{\oplus}\limits_{\substack{\tilde{\sigma} \in \widetilde{K} \\ \mathbf{p}(\tilde{\sigma}) \in S}}C_r(\mathbf{p}\tilde{\sigma})^*\ar{rrrr}{(-1)^{r(l-r)}Tf_{l-r}(\widetilde{S}) \circ \Upsilon^r[\widetilde{S}]} \dar{=} & & & & \mathop{\oplus}\limits_{\substack{\tilde{\sigma} \in \widetilde{K} \\ \mathbf{p}(\tilde{\sigma}) \in S}}C_{l-r}(\mathbf{p}\tilde{\sigma}) \dar{=} \\
			\mathop{\oplus}\limits_{\substack{\tilde{\sigma} \in \widetilde{K} \\ \mathbf{p}(\tilde{\sigma}) \in S}}C_r(\mathbf{p}\tilde{\sigma})^* \ar{rrrr}{ \Upsilon^{l-r}[\widetilde{S}]^{cd} \circ f_r(\widetilde{S})^{cd} } & & & & \mathop{\oplus}\limits_{\substack{\tilde{\sigma} \in \widetilde{K} \\ \mathbf{p}(\tilde{\sigma}) \in S}}C_{l-r}(\mathbf{p}\tilde{\sigma}) \\ 
		\end{tikzcd}
	\end{equation*}
\end{Remark}

A very important corollary is:

\begin{Corollary}
	\label{partialquad}
	Given the same setup as above, suppose that $\psi$ is a quadratic structure on $C$. Let $C_*(\widetilde{S}^-)$ be the partial assembly of the chain complex $(C_r)_{r \in \mathbb{Z}}$ over $S^-$ with respect to the Galois covering $\mathbf{p}:S^{cover} \longrightarrow \overline{S}$. Denote $p_S^{dual},i^{S,dual}$ to be the following maps:
	\begin{equation*}
		\begin{aligned}
			& \ \ p_S^{dual}:C^*(\widetilde{S}) \stackrel{projection}{\longrightarrow} C^*(\widetilde{S}^-) \\
			& i^{S,dual}:C^*(\widetilde{S}^{-})^{cd} \stackrel{inclusion}{\longrightarrow} C^*(\widetilde{S})^{cd} \\
		\end{aligned}
	\end{equation*}
	
	Let
	\begin{equation*}
		\begin{aligned}
			& \ \ \psi_s^r(\widetilde{S}^-)^{ass}=\psi_s^r(\widetilde{S}^-)p_S^{dual}\Upsilon^{n-s-r}[\widetilde{S}^-]: C_{n-s-r}(\widetilde{S}^-)^{cd} \longrightarrow C_r(\widetilde{S}^-) \\
			& \psi_s^r(\widetilde{S}^-)^{ass}_{cd}=\Upsilon^{n-s-r}[\widetilde{S}^-]^{cd}i^{S,dual}\psi_s^{r}(\widetilde{S}^-)^{cd}: C_{r}(\widetilde{S}^-)^{cd} \longrightarrow C_{n-s-r}(\widetilde{S}^-) \\
		\end{aligned}
	\end{equation*}

	and $r'=n-s-r-1$, then:
	
	$(1)$ $C_*(\widetilde{S}^-)$ and $C_*(\widetilde{S}^-)^{cd}$ are chain complexes.
	
	$(2)$ The following equation holds:
	\begin{equation*}
		\begin{aligned}
			0=& \ d_{C}(\widetilde{S}^-)\psi_s^{r+1}(\widetilde{S}^-)^{ass}+(-1)^r\psi_s^{r}(\widetilde{S}^-)^{ass}d_{C}(\widetilde{S}^-)^{cd} \\
			& \ +(-1)^{n-s-1}\psi_{s+1}^{r}(\widetilde{S}^-)^{ass}+(-1)^{n+rr'}\psi_{s+1}^{r'}(\widetilde{S}^-)^{ass}_{cd} \\
		\end{aligned}
	\end{equation*}
\end{Corollary}

\begin{proof}
	(1) follows from the fact that $S$ is upper closed and Lemma \ref{partialassemblyfun}.
	
	To prove (2), note that since $\psi$ is a quadratic structure on $C$, if we write it in components, $\psi_s^r:C^{n-s-r} \longrightarrow C_r$ ($s \in \mathbb{N}$, $r \in \mathbb{Z}$), we have the following equality:
	\begin{equation*}
		d_{C_*}\psi_s^{r+1}-(-1)^{n-s}\psi_s^rd_{C^{n-*}}+(-1)^{n-s-1}(\psi_{s+1}^{r}+(-1)^{s+1}(T\psi_{s+1})^{r})=0
	\end{equation*}

	Since $S$ is upper closed, partial assembly over $S$ with respect to $\mathbf{p}$ is a functor. We can apply the functor to the equation above, compose it with $p_S$ on the left and with $\Upsilon^{r'}[\widetilde{S}^-]$ on the right. We have:
	\begin{equation}
		\label{Eq6.14}
		\begin{aligned}
			0=& \ p_Sd_{C_*}(\widetilde{S})\psi_s^{r+1}(\widetilde{S})\Upsilon^{r'}[\widetilde{S}^-]-(-1)^{n-s}p_S\psi_s^r(\widetilde{S})d_{C^{n-*}}(\widetilde{S})\Upsilon^{r'}[\widetilde{S}^-] \\
			&\ +(-1)^{n-s-1}p_S\psi_{s+1}^{r}(\widetilde{S})\Upsilon^{r'}[\widetilde{S}^-]+(-1)^{n}p_ST\psi_{s+1}^{r}(\widetilde{S})\Upsilon^{r'}[\widetilde{S}^-]
		\end{aligned}
	\end{equation}

	We will compute each term in the equation separately:
	
	Since $S$ is upper closed and $S^-$ is a subcomplex, we have:
	\begin{equation*}
		\begin{aligned}
			p_Sd_{C_*}(\widetilde{S})&=\mathop{\boxplus}\limits_{\mathbf{p}\tilde{\sigma} \in S}\mathop{\oplus}\limits_{\mathbf{p}\tilde{\tau} \in S^-}d_C(\tilde{\tau},\tilde{\sigma}) \\
			&=\mathop{\boxplus}\limits_{\substack{\mathbf{p}\tilde{\sigma} \in S \\ \tilde{\sigma} \leq \tilde{\tau}}}\mathop{\oplus}\limits_{\mathbf{p}\tilde{\tau} \in S^-}d_C(\tilde{\tau},\tilde{\sigma}) \\
			&=\mathop{\boxplus}\limits_{\substack{\mathbf{p}\tilde{\sigma} \in S^- \\ \tilde{\sigma} \leq \tilde{\tau}}}\mathop{\oplus}\limits_{\mathbf{p}\tilde{\tau} \in S^-}d_C(\tilde{\tau},\tilde{\sigma}) \\
			&=\mathop{\boxplus}\limits_{\mathbf{p}\tilde{\sigma} \in S^-}\mathop{\oplus}\limits_{\mathbf{p}\tilde{\tau} \in S^-}d_C(\tilde{\tau},\tilde{\sigma}) \\
			&=d_{C_*}(\widetilde{S}^-)p_S\\
		\end{aligned}
	\end{equation*}
	
	Similarly we get:
	\begin{equation*}
		\begin{aligned}
			p_S\psi_{s}^{r}(\widetilde{S})=\psi_s^r(\widetilde{S}^-)p_S^{dual} \qquad \qquad \qquad \ &\\
			\psi_s^r(\widetilde{S})^{cd}i^S=i^{S,dual}\psi_s^r(\widetilde{S}^{-})^{cd} \qquad \qquad \qquad &\\
			d_{C^{n-*}}(\widetilde{S})\Upsilon^{r'}[\widetilde{S}^-]=(-1)^{n-s-r-1}\Upsilon^{n-s-r}[\widetilde{S}^-]d_{C_*}(\widetilde{S}^-)^{cd} &\\
		\end{aligned}
	\end{equation*}
	
	Therefore:
	\begin{equation*}
		p_Sd_{C_*}(\widetilde{S})\psi_s^{r+1}(\widetilde{S})\Upsilon^{r'}[\widetilde{S}^-]=d_{C_*}(\widetilde{S}^-)\psi_s^{r+1}(\widetilde{S}^-)^{ass}
	\end{equation*}
	\begin{equation*}
		\begin{aligned}
			p_S\psi_s^{r}(\widetilde{S})d_{C^{n-*}}(\widetilde{S})\Upsilon^{r'}[\widetilde{S}^-]&=(-1)^{n-s-r-1}\psi_s^{r}(\widetilde{S}^-)p_S^{dual}\Upsilon^{n-s-r}[\widetilde{S}^-]d_{C}(\widetilde{S}^-)^{cd} \\
			&=(-1)^{n-s-r-1}\psi_s^{r}(\widetilde{S}^-)^{ass}d_{C}(\widetilde{S}^-)^{cd} \\
		\end{aligned}
	\end{equation*}
	\begin{equation*}
		p_S\psi_{s+1}^{r}(\widetilde{S})\Upsilon^{r'}[\widetilde{S}^-]=\psi_{s+1}^{r}(\widetilde{S}^-)^{ass}
	\end{equation*}
	\begin{equation*}
		\begin{aligned}
			p_ST\psi_{s+1}^{r}(\widetilde{S})\Upsilon^{r'}[\widetilde{S}^-] &=(-1)^{rr'}\Upsilon^{r}[\widetilde{S}^-]^{cd}\psi^{r'}_{s+1}(\widetilde{S})^{cd}i^S \quad (\text{By lemma } \ref{Dualexchange}) \\
			&=(-1)^{rr'}\psi_{s+1}^{r'}(\widetilde{S}^-)^{ass}_{cd}
		\end{aligned}
	\end{equation*}
	
	Substuting the equations above into equation \ref{Eq6.14}, we get:
	\begin{equation*}
		\begin{aligned}
			0=& \ d_{C}(\widetilde{S}^-)\psi_s^{r+1}(\widetilde{S}^-)^{ass}+(-1)^r\psi_s^{r}(\widetilde{S}^-)^{ass}d_{C}(\widetilde{S}^-)^{cd} \\
			&\ +(-1)^{n-s-1}\psi_{s+1}^{r}(\widetilde{S}^-)^{ass}+(-1)^{n+rr'}\psi_{s+1}^{r'}(\widetilde{S}^-)^{ass}_{cd} \\
		\end{aligned}
	\end{equation*}

	Which is the same as the equation in the statement of the Corollary.
\end{proof}

\begin{Remark}
	If $S$ is a finite set, then we have $C_*(\widetilde{S}^-)^{cd}=C_*(\widetilde{S}^-)^*$ and $\psi_s^r(\widetilde{S}^-)^{ass}_{cd}=(\psi_s^r(\widetilde{S}^-)^{ass})^*$. Therefore, the above construction gives a quadratic chain complex.
\end{Remark}

\begin{Remark}
	\label{partialquadpair}
	It is possible to relate the construction above with a "Poincare pair". The detail of this will be discussed in the Appendix.
\end{Remark}

A standard way to construct some quadratic chain complex in $M^h(R)_*^{lf}(K)$ is the infinite transfer via covering maps. The construction is as follows:

\begin{Definition}[Infinite transfer]
	\
	
	Let $K_0$ be a locally finite, finite dimensional, ordered simplicial complex and $R$ be a ring with involution. Let $M,N \in M^h(R)_*(K_0)$ be two objects and $f:M \longrightarrow N$ be a morphism. Let $\mathbf{p}: \widetilde{K}_0 \longrightarrow K_0$ be a covering map. We can define the following notations:
	
	$(1)$ For every simplex $\tilde{\sigma} \in \widetilde{K}_0$, let $\dot{M}(\tilde{\sigma})= M(\mathbf{p}\tilde{\sigma})$. It is an object in $M^h(R)_*^{lf}(\widetilde{K}_0)$.
	
	$(2)$ For every simplex $\tilde{\sigma},\tilde{\tau} \in K_0$, let $f(\tilde{\tau},\tilde{\sigma})=\begin{cases} f(\mathbf{p}\tilde{\tau},\mathbf{p}\tilde{\sigma}) & \text{If } \tilde{\sigma} \leq \tilde{\tau} \\ \ \ \ \ \ 0 & \text{else} \\ \end{cases}:\dot{M}(\tilde{\sigma}) \longrightarrow \dot{N}(\tilde{\tau})$. It gives a morphism between $\dot{M}$ and $\dot{N}$.
\end{Definition}

\begin{Lemma}[Appendix C in \cite{ranickiltheory}]
	\label{infinitetransfer}
	\
	
	The construction above gives a functor from $M^h(R)_*(K_0)$ to $M^h(R)_*^{lf}(\widetilde{K}_0)$. Moreover, the functor carries the chain duality on $M^h(R)_*(K_0)$ to the chain duality on $M^h(R)_*^{lf}(\widetilde{K}_0)$. We call it the infinite transfer functor with respect to $\mathbf{p}$.
\end{Lemma}

\subsection{Ball complex description of $L$-theory}
\

We will introduce the generalization of construction above to ball complexes given by Laures and McClure. The main reference is \cite{LauresBallL}. The description will help us to simplify the computation greatly for complexes of the form "$K \times I$". We recall some definitions from their article first.

\begin{Definition}[Finite ball complex]
	\
	
	$(1)$ Let $K$ be a finite collection of PL balls in some Euclidean space $\mathbb{R}^n$. Denote $|K|_{geo}$ for the union $\mathop{\cup}\limits_{\sigma \in K}\sigma$. We call $K$ a finite ball complex if:
	
	$(a)$ The interiors of the balls in $K$ are disjoint.
	
	$(b)$ The boundary of each ball in $K$ is a union of balls in $K$.
	
	The balls in $K$ are also called closed cells in $K$.
	
	$(2)$ Let $K,L$ be finite ball complexes and $f:|K|_{geo} \longrightarrow |L|_{geo}$ be a PL homeomorphism. We call $f$ an isomorphism from $K$ to $L$, if it takes closed cells in $K$ to closed cells in $L$.
	
	$(3)$ A subcomplex of a finite ball complex $K$ is a subset of $K$ that is a finite ball complex.
	
	$(4)$ A morphism of finite ball complexes is the composite of an isomorphism with an inclusion of a subcomplex.
	
	$(5)$ Let $K,L$ be finite ball complexes. The product $K \times L$ is a finite ball complex with a closed cell $\sigma \times \tau$ for each closed cell $\sigma \in K,\tau \in L$.
	
	We denote the category of finite ball complexes to be $\mathbf{Bi}$.
\end{Definition}

\begin{Remark}
	We will denote $\Delta^1$ to be the unit interval with its standard structure as a finite ball complex, which has two 0-cells and one 1-cell. It is also a simplicial complex with the same structure.
\end{Remark}

We denote $\mathbb{Z}^{cat}$ to be the discrete category of the poset $\mathbb{Z}$. It is endowed with the trivial involution.

\begin{Definition}[Dimensioned category, Definition 3.3 in \cite{LauresBallL}]
	\
	
	A dimensioned category is a small category $\mathbb{A}$ with involution $i_{\mathfrak{A}}$, together with involution-preserving functors $d: \mathfrak{A} \longrightarrow \mathbb{Z}^{cat}$ (called the dimension function) and $\emptyset: \mathbb{Z}^{cat} \longrightarrow \mathfrak{A}$ such that:
	
	$(1)$ $d\emptyset$ is equal to the identity functor.
	
	$(2)$ If $f: a \longrightarrow b$ is a non-identity morphism in $\mathfrak{A}$, then $d(a) < d(b)$.
	
	For any $k \in \mathbb{Z}$, a $k$-morphism between dimensioned categories is a functor that decreases the
	dimensions of objects by $k$ and strictly commutes with $\emptyset$ and the involution.
\end{Definition}

\begin{Remark}
	In Laures and McClure's paper, dimensioned category is called $\mathbb{Z}$-graded category. We change the term here to avoid confusion with definition in previous sections.
\end{Remark}

\begin{Example}[Example 3.6 in \cite{LauresBallL}]
	\
	
	Let $K$ be a finite ball complex and $L$ be its subcomplex. Define $Cell(K,L)$ to be the dimensioned category whose objects in dimension $k$ are the oriented closed $k$-cells $(\sigma, o)$ which are not in $L$, together with an object $\emptyset_k$. All the morphisms are one of the following:
	
	$(1)$ The identity morphism for every object.
	
	$(2)$ If $\sigma \varsubsetneq \sigma'$, there is a unique morphism $(\sigma,o) \longrightarrow (\sigma',o')$.
	
	$(3)$ For every $k \leq |\sigma|$, there is a unique morphism $\emptyset_k \longrightarrow (\sigma,o)$.
	
	The involution is given by $(\sigma,o) \mapsto (\sigma,-o)$.
	
	For $L=\emptyset$, the category is written as $Cell(K)$.
\end{Example}

\begin{Definition}[Definition 6.1 in \cite{LauresBallL}]
	\
	
	Let $K$ be a finite ball complex and $L$ be its subcomplex. Denote $Cell^b(K,L)$ to be the category with objects the cells in $K\backslash L$, together with an empty cell $\emptyset_k$ for every $k \in \mathbb{N}$. The morphisms are inclusions of cells. Here the empty cell $\emptyset_k$ is included into every cell of dimension greater than $k$.
\end{Definition}

\begin{Remark}
	There is a natural functor $Cell(K,L) \longrightarrow Cell^b(K,L)$ given by $(\sigma,o) \mapsto \sigma,\emptyset_k \longrightarrow \emptyset_k$ on objects.
\end{Remark}

The following definition also comes from \cite{LauresBallL}:

\begin{Definition}[Balanced category, Definition 5.1 in \cite{LauresBallL}]
	\
	
	A balanced category is a dimensioned category $\mathfrak{A}$ together with a natural bijection:
	\begin{equation*}
		\eta:Hom_{\mathfrak{A}}(A,B) \longrightarrow Hom_{\mathfrak{A}}(A,i_{\mathfrak{A}}B)
	\end{equation*}
	for objects $A,B$ with $d(A)<d(B)$, such that:

	$(1)$ $\eta \circ i_{\mathfrak{A}}=i_{\mathfrak{A}} \circ \eta: Hom_{\mathfrak{A}}(A,B) \longrightarrow Hom_{\mathfrak{A}}(i_{\mathfrak{A}}A,B)$

	$(2)$ $\eta \circ \eta=Id$
	
	We call $\eta$ the balance structure of $\mathfrak{A}$.
\end{Definition}

\begin{Remark}
	$\mathbb{Z}^{cat},Cell(K,L)$ are balanced categories. Since if the morphism set is nonempty, it will only contain one morphism, the balance structure $\eta$ of these two categories is given in a unique way.
\end{Remark}

Now we can introduce the notion of $K$-ad for a finite ball complex $K$ in \cite{LauresBallL}. For that, we need the following definition:

\begin{Definition}[pre-ad, Definition 3.7 in \cite{LauresBallL}]
	\
	
	Let $\mathfrak{A}$ be a dimensioned category and $K$ be a ball complex with $L$ subcomplex.
	
	$(1)$ A pre $K$-ad of degree $k$ in $\mathfrak{A}$ is a $k$-morphism from $Cell(K)$ to $\mathfrak{A}$.
	
	$(2)$ The trivial pre $K$-ad of degree $k$ in $\mathfrak{A}$ is:
	\begin{equation*}
		Cell(K) \stackrel{d}{\longrightarrow} \mathbb{Z}^{cat} \stackrel{-k}{\longrightarrow} \mathbb{Z}^{cat} \stackrel{\emptyset}{\longrightarrow} \mathfrak{A}
	\end{equation*}

	$(3)$ A pre $(K,L)$-ad of degree $k$ in $\mathfrak{A}$ is a pre $K$-ad of degree $k$ in $\mathfrak{A}$ that is trivial when restricted to $L$.
\end{Definition}

The $(K,L)$-ads are chosen properly in pre $(K,L)$-ads for every pair of ball complexes. They have to satisfy some compatibility condition, namely $(a)-(g)$ in Definition 3.10 in \cite{LauresBallL}.

Next, we begin to introduce the constructions in \cite{LauresBallL} of some special $(K,L)$-ad theory that finally gives a description of the generalized cohomology theory of ball complexes related to the quadratic $L$-group. 

In \cite{LauresBallL}, a special set $\mathfrak{S}$ is introduced to deal with set theoretic problems. Here we only need the fact that $\mathfrak{S}$ contains all finite subsets of $\mathbb{Z}$. Then we take the following definitions from \cite{LauresBallL}:

\begin{Definition}[Definition 9.2 in \cite{LauresBallL}]
	\
	
	Let $R$ be ring with involution, then:
	
	$(1)$ Let $M^{\mathfrak{S}}(R)$ be the category of right $R$ modules of the form $R\textlangle S\textrangle$ with $S \in \mathfrak{S}$. Here $R\textlangle S\textrangle$ is the set of finitely supported functions from $S$ to $R$, endowed with the natural $R$ module structure given by right multiplication.
	
	$(2)$ A chain complex in $M^{\mathfrak{S}}(R)$ is called homotopically finite if it is chain homotopic to a finite chain complex.
	
	$(3)$ Let $\mathbb{B}^{hf}(M^{\mathfrak{S}}(R))$ be the full subcategory of $\mathbb{B}(M^{\mathfrak{S}}(R))$ consisting of objects that are homotopically finite.
\end{Definition}

Similar to the definition of previous section, we have the notion of a quasi quadratic (symmetric) chain complex, see also the definition 9.3,11.1 in \cite{LauresBallL}. In order to give the ad theory, we need to state what the dimensioned or balanced category is.

\begin{Definition}[Definition 11.2 in \cite{LauresBallL}]
	\
	
	Let $R$ be ring with involution, we define a balanced category $\mathfrak{A}_R$ as follows:
	
	
	The objects of $\mathfrak{A}_R$ are quasi quadratic chain complexes in $\mathbb{B}^{hf}(M^{\mathfrak{S}}(R))$. Non-identity morphisms $(C,\psi) \longrightarrow (C',\psi')$ are defined only for $\dim (C,\psi) <\dim (C',\psi')$ and are the sets of chain maps from $C$ to $C'$. The involution is defined by $(C,\psi) \mapsto (C,-\psi)$. The balance structure is given by the identity map.
\end{Definition}

Then we have to specify some properties to define $(K,L)$-ad:
\begin{Definition}[Balanced pre $\mathbb{K}$-ad]
	\
	
	Let $\mathfrak{A}$ be a balanced category and $(K,L)$ be a finite ball complex pair, a pre $(K,L)$-ad is called balanced, if it commutes with the balance structure $\eta$.
\end{Definition}

\begin{Remark}
	Notice that for a balanced pre $K$-ad $F$ in $\mathfrak{A}_R$, every cell $(\sigma,o)$ in $K$ gives a quasi quadratic chain complex. Since the functor is balanced, the chain complex part is independent of the orientation. Thus we can write $F(\sigma,o)=(C_{\sigma},\psi_{\sigma,o})$ and $\sigma \mapsto C_{\sigma},\emptyset_k \mapsto \emptyset_k$ gives a functor $C:Cell^b(K) \longrightarrow \mathbb{B}^{hf}(M^{\mathfrak{S}}(R))$.
\end{Remark}

\begin{Definition}
	\label{Deffib}
	A map between chain complexes in $R$ is called a cofibration if, on every dimension, it is an inclusion of a direct summand.
\end{Definition}

\begin{Definition}
	A functor $C$ from $Cell^b(K)$ to chain complexes in $R$ is called well-behaved if:
	
	$(1)$ $C$ maps each morphism to a cofibration.
	
	$(2)$ For every ball $\sigma$ in $K$, the map 
	\begin{equation*}
		\mathop{colim}\limits_{\tau \subsetneq \sigma} C_{\tau} \longrightarrow C_{\sigma}
	\end{equation*}

 	is a cofibration.
 	We denote $C_{\partial\sigma}$ to be $	\mathop{colim}\limits_{\tau \subsetneq \sigma}C_{\tau}$.
\end{Definition}

\begin{Definition}
	\
	
	Let $F$ be a pre $K$-ad in $\mathfrak{A}_R$. Call $F$ is closed, if for every $\sigma$, denote $cl(\sigma)$ to be the cellular chain complex of $|\sigma|_{geo}$, the following map is a chain map:
	\begin{equation*}
		\begin{aligned}
			cl(\sigma) \longrightarrow W \otimes_{\mathbb{Z}[\mathbb{Z}_2]} C_{\sigma}^t \otimes_{R} C_{\sigma} \\
		\end{aligned}
	\end{equation*}

	Where $\textlangle \tau,o\textrangle$ is mapped to the image of $\psi_{\tau,o}$ under the morphism induced by $F((\sigma,o_1)-(\tau,o))$ with any orientation $o_1$ on $\sigma$.
\end{Definition}

\begin{Definition}[$(K,L)$-ad]
	\
	\label{Defad}
	
	Let $(K,L)$ be a pair of finite ball complexes, then
	
	$(1)$ A pre $K$-ad $F$ is a $K$-ad if:
	
	$(a)$ $F$ is balanced and closed and the associated functor $C$ is well-behaved.
	
	$(b)$ For any $\sigma$, denote $\bar{\psi}_{\sigma}$ the composite 
	\begin{equation*}
		\begin{tikzcd}
			W \ar{rr}{(1+T)\psi_{\sigma,o}} & & C_{\sigma}^t \otimes_{R} C_{\sigma} \ar{rr}{\text{projection}} & & (C_{\sigma}/C_{\partial\sigma})^t \otimes_{R} C_{\sigma} \\
		\end{tikzcd}
	\end{equation*}

	It is a chain map. Let $1_0$ be the element $1 \in W_0$, then we require that $\bar{\psi}_{\sigma}(1_0):Hom(C_{\sigma},R)^{{|\sigma|-deg F-*}} \longrightarrow (C_{\sigma}/C_{\partial \sigma})_*$ is a chain homotopy equivalence.
	
	Here $1+T$ is the symmetrization map.
	
	$(2)$ A pre $(K,L)$-ad is a $(K,L)$-ad if it is a $K$-ad as
	a pre $K$-ad.
\end{Definition}

As proved in \cite{LauresBallL}, the cobordism groups of $(K,L)$-ad form a cohomology theory and it can be identified with the quadratic L-cohomology. We will make a description of this identification below.

\begin{Definition}[Definition 14.1 in \cite{LauresBallL}]
	\
	
	Let $K$ be a ball complex with a subcomplex $L$ and $k \in \mathbb{Z}$. We call two $(K,L)$-ads $F,G$ of degree $k$ cobordant, if there is a $(K \times \Delta^1, L \times \Delta^1)$-ad that restricts to $F$ on $K \times 0$ and $G$ on $K \times 1$.
	
	The set of cobordism classes of $(K,L)$-ads of degree $k$ is denoted by $T^k(K,L)$.
\end{Definition}

There are abelian group structures on these sets, which are discussed in Section 14 in \cite{LauresBallL}. These groups give a cohomology theory. Before we make the statement, let us recall a 
definition from	\cite{LauresBallL} that is related to the connecting homomorphism:

\begin{Definition} [Definition 14.3 in \cite{LauresBallL}]
	\
	\label{DefK}
	
	Let $\mathcal{K}: Cell(K \times \Delta^1,K \times \partial \Delta^1 \cup L \times \Delta^1) \longrightarrow Cell(K, L)$ be the isomorphism of categories that takes $(\sigma \times \Delta^1,o \times o_{\Delta^1})$ to $(\sigma,(-1)^{|\sigma|}o)$, where $o_{\Delta^1}$ is the standard orientation on $\Delta^1$.
\end{Definition}

\begin{Remark}
	There is sign difference between the definition above and the Definition 14.3 in \cite{LauresBallL}, which arises as we switch $\sigma$ and $\Delta^1$ in the definition.
\end{Remark}


\begin{Theorem}[Lemma 14.8 in \cite{LauresBallL}]
	$\mathcal{K}$ induces an isomorphism:
	\begin{equation*}
		T^k(K,L) \stackrel{\cong}{\longrightarrow} T^{k+1}(K \times \Delta^1,K \times \partial \Delta^1 \cup L \times \Delta^1)
	\end{equation*}
\end{Theorem}

Then we have:

\begin{Theorem}[Theorem 14.11 in \cite{LauresBallL}]
	\label{deltaofT}
	\
	
	$T^*$ is a cohomology theory, with the connecting homomorphism $T^k(L) \longrightarrow T^{k+1}(K,L)$ given by the negative of the composite of the following morphisms:
	\begin{equation*}
		\begin{tikzcd}
			T^k(L) \rar{\mathcal{K}^*} & T^{k+1}(L \times \Delta^1, L \times \partial \Delta^1) \\
			T^{k+1}(K \times \Delta^1, K \times 1 \cup L \times 0) \ar{ur}{res}[swap]{\cong} \rar{res}[swap]{\cong} & T^{k+1}(K,L) \\
		\end{tikzcd}
	\end{equation*}
\end{Theorem}

The cohomology theory above can be identified with the quadratic L-cohomology. The following theorem gives a detailed description of the identification:

\begin{Theorem}
	\
	\label{Lcohomologywithadcord}
	
	Let $(K,L)$ be a pair of finite ordered geometric simplicial complexes. Then there is an isomorphism $H^{k}(K,L;\underline{L}(R)) \cong T^k(K,L)$ given as follows:
	
	Let $x \in H^{k}(K,L;\underline{L}(R))$, choose a $(-k)$-dimensional Poincare quadratic chain complex $(B,\psi_B)$ in $M^h(R)^*(K,L)$ representing $x$ under the identification in Theorem \ref{Lhomology}. Then the identification of $x$ in $T^k(K,L)$ is given by the cobordism class of a $(K,L)$-ad $F$ of degree $k$ given as follows:
	
	For any oriented closed cell $(\zeta,o)$ of $K$, denote $sgn(o)=1$ if $o$ agrees with the standard one and $-1$ otherwise. Then the functor $F$, denoted by $F(\zeta,o)=(C_{\zeta},\psi_{\zeta,o})$, is given as follows:
	
	On objects, $F$ is given by:
	\begin{equation*}
		\begin{aligned}
			&\text{For } \sigma \in K \backslash L: \\
			& (C_{\sigma})_r= [B_r][\sigma] \\
			&d_{C_{\sigma}}=[d_{B}][\sigma]:(C_{\sigma})_r \longrightarrow (C_{\sigma})_{r-1} \\
			&\psi_{\sigma,o}^{u,r}: C^{|\sigma|-u-deg F-r}_{\sigma}=[B_{|\sigma|-u-k-r}][\sigma]^* \longrightarrow (C_{\sigma^*})_r=[B_r][\sigma]\\
			&\psi_{\sigma,o}^{u,r}=(-1)^{\frac{|\sigma|(|\sigma|-1)}{2}}(-1)^{|\sigma|r}sgn(o)\mathop{\oplus}\limits_{\tau \leq \sigma} \psi_{B,u}^r(\tau,\sigma) \\
			&\text{For all other closed cells } \zeta: (C_{\zeta},\psi_{\zeta,o})=\emptyset_{|\zeta|-k-1}. \\
			&\text{For all } l \in \mathbb{Z}: F(\emptyset_l)=\emptyset_{l-k-1}. \\
		\end{aligned}
	\end{equation*}
	
	On morphisms, $F$ is given by:
	
	\quad Let $\zeta_1 \leq \zeta_2$, then:
	
	\quad $(1)$ If $\zeta_1=\tau, \zeta_2=\sigma$, with $\tau,\sigma \in K \backslash L$.
	
	\quad We have $\tau \leq \sigma$, then $F((\zeta_1,o_1)-(\zeta_2,o_2))$ is given by the inclusion map:
	\begin{equation*}
		\mathop{\oplus}\limits_{\kappa \leq \tau}B_*(\kappa) \longrightarrow \mathop{\oplus}\limits_{\kappa \leq \sigma}B_*(\kappa)
	\end{equation*}
	
	\quad $(2)$ If the case above does not happen, then we define $F((\zeta_1,o_1)-(\zeta_2,o_2))$ 
	
	\quad to be zero. (In fact in this case the domain of $F((\zeta_1,o_1)-(\zeta_2,o_2))$ will
	
	\quad always be zero.)
	
	Moreover, the isomorphism is an isomorphism between cohomology theories on simplicial pairs, that is, it commutes with the inclusion maps of simplicial pairs and the following diagram commutes for all $J \subset K \subset L$:
	
	\begin{equation*}
		\begin{tikzcd}
			H^k(K,J;\underline{L}(R))\dar{\delta} \rar{\cong} & T^k(K,J) \dar{\delta}\\
			H^{k+1}(L,K;\underline{L}(R)) \rar{\cong} & T^{k+1}(L,K) \\
		\end{tikzcd}
	\end{equation*}
\end{Theorem}

\begin{proof}
	We provide a brief explanation of the proof here. For the first part, using the Remark \ref{AppRemquadpair}, one can check the condition of a $(K,L)$-ad on $F$ following the similar computations as in Theorem \ref{computationdelta} in the Appendix and it will follow similarly that the construction is well-defined. For details, see Theorem \ref{computationdelta} in the Appendix.
	
	For the second part, it can be seen obviously from the construction that the isomorphism is natural with respect to inclusions of simplicial pairs. For the commutative diagram, we first introduce some spaces as in Definition 15.4 in \cite{LauresBallL}. For any $q \in \mathbb{Z}$, let $\mathbf{P}_q$ be the $\Delta$-set given as follows:
	
	For every $s \in \mathbb{N}$, $\mathbf{P}_q(\Delta^s)$ is the set of all $\Delta^s$-ad of degree $q$. Then the construction stated in the theorem actually gives a $\Delta$-map $\mathcal{F}_q:\mathbb{L}_{-q}(R) \longrightarrow \mathbf{P}_q$. Furthermore, by Proposition 15.9 and 16.4 in \cite{LauresBallL}, $|\mathbf{P}_q|_{q \in \mathbb{Z}}$ gives an $\Omega$-spectrum and we have an identification $T^k(K,L) \cong [K,L;\mathbf{P}_k,*]$. It is straightforward to show that the following diagram commutes:
	\begin{equation*}
		\begin{tikzcd}
			H^k(K,L;\underline{L}(R)) \rar{\cong} \dar {\cong} & T^k(K,L) \dar{\cong}  \\
			\mathop{[}K,L;\mathbb{L}_{-k}(R),*\mathop{]} \rar{\mathcal{F}_k} & \mathop{[}K,L;\mathbf{P}_k,*\mathop{]} \\
		\end{tikzcd}
	\end{equation*}

	Notice that $\mathcal{F}_q:\mathbb{L}_{-q}(R) \longrightarrow \mathbf{P}_q$ is a $\Delta$-map and we can deduce from the exact form of the construction that $|\mathcal{F}_*|$ is a morphism between $\Omega$-spectra. From that we see that the diagram commutes.
\end{proof}

\subsection{Proof of Theorem \ref{L theory transfer}}
\

Now we begin introducing the simplicial structures on the spaces we set in the geometric setting in section \ref{geometric setting}.

\begin{Construction}[Simplicial setting]
	\label{simplicial setting}
	\
	
	$(1)$ Fix a triangulation $K_M$ of $M$, such that $N$ and $N \times S^1$ are ordered subcomplexes of $M$ and we have a decomposition $M=M_0 \cup_{(N \times S^1) \otimes \partial [-1,1]} (N \times S^1) \otimes [-1,1]$, with $(N \times S^1) \otimes [0,1] \subset W$ being a collar of $(N \times S^1)$ in $W$. Furthermore, we require that $(N \times S^1) \otimes [-1,1] \subset N \times \mathbb{R}^2$. 
	
	
	$(2)$ The triangulation $K_M$ of $M$ lifts to a triangulation of $\widetilde{M}$ and we can view all the spaces appeared in the geometric setting in this way as locally finite ordered simplicial complexes. Denote $W_0$ to be $W_{\infty} \backslash \bar{e}(N \otimes [0,1))$, it is a subcomplex of $\overline{M}$.
	
	$(3)$ By the simplicial approximation theorem, there is a triangulation $K_{M'}$ of $M'$ and a simplicial map $F:K_{M'} \longrightarrow K_M$, such that $f$ is homotopic to $F$. Since surgery obstruction is invariant under cobordisms, we can, without loss of generality, assume that $f$ is simplicial with respect to the triangulation $K_{M'},K_M$.
	
	$(4)$ Denote $SdK_M$ be the barycentric subdivision of the triangulation. By a remark of Cohen, for each simplex $\sigma \in K_{M'}$, if we choose its barycenter $b(\sigma)$ to be an interior point of the convex cell $(f|_{\sigma})^{-1}\big(b(f(\sigma))\big)$, then $f:Sd K_{M'} \longrightarrow Sd K_M$ is simplicial, see the first Remark on page 225 of \cite{Cohen1967simplicial}.
	
	$(5)$ For any simplex $\sigma \in K_M$, let $\sigma^* \subset Sd K_M$ be its dual cell. By Proposition 5.2 and 5.6 in \cite{Cohen1967simplicial}, $f^{-1}(\sigma^*)$ is a $(m-|\sigma|)$-dimensional PL manifold with boundary $f^{-1}(\partial \sigma^*)$.
\end{Construction}
	
 Then we have the following description of $\sigma(f,b)$ given by Ranicki:

\begin{Theorem}[Proposition 18.3 in \cite{ranickiltheory}]
	\
	\label{simplicialsurgery}
	
	For any simplex $\sigma \in M$, by $(4)$ and $(5)$ in the simplicial setting \ref{simplicial setting}, the restriction of $(f,b)$ gives a degree 1 normal map from the $(n-|\sigma|)$-dimensional geometric Poincare pairs
	$(f^{-1}(\sigma^*),f^{-1}(\partial\sigma^*))$ to the $(n-|\sigma|)$-dimensional geometric normal pairs
	$(\sigma^*,\partial\sigma^*)$. Then there is a quadratic Poincare chain complex $(C,\psi)$ in $M^h(\mathbb{Z})_*(K_M)$ that corresponds to these normal maps, such that $(C,\psi)$ universal assembles to a quadratic Poincare $\mathbb{Z}\pi_1(M)$-chain complex $(C(\widetilde{M}),\psi(\widetilde{M}))$ representing $\sigma(f,b) \in L_n(\mathbb{Z}\pi_1(M))$.
\end{Theorem}

In the following step we will make a description of the quadratic chain complex $\Theta\rho_*(C(\widetilde{M},\psi(\widetilde{M})))$. Recall that $\Theta:M^h(\Sigma \mathbb{Z}\Pi) \longrightarrow \mathbb{F}_{\mathbb{N},b}(M^h(\mathbb{Z}\Pi))$ is the functor constructed in Section \ref{proofsuspension} in the case $R=\mathbb{Z}\Pi$. It induces an isomorphism on L-theory and deduces the isomorphism $L^h_*(\Sigma\mathbb{Z}\Pi) \cong L^h_*(\mathbb{F}_{\mathbb{N},b}(M^h(\mathbb{Z}\Pi))) \cong L^p_{*-1}(\mathbb{Z}\Pi)$. Therefore, it is important to get a description of $\Theta\rho_*(C(\widetilde{M},\psi(\widetilde{M})))$. Before stating the theorem, we make some preparations first:

The first definition is a function $d_{\infty}:W_{\infty} \longrightarrow [0,+\infty)$ which serves roughly as a distance function to $\partial W_{\infty}$.

\begin{Definition}
	\
	
	We construct a continous function $d_{\infty}:W_{\infty} \longrightarrow [0,\infty)$ by the following steps:
	
	$(1)$ For each vertex $v \in W_{\infty}$, let $d_{\infty}(v) \in \mathbb{N}$ to be the minimal number $l$, such that there exist $l+1$ vertices $v_0,v_1,...,v_l$, satisfying the following properties:
	
	$(a)$ $v_0=v,v_l \in \partial W_{\infty}=\bar{e}(N \times S^1)$.
	
	$(b)$ $v_i$ and $v_{i+1}$ is connected by a 1-dimensional simplex in $W_{\infty}$.
	
	$(2)$ We extend the map to the entire $W_{\infty}$ by linearly extending the map on every simplex. 
	
\end{Definition}

There are some straightforward properties of the map $d_{\infty}$, namely:

\begin{Lemma}
	\
	\label{distance}
	
	$(1)$ Let $\sigma,\tau \in W_{\infty}$ be simplices with $\tau \leq \sigma$. Suppose $v$ is a vertex of $\sigma$, such that $d_{\infty}(v)=d$. Then for all vertices $w$ of $\tau$, we have $d-1 \leq d_{\infty}(w) \leq d+1$.
	
	$(2)$ For every $s \geq 0$, $d_{\infty}^{-1}([0,s])$ is compact.
\end{Lemma}

Denote $f_{M,N}$ to be the quotient map: $M \longrightarrow M/M_0=\Sigma (N \times S^1)_+$, we also need to compute the map
\begin{equation*}
	(f_{M,N})_*:H_n(M;\underline{L}(\mathbb{Z})) \longrightarrow H_n(\Sigma (N \times S^1)_+;\underline{L}(\mathbb{Z})) \cong H_{n-1}(N \times S^1;\underline{L}(\mathbb{Z}))
\end{equation*}

in terms of Poincare quadratic chain complexes in $M^h(\mathbb{Z})_*(K_M)$. Before stating the lemma, we introduce a notation first.

\begin{Definition}
	For any simplex $\sigma \in N \times S^1$, define $A_{\sigma},B_{\sigma}$ to be the following sets of simplices:
	\begin{equation*}
		A_{\sigma}=\{s \in (N \times S^1) \otimes [0,1] \backslash (N \times S^1) \otimes \partial[0,1] \ | \ s \geq \sigma\}, \ B_{\sigma}=A_{\sigma} \backslash \mathop{\cup}\limits_{\sigma'>\sigma}A_{\sigma'}
	\end{equation*}
\end{Definition}

\begin{Remark}
	\label{RemB}
	For any simplex $s \in (N \times S^1) \otimes [0,1] \backslash (N \times S^1) \otimes \partial[0,1]$, denote $s_0$ to be the intersection of $s$ with $N \times S^1 \otimes 0$, it is a face of $s$. By Remark \ref{RemAB}, $B_{\sigma}=\{s \in (N \times S^1) \otimes [0,1] \backslash (N \times S^1) \otimes \partial[0,1] \ | \ s_0=\sigma \}$.
\end{Remark}

Then we can state the lemma:

\begin{Lemma}
	\
	\label{surgeryres}
	
	Let $(C,\psi)$ be an $n$-dimensional Poincare quadratic chain complex in  \\ $M^h(\mathbb{Z})_*(K_M)$. For any $r \in \mathbb{Z},s \in A_{\sigma}$, let $V_1$ be the set of all the vertices of $s$ that are in $N \times S^1 \otimes 1$. For any $v \in V_1$, denote $\sigma_v$ to be the simplex spanned by $\sigma$ and $v$. Denote $\iota_{v}$ to be the following inclusion map:
	\begin{equation*}
		\iota_{v}:C_{r-|\sigma|-1}(s)^* \longrightarrow C^r(\sigma_v)= \mathop{\oplus}\limits_{\kappa \geq \sigma_v}C_{r-|\sigma|-1}(\kappa)^*
	\end{equation*} 
	
	For every $r \in \mathbb{Z}$,
	let $\mho_{\sigma}^r:\mathop{\oplus}\limits_{s \in A_{\sigma}} C_{r-|\sigma|-1}(s)^* \longrightarrow \mathop{\oplus}\limits_{s \in B_{\sigma}} C^r(s)$ be the morphism determined by the following property:
	
	For any $s \in A_{\sigma}$, $\mho_{\sigma}^r$ restricted on the $C_{r-|\sigma|-1}(s)^*$ component is given by:
	\begin{equation*}
		\mho_{\sigma}^r(z)=\mathop{\oplus}\limits_{v \in V_1}\iota_{v}(z) \in \mathop{\oplus}\limits_{v \in V_1} C^r(\sigma_v) \subset \mathop{\oplus}\limits_{s \in B_{\sigma}}C^r(s)
	\end{equation*} 
	
	Let $x_0 \in L_n(M^h(\mathbb{Z})_*(K_M))$ be the element represented by $(C,\psi)$. Let $x_1$ be the element in $H_n(M;\underline{L}(\mathbb{Z}))$ corresponding to $x_0$ under the identification in Theorem \ref{localdual}. Then under the identification in Theorem \ref{localdual}, $(f_{M,N})_*(x_1)$ is represented by the following Poincare quadratic chain complex $(DL,\theta L)$ in $M^h(\mathbb{Z})_*(K_{N \times S^1})$:

	For $\sigma,\tau \in N \times S^1$ with $\tau \leq \sigma$ and $u \in \mathbb{N},r \in \mathbb{Z}$, the $(DL,\theta L)$ is given as follows:
	\begin{equation*}
		DL_r(\sigma)=\mathop{\oplus}\limits_{s \in B_{\sigma}}C_r(s)
	\end{equation*}
	\begin{equation*}
		d_{DL,r}(\tau,\sigma)=\mathop{\boxplus}\limits_{s \in B_{\sigma}}\mathop{\oplus}\limits_{s' \in B_{\tau}}d_{C,r}(s',s):DL_r(\sigma) \longrightarrow DL_{r-1}(\tau)
	\end{equation*}
	\begin{equation*}
		\begin{aligned}
			\theta L_u^r(\tau,\sigma):DL^{n-1-u-r}(\sigma)=\mathop{\oplus}\limits_{s \in A_{\sigma}}C_{n-1-u-r-|\sigma|}(s)^* \longrightarrow DL_r(\tau)=\mathop{\oplus}\limits_{s \in B_{\tau}}C_r(s) \\
		\end{aligned}
	\end{equation*}
	\begin{equation*}
			\theta L_u^r(\tau,\sigma)=(-1)^{n+|\sigma|+r+1}\big( \mathop{\boxplus}\limits_{s \in B_{\sigma}}\mathop{\oplus}\limits_{s' \in B_{\tau}}\psi_u^r(s',s) \big) \circ \mho_{\sigma}^{n-u-r}
	\end{equation*}
\end{Lemma}

\begin{proof}
	We briefly recall first how the identification
	\begin{equation*}
		H_n(M,\underline{L}(\mathbb{Z})) \cong L_n(M^h(\mathbb{Z})_*(K_M)) 
	\end{equation*}
	
	is obtained. Choose $l \in \mathbb{Z}$ sufficiently large so that we can embed $M$ simplicially and order-preservingly in $\partial\Delta^{l+1}$. Using the same notation in the previous subsection, by Theorem \ref{Lhomologyandco}, we can represent $x_1$ by the homotopy class of a $\Delta$-set map $f_{\Delta}:(\Sigma^l,\underline{M}) \longrightarrow (\mathbb{L}_{n-l}(M^h(\mathbb{Z})),\emptyset)$. By Theorem \ref{Lhomology}, $f_{\Delta}$ corresponds to a $(n-l)$-dimensional quadratic chain complex $(\widecheck{C},\widecheck{\psi})$ in $M^h(R)_*(\Sigma^l,\underline{M})$. Then its local dual  defined in Theorem \ref{localdual} is $(C,\psi)$.
	
	Then we proceed to compute $(f_{M,N})_*(x_1)$. Denote $M_{1}=M_0 \cup N \times S^1 \otimes [0,1], M_{-1}=M_0 \cup N \times S^1 \otimes [-1,0]$ and $\Sigma_1=f_{M,N}(M_1),\Sigma_{-1}=f_{M,N}(M_{-1})$.
	Then we have the following commutative diagram:
	\begin{equation*}
		\begin{tikzcd}
			H_n(M;\underline{L}(\mathbb{Z})) \dar{(j_-)_*}\rar{(f_{M,N})_*} & H_n(\Sigma (N \times S^1)_+;\underline{L}(\mathbb{Z})) \dar \\
			H_n(M,M_{-1};\underline{L}(\mathbb{Z})) \rar{(f_{M,N})_*} & H_n(\Sigma (N \times S^1)_+,\Sigma_{-1};\underline{L}(\mathbb{Z})) \\
			H_n(M_1,M_0 \cup (N \times S^1);\underline{L}(\mathbb{Z})) \uar{(i_+)_*}[swap]{\cong}\dar{\partial}\rar{(f_{M,N})_*} & H_n(\Sigma_1 ,(N \times S^1)_+;\underline{L}(\mathbb{Z})) \uar{\cong}\dar{\partial} \\
			H_{n-1}(M_0 \cup (N \times S^1),M_0;\underline{L}(\mathbb{Z})) \rar{(f_{M,N})_*} & H_{n-1}((N \times S^1)_+,pt;\underline{L}(\mathbb{Z})) \\
			H_{n-1}(N \times S^1;\underline{L}(\mathbb{Z})) \rar{Id}\uar{(i_0)_*}[swap]{\cong} & H_{n-1}((N \times S^1)_+,pt;\underline{L}(\mathbb{Z})) \uar{\cong} \\
		\end{tikzcd}
	\end{equation*}
	
	Here $i_0,i_+,j_-$ denote the corresponding inclusion maps of pairs of spaces. Therefore, we have $(f_{M,N})_*(x_1)=(i_0)_*^{-1}\partial(i_+)_*^{-1}(j_-)_*(x_1)$. Denote $M_0'=M_0 \cup (N \times S^1)$, we begin with the computation of $(i_+)_*^{-1}(j_-)_*(x_1)$ by the following commutative diagram:
		\begin{equation*}
		\begin{tikzcd}
			H_n(M;\underline{L}(\mathbb{Z})) \dar{(j_-)_*}\rar{\cong} & \mathop{[}\Sigma^l,\underline{M};\mathbb{L}_{n-l}(M^h(\mathbb{Z})),\emptyset\mathop{]} \dar{res} \\
			H_n(M,M_{-1};\underline{L}(\mathbb{Z})) \rar{\cong} & \mathop{[}\underline{M_{-1}},\underline{M};\mathbb{L}_{n-l}(M^h(\mathbb{Z})),\emptyset\mathop{]} \\
			H_n(M_1,M_0';\underline{L}(\mathbb{Z})) \uar{(i_+)_*}[swap]{\cong}\rar{\cong} & \mathop{[}\underline{M_0'},\underline{M_1};\mathbb{L}_{n-l}(M^h(\mathbb{Z})),\emptyset\mathop{]} \uar{\cong}[swap]{res} \\
		\end{tikzcd}
	\end{equation*}
	
	Here, the inverse of the isomorphism
	\begin{equation*}
		\mathop{[}\underline{M_0'},\underline{M_1};\mathbb{L}_{n-l}(M^h(\mathbb{Z})),\emptyset\mathop{]} \stackrel{res}{\longrightarrow} \mathop{[}\underline{M_{-1}},\underline{M};\mathbb{L}_{n-l}(M^h(\mathbb{Z})),\emptyset\mathop{]} 
	\end{equation*}
	is given by keeping the definition on $\underline{M_{-1}}$ and defining the corresponding chain complex on simplices in $\underline{M_0'} \backslash \underline{M_{-1}}$ to be $0$.
	
	Since $x_1$ corresponds to $f_{\Delta}$ by our definition, by the commutative diagram above, we have that $(i_+)_*^{-1}(j_-)_*(x_1)$ corresonding to the following map $f_{\Delta}'$:
	\begin{equation*}
		f_{\Delta}': (\underline{M_0'},\underline{M_1}) \longrightarrow (\mathbb{L}_{n-l}(M^h(\mathbb{Z})),\emptyset)
	\end{equation*}
	\begin{equation*}
		f_{\Delta}'(\sigma^*)=\begin{cases} f_{\Delta}(\sigma^*) & \text{If } \sigma \notin M_{-1} \\ \ \ \ \ 0 & \text{else} \end{cases}
	\end{equation*}

	By the formula in Theorem \ref{Lhomology}, $f_{\Delta}'$ corresponds to the following Poincare quadratic chain complex $(\widecheck{D},\widecheck{\theta})$ in $M^h(R)_*(\underline{M_0'},\underline{M_1})$:
	
	For all $u \in \mathbb{N},r \in \mathbb{Z},\sigma^*,\tau^* \in \underline{M_0 '}$ with $\tau^* \leq \sigma^*$, we have :
	\begin{equation*}
		\widecheck{D}_r(\sigma^*)=\begin{cases} \widecheck{C}_r(\sigma^*) & \text{If } \sigma \notin M_{-1} \\ \quad \ 0 & \text{else} \end{cases}, \ \widecheck{\theta}_u^r(\tau^*,\sigma^*)=\begin{cases} \widecheck{\psi}_u^r(\tau^*,\sigma^*) & \text{If } \sigma, \tau \notin M_{-1} \\ \qquad 0 & \text{else} \end{cases}
	\end{equation*}

	Thus its local dual $(D,\theta)$ is given by:
	\begin{equation*}
		D_r(\sigma)=\begin{cases} C_r(\sigma) & \text{If } \sigma \notin M_{-1} \\ \quad 0 & \text{else} \end{cases}, \ \theta_u^r(\tau,\sigma)=\begin{cases} \psi_u^r(\tau,\sigma) & \text{If } \sigma, \tau \notin M_{-1} \\ \quad \ 0 & \text{else} \end{cases}
	\end{equation*}
	
	For $\sigma,\tau \in N \times S^1$ and $u \in \mathbb{N},r \in \mathbb{Z}$, let $A_{\sigma},B_{\sigma}$ be the sets in the setting of Appendix with $K=M,L=N \times S^1$ and let $\mho_{\sigma}^r$ be the map defined in \ref{expressmho}, then these definitions agree with what we defined in the Lemma. By Corollary \ref{computationpartial}, we have that $(i_0)^{-1}\partial (i_+)_*^{-1}(j_-)_* (x_1)$ is given by the following Poincare quadratic chain complex $(DL,\theta L)$ in $M^h(R)_*(N \times S^1)$:
	
	\begin{equation*}
		DL_r(\sigma)=\mathop{\oplus}\limits_{s \in B_{\sigma}}D_r(s)=\mathop{\oplus}\limits_{s \in B_{\sigma}}C_r(s)
	\end{equation*}
	\begin{equation*}
		d_{DL,r}(\tau,\sigma)=\mathop{\boxplus}\limits_{s \in B_{\sigma}}\mathop{\oplus}\limits_{s' \in B_{\tau}}d_{D,r}(s',s)=\mathop{\boxplus}\limits_{s \in B_{\sigma}}\mathop{\oplus}\limits_{s' \in B_{\tau}}d_{C,r}(s',s)
	\end{equation*}
	\begin{equation*}
		\begin{aligned}
			\theta L_u^r(\tau,\sigma)&=(-1)^{n+|\sigma|+r+1}\big( \mathop{\boxplus}\limits_{s \in B_{\sigma}}\mathop{\oplus}\limits_{s' \in B_{\tau}}\theta_u^r(s',s) \big) \circ \mho_{\sigma}^{n-u-r} \\
			&=(-1)^{n+|\sigma|+r+1}\big( \mathop{\boxplus}\limits_{s \in B_{\sigma}}\mathop{\oplus}\limits_{s' \in B_{\tau}}\psi_u^r(s',s) \big) \circ \mho_{\sigma}^{n-u-r}
		\end{aligned}
	\end{equation*}
	
	Therefore, the Lemma holds.
\end{proof}

Now we can state the main theorem in this section to prove Theorem \ref{L theory transfer}.

\begin{Theorem}
	\label{mainproof}
	\
	
	Let $(C,\psi)$ be an $n$-dimensional Poincare quadratic chain complex in \\
	$M^h(\mathbb{Z})_*(K_M)$ and let $\Upsilon_{\hat{x},\hat{\sigma}}^r$ be as in Theorem \ref{assemblydual} with $K=W_{\infty},\widetilde{K}=\widehat{W}_{\infty},\mathbf{p}=\hat{p}$. Denote $(C(\widetilde{M}),\psi(\widetilde{M}))$ to be the universal assembly of $(C,\psi)$. For every $b \in \mathbb{N}$, denote $K^b \subset W_{\infty}$ to be the minimal subcomplex containing all vertices $x$ with $d_{\infty}(x) \leq b$. Let $K_0=K^0$ and $K_b=K^b \backslash K^{b-1}$ for $b \geq 1$. Denote $\widehat{K}_b=\hat{p}^{-1}(K_b)$. For any $\hat{\tau},\hat{\sigma} \in \widehat{W}_{\infty}$, $u \in \mathbb{N}$ and $r \in \mathbb{Z}$, let $d_r(\hat{\tau},\hat{\sigma})$ and $\psi_u^r(\hat{\tau},\hat{\sigma})$ be the morphism given by:
	
	\begin{align}
		&d_r(\hat{\tau},\hat{\sigma}): C_r(p\hat{p}\hat{\sigma}) \longrightarrow C_{r-1}(p\hat{p}\hat{\tau}) \\
		&d_r(\hat{\tau},\hat{\sigma}):=\begin{cases} d_{C,r}(p\hat{p}\hat{\tau},p\hat{p}\hat{\sigma}) & if \ \hat{\sigma} \leq \hat{\tau} \\ \ \ \ \ \ \ \ 0 & else\\ \end{cases} \label{dr} \\
		&\psi_u^r(\hat{\tau},\hat{\sigma}): C_{n-u-r}(p\hat{p}\hat{\sigma})^* \longrightarrow C_r(p\hat{p}\hat{\tau}) \\
		&\psi_u^r(\hat{\tau},\hat{\sigma}):=\begin{cases} \sum\limits_{\substack{\hat{x} \in \hat{\tau} \cap \hat{\sigma} \\ |\hat{x}|=0}} \psi_u^r(p\hat{p}\hat{\tau},p\hat{p}\hat{x}) \Upsilon^{n-u-r}_{\hat{x},\hat{\sigma}} & if \  \hat{\tau} \cap \hat{\sigma} \neq \emptyset \\ \ \ \ \ \ \ \ \ \ \ \ \ \ \ 0 & else\end{cases} \label{psir}
	\end{align}

	Let $(D,\theta)$ be the following quadratic chain complex in $\mathbb{F}_{\mathbb{N},b}(M^h(\mathbb{Z}\Pi))$:
	
	$D_r(b)=\mathop{\oplus}\limits_{\hat{\sigma} \in \widehat{K}_{b}} C_r(p\hat{p}\hat{\sigma})$ for $b>0$ 
	and $D_r(0)=\mathop{\oplus}\limits_{\hat{\sigma} \in \widehat{K}_0} C_r(p\hat{p}\hat{\sigma}) \oplus \big(\mathop{\oplus}\limits_{\sigma \notin W} C_r(\sigma) \otimes_{\mathbb{Z}}\mathbb{Z}\Pi \big)$
	 
	 $d_{D,r}$ is the equivalence class of the following morphism in $\mathbb{F}_{\mathbb{N}}(M^h(\mathbb{Z}\Pi))$:
	 \begin{equation}
	 	\label{bardD}
	 	 \bar{d}_{D,r}(b',b)=\mathop{\boxplus}\limits_{\hat{\sigma} \in \widehat{K}_{b}}\mathop{\oplus}\limits_{\hat{\tau} \in \widehat{K}_{b'}} d_r(\hat{\tau},\hat{\sigma}): D_r(b) \longrightarrow D_{r-1}(b')
	 \end{equation}
	 
	 $\theta_u^r:D_{n-u-r}^* \longrightarrow D_r$ is the equivalence class of the following morphism in $\mathbb{F}_{\mathbb{N}}(M^h(\mathbb{Z}\Pi))$:
	 \begin{equation}
	 	\label{bartheta}
	 	\overline{\theta}_u^r(b',b)=\mathop{\boxplus}\limits_{\hat{\sigma} \in \widehat{K}_{b}}\mathop{\oplus}\limits_{\hat{\tau} \in \widehat{K}_{b'}}\psi_u^r(\hat{\tau},\hat{\sigma}):D_{n-u-r}(b)^* \longrightarrow D_r(b')
	 \end{equation}

	Suppose that $\Gamma/\pi$ is an infinite set, then $(D,\theta) \cong \Theta\rho_*(C(\widetilde{M}),\psi(\widetilde{M}))$.
	
	Moreover, we have that $(C,\psi)$ represents an element in $L_n(M^h(\mathbb{Z})_*(K_M))=H_n(K_M,\underline{L}(\mathbb{Z}))$. Let $(C,\psi)|_{N \times S^1} \in H_{n-1}(N \times S^1,\underline{L}(\mathbb{Z}))$ be its image under the homomorphism $(f_{M,N})_*$, then there is a quadratic chain complex $(E,\xi)$ in $\mathbb{F}_{\mathbb{N}}(M^h(\mathbb{Z}\Pi))$, such that $[(E,\xi)]=(D,\theta)$ and $\partial(E,\xi)$ is cobordant to $(C_{res}^{uni},\psi_{res}^{uni})$, where $(C_{res}^{uni},\psi_{res}^{uni})$ is the universal assembly of $(C,\psi)|_{N \times S^1}$.
\end{Theorem}

\begin{proof}
	Let us briefly describe the ideas of the proof. We can lift $(C,\psi) $ with respect to the covering $p:\overline{M} \longrightarrow M$. Consider the assembly of its restriction to $W_{\infty}$ with respect to the covering $\hat{p}:\widehat{W}_{\infty} \longrightarrow W_{\infty}$, there are two ways to decompose the module in every degree of the chain complex $\mathbf{C}_r=\mathop{\oplus}\limits_{\hat{\sigma} \in \widehat{W}_{\infty}}C_r(p\hat{p}\hat{\sigma})$ into a direct sum of modules indexed by $\mathbb{N}$.
	
	The first way is to use the function $d_{\infty}$ and this way corresponds to the expression of $D$. The second way is as follows: note that $\hat{p}\hat{\sigma} \in \overline{M}$ can be indexed by $\Gamma/\pi$ and simplices in $M$. Since $\Gamma/\pi \cong \mathbb{N}$, we can get an $\mathbb{N}$-index and this way corresponds to $\Theta\rho_*C$. Since reordering is an isomorphism in $\mathbb{F}_{\mathbb{N},b}(M^h(\mathbb{Z}\Pi))$. The first statement will hold.
	
	Moreover, roughly speaking, $\mathbf{C}_*$ represents the "surgery problem" of some normal map $\widehat{W}''_{\infty} \longrightarrow \widehat{W}_{\infty}$ and taking boundary in quadratic chain complex corresponds to restricting the surgery problem to the boundary. Thus we have the second statement. We will now present the details of the proof.

	($\mathbf{1}$) In order to give an explicit description for $C_*(\widetilde{M})$ and $D_*$, we need to first write down all the simplices in $M$ and choose their liftings to $\widetilde{M}$ and $\widehat{W}_{\infty}$. We introduce some notations first, recall the commutative diagram in the geometric setting of maps between fundamental groups induced by inclusions and covering projections:
	$$
	\begin{tikzcd}
		\pi_1(\overline{W})=H \rar{i_*} \dar[hook]{p_*} & \pi_1(W_{\infty}) \rar{i'_*} & \pi_1(\overline{M})=\pi \dar[hook]{p_*} \\
		\pi_1(W)=G \ar[two heads]{rr}{j_*} & & \pi_1(M)=\Gamma 
	\end{tikzcd}
	$$
	
	By the properties of the covering map, we have $G/H \cong \Gamma/\pi$. Since $\Gamma / \pi$ is an infinite set, let $\{g_b\}_{b \in \mathbb{N}}$ be a sequence of representatives of $\Gamma/\pi$ with $g_0=E$, where $E$ denotes the identity element. Let $\{\tilde{g}_b\}_{b \in \mathbb{N}}$ be a sequence of representatives of $G/H$, such that $\tilde{g}_0=E$ and $j_*(\tilde{g}_b)=g_b$ for all $b \in \mathbb{N}$. Let $\sigma_1,\sigma_2,...,\sigma_{l_0}$ be all the simplices that are in $W=M \backslash (N \times \mathring{D}^2) \subset M$ but not in $N \times S^1$. Let $\sigma_{l_0+1},....,\sigma_{l_0+l_1}$ be all the simplices in $N \times S^1$ and $\sigma_{l_0+l_1+1},...,\sigma_{l_0+l_1+l_2}$ be all the remaining simplices in $M$. Then $W=\mathop{\cup}\limits_{1 \leq i \leq l_0+l_1}\sigma_i$ and $N \times D^2=\mathop{\cup}\limits_{l_0+1 \leq i \leq l_0+l_1+l_2}\sigma_i$.
	
	($\mathbf{2}$) In the next step we choose carefully the liftings of these simplices to $\widetilde{M}$ and to $\widehat{W}_{\infty}$. Before the choice of the liftings, we investigate more about the geometric settings. For $b>0$, since $g_b^{-1}\tilde{e}(\widetilde{N} \times D^2) \cap \tilde{e}(\widetilde{N} \times D^2)=\emptyset$, by (10) in the geometric setting \ref{geometric setting}, we have $\acute{p}^{-1}(g_b^{-1}\tilde{e}(\widetilde{N} \times D^2)) \cong \widetilde{N} \times D^2 \times \mathbb{Z}$. By the equivariance of covering maps in (9) of the setting \ref{geometric setting}, we have $\acute{p}\vec{p}(\tilde{g}_b^{-1}\vec{e}(\widetilde{N} \times \mathbb{R}^1))=g_b^{-1}\bar{e}(\widetilde{N} \times S^1)$. Therefore, we have:
	\begin{equation*}
		 \vec{p}(\tilde{g}_b^{-1}\vec{e}(\widetilde{N} \times \mathbb{R}^1)) \subset \acute{p}^{-1}(g_b^{-1}\tilde{e}(\widetilde{N} \times S^1)) \cong \widetilde{N} \times S^1 \times \mathbb{Z}
	\end{equation*}
	
	Since $\vec{p}(\tilde{g}_b^{-1}\vec{e}(\widetilde{N} \times \mathbb{R}^1))$ is path-connected, we can denote the unique path-connected component of $\acute{p}^{-1}(g_b^{-1}\tilde{e}(\widetilde{N} \times D^2))$ where $\vec{p}(\tilde{g}_b^{-1}\vec{e}(\widetilde{N} \times \mathbb{R}^1))$ lies in to be $N_b$. Then $\acute{p}:N_b \longrightarrow g_b^{-1}\tilde{e}(\widetilde{N} \times D^2)$ is a homeomorphism. 
	
	Now we begin choosing the liftings of the simplices. For every $1 \leq i \leq l_0$, choose a simplex $\vec{\sigma}_i \in \overrightarrow{W}$ that projects to $\sigma_i$. Let $\tilde{\sigma}_i=\acute{p}\vec{p}(\vec{\sigma}_i) \in \widetilde{M}$ and for every $b \in \mathbb{N}$, let $\hat{\sigma}_{b,i}=\vec{p}(\tilde{g}_b^{-1}\vec{\sigma}_i)$. For every $l_0+1 \leq i \leq l_0+l_1$, choose a simplex $\vec{\sigma}_i \in \vec{e}(\widetilde{N} \times \mathbb{R}^1) \subset \overrightarrow{W}$ that projects to $\sigma_i$. Let $\tilde{\sigma}_i=\acute{p}\vec{p}(\vec{\sigma}_i) \in \widetilde{M}$ and for every $b \in \mathbb{N}$, let $\hat{\sigma}_{b,i}=\vec{p}(\tilde{g}_b^{-1}\vec{\sigma}_i)$, then $\tilde{\sigma}_i \in \tilde{e}(\widetilde{N} \times D^2)$ by definition. For every $l_0+l_1+1 \leq i \leq l_0+l_1+l_2$, choose a simplex $\tilde{\sigma}_i \in \tilde{e}(\widetilde{N} \times D^2)$ that projects to $\sigma_i$. For every $b>0$, since $\acute{p}:N_b \longrightarrow g_b^{-1}\tilde{e}(\widetilde{N} \times D^2)$ is a homeomorphism, there is a unique simplex $\hat{\sigma}_{b,i} \in N_b$, such that $\acute{p}(\hat{\sigma}_{b,i})=g_b^{-1}\tilde{\sigma}_i$. In summary, we get simplices $\tilde{\sigma}_i \in \widetilde{M}$ that lift $\sigma_i$ and simplices $\hat{\sigma}_{b,i}$ ($b \neq 0$ if $l_0+l_1+1 \leq i \leq  l_0+l_1+l_2$) that lifts $g_b^{-1}\tilde{\sigma}_i$.
	
	Then we consider the simplex decomposition of the spaces in the geometric setting \ref{geometric setting}. Since $W=\mathop{\cup}\limits_{1 \leq i \leq l_0+l_1}\sigma_i$, we get $\widetilde{W}=\mathop{\cup}\limits_{\substack{1 \leq i \leq l_0+l_1 \\ g \in \Gamma}}g\tilde{\sigma}_i$. Note that by our choice, for any $\sigma_i \in N \times D^2$ (that is the same with $l_0+1 \leq i \leq l_0+l_1+l_2$), we have $\tilde{\sigma}_i \in \tilde{e}(\widetilde{N} \times D^2)$. Thus $\widetilde{W}_{\infty}=\widetilde{M} \backslash \tilde{e}(\widetilde{N} \times \mathring{D}^2)=\widetilde{W}\cup (\mathop{\cup}\limits_{\substack{l_0+l_1+1 \leq i \leq  l_0+l_1+l_2 \\ b>0,v \in \pi}}vg_b^{-1}\tilde{\sigma}_i)$. We have  $W_{\infty}=\big(\mathop{\cup}\limits_{\substack{1 \leq i \leq  l_0+l_1 \\ b \geq 0}}\underline{p}(g_b^{-1}\tilde{\sigma}_i)\big) \cup \big(\mathop{\cup}\limits_{\substack{l_0+l_1+1 \leq i \leq  l_0+l_1+l_2 \\ b>0}}\underline{p}(g_b^{-1}\tilde{\sigma}_i)\big)$. Note that by our choice, every $\hat{\sigma}_{b,i}$ is a lifting of $\underline{p}(g_b^{-1}\tilde{\sigma}_i)$, thus $\widehat{W}_{\infty}=\big(\mathop{\cup}\limits_{\substack{1 \leq i \leq  l_0+l_1 \\ b \geq 0, \omega \in \Pi}}\omega\hat{\sigma}_{b,i}\big) \cup \big(\mathop{\cup}\limits_{\substack{l_0+l_1+1 \leq i \leq  l_0+l_1+l_2 \\ b>0,\omega \in \Pi}}\omega\hat{\sigma}_{b,i}\big)$. In other words, the $\hat{\sigma}_{b,i}$ constructed above translated under $\Pi$ action will formulate bijectively all the simplices in $\widehat{W}_{\infty}$.
	
	The crucial property of the choice of $\tilde{\sigma}_i$ and $\hat{\sigma}_{b,i}$ is the following claim:
	\begin{equation*}
		\label{cclaim}
		\begin{aligned}
			& \textbf{Claim \ref{cclaim}: } \text{For every two simplices } \sigma_i,\sigma_j \in M \text{ with } \sigma_i \leq \sigma_j, \text{let } g_{i,j} \text{ be the}\\
			& \text{unique element in } \Gamma \text{ such that } \tilde{\sigma}_i \leq g_{i,j}\tilde{\sigma}_j. \text{ Then there is an element } \tilde{g}_{i,j} \in G,\\ 
			& \text{such that:}
		\end{aligned}
	\end{equation*}
	
	(1) $j_*(\tilde{g}_{i,j})=g_{i,j}$.
	
	(2) Let $b \in \mathbb{N}$ be any number, let $b' \in \mathbb{N},h \in H$ be the unique element such that $\tilde{g}_{i,j}^{-1}\tilde{g}_b=\tilde{g}_{b'}h$. If $\underline{p}(g_{b}^{-1}\tilde{\sigma}_i),\underline{p}(g_{b'}^{-1}\tilde{\sigma}_j) \in W_{\infty}$, (which means $b>0$ if $l_0+l_1+1 \leq i \leq l_0+l_1+l_2$, similarly for $b',j$), then $\hat{\sigma}_{b,i} \leq r_{\Pi}(h^{-1})\hat{\sigma}_{b',j}$.
	
	To prove the claim, we divide it into several cases, depending on whether $\sigma_i,\sigma_j$ lies in $W$ or not.
	
	(a) If $\sigma_i,\sigma_j \in W$, then we have $\underline{p}(g_{b}^{-1}\tilde{\sigma}_i),\underline{p}(g_{b'}^{-1}\tilde{\sigma}_j) \in W_{\infty}$. Since we have chosen their lift to the universal cover of $W$, there is a unique element $\tilde{g}_{i,j} \in G$, such that $\overrightarrow{\sigma}_i \leq \tilde{g}_{i,j}\overrightarrow{\sigma}_j$. Then $\tilde{\sigma}_i=\acute{p}\vec{p}(\vec{\sigma}_i) \leq \acute{p}\vec{p}(\tilde{g}_{i,j}\overrightarrow{\sigma}_j)=j_*(\tilde{g}_{i,j})\tilde{\sigma}_j$, thus we get $j_*(\tilde{g}_{i,j})=g_{i,j}$. Property (1) is verified.
	
	By definition of $\hat{\sigma}_{b,i}$, we have:
	\begin{equation*}
		\hat{\sigma}_{b,i}=\vec{p}(\tilde{g}_b^{-1}\vec{\sigma}_i)\leq \vec{p}(\tilde{g}_b^{-1}\tilde{g}_{i,j}\overrightarrow{\sigma}_j)=\vec{p}(h^{-1}\tilde{g}_{b'}^{-1}\overrightarrow{\sigma}_j)=r_{\Pi}(h^{-1})\hat{\sigma}_{b',j}
	\end{equation*}

	Therefore, property (2) is verified.
	
	(b) If $\sigma_i \in W, \sigma_j \notin W$, since $\sigma_i \leq \sigma_j$, we have $\sigma_i \in N \times S^1$, that is, $l_0+1 \leq i \leq l_0+l_1$. Then by the choice of $\tilde{\sigma}_i$, we have $\tilde{\sigma}_i,\tilde{\sigma}_j \in \tilde{e}(\widetilde{N} 
	\times D^2)$, thus $g_{i,j} \in \pi$. Choose $\tilde{g}_{i,j}=g_{i,j} \in \pi \subset G$, then $j_*(\tilde{g}_{i,j})=g_{i,j}$.
	
	For the second property, suppose $r_{\Pi}(h^{-1})=cv \in \Pi$ with $v \in \pi$, $c \in \mathbb{Z}$. By (5) in the geometric setting \ref{geometric setting}, we have:
	\begin{equation*}
		v=p_{\pi}r_{\Pi}(h^{-1})=p_{\pi}r_0i_*(h^{-1})=i'_*i_*(h^{-1})=j_*(h^{-1})
	\end{equation*}

	Thus we get $j_*(h)=v^{-1}$. Since by assumption $\tilde{g}_{i,j}^{-1}\tilde{g}_b=\tilde{g}_{b'}h$, we get $g_{i,j}^{-1}g_b=g_{b'}v^{-1}$. Since $\sigma_j \notin W$, if $\underline{p}(g_{b'}^{-1}\sigma_j) \in W_{\infty}$, then we have $b'>0$ and thus $g_{b'} \notin \pi$. Since $g_{i,j} \in \pi$, we get $g_b \notin \pi$ and thus $b>0$. Therefore, we have the definition for $N_b$ and $N_{b'}$ as above. Now by definition of $\hat{\sigma}_{b',j}$, it is the unique simplex in $N_{b'}$ that projects to $g_{b'}^{-1}\tilde{\sigma}_j$ under the homeomorphism $\acute{p}:N_{b'} \longrightarrow g_{b'}^{-1}\tilde{e}(\widetilde{N} \times D^2)$. Then we have $r_{\Pi}(h^{-1})\hat{\sigma}_{b',j} \in r_{\Pi}(h^{-1})N_{b'}=cvN_{b'}$. By the equivariance of $\acute{p}$, we have the following commutative diagram:
	\begin{equation*}
			\begin{tikzcd}
			N_{b'} \rar{\acute{p}}[swap]{\cong} \dar{\cong}[swap]{cv \cdot} & g_{b'}^{-1}\tilde{e}(\widetilde{N} \times D^2) \dar{v \cdot}[swap]{\cong} \\
			cvN_{b'} \rar{\acute{p}}  & vg_{b'}^{-1}\tilde{e}(\widetilde{N} \times D^2) \\
		\end{tikzcd}
	\end{equation*}

	Thus $\acute{p}:cvN_{b'} \longrightarrow vg_{b'}^{-1}\tilde{e}(\widetilde{N} \times D^2)$ is a homeomorphism. Now by definition we have $\hat{\sigma}_{b,i}=\vec{p}(\tilde{g}_b^{-1}\vec{\sigma}_i)$ and $\vec{\sigma}_i \in \vec{e}(\widetilde{N} \times \mathbb{R}^1)$, thus $\hat{\sigma}_{b,i} \in \vec{p}(\tilde{g}_b^{-1} \vec{e}(\widetilde{N} \times \mathbb{R}^1)) \subset N_b$. Thus it leaves to prove that $N_b=cvN_{b'}$. If this holds, notice that $\acute{p}:cvN_{b'} \longrightarrow vg_{b'}^{-1}\tilde{e}(\widetilde{N} \times D^2)$ is a homeomorphism and $\hat{\sigma}_{b,i},r_{\Pi}(h^{-1})\hat{\sigma}_{b',j} \in cvN_{b'}$ are liftings of $g_b^{-1}\tilde{\sigma}_i ,vg_{b'}^{-1}\tilde{\sigma}_j$. Since we have $g_b^{-1}\tilde{\sigma}_i \leq g_b^{-1}g_{i,j}\tilde{\sigma}_j=vg_{b'}^{-1}\tilde{\sigma}_j\in vg_{b'}^{-1}\tilde{e}(\widetilde{N} \times D^2)$, we get $\hat{\sigma}_{b,i} \leq r_{\Pi}(h^{-1})\hat{\sigma}_{b',j}$.
	
	To prove $N_b=cvN_{b'}$, note first that since $g_{i,j}^{-1}g_b=g_{b'}v^{-1}$ and $g_{i,j} \in \pi$, we have $vg_{b'}^{-1}\tilde{e}(\widetilde{N} \times D^2)=g_b^{-1}g_{i,j}\tilde{e}(\widetilde{N} \times D^2)=g_b^{-1}\tilde{e}(\widetilde{N} \times D^2)$. Thus by definition of $N_b$ and the fact that $\acute{p}$ is homeomorphism on $cvN_{b'}$, both of them are path-connected components of $\acute{p}^{-1}(g_b^{-1}\tilde{e}(\widetilde{N} \times D^2)) \cong \widetilde{N} \times D^2 \times \mathbb{Z}$. By definition, we have $\vec{p}(\tilde{g}_b^{-1} \vec{e}(\widetilde{N} \times \mathbb{R}^1)) \subset N_b$ and $\vec{p}(\tilde{g}_{b'}^{-1} \vec{e}(\widetilde{N} \times \mathbb{R}^1)) \subset N_{b'}$. Then we have:
	\begin{equation*}
		\begin{aligned}
			cvN_{b'} \supset cv\vec{p}(\tilde{g}_{b'}^{-1} \vec{e}(\widetilde{N} \times \mathbb{R}^1))=r_{\Pi}(h^{-1})\vec{p}(\tilde{g}_{b'}^{-1} \vec{e}(\widetilde{N} \times \mathbb{R}^1))&=\vec{p}(h^{-1}\tilde{g}_{b'}^{-1} \vec{e}(\widetilde{N} \times \mathbb{R}^1)) \\
			&=\vec{p}(\tilde{g}_{b}^{-1}\tilde{g}_{i,j} \vec{e}(\widetilde{N} \times \mathbb{R}^1)) \\
			&=\vec{p}(\tilde{g}_{b}^{-1} \vec{e}(\widetilde{N} \times \mathbb{R}^1)) \\
		\end{aligned}
	\end{equation*}
	
	Thus we get $N_{b} \cap cvN_{b'} \neq \emptyset$, but they are also path connected components of $\acute{p}^{-1}(g_b^{-1}\tilde{e}(\widetilde{N} \times D^2))$. Therefore, we have $N_b=cvN_{b'}$, completing the proof for this case.
	
	(c) If $\sigma_i \notin W,\sigma_j \in W$, since $\sigma_i \leq \sigma_j$ and $W$ is a subcomplex, this case can not happen.
	
	(d) If $\sigma_i,\sigma_j \notin W$, then by the choice of $\tilde{\sigma}_i$, we have $\tilde{\sigma}_i,\tilde{\sigma}_j \in \tilde{e}(\widetilde{N} \times D^2)$, thus $g_{i,j} \in \pi$. Choose $\tilde{g}_{i,j}=g_{i,j}$, then $j_*(\tilde{g}_{i,j})=g_{i,j}$.
	
	For the second property, the proof idea is similar to case (b). Suppose also that $r_{\Pi}(h^{-1})=cv$ with $v \in \pi$, $c \in \mathbb{Z}$. Similarly to the proof in (b), we get $g_{i,j}^{-1}g_b=g_{b'}v^{-1}$ and furthermore, $N_b=cvN_{b'}$ and $\acute{p}:cvN_{b'} \longrightarrow vg_{b'}^{-1}\tilde{e}(\widetilde{N} \times D^2)$ is a homeomorphism. By our choice of $\hat{\sigma}_{b,i}$, we have that $\hat{\sigma}_{b,i},cv\hat{\sigma}_{b',j}$ are liftings  in $N_b=cvN_{b'}$ of $g_b^{-1}\tilde{\sigma}_i, vg_{b'}^{-1}\tilde{\sigma}_j \in vg_{b'}^{-1}\tilde{e}(\widetilde{N} \times D^2)$ and $g_b^{-1}\tilde{\sigma}_i \leq g_b^{-1}g_{i,j}\tilde{\sigma}_j=vg_{b'}^{-1}\tilde{\sigma}_j$. Thus we have $\hat{\sigma}_{b,i} \leq cv\hat{\sigma}_{b',j}=r_{\Pi}(h^{-1})\hat{\sigma}_{b',j}$.

	($\mathbf{3}$) Now we begin to write down the assembly of $(C,\psi)$ explictly. Let $l=l_0+l_1+l_2$. For every $1 \leq i \leq l$, let $S_i=\{j \ | \ \sigma_i \leq \sigma_j\}$. For every $1 \leq i,j \leq l$, if $j \in S_i$, then there is a unique element $g_{i,j} \in \Gamma$, such that, $\tilde{\sigma}_i \leq g_{i,j}\tilde{\sigma}_j$. Let $g(j,i)=\begin{cases} g_{i,j}^{-1} & if \ j \in S_i \\	0 & else\\	\end{cases}$. Let $\tilde{g}_{i,j} \in G$ be the element stated in Claim \ref{cclaim} for every pair $(i,j)$ with $j \in S_i$ and let $\tilde{g}(j,i)=\begin{cases} \tilde{g}_{i,j}^{-1} & if \ j \in S_i \\	0 & else\\	\end{cases}$. Then by the property (1) in Claim \ref{cclaim}, we have $j_*(\tilde{g}(j,i))=g(j,i)$.
	
	For $k \in \mathbb{N}$ and a discrete group $G'$, denote $\epsilon_1^{G'},...,\epsilon_k^{G'}$ to be the standard $\mathbb{Z}G'$-basis of $\mathbb{Z}G'^k$ and $(\epsilon_1^{G'})^*,...,(\epsilon_k^{G'})^*$ to be the $\mathbb{Z}G'$-dual basis. For $r \in \mathbb{Z}$ and $1 \leq i \leq l$, let $N_r(i)=rank \ C_r(\sigma_i)$ and let $\{e_r^s(i)\}_{1 \leq s \leq N_r(i)}$ be a $\mathbb{Z}$-basis of $C_r(\sigma_i)$. Denote $\{e_r^s(i)^*\}_{1 \leq s \leq N_r(i)}$ to be $\mathbb{Z}$-dual basis. Denote $\mathcal{I}_{r,i}$ and $\mathcal{P}_{r,i}$ to be the inclusion map $C_r(\sigma_i) \longrightarrow \mathop{\oplus}\limits_{i=1}^l C_r(\sigma_i)$ and the projection map $\mathop{\oplus}\limits_{i=1}^l C_r(\sigma_i) \longrightarrow C_r(\sigma_i)$ respectively. For every $r,u \in \mathbb{Z}, 1 \leq i,j \leq l$ and vertex $x \in \sigma_i \cap \sigma_j$, let $d_r(j,i) \in M_{N_{r-1}(j),N_r(i)}(\mathbb{Z})$ be the matrix of the morphism $\mathcal{P}_{r-1,j} \circ d \circ \mathcal{I}_{r,i}:C_r(\sigma_i) \longrightarrow C_{r-1}(\sigma_j)$ with respect to the basis $\{e_r^s(i)\}_{1 \leq s \leq N_r(i)},\{e_{r-1}^{s'}(i)\}_{1 \leq s' \leq N_{r-1}(i)}$. Let $\psi_{u,x}^r(j,i) \in M_{N_r(j),N_{n-u-r}(i)}(\mathbb{Z})$ be the matrix of the morphism 
	\begin{equation*}
		\varphi_{u,x}^r(j,i): \begin{tikzcd} C_{n-u-r}(\sigma_i)^* \ar{rr}{\text{inclusion}} & & C^{n-u-r}(x) \ar{rr}{\psi_u^r(x,\sigma_j)} & & C_r(\sigma_j) \end{tikzcd}
	\end{equation*}
	with respect to the basis $\{e_{n-u-r}^s(i)^*\}_{1 \leq s \leq N_{n-u-r}(i)},\{e_{r}^{s'}(i)\}_{1 \leq s' \leq N_{r}(i)}$.
	
	Then using the isomorphism:
	\begin{equation*}
		\mathop{\oplus}\limits_{g \in \Gamma}C_r(\tilde{p}g\tilde{\sigma}_i) \stackrel{ \cong}{\longrightarrow}\mathbb{Z}\Gamma^{N_r(i)}: \mathop{\oplus}\limits_{g \in \Gamma}\sum\limits_{1 \leq s \leq N_r(i)}a(s,g)e_r^s(i) \mapsto \sum\limits_{1 \leq s \leq N_r(i)}\sum\limits_{g \in \Gamma}a(s,g)g^{-1}\epsilon_s^{\Gamma}
	\end{equation*}
	we can identify the assembly of $(C,\psi)$ with the following quadratic chain complex $(C',\psi')$ in $M^h(\mathbb{Z}\Gamma)$:
	
	For all $r \in \mathbb{Z}$,
		\begin{gather*}
			C_r'=\mathop{\oplus}\limits_{1 \leq i \leq l}\mathbb{Z}\Gamma^{N_r(i)} \\
			d_r'=\mathop{\boxplus}\limits_{i=1}^l\mathop{\oplus}\limits_{j=1}^ld_r(j,i)g(j,i):C_r' \longrightarrow C_{r-1}' \\
			{\psi'_u}^r=\mathop{\boxplus}\limits_{i=1}^{l}\mathop{\oplus}\limits_{j=1}^l \sum\limits_{\substack{\sigma_k \leq \sigma_i,\sigma_j \\ |\sigma_k|=0}}\psi_{u,\sigma_k}^r(j,i)g(j,k)g(i,k)^{-1} :{C_{n-u-r}'}^* \longrightarrow C_r' \\
		\end{gather*}

	Thus $(C'',\psi'')=\Theta\rho_*(C,\psi)$ is the following quadratic chain complex:
	\begin{equation*}
		C''_r(b)=\mathop{\oplus}\limits_{i=1}^l\mathbb{Z}\Pi^{N_r(i)} \text{ for all }b \in \mathbb{N}. 
	\end{equation*}
	
	$d''_r:C''_r \longrightarrow C''_{r-1}$ is the equivalence class of the following morphism in $\mathbb{F}_{\mathbb{N}}(M^h(\mathbb{Z}\Pi))$:
	\begin{equation}
		\label{bard}
		\begin{aligned}
			\bar{d}_r(b',b): C''_r(b) \longrightarrow C''_{r-1}(b'), \ b,b' \in \mathbb{N}\\
			\bar{d}_r(b',b)=\mathop{\boxplus}\limits_{i=1}^{l}\mathop{\oplus}\limits_{j=1}^l d_r(j,i)\rho_G\big( \tilde{g}(j,i) \big)[b',b] \\
		\end{aligned}
	\end{equation}
	
	${\psi''_u}^{r}:{C''_{n-u-r}} ^*\longrightarrow C''_{r}$ is the equivalence class of the following morphism in $\mathbb{F}_{\mathbb{N}}(M^h(\mathbb{Z}\Pi))$:
	\begin{equation}
		\label{barpsi}
		\begin{aligned}
			\bar{\psi}_u(b',b):C''_{n-u-r}(b)^* \longrightarrow C''_{r}(b'), \ b,b' \in \mathbb{N} \ \ \ \ \ \ \ \ \ & \\
			\bar{\psi}_u(b',b)=\mathop{\boxplus}\limits_{i=1}^{l}\mathop{\oplus}\limits_{j=1}^l \sum\limits_{\substack{\sigma_k \leq \sigma_i,\sigma_j \\ |\sigma_k|=0}}\psi_{u,\sigma_k}^r(j,i)\rho_G\big(\tilde{g}(j,k)\tilde{g}(i,k)^{-1}\big)[b',b] & \\
		\end{aligned}
	\end{equation}
	
	Let us make some explanations of the expressions above, recall that $\rho:\Gamma \longrightarrow \Sigma \mathbb{Z}\Pi$ is induced from $\rho_G:G \longrightarrow M_{\infty}(\mathbb{Z}\Pi)$. We can extend $\rho_G$ linearly to a homomorphism, still denoted by $\rho_G:\mathbb{Z}G \longrightarrow M_{\infty}(\mathbb{Z}\Pi)$. Then $\rho_G\big(\tilde{g}(j,i)\big)$ are elements in $M_{\infty}(\mathbb{Z}\Pi)$ and $\rho_G\big(\tilde{g}(j,i)\big)[b',b] \in \mathbb{Z}\Pi$ is the element in the $b'$th row, $b$th column in the matrix, similarly for $\rho_G\big(\tilde{g}(j,k)\tilde{g}(i,k)^{-1}\big)[b',b]$.
	
	($\mathbf{4}$) In the next step we proceed to construct an isomorphism of $D$ and $C''$. To start with, for $b_0 \in \mathbb{N}$, let $\bar{S}_{b_0}=\{(b,i) \ | \ g_b^{-1}\tilde{\sigma}_i \in \widetilde{W}_{\infty}, \underline{p}(g_b^{-1}\tilde{\sigma}_i) \in K_{b_0}\}$, it is a finite set by Lemma \ref{distance}. Let $\bar{S}=\mathop{\cup}\limits_{b_0=0}^{+\infty} \bar{S}_{b_0}$, then $W_{\infty}=\mathop{\cup}\limits_{(b,i) \in \bar{S}}\underline{p}(g_b^{-1}\tilde{\sigma}_i)$. Note that for $b_0>0$, we have:
	\begin{equation}
		\label{D}
		\begin{aligned}
			D(b_0)=\mathop{\oplus}\limits_{\hat{\sigma} \in \widehat{K}_{b_0}} C_*(p\hat{p}\hat{\sigma})=\mathop{\oplus}\limits_{(b,i) \in \bar{S}_{b_0}}\mathop{\oplus}\limits_{\omega \in \Pi}C_*(p\hat{p}\omega\hat{\sigma}_{b,i}) & \\
			D(0)=\big(\mathop{\oplus}\limits_{(b,i) \in \bar{S}_{0}}\mathop{\oplus}\limits_{\omega \in \Pi}C_*(p\hat{p}\omega\hat{\sigma}_{b,i})\big) \oplus \big(\mathop{\oplus}\limits_{\sigma \notin W} C_*(\sigma) \otimes_{\mathbb{Z}}\mathbb{Z}\Pi\big) \ \ \ \ & \\
		\end{aligned}
	\end{equation}
	
	Now denote $\Xi_{r,b,i}$ and $\varXi_r$ to be the following $\mathbb{Z}\Pi$-module isomorphism:
	\begin{equation}
		\label{Xi}
		\begin{aligned}
			\Xi_{r,b,i}:\mathop{\oplus}\limits_{\omega \in \Pi}C_r(p\hat{p}\omega\hat{\sigma}_{b,i}) \longrightarrow \mathbb{Z}\Pi^{N_r(i)} \ \ \ \ \ \ \ \ \ \ \ \ \ \ \ \ \ \ \ \ \ \ \ \ \ \ \ \ \  & \\
			\mathop{\oplus}\limits_{\omega \in \Pi}\sum\limits_{1 \leq s \leq N_r(i)}a(s,\omega)e_r^s(i) \mapsto \sum\limits_{1 \leq s \leq N_r(i)}\sum\limits_{\omega \in \Pi}a(s,\omega)\omega^{-1}\epsilon_s^{\Pi} \ \ (a(s,\omega) \in \mathbb{Z}) \ \ \ \ \ \ \ & \\
			\varXi_r:\mathop{\oplus}\limits_{\sigma \notin W} C_r(\sigma) \otimes_{\mathbb{Z}}\mathbb{Z}\Pi \longrightarrow \mathop{\oplus}\limits_{i=l_0+l_1+1}^{l_0+l_1+l_2}\mathbb{Z}\Pi^{N_r(i)} \ \ \ \ \ \ \ \ \ \ \ \ \ \ \ \ \ \ \ \ \ \ & \\
			\mathop{\oplus}\limits_{i=l_0+l_1+1}^{l_0+l_1+l_2}(\sum\limits_{1 \leq s \leq N_r(i)}a(s,i)e_r^s(i) \otimes \sum\limits_{\omega \in \Pi}b(i,\omega)\omega) \mapsto 	\mathop{\oplus}\limits_{i=l_0+l_1+1}^{l_0+l_1+l_2}\sum\limits_{\substack{1 \leq s \leq N_r(i) \\ \omega \in \Pi}}a(s,i)b(i,\omega)\omega\epsilon_s^{\Pi} & \\
			(a(s,i),b(i,\omega) \in \mathbb{Z}) \ \ \ \ \ \ \ \ \ \ \ \ \ \ \ \ \ \ \ \ \ \ \ \ \ \ \ \ \ \ \ \ \ \ \ \ \ & \\
		\end{aligned}
	\end{equation}
	
	For $r \in \mathbb{Z}$, let $f_r:D_r \longrightarrow C''_r, f_r \in \mathbb{F}_{\mathbb{N},b}(M^h(\mathbb{Z}\Pi))$ to be the equivalence class of the following morphism $\bar{f}_r$ in $\mathbb{F}_{\mathbb{N}}(M^h(\mathbb{Z}\Pi))$:
	
	If $b_1 \neq 0$,
	\begin{equation*}
		\begin{aligned}
			&\bar{f}_r(b_2,b_1):D_r(b_1) \longrightarrow C''_r(b_2), \mathop{\oplus}\limits_{(b,i) \in \bar{S}_{b_1}}z_{b,i} \mapsto \mathop{\oplus}\limits_{\substack{i \\ (b_2,i) \in \bar{S}_{b_1}}}\Xi_{r,b_2,i}z_{b_2,i}, \text{ where }	\\
			&z_{b,i} \in  \mathop{\oplus}\limits_{\omega \in \Pi} C_r(p\hat{p}\omega\hat{\sigma}_{b,i}) \text{ for all } (b,i) \in \bar{S}_{b_1}. \\
		\end{aligned}
	\end{equation*}

	If $b_1 = 0$ and $b_2 \neq 0$,
	\begin{equation}
		\label{barf}
		\begin{aligned}
			&\bar{f}_r(b_2,0):D_r(0) \longrightarrow C''_r(b_2), \mathop{\oplus}\limits_{(b,i) \in \bar{S}_{0}}z_{b,i} \oplus z \mapsto \mathop{\oplus}\limits_{\substack{i \\ (b_2,i) \in \bar{S}_{0}}}\Xi_{r,b_2,i}z_{b_2,i}, \text{ where } \\
			&z_{b,i} \in  \mathop{\oplus}\limits_{\omega \in \Pi} C_r(p\hat{p}\omega\hat{\sigma}_{b,i}),z \in \mathop{\oplus}\limits_{\sigma \notin W} C_r(\sigma) \otimes_{\mathbb{Z}}\mathbb{Z}\Pi \\
		\end{aligned}
	\end{equation}
	
	If $b_1=b_2=0$,
	\begin{equation*}
		\begin{aligned}
			&\bar{f}_r(0,0):D_r(0) \longrightarrow C''_r(0), \mathop{\oplus}\limits_{(b,i) \in \bar{S}_{0}}z_{b,i} \oplus z \mapsto \mathop{\oplus}\limits_{\substack{i \\ (0,i) \in \bar{S}_{0}}}\Xi_{r,0,i}z_{0,i} \oplus \varXi_r(z) \\
			& \text{where }z_{b,i} \in  \mathop{\oplus}\limits_{\omega \in \Pi} C_r(p\hat{p}\omega\hat{\sigma}_{b,i}),z \in \mathop{\oplus}\limits_{\sigma \notin W} C_r(\sigma) \otimes_{\mathbb{Z}}\mathbb{Z}\Pi \\
		\end{aligned}
	\end{equation*}
	
	(Note that by definition, $(0,i) \in \bar{S}_0$ implies $\tilde{\sigma}_i \in \widetilde{W}_{\infty}$ and thus $1 \leq i \leq l_0+l_1$. We have $\mathop{\oplus}\limits_{\substack{i \\ (0,i) \in \bar{S}_{0}}}\Xi_{r,0,i}z_{0,i} \in \mathop{\oplus}\limits_{i=1}^{l_0+l_1} \mathbb{Z}\Pi^{N_r(i)}$, while $\varXi_r(z) \in \mathop{\oplus}\limits_{i=l_0+l_1+1}^{l_0+l_1+l_2}\mathbb{Z}\Pi^{N_r(i)}$, thus the definition of $\bar{f}_r(0,0)$ is well defined.)
	
	We claim that $f:=\{f_r\}_{r \in \mathbb{Z}}$ induces an isomorphism of chain complexes in $\mathbb{F}_{\mathbb{N},b}(M^h(\mathbb{Z}\Pi))$ and furthermore, $f$ pushes the structure $\theta$ to $\psi''$. Then the first part of theorem is proved. The proof of the claim is divided into several steps:
	
	(e) To start with, we need to check that $\bar{f}_r$ is well-defined, that is, for any $b \in \mathbb{N}$, the sets $\{b'\ | \ \bar{f}_r(b',b) \neq 0\} ,\{b' \ | \ \bar{f}_r(b,b') \neq 0\}$ are finite.
	
	Note first that we only have to prove that the sets $\{b' \neq 0\ | \ \bar{f}_r(b',b) \neq 0\} ,\{b' \neq 0\ | \ \bar{f}_r(b,b') \neq 0\}$ are finite. By definition \ref{barf}, if $b' \neq 0$, then $\bar{f}_r(b',b) \neq 0 \Leftrightarrow \exists 1 \leq i \leq l$, such that $(b',i) \in \bar{S}_{b}$. Since $\bar{S}_{b}$ is a finite set, $\{b' \neq 0\ | \ \bar{f}_r(b',b) \neq 0\}$ is a finite set. For the other set, note that by definition \ref{barf}, if $b' \neq 0$, then $\bar{f}_r(b,b') \neq 0 \Leftrightarrow \exists 1 \leq i \leq l$, such that $(b,i) \in \bar{S}_{b'}$. By the definition of $\bar{S}_{b'}$, for every $1 \leq i \leq l$, there is at most one $b' \in \mathbb{N}$, such that $(b,i) \in \bar{S}_{b'}$. Thus $\{b' \neq 0 \ | \ \bar{f}_r(b,b') \neq 0\}$ is a finite set.
	
	(f) We check now that $f$ is an isomorphism, in fact, we claim that $\bar{f}_r$ is an isomorphism with the inverse given by:
	
	If $b_1 \neq 0$,
	\begin{equation}
		\label{barf-1neq0}
		\begin{aligned}
			&\bar{f}_r^{-1}(b_2,b_1):C''_r(b_1) \longrightarrow D_r(b_2), \mathop{\oplus}\limits_{i=1}^lz_i \mapsto \mathop{\oplus}\limits_{i=1}^lw_i \\
			&\text{where } z_i \in \mathbb{Z}\Pi^{N_r(i)} \text{ and } w_i=\begin{cases} \Xi_{r,b_1,i}^{-1}z_i & if \ (b_1,i) \in \bar{S}_{b_2} \\ \ \ \ \ 0 & else\\
			\end{cases}\\
		\end{aligned}
	\end{equation}

	If $b_1=0$,
	\begin{equation}
		\label{barf-10}
		\begin{aligned}
			&\bar{f}_r^{-1}(b_2,0):C''_r(0) \longrightarrow D_r(b_2), \mathop{\oplus}\limits_{i=1}^{l_1+l_2}z_i \oplus z \mapsto \mathop{\oplus}\limits_{i=1}^{l_1+l_2}w_i \oplus w \\
			&\text{where } w_i=\begin{cases} \Xi_{r,0,i}^{-1}z_i & if \ (0,i) \in \bar{S}_{b_2} \\ \ \ \ \ 0 & else\\
			\end{cases} \text{ and } w=\begin{cases} \varXi_{r}^{-1}z & if \ b_2=0 \\ \ \ \ \ 0 & else\\\end{cases} \\
		\end{aligned}
	\end{equation}

	To prove the claim, note first that similar to the proof in (e), $\bar{f}_r^{-1}$ is well defined, thus we only have to check that $\bar{f}_r\bar{f}_r^{-1}=\bar{f}_r^{-1}\bar{f}_r=id$. We will prove $\bar{f}_r\bar{f}_r^{-1}=id$ and the other one is similar.
	
	Now for any $b_1,b_3 \in \mathbb{N}$, we have:
	\begin{equation*}
		\bar{f}_r\bar{f}_r^{-1}(b_3,b_1)=\sum\limits_{b_2 \in \mathbb{N}}\bar{f}_r(b_3,b_2)\bar{f}_r^{-1}(b_2,b_1):C''(b_1) \longrightarrow C''(b_3)
	\end{equation*}
	
	If $b_1 \neq 0$, for any $1 \leq i \leq l$, since $K_b$ is disjoint for different $b$, there is a unique $b_2 \in \mathbb{N}$, such that $(b_1,i) \in \bar{S}_{b_2}$. Let $z_i \in \mathbb{Z}\Pi^{N_r(i)}$, then by definition \ref{barf-1neq0} of $\bar{f}_r^{-1}$, we get $\bar{f}_r\bar{f}_r^{-1}(b_3,b_1)(z_i)=\bar{f}_r(b_3,b_2)(\Xi_{r,b_1,i}^{-1}z_i)$ with $\Xi_{r,b_1,i}^{-1}z_i \in \mathop{\oplus}\limits_{\omega \in \Pi} C_r(p\hat{p}\omega\hat{\sigma}_{b_1,i})$. Then by definition \ref{barf} of $\bar{f}_r$, we get $\bar{f}_r(b_3,b_2)(\Xi_{r,b_1,i}^{-1}z_i)=\begin{cases} z_i & \text{If} \ b_3=b_1 \\ 0 & \text{else} \\ \end{cases}$.
	
	If $b_1=0$, for any $1 \leq i \leq l_1+l_2$, since $K_b$ is disjoint for different $b$, there is a unique $b_2 \in \mathbb{N}$, such that $(b_1,i) \in \bar{S}_{b_2}$. Let $z_i \in \mathbb{Z}\Pi^{N_r(i)}$, similar to the proof in the case $b_1 \neq 0$, we can get $\bar{f}_r\bar{f}^{-1}_r(b_3,b_1)(z_i)=\begin{cases} z_i & \text{If} \ b_3=b_1 \\ 0 & \text{else} \\ \end{cases}$.
	
	For any $l_1+l_2+1 \leq i \leq l$, since $g_0^{-1}\tilde{\sigma}_i \notin \widetilde{W}_{\infty}$, we have $(0,i) \notin \bar{S}_{b}$ for all $b \in \mathbb{N}$. Let $z_i \in \mathbb{Z}\Pi^{N_r(i)}$, by definition \ref{barf-10} of $\bar{f}^{-1}_r$, we have $\bar{f}_r\bar{f}^{-1}_r(b_3,b_1)(z_i)=\bar{f}_r(b_3,0)(\varXi_r^{-1}z_i)$, with $\varXi_r^{-1}z_i \in \mathop{\oplus}\limits_{\sigma \notin W} C_*(\sigma) \otimes_{\mathbb{Z}}\mathbb{Z}\Pi$. Then by definition \ref{barf} of $\bar{f}_r$, we get $\bar{f}_r(b_3,0)(\varXi_r^{-1}z_i)=\begin{cases} z_i & \text{If} \ b_3=0 \\ 0 & \text{else} \\ \end{cases}$. Therefore, combining the proofs in two cases we prove our claim.
	
	(g) We check now that $f$ is chain map, that is, the following diagram commutes:
	\begin{equation*}
		\begin{tikzcd}
			D_r \rar{f_r} \dar{d_{D,r}} & C_r'' \dar{d_r''} \\
			D_{r-1} \rar{f_{r-1}} & C_{r-1}'' \\
		\end{tikzcd}
	\end{equation*}

	Since all the morphisms that we consider are equivalence classes of certain morphisms in $\mathbb{F}_{\mathbb{N}}(M^h(\mathbb{Z}\Pi))$, it is equivalent to prove that $\bar{d}_r\bar{f}_r-\bar{f}_{r-1}\bar{d}_{D,r}$ is equivalent to the $0$ map. Recall that the morphisms of $\mathbb{F}_{\mathbb{N},b}(M^h(\mathbb{Z}\Pi))$ are the quotient of $\mathbb{F}_{\mathbb{N}}(M^h(\mathbb{Z}\Pi))$ by those morphisms $F$ where $\{(b_3,b_1) \in \mathbb{N}^2 \ | \ F(b_3,b_1) \neq 0\}$ is a finite set. Thus it is equivalent to prove that $\{(b_3,b_1) \in \mathbb{N}^2 \ | \ (\bar{d}_r\bar{f}_r-\bar{f}_{r-1}\bar{d}_{D,r})(b_3,b_1) \neq 0\}$ is a finite set. Note that since $\bar{d}_r\bar{f}_r-\bar{f}_{r-1}\bar{d}_{D,r}$ are morphisms in $\mathbb{F}_{\mathbb{N}}(M^h(\mathbb{Z}\Pi))$, the following sets are finite sets:
	\begin{equation*}
		\begin{aligned}
			\{(b_3,0) \in \mathbb{N}^2 \ | \ (\bar{d}_r\bar{f}_r-\bar{f}_{r-1}\bar{d}_{D,r})(b_3,0) \neq 0\} \\
			\{(0,b_1) \in \mathbb{N}^2 \ | \ (\bar{d}_r\bar{f}_r-\bar{f}_{r-1}\bar{d}_{D,r})(0,b_1) \neq 0\} \\
		\end{aligned}
	\end{equation*}
	
	Thus it suffices to prove that $\{(b_3,b_1) \in \mathbb{N}^2 \ | \ b_1 \neq 0, b_3 \neq 0, (\bar{d}_r\bar{f}_r-\bar{f}_{r-1}\bar{d}_{D,r})(b_3,b_1) \neq 0\}$ is a finite set.
	
	Let us make a computation of the morphsim $(\bar{d}_r\bar{f}_r-\bar{f}_{r-1}\bar{d}_{D,r})(b_3,b_1):D_r(b_1) \longrightarrow C_{r-1}''(b_3)$ under the condition $b_1 \neq 0,b_3 \neq 0$:
	
	Since $b_1 \neq 0$, by definition \ref{D} we have $D_r(b_1)=\mathop{\oplus}\limits_{(b,i) \in \bar{S}_{b_1}}\mathop{\oplus}\limits_{\omega \in \Pi}C_*(p\hat{p}\omega\hat{\sigma}_{b,i})$. Fix $(b,i) \in \bar{S}_{b_1}, \omega \in \Pi$ and choose $z_{b,i,\omega} \in C_*(p\hat{p}\omega\hat{\sigma}_{b,i})$, in order to understand the morphism, it suffices to compute the image of $z_{b,i,\omega}$. 
	
	For $b',b'' \in \mathbb{N}$ and $1 \leq j \leq l$, let $\delta_{b',j}^{b''}=\begin{cases} 1 & \text{If} \ (b',j) \in \bar{S}_{b''} \\ 0 & \text{else} \\
	\end{cases}$ and $\delta_{b',b''}=\begin{cases} 1 & \text{If} \ b'=b'' \\ 0 & \text{else} \\ \end{cases}$. For every $j \in S_i$, let $b_{i,j} \in \mathbb{N},h_{i,j} \in H$ be the unique element such that $\tilde{g}_{i,j}^{-1}\tilde{g}_b=\tilde{g}_{b_{i,j}}h_{i,j}$. Now by definition of the composition, we have:
	\begin{equation}
		\label{Eq6.15}
		\begin{aligned}
			(\bar{d}_r\bar{f}_r-\bar{f}_{r-1}\bar{d}_{D,r})(b_3,b_1)(z_{b,i,\omega})=&\sum\limits_{b_2 \in \mathbb{N}}\bar{d}_r(b_3,b_2)\bar{f}_r(b_2,b_1)(z_{b,i,\omega}) \\
			&-\sum\limits_{b_2 \in \mathbb{N}}\bar{f}_{r-1}(b_3,b_2)\bar{d}_{D,r}(b_2,b_1)(z_{b,i,\omega}) \\
		\end{aligned}
	\end{equation}

	We will compute the two terms separately, for the first term we have:
	\begin{equation*}
		\begin{aligned}
			& \quad \, \sum\limits_{b_2 \in \mathbb{N}}\bar{d}_r(b_3,b_2)\bar{f}_r(b_2,b_1)(z_{b,i,\omega}) \\
			& \ =\bar{d}_r(b_3,b)(\Xi_{r,b,i}z_{b,i,\omega}) \ \ (\text{By definition } \ref{barf}) \\
			& \ =\mathop{\oplus}\limits_{j=1}^l \rho_G\big(\tilde{g}(j,i)\big)[b_3,b] \cdot d_r(j,i)(\Xi_{r,b,i}z_{b,i,\omega})  \ \ (\text{By definition } \ref{bard}) \\
		\end{aligned}
	\end{equation*}

	If $j \in S_i$, recall that by definition $\rho_G$ is the induced action of $H$ on $\mathbb{Z}[\Pi]$ given by $h \cdot w=r_{\Pi}(h)w$. Since $\tilde{g}(j,i)\tilde{g}_b=\tilde{g}_{i,j}^{-1}\tilde{g}_b=\tilde{g}_{b_{i,j}}h_{i,j}$, we can get $\rho_G\big(\tilde{g}(j,i)\big)[b_3,b]=\rho_G(\tilde{g}_{i,j}^{-1})[b_3,b]=\delta_{b_{i,j},b_3}r_{\Pi}(h_{i,j})$.
	
	If $j \notin S_i$, then $\tilde{g}(j,i)$ is defined to be $0$ and thus $\rho_G\big((\tilde{g}(j,i))\big)[b_3,b]=0$. Combining these two cases we get:
	\begin{equation}
		\label{onesidediff}
		\sum\limits_{b_2 \in \mathbb{N}}\bar{d}_r(b_3,b_2)\bar{f}_r(b_2,b_1)(z_{b,i,\omega})=\mathop{\oplus}\limits_{j \in S_i}\delta_{b_{i,j},b_3}r_{\Pi}(h_{i,j}) \cdot d_r(j,i)(\Xi_{r,b,i}z_{b,i,\omega})
	\end{equation}
	
	Next we compute the second term, note first that we have:
	\begin{align*}
		\bar{d}_D(b_2,b_1)(z_{b,i,\omega})&=\mathop{\oplus}\limits_{\hat{\sigma} \in \widehat{K}_{b_2}}d_r(\hat{\sigma},\omega\hat{\sigma}_{b,i})(z_{b,i,\omega}) \ \ (\text{By definition } \ref{bardD})  \\
		&=\mathop{\oplus}\limits_{j \in S_i}\delta_{b_{i,j},j}^{b_2}d_r(\sigma_j,\sigma_i)(z_{b,i,\omega}) (\text{By definition } \ref{dr} )\\
		&\in \mathop{\oplus}\limits_{\substack{j \in S_i \\ (b_{i,j},j) \in \bar{S}_{b_2}}}C_{r-1}(p\hat{p}\omega r_{\Pi}(h_{i,j}^{-1})\hat{\sigma}_{b_{i,j},j}) \ \ (\text{By (1) in Claim \ref{cclaim}}) \\
	\end{align*}

	Then we have:
	\begin{equation}
		\label{Eq6.16}
		\begin{aligned}
			\sum\limits_{b_2 \in \mathbb{N}}\bar{f}_{r-1}(b_3,b_2)\bar{d}_{D,r}(b_2,b_1)(z_{b,i,\omega})=\sum\limits_{b_2 \in \mathbb{N}}\bar{f}_{r-1}(b_3,b_2)(\mathop{\oplus}\limits_{j \in S_i}\delta_{b_{i,j},j}^{b_2}d_r(\sigma_j,\sigma_i)(z_{b,i,\omega})) \\
		\end{aligned}
	\end{equation}
	
	If $1 \leq j \leq l_0+l_1$, then by definition of $\delta_{b_{i,j},j}^{b_2}$, there is a unique $b_2 \in \mathbb{N}$ such that $\delta_{b_{i,j},j}^{b_2}=1$.
	
	If $l_0+l_1+1 \leq j \leq l$, then we have $\sigma_j \notin W$ and $\sigma_i \leq \sigma_j$. We are in the case (b) or (d) in the proof of Claim \ref{cclaim} and we can find in the proof that $b_{i,j} \neq 0$. Thus we have $g_{b_{i,j}}^{-1}\tilde{\sigma}_j \in \widetilde{W}_{\infty}$. By definition of $\delta_{b_{i,j},j}^{b_2}$, there is a unique $b_2 \in \mathbb{N}$, such that $\delta_{b_{i,j},j}^{b_2}=1$, thus: 
	\begin{equation}
		\label{Eq6.17}
		\begin{aligned}
			& \quad \, \sum\limits_{b_2 \in \mathbb{N}}\bar{f}_{r-1}(b_3,b_2)\big(\mathop{\oplus}\limits_{j \in S_i}\delta_{b_{i,j},j}^{b_2}d_r(\sigma_j,\sigma_i)(z_{b,i,\omega})\big) \\
			&=\mathop{\oplus}\limits_{j \in S_i}\delta_{b_3,b_{i,j}}\Xi_{r,b_{i,j},j}\big(d_r(\sigma_j,\sigma_i)(z_{b,i,\omega})\big) \ \ (\text{By definition } \ref{barf}) \\
		\end{aligned}
	\end{equation}

	Suppose that $z_{b,i,\omega}=\sum\limits_{s=1}^{N_r(i)} a_{s}e_{r}^{s}(i)$ with $a_s \in \mathbb{Z}$. For any $1 \leq s \leq N_r(i), 1 \leq s' \leq N_{r-1}(j)$, let $d_r(j,i)[s',s] \in \mathbb{Z}$ be the element in the $s'$-th row and $s$-th column of $d_r(j,i)$. By definition of $d_r(j,i)$, we have $d_r(\sigma_j,\sigma_i)(z_{b,i,\omega})=\sum\limits_{s=1}^{N_r(i)}\sum\limits_{s'=1}^{N_{r-1}(j)}a_sd_r(j,i)[s',s]e_{r-1}^{s'}(j)$. By definition \ref{Xi}, we have:
	\begin{equation*}
		\begin{aligned}
			\Xi_{r,b_{i,j},j}\big(d_r(\sigma_j,\sigma_i)(z_{b,i,\omega})\big)&=\sum\limits_{s=1}^{N_r(i)}\sum\limits_{s'=1}^{N_{r-1}(j)}a_sd_r(j,i)[s',s] r_{\Pi}(h_{i,j})\omega^{-1} \epsilon_{s'}^{\Pi} \\
			&=r_{\Pi}(h_{i,j}) \cdot d_r(j,i) \big(\sum\limits_{s=1}^{N_r(i)}a_s\omega^{-1} \epsilon_{s}^{\Pi}\big) \\
			& \quad (\text{By definition } \ref{Xi}) \\
			&=r_{\Pi}(h_{i,j}) \cdot d_r(j,i) \big(\Xi_{r,b,i}z_{b,i,\omega}\big)
		\end{aligned}
	\end{equation*}
	
	Substituting the equation above to equation \ref{Eq6.16} and \ref{Eq6.17}, we have:
	\begin{equation}
		\label{anothersidediff}
		\begin{aligned}
			\sum\limits_{b_2 \in \mathbb{N}}\bar{f}_{r-1}(b_3,b_2)\bar{d}_{D,r}(b_2,b_1)(z_{b,i,\omega})=\mathop{\oplus}\limits_{j \in S_i}\delta_{b_3,b_{i,j}}r_{\Pi}(h_{i,j}) \cdot d_r(j,i)(\Xi_{r,b,i}z_{b,i,\omega})
		\end{aligned}
	\end{equation}

	Substituting the two results \ref{onesidediff} and \ref{anothersidediff} into equation \ref{Eq6.15}, we get $	(\bar{d}_r\bar{f}_r-\bar{f}_{r-1}\bar{d}_D)(b_3,b_1)=0$. In other words, we have that $\{(b_3,b_1) \in \mathbb{N}^2 \ | \ b_3 \neq 0,b_1 \neq 0, (\bar{d}_r\bar{f}_r-\bar{f}_{r-1}\bar{d}_D)(b_3,b_1) \neq 0\}$ is an empty set, in particular, it is a finite set. Thus we have checked that $f$ is a chain map. 
	
	(h) We check now that $f$ pushes the strucure $\theta$ to $\psi''$, that is, the following diagram commutes:
	$$
	\begin{tikzcd}
		D_{n-u-r}^* \dar{\theta_u} & {C_{n-u-r}''}^* \dar{\psi_u''} \lar[swap]{f_{n-u-r}^*}\\
		D_{r} \rar{f_{r}} & C_{r}'' \\
	\end{tikzcd}
	$$
	
	The proof method is the same with (g), we can similarly get that it suffices to prove that for any $u \in \mathbb{N},r \in \mathbb{Z}$, the set $\{(b_3,b_1) \in \mathbb{N}^2 \ | \ b_1 \neq 0, b_3 \neq 0, (\bar{f}_r\bar{\theta}_u^r-\bar{\psi}_u^r(\bar{f}_{n-u-r}^{-1})^*)(b_3,b_1) \neq 0\}$ is a finite set.
	
	Let us make a computation of $(\bar{f}_r\bar{\theta}_u^r-\bar{\psi}_u^r(\bar{f}_{n-u-r}^{-1})^*)(b_3,b_1):D_{n-u-r}(b_1)^* \longrightarrow C_r''(b_3)$ under the condition $b_1 \neq 0,b_3 \neq 0$:
	
	Since $b_1 \neq 0$, we have $D_{n-u-r}(b_1)^*=\mathop{\oplus}\limits_{(b,i) \in \bar{S}_{b_1}}\mathop{\oplus}\limits_{\omega \in \Pi}C_{n-u-r}(p\hat{p}\omega\hat{\sigma}_{b,i})^*$, it suffices to compute the morphism on every direct summand. Fix $(b,i) \in \bar{S}_{b_1}$ and choose $z_{b,i}=\mathop{\oplus}\limits_{\omega \in \Pi}z_{b,i,\omega} \in \mathop{\oplus}\limits_{\omega \in \Pi}C_{n-u-r}(p\hat{p}\omega\hat{\sigma}_{b,i})^*$. Let $\hat{x}_1,\hat{x}_2,...,\hat{x}_c$ be all vertices in $\hat{\sigma}_{b,i}$ and for every $1 \leq v \leq c$, let $x_v=p\hat{p}\hat{x}_v \in \sigma_i$. For every $1 \leq v \leq c$, let $i(v)$ be the index such that $\sigma_{i(v)}=x_v$ and let $\mathbf{S}_v=\{j \ | \ x_v \in \sigma_j\}$. For every $1 \leq v \leq c$ and $j \in \mathbf{S}_v$, let $\mathbf{b}_v,\mathbf{b}_{j,v} \in \mathbb{N}$ and $\mathbf{h}_v,\mathbf{h}_{j,v} \in H$ be the unique element such that:
	\begin{equation}
		\label{Grelquad}
		\tilde{g}_{i(v),i}^{-1}\tilde{g}_{\mathbf{b}_v}=\tilde{g}_{b_1}\mathbf{h}_v,\tilde{g}_{i(v),j}^{-1}\tilde{g}_{\mathbf{b}_v}=\tilde{g}_{\mathbf{b}_{j,v}}\mathbf{h}_{j,v}
	\end{equation}
	
	For $b',b'' \in \mathbb{N}, 1 \leq j \leq l$, let $\delta_{b',j}^{b''}=\begin{cases} 1 & if \ (b',j) \in \bar{S}_{b''} \\ 0 & else \\
	\end{cases}$, $\delta_{b',b''}=\begin{cases} 1 & if \ b'=b'' \\ 0 & else \\ \end{cases}$. Fix $b_1 \neq 0$, $b_3 \neq 0$, we begin to compute $\bar{f}_r\bar{\theta}_u^r(b_3,b_1)(z_{b,i})$,$\bar{\psi}_u^r(\bar{f}_{n-u-r}^{-1})^*(b_3,b_1)(z_{b,i})$ respectively, note first by the definition of composition, we have:
	\begin{equation*}
		\bar{f}_r\bar{\theta}_u^r(b_3,b_1)=\sum\limits_{b_2 \in \mathbb{N}}\bar{f}_r(b_3,b_2)\bar{\theta}_u^r(b_2,b_1)
	\end{equation*}
	
	Then for any $b_2 \in \mathbb{N}$ we have:
	\begin{equation}
		\label{Eq6.18}
		\begin{aligned}
			\bar{\theta}_u^r(b_2,b_1)(z_{b,i})&=\sum\limits_{\omega \in \Pi}\mathop{\oplus}_{\hat{\sigma} \in \widehat{K}_{b_2}}\psi_u^r(\hat{\sigma},\omega\hat{\sigma}_{b,i})(z_{b,i,\omega}) \ \ (\text{By definition }\ref{bartheta})\\
			&=\sum\limits_{\omega \in \Pi}\mathop{\oplus}_{(b',j) \in \bar{S}}\mathop{\oplus}_{\omega' \in  \Pi}\delta_{b',j}^{b_2} \psi_u^r(\omega'\hat{\sigma}_{b',j},\omega\hat{\sigma}_{b,i})(z_{b,i,\omega})
		\end{aligned}
	\end{equation}

	Now since $(b',j) \in \bar{S}$ and $K_b$ is disjoint for different $b$, there is a unique $b_2 \in \mathbb{N}$, such that $\delta_{b',j}^{b_2}=1$. Thus by definition \ref{barf} of $\bar{f}$, we have:
	\begin{equation}
			\label{Eq6.19}
			\begin{aligned}
			&\quad \sum\limits_{b_2 \in \mathbb{N}}\bar{f}_r(b_3,b_2)\bar{\theta}_u(b_2,b_1)(z_{b,i}) \\
			& \quad (\text{By \ref{Eq6.18}} )\\
			&=\sum\limits_{b_2 \in \mathbb{N}}\bar{f}_r(b_3,b_2)\big(\sum\limits_{\omega \in \Pi}\mathop{\oplus}_{(b',j) \in \bar{S}}\mathop{\oplus}_{\omega' \in  \Pi}\delta_{b',j}^{b_2}\psi_u^r(\omega'\hat{\sigma}_{b',j},\omega\hat{\sigma}_{b,i})(z_{b,i,\omega})\big) \\
			&=\sum\limits_{\omega \in \Pi}\sum\limits_{b_2 \in \mathbb{N}}\bar{f}_r(b_3,b_2)\big(\mathop{\oplus}_{(b',j) \in \bar{S}}\mathop{\oplus}_{\omega' \in  \Pi}\delta_{b',j}^{b_2} \psi_u^r(\omega'\hat{\sigma}_{b',j},\omega\hat{\sigma}_{b,i})(z_{b,i,\omega})\big) \\
			& \quad (\text{By definition }\ref{barf}) \\
			&=\sum\limits_{\omega \in \Pi}\mathop{\oplus}_{\substack{ j \\ (b_3,j) \in \bar{S}}}\Xi_{r,b_3,j}\big(\mathop{\oplus}_{\omega' \in  \Pi} \psi_u^r(\omega'\hat{\sigma}_{b_3,j},\omega\hat{\sigma}_{b,i})(z_{b,i,\omega})\big)\\
			&=\mathop{\oplus}_{\substack{ j \\ (b_3,j) \in \bar{S}}}\Xi_{r,b_3,j}\big(\mathop{\oplus}_{\omega' \in  \Pi}\sum\limits_{\omega \in \Pi} \psi_u^r(\omega'\hat{\sigma}_{b_3,j},\omega\hat{\sigma}_{b,i})(z_{b,i,\omega})\big) \\
		\end{aligned}
	\end{equation}

	Note that for any $1 \leq v \leq c$ we have $x_v=\sigma_{i(v)} \leq \sigma_i$ and $\hat{x}_v \leq \hat{\sigma}_{b_1,i} \in \widehat{W}_{\infty}$. By (2) in Claim \ref{cclaim}, we have $\hat{x}_v=r_{\Pi}(\mathbf{h}_v)\hat{\sigma}_{\mathbf{b}_v,i(v)}$. For $\omega' \in \Pi,(b_3,j) \in \bar{S}, 1 \leq v \leq c$, by (2) in Claim \ref{cclaim}, we have that $\hat{x}_v \in \omega'\hat{\sigma}_{b_3,j}$ if and only if $j \in \mathbf{S}_v$, $b_3=\mathbf{b}_{j,v}$, $\omega'=r_{\Pi}(\mathbf{h}_v\mathbf{h}_{j,v}^{-1})$. Then, for any $1 \leq j \leq l,\omega' \in \Pi$ with $(b_3,j) \in \bar{S}$, we have:
	\begin{equation}
		\label{Eq6.20}
			\begin{aligned}
			&\quad \sum\limits_{\omega \in \Pi}\psi_u^r(\omega'\hat{\sigma}_{b_3,j},\omega\hat{\sigma}_{b,i})(z_{b,i,\omega}) \\
			& \quad (\text{By definition } \ref{psir}) \\
			&=\sum\limits_{\omega \in \Pi}\sum\limits_{\substack{\hat{x} \in \omega'\hat{\sigma}_{b_3,j} \cap \omega\hat{\sigma}_{b,i} \\ |\hat{x}|=0}} \psi_u^r(p\hat{p}\omega'\hat{\sigma}_{b_3,j},p\hat{p}\hat{x}) \Upsilon^{n-u-r}_{\hat{x},\omega\hat{\sigma}_{b,i}}(z_{b,i,\omega}) \\
			&=\sum\limits_{\omega \in \Pi}\sum\limits_{\substack{1 \leq v \leq c \\ j \in \mathbf{S}_v}}\delta_{b_3,\mathbf{b}_{j,v}}\delta_{\omega',\omega r_{\Pi}(\mathbf{h}_v\mathbf{h_{j,v}^{-1}})}\psi_u^r(\sigma_{j},x_v) \Upsilon^{n-u-r}_{\omega\hat{x}_v,\omega\hat{\sigma}_{b,i}}(z_{b,i,\omega}) \\
		\end{aligned}
	\end{equation}

	When $b_3=\mathbf{b}_{j,v}$ for some $v$ and  $\omega'=\omega r_{\Pi}(\mathbf{h}_v\mathbf{h_{j,v}^{-1}})$, $	\psi_u^r(\sigma_{j},x_v)\Upsilon^{n-u-r}_{\omega\hat{x}_v,\omega\hat{\sigma}_{b,i}}$ is the following morphism:
	\begin{equation*}
		\begin{tikzcd}
			C_{n-u-r}(p\hat{p}\omega\hat{\sigma}_{b,i})^* \ar{rr}{\text{inclusion}} & &  C^{n-u-r}(p\hat{p}\omega\hat{x}_v) \ar{rr}{\psi_u^r(\sigma_{j},x_v)} & & C_r(p\hat{p}\omega'\hat{\sigma}_{b_3,j})
		\end{tikzcd}
	\end{equation*}
	
	which is the same as $\varphi_{u,x_v}^r(j,i)$:
	\begin{equation*}
		\begin{tikzcd}
			C_{n-u-r}(\sigma_{i})^* \ar{rr}{\text{inclusion}} & & C^{n-u-r}(x_v) \ar{rr}{\psi_u^r(\sigma_{j},x_v)} & & C_r(\sigma_j)
		\end{tikzcd}
	\end{equation*}

	Therefore, for any $\omega' \in \Pi$, we have:
	\begin{equation}
		\label{Eq6.21}
		\begin{aligned}
			&\quad \ \ \sum\limits_{\omega \in \Pi}\sum\limits_{\substack{1 \leq v \leq c \\ j \in \mathbf{S}_v}}\delta_{b_3,\mathbf{b}_{j,v}}\delta_{\omega',\omega r_{\Pi}(\mathbf{h}_v\mathbf{h_{j,v}^{-1}})}\psi_u^r(\sigma_{j},x_v) \Upsilon^{n-u-r}_{\omega\hat{x}_v,\omega\hat{\sigma}_{b,i}}(z_{b,i,\omega}) \\
			&=\sum\limits_{\substack{1 \leq v \leq c \\ j \in \mathbf{S}_v}}\delta_{b_3,\mathbf{b}_{j,v}}\psi_u^r(\sigma_{j},x_v) \Upsilon^{n-u-r}_{\omega'r_{\Pi}(\mathbf{h}_{j,v}\mathbf{h}_v^{-1})\hat{x}_v,\omega'r_{\Pi}(\mathbf{h}_{j,v}\mathbf{h}_v^{-1})\hat{\sigma}_{b,i}}(z_{b,i,\omega'r_{\Pi}(\mathbf{h}_{j,v}\mathbf{h}_v^{-1})}) \\
			&=\sum\limits_{\substack{1 \leq v \leq c \\ j \in \mathbf{S}_v}}\delta_{b_3,\mathbf{b}_{j,v}}\varphi_{u,x_v}^r(j,i)\big(\sum\limits_{s=1}^{N_{n-u-r}(i)}z_{b,i,\omega'r_{\Pi}(\mathbf{h}_{j,v}\mathbf{h}_v^{-1}),s}e_{n-u-r}^s(i)^*\big) \\
		\end{aligned}
	\end{equation}
	
	By definition, $\psi_{u,x_v}^{r}(j,i)$ is the matrix of the homomorphism $\varphi_{u,x_v}^r(j,i)$ under the $\mathbb{Z}$-basis $\{e_{n-u-r}^s(i)^*\}_{s=1}^{N_{n-u-r}(i)},\{e_r^{s'}(j)\}_{s'=1}^{N_r(j)}$, we can interpret it as follows. Let $\psi_{u,x_v}^{r}(j,i)[s',s] \in \mathbb{Z}$ be the element in the $s'$th row and $s$th column. Write $z_{b,i,\omega}=\sum\limits_{s=1}^{N_{n-u-r}(i)}z_{b,i,\omega,s}e_{n-u-r}^s(i)^*$, then we have:
	
	\begin{equation}
		\label{Eq6.22}
		\begin{aligned}
			&\quad \ \sum\limits_{\substack{1 \leq v \leq c \\ j \in \mathbf{S}_v}}\delta_{b_3,\mathbf{b}_{j,v}}\varphi_{u,x_v}^r(j,i)\big(\sum\limits_{s=1}^{N_{n-u-r}(i)}z_{b,i,\omega'r_{\Pi}(\mathbf{h}_{j,v}\mathbf{h}_v^{-1}),s}e_{n-u-r}^s(i)^*\big) \\
			&=\sum\limits_{\substack{1 \leq v \leq c \\ j \in \mathbf{S}_v}}\delta_{b_3,\mathbf{b}_{j,v}}\sum\limits_{s'=1}^{N_r(j)}\sum\limits_{s=1}^{N_{n-u-r}(i)}\psi_{u,x_v}^r(j,i)[s',s]z_{b,i,\omega'r_{\Pi}(\mathbf{h}_{j,v}\mathbf{h}_v^{-1}),s}e_{r}^{s'}(j) \\
			&\in C_r(p\hat{p}\omega'\hat{\sigma}_{b_3,j}) \\
		\end{aligned}
	\end{equation}
	Combining \ref{Eq6.20}, \ref{Eq6.21} and \ref{Eq6.22}, we get for any $1 \leq j \leq l$:
	\begin{equation*}
			\begin{aligned}
			&\quad \ \Xi_{r,b_3,j}\big(\mathop{\oplus}_{\omega' \in  \Pi}\sum\limits_{\omega \in \Pi} \psi_u^r(\omega'\hat{\sigma}_{b_3,j},\omega\hat{\sigma}_{b,i})(z_{b,i,\omega})\big)\\
			& \quad (\text{By } \ref{Eq6.20},\ref{Eq6.21},\ref{Eq6.22}) \\
			&=\Xi_{r,b_3,j}\big(\mathop{\oplus}_{\omega' \in  \Pi}\sum\limits_{\substack{1 \leq v \leq c \\ j \in \mathbf{S}_v}}\delta_{b_3,\mathbf{b}_{j,v}}\sum\limits_{s'=1}^{N_r(j)}\sum\limits_{s=1}^{N_{n-u-r}(i)}\psi_{u,x_v}^r(j,i)[s',s]z_{b,i,\omega'r_{\Pi}(\mathbf{h}_{j,v}\mathbf{h}_v^{-1}),s}e_{r}^{s'}(j)\big) \\
			& \quad (\text{By definition } \ref{Xi}) \\
			&=\sum\limits_{\omega' \in  \Pi}\sum\limits_{\substack{1 \leq v \leq c \\ j \in \mathbf{S}_v}}\delta_{b_3,\mathbf{b}_{j,v}}\sum\limits_{s'=1}^{N_r(j)}\sum\limits_{s=1}^{N_{n-u-r}(i)}\psi_{u,x_v}^r(j,i)[s',s]z_{b,i,\omega'r_{\Pi}(\mathbf{h}_{j,v}\mathbf{h}_v^{-1}),s}\omega'^{-1}\epsilon_{s'}^{\Pi} \\
		\end{aligned}
	\end{equation*}
	
	Changing the summing index $ \omega=\omega'r_{\Pi}(\mathbf{h}_{j,v}\mathbf{h}_v^{-1})$, we get:
	\begin{equation}
		\label{Eq6.23}
		\begin{aligned}
			&\quad \ \  \Xi_{r,b_3,j}\big(\mathop{\oplus}_{\omega' \in  \Pi}\sum\limits_{\omega \in \Pi} \psi_u^r(\omega'\hat{\sigma}_{b_3,j},\omega\hat{\sigma}_{b,i})(z_{b,i,\omega})\big)\\
			&=\sum\limits_{\omega \in  \Pi}\sum\limits_{\substack{1 \leq v \leq c \\ j \in \mathbf{S}_v}}\delta_{b_3,\mathbf{b}_{j,v}}\sum\limits_{s'=1}^{N_r(j)}\sum\limits_{s=1}^{N_{n-u-r}(i)}\psi_{u,x_v}^r(j,i)[s',s]z_{b,i,\omega,s}r_{\Pi}(\mathbf{h}_{j,v}\mathbf{h}_v^{-1})\omega^{-1}\epsilon_{s'}^{\Pi} \\
			&=\sum\limits_{\omega \in  \Pi}\sum\limits_{\substack{1 \leq v \leq c \\ j \in \mathbf{S}_v}}\delta_{b_3,\mathbf{b}_{j,v}}r_{\Pi}(\mathbf{h}_{j,v}\mathbf{h}_v^{-1}) \cdot \bigg(\sum\limits_{s=1}^{N_{n-u-r}(i)}z_{b,i,\omega,s}\psi_{u,x_v}^r(j,i)\big((\epsilon_{s}^{\Pi})^*\big) \omega^{-1}\bigg) \\
			&=\sum\limits_{\omega \in  \Pi}\sum\limits_{\substack{1 \leq v \leq c \\ j \in \mathbf{S}_v}}\delta_{b_3,\mathbf{b}_{j,v}}r_{\Pi}(\mathbf{h}_{j,v}\mathbf{h}_v^{-1}) \cdot \bigg(\psi_{u,x_v}^r(j,i)\big(\sum\limits_{s=1}^{N_{n-u-r}(i)}z_{b,i,\omega,s}\omega^{-1}(\epsilon_{s}^{\Pi})^*\big)\bigg) \\
		\end{aligned}
	\end{equation}

	Summarizing the computations above, we get:
	\begin{equation*}
		\begin{aligned}
			& \quad \ \ \bar{f}_r\bar{\theta}_u(b_3,b_1)(z_{b,i})\\
			&=\sum\limits_{b_2 \in \mathbb{N}}\bar{f}_r(b_3,b_2)\bar{\theta}_u(b_2,b_1)(z_{b,i}) \\
			& \quad (\text{By } \ref{Eq6.19}) \\
			&=\mathop{\oplus}_{\substack{ j \\ (b_3,j) \in \bar{S}}}\Xi_{r,b_3,j}\big(\mathop{\oplus}_{\omega' \in  \Pi}\sum\limits_{\omega \in \Pi} \psi_u^r(\omega'\hat{\sigma}_{b_3,j},\omega\hat{\sigma}_{b,i})(z_{b,i,\omega})\big) \\
			& \quad (\text{By } \ref{Eq6.23}) \\
			&=\mathop{\oplus}_{\substack{ j \\ (b_3,j) \in \bar{S}}}\sum\limits_{\omega \in  \Pi}\sum\limits_{\substack{1 \leq v \leq c \\ j \in \mathbf{S}_v}}\delta_{b_3,\mathbf{b}_{j,v}}r_{\Pi}(\mathbf{h}_{j,v}\mathbf{h}_v^{-1}) \cdot \bigg(\psi_{u,x_v}^r(j,i)\big(\sum\limits_{s=1}^{N_{n-u-r}(i)}z_{b,i,\omega,s}\omega^{-1}(\epsilon_{s}^{\Pi})^*\big)\bigg) \\
		\end{aligned}
	\end{equation*}
	
	Since $b_3 \neq 0$, we have $(b_3,j) \in \bar{S}$ for $\forall 1 \leq j \leq l$, therefore,
	\begin{equation}
		\label{onesidequad}
		\begin{aligned}
			& \quad \ \ \bar{f}_r\bar{\theta}_u(b_3,b_1)(z_{b,i}) \\
			&=\mathop{\oplus}_{j=1}^l\sum\limits_{\omega \in  \Pi}\sum\limits_{\substack{1 \leq v \leq c \\ j \in \mathbf{S}_v}}\delta_{b_3,\mathbf{b}_{j,v}}r_{\Pi}(\mathbf{h}_{j,v}\mathbf{h}_v^{-1}) \cdot \bigg(\psi_{u,x_v}^r(j,i)\big(\sum\limits_{s=1}^{N_{n-u-r}(i)}z_{b,i,\omega,s}\omega^{-1}(\epsilon_{s}^{\Pi})^*\big)\bigg) \\
		\end{aligned}
	\end{equation}
	
	We continue to compute $\bar{\psi}_u^r(\bar{f}_{n-u-r}^{-1})^*(b_3,b_1)$, by composition law,
	\begin{equation*}
		\bar{\psi}_u^r(\bar{f}_{n-u-r}^{-1})^*(b_3,b_1)=\sum\limits_{b_2 \in \mathbb{N}}\bar{\psi}_u^r(b_3,b_2)(\bar{f}_{n-u-r}^{-1})^*(b_2,b_1)
	\end{equation*}
	
	Note that for any $b_2 \in \mathbb{N}$, we have:
	\begin{equation*}
		(\bar{f}_{n-u-r}^{-1})^*(b_2,b_1)(z_{b,i})=(\bar{f}_{n-u-r}^{-1}(b_1,b_2))^*(z_{b,i}) \in C_{n-u-r}''(b_2)^*=\mathop{\oplus}\limits_{j=1}^{l}(\mathbb{Z}\Pi^{N_{n-u-r}(j)})^*
	\end{equation*}
	
	Choose any element in $\mathop{\oplus}\limits_{j=1}^{l}\mathbb{Z}\Pi^{N_{n-u-r}(j)}$, write it as $\mathop{\oplus}\limits_{j=1}^l\sum\limits_{s=1}^{N_{n-u-r}(j)}\sum\limits_{\omega' \in \Pi}a(j,s,\omega')\omega'\epsilon_s^{\Pi}$ with $a(j,s,\omega') \in \mathbb{Z}$, then we have:
	\begin{equation*}
			\begin{aligned}
			& \quad \big((\bar{f}_{n-u-r}^{-1}(b_1,b_2))^*(z_{b,i})\big)\big(\mathop{\oplus}\limits_{j=1}^l\sum\limits_{s=1}^{N_{n-u-r}(j)}\sum\limits_{\omega' \in \Pi}a(j,s,\omega')\omega'\epsilon_s^{\Pi}\big) \\
			& \quad (\text{By definition of dual morphism}) \\
			&=z_{b,i}\bigg(\bar{f}_{n-u-r}^{-1}(b_1,b_2)\big(\mathop{\oplus}\limits_{j=1}^l\sum\limits_{s=1}^{N_{n-u-r}(j)}\sum\limits_{\omega' \in \Pi}a(j,s,\omega')\omega'\epsilon_s^{\Pi}\big)\bigg) \\
			& \quad (\text{By definition } \ref{barf-1neq0},\ref{barf-10}) \\
			&=z_{b,i}\bigg(\mathop{\oplus}\limits_{j=1}^l\delta_{b_2,j}^{b_1}\Xi_{n-u-r,b_2,j}^{-1}\big(\sum\limits_{s=1}^{N_{n-u-r}(j)}\sum\limits_{\omega' \in \Pi}a(j,s,\omega')\omega'\epsilon_s^{\Pi}\big)\bigg) \\
		\end{aligned}
	\end{equation*}
	
	Since $z_{b,i} \in (\mathop{\oplus}\limits_{\omega \in \Pi}C_{n-u-r}(p\hat{p}\omega\hat{\sigma}_{b,i}))^*$ and the range of $\Xi_{n-u-r,b_2,j}^{-1}$ is $C_{n-u-r}(p\hat{p}\omega\hat{\sigma}_{b_2,j})$, we have:
	\begin{equation*}
		\begin{aligned}
			& \quad \big((\bar{f}_{n-u-r}^{-1}(b_1,b_2))^*(z_{b,i})\big)\big(\mathop{\oplus}\limits_{j=1}^l\sum\limits_{s=1}^{N_{n-u-r}(j)}\sum\limits_{\omega' \in \Pi}a(j,s,\omega')\omega'\epsilon_s^{\Pi}\big) \\
			&=\delta_{b,b_2}z_{b,i}\bigg(\Xi_{n-u-r,b,i}^{-1}\big(\sum\limits_{s=1}^{N_{n-u-r}(i)}\sum\limits_{\omega' \in \Pi}a(i,s,\omega')\omega'\epsilon_s^{\Pi}\big)\bigg) \\
			&=\delta_{b,b_2}z_{b,i}\bigg(\mathop{\oplus}\limits_{\omega' \in \Pi}\sum\limits_{s=1}^{N_{n-u-r}(i)}a(i,s,\omega'^{-1})e_{n-u-r}^s(i)\bigg) \ \ (\text{By Definition } \ref{Xi}) \\
		\end{aligned}
	\end{equation*}
	It is straightforward to check that the following map is an isomorphism between (right) $\mathbb{Z}\Pi$-modules:
	\begin{equation*}
		\mathop{\oplus}\limits_{\omega \in \Pi}C_{n-u-r}(p\hat{p}\omega\hat{\sigma}_{b,i})^* \longrightarrow \big(\mathop{\oplus}\limits_{\omega \in \Pi}C_{n-u-r}(p\hat{p}\omega\hat{\sigma}_{b,i})\big)^*
	\end{equation*}
	\begin{equation*}
		\mathop{\oplus}\limits_{\omega \in \Pi} f_\omega \mapsto (\mathop{\oplus}\limits_{\omega' \in \Pi}y_{\omega'} \mapsto \sum\limits_{\omega \in \Pi}\sum\limits_{\omega' \in \Pi}f_{\omega}(y_{\omega'})\omega\omega'^{-1})
	\end{equation*}

	Thus, we have:
	\begin{equation}
		\label{quad6}
		\begin{aligned}
				& \quad \big((\bar{f}_{n-u-r}^{-1}(b_1,b_2))^*(z_{b,i})\big)\big(\mathop{\oplus}\limits_{j=1}^l\sum\limits_{s=1}^{N_{n-u-r}(j)}\sum\limits_{\omega' \in \Pi}a(j,s,\omega')\omega'\epsilon_s^{\Pi}\big) \\
				&=\delta_{b,b_2}z_{b,i}\bigg(\mathop{\oplus}\limits_{\omega' \in \Pi}\sum\limits_{s=1}^{N_{n-u-r}(i)}a(i,s,\omega'^{-1})e_{n-u-r}^s(i)\bigg)  \\
				&=\delta_{b,b_2}\sum\limits_{\omega \in \Pi}\sum\limits_{\omega' \in \Pi}z_{b,i,\omega}\big(\sum\limits_{s=1}^{N_{n-u-r}(i)}a(i,s,\omega'^{-1})e_{n-u-r}^s(i)\big)\omega\omega'^{-1}\\
				&=\delta_{b,b_2}\sum\limits_{\omega \in \Pi}\sum\limits_{\omega' \in \Pi}\sum\limits_{s=1}^{N_{n-u-r}(i)}z_{b,i,\omega,s}a(i,s,\omega')\omega\omega' \\
		\end{aligned}
	\end{equation}
	
	Note that, for any $\omega \in \Pi$, we have:
	\begin{equation}
		\label{quad7}
		\begin{aligned}
			& \quad \big(\mathop{\oplus}\limits_{j=1}^l\delta_{i,j}\sum\limits_{s'=1}^{N_{n-u-r}(j)}z_{b,i,\omega,s'}\omega^{-1}(\epsilon_{s'}^{\Pi})^*\big)\big(\mathop{\oplus}\limits_{j=1}^l\sum\limits_{s=1}^{N_{n-u-r}(j)}\sum\limits_{\omega' \in \Pi}a(j,s,\omega')\omega'\epsilon_s^{\Pi}\big) \\
			&=\big(\sum\limits_{s'=1}^{N_{n-u-r}(i)}z_{b,i,\omega,s'}\omega^{-1}(\epsilon_{s'}^{\Pi})^*\big)\big(\sum\limits_{s=1}^{N_{n-u-r}(i)}\sum\limits_{\omega' \in \Pi}a(i,s,\omega')\omega'\epsilon_s^{\Pi}\big) \\
			& \quad (\text{By definition of the right module structure on dual module}) \\
			&=\sum\limits_{s'=1}^{N_{n-u-r}(j)}z_{b,i,\omega,s'}\omega \cdot \big(\epsilon_{s'}^{\Pi})^* (\sum\limits_{s=1}^{N_{n-u-r}(j)}\sum\limits_{\omega' \in \Pi}a(i,s,\omega')\omega'\epsilon_s^{\Pi}\big)\\
			&=\sum\limits_{s=1}^{N_{n-u-r}(j)}\sum\limits_{\omega' \in \Pi}z_{b,i,\omega,s}a(i,s,\omega')\omega\omega'
		\end{aligned}
	\end{equation}
	
	Comparing \ref{quad6} and \ref{quad7} we get:
	\begin{equation}
		\label{quad8}
		(\bar{f}_{n-u-r}^{-1}(b_1,b_2))^*(z_{b,i})=\delta_{b,b_2}\sum\limits_{\omega \in \Pi}\mathop{\oplus}\limits_{j=1}^l\delta_{i,j}\sum\limits_{s=1}^{N_{n-u-r}(j)}z_{b,i,\omega,s}\omega^{-1}(\epsilon_{s}^{\Pi})^*
	\end{equation}
	
	For $1 \leq v \leq c$, $j \in \mathbf{S}_v$, by definition \ref{Grelquad} we can get:
	\begin{equation*}
		\tilde{g}(j,i(v))\tilde{g}(i,i(v))^{-1}\tilde{g}_{b_1}=\tilde{g}_{i(v),j}^{-1}\tilde{g}_{i(v),i}\tilde{g}_{b_1}=\tilde{g}_{i(v),j}^{-1}\tilde{g}_{\mathbf{b}_v}\mathbf{h}_v^{-1}=\tilde{g}_{\mathbf{b}_{j,v}}\mathbf{h}_{j,v}\mathbf{h}_v^{-1}
	\end{equation*}
	
	Recall that $\rho_G:G \longrightarrow M_{\infty}(\mathbb{Z}\Pi)$ is the induced action of $H$ on $\mathbb{Z}\Pi$ given by $h \cdot w=r_{\Pi}(h)w$ and the indexing is given by $G/H$. Thus:
	\begin{equation}
		\label{quad9}
		\rho_G\big(\tilde{g}(j,i(v))\tilde{g}(i,i(v))^{-1}\big)[b_3,b]=\begin{cases} r_{\Pi}(\mathbf{h}_{j,v}\mathbf{h}_v^{-1}) & \text{If} \ j \in \mathbf{S}_v \text{ and } b_3=\mathbf{b}_{j,v}\\ 0 & \text{else} \\
		\end{cases}
	\end{equation}
	
	Then we can continue our computation:
	\begin{equation}
		\label{othersidequad}
			\begin{aligned}
			& \quad \sum\limits_{b_2 \in \mathbb{N}}\bar{\psi}_u^r(b_3,b_2)(\bar{f}_{n-u-r}^{-1}(b_1,b_2))^*(z_{b,i}) \\
			& \quad (\text{By } \ref{quad8}) \\
			&=\bar{\psi}_u^r(b_3,b)\big(\sum\limits_{\omega \in \Pi}\mathop{\oplus}\limits_{j=1}^l\delta_{i,j}\sum\limits_{s=1}^{N_{n-u-r}(j)}z_{b,i,\omega,s}\omega^{-1}(\epsilon_{s}^{\Pi})^*\big)\\
			&=\bar{\psi}_u^r(b_3,b)\big(\mathop{\oplus}\limits_{j=1}^l\delta_{i,j}\sum\limits_{\omega \in \Pi}\sum\limits_{s=1}^{N_{n-u-r}(j)}z_{b,i,\omega,s}\omega^{-1}(\epsilon_{s}^{\Pi})^*\big) \\
			& \quad (\text{By definition } \ref{barpsi}) \\
			&=\mathop{\oplus}\limits_{j=1}^l\sum\limits_{\omega \in \Pi}\sum\limits_{\substack{\sigma_k \leq \sigma_i,\sigma_j \\ |\sigma_k|=0}}\rho_G(\tilde{g}(j,k)\tilde{g}(i,k)^{-1})[b_3,b] \cdot \bigg(\psi_{u,\sigma_k}^r(j,i)\big(\sum\limits_{s=1}^{N_{n-u-r}(i)}z_{b,i,\omega,s}\omega^{-1}(\epsilon_{s}^{\Pi})^*\big)\bigg)\\
			&=\mathop{\oplus}\limits_{j=1}^l\sum\limits_{\omega \in \Pi}\sum\limits_{\substack{1 \leq v \leq c \\ j \in \mathbf{S}_v}}\rho_G(\tilde{g}(j,i(v))\tilde{g}(i,i(v))^{-1})[b_3,b] \cdot \bigg(\psi_{u,x_v}^r(j,i)\big(\sum\limits_{s=1}^{N_{n-u-r}(i)}z_{b,i,\omega,s}\omega^{-1}(\epsilon_{s}^{\Pi})^*\big)\bigg)\\
			& \quad (\text{By } \ref{quad9}) \\
			&=\mathop{\oplus}\limits_{j=1}^l\sum\limits_{\omega \in \Pi}\sum\limits_{\substack{1 \leq v \leq u \\ j \in \mathbf{S}_v}}\delta_{b_3,\mathbf{b}_{j,v}}r_{\Pi}(\mathbf{h}_{j,v}\mathbf{h}_v^{-1}) \cdot \bigg(\psi_{u,x_v}^r(j,i)\big(\sum\limits_{s=1}^{N_{n-u-r}(i)}z_{b,i,\omega,s}\omega^{-1}(\epsilon_{s}^{\Pi})^*\big)\bigg) \\
		\end{aligned}
	\end{equation}
	
 	Comparing \ref{onesidequad} and \ref{othersidequad} we can get $(\bar{f}_r\bar{\theta}_u-\bar{\psi}_u^r(\bar{f}_{n-u-r}^{-1})^*)(b_3,b_1)=0$. That is, the set $\{(b_3,b_1) \in \mathbb{N}^2 \ | \ b_3 \neq 0,b_1 \neq 0, (\bar{f}_r\bar{\theta}_u-\bar{\psi}_u^r(\bar{f}_{n-u-r}^{-1})^*)(b_3,b_1)\}$ is an empty set, in particular, it is a finite set. Thus we have checked that $f$ pushes the structure $\theta$ to $\psi''$. Summarizing the result in (g) and (h), we have proven $(D,\theta) \cong \Theta\rho_*(C_*(\widetilde{M}),\psi(\widetilde{M}))$.
 	
 	($\mathbf{5}$) Now we begin to prove the second statement, that is, there is a quadratic chain complex $(E,\xi)$ in $\mathbb{F}_{\mathbb{N}}(M^h(\mathbb{Z}\Pi))$, such that $[(E,\xi)]=(D,\theta)$ and $\partial(E,\xi)$ is cobordant to $(C_{res}^{uni},\psi_{res}^{uni})$.
 	
 	Since the differential in the chain complex and the quadratic structure is defined to be the equivalence class of $\bar{d}_D,\bar{\theta}$, it is natural to guess that they give a quadratic chain complex in $\mathbb{F}_{\mathbb{N}}(M^h(\mathbb{Z}\Pi))$. Unfortunately, this is not the case, we need to modify the morphisms slightly.
 	
 	Let us briefly describe the idea of the modification, the problem of the above choice lies in the following fact: the unions of all dual simplices of simplices in $W_{\infty}$, $\mathop{\cup}\limits_{\bar{\sigma} \in W_{\infty}}\bar{\sigma}^*$, is not a subcomplex of $SdK_{\overline{M}}$. We will resolve this by choosing the simplices in the "interior" of $W_{\infty}$: $\mathop{\cup}\limits_{\bar{\sigma} \in W_{\infty} \backslash \bar{e}(N \times S^1)}\bar{\sigma}^*$. Moreover, the boundary of this space is homeomorphic to $N \times S^1$.
 	
 	\begin{gather*}
 		\begin{tikzpicture}
 			\node at (0,0) {$N \times D^2$};
 			\draw (-2,0) -- (-1,1.7);
 			\draw (-1,1.7)--(1,1.7);
 			\draw (2,0) -- (1,1.7);
 			\draw (-2,0) -- (-1,-1.7);
 			\draw (1,-1.7) -- (-1,-1.7);
 			\draw (2,0) -- (1,-1.7);
 			\draw (-4,0) -- (-2,3.4);
 			\draw (-2,3.4)--(2,3.4);
 			\draw (4,0) -- (2,3.4);
 			\draw (-4,0) -- (-2,-3.4);
 			\draw (2,-3.4) -- (-2,-3.4);
 			\draw (4,0) -- (2,-3.4);
 			\draw (-4,0) -- (-2,0);
 			\draw (-4,0) -- (-1,1.7);
 			\draw (-2,3.4) -- (-1,1.7);
 			\draw (-2,3.4) -- (1,1.7);
 			\draw (2,3.4) -- (1,1.7);
 			\draw (2,3.4) -- (2,0);
 			\draw (4,0) -- (2,0);
 			\draw (4,0) -- (1,-1.7);
 			\draw (2,-3.4) -- (1,-1.7);
 			\draw (2,-3.4) -- (-1,-1.7);
 			\draw (-2,-3.4) -- (-1,-1.7);
 			\draw (-2,-3.4) -- (-2,0);
 			\filldraw[draw=white,fill=blue!20] (-4,0) -- (-6,1.7) -- (-4,3.4) -- (-2,3.4)--cycle;
 			\filldraw[draw=white,fill=blue!20] (-2,3.4) -- (-4,3.4) -- (-1,5.1)--cycle;
 			\filldraw[draw=white,fill=blue!20] (-2,3.4) -- (-1,5.1) -- (1,5.1) -- (2,3.4)--cycle;
 			\filldraw[draw=white,fill=blue!20] (2,3.4) -- (1,5.1) -- (4,3.4)--cycle;
 			\filldraw[draw=white,fill=blue!20] (2,3.4) -- (4,3.4) -- (6,1.7) -- (4,0)--cycle;
 			\filldraw[draw=white,fill=blue!20] (6,1.7) -- (4,0) -- (6,-1.7)--cycle;
 			\filldraw[draw=white,fill=blue!20] (6,-1.7) -- (4,0) -- (2,-3.4) -- (4,-3.4)--cycle;
 			\filldraw[draw=white,fill=blue!20] (4,-3.4) -- (2,-3.4) -- (1,-5.1);
 			\filldraw[draw=white,fill=blue!20] (1,-5.1) -- (2,-3.4) -- (-2,-3.4) -- (-1,-5.1)--cycle;
 			\filldraw[draw=white,fill=blue!20] (-1,-5.1) -- (-2,-3.4) -- (-4,-3.4)--cycle;
 			\filldraw[draw=white,fill=blue!20] (-4,-3.4) -- (-2,-3.4) -- (-4,0) -- (-6,-1.7)--cycle;
 			\filldraw[draw=white,fill=blue!20] (-6,-1.7) -- (-4,0) -- (-6,1.7)--cycle;
 			\draw (-4,0) -- (-6,1.7);
 			\draw (-4,0) -- (-6,-1.7);
 			\draw (-2,3.4) -- (-4,3.4);
 			\draw (-2,3.4) -- (-1,5.1);
 			\draw (2,3.4) -- (4,3.4);
 			\draw (2,3.4) -- (1,5.1);
 			\draw (4,0) -- (6,1.7);
 			\draw (4,0) -- (6,-1.7);
 			\draw (2,-3.4) -- (4,-3.4);
 			\draw (2,-3.4) -- (1,-5.1);
 			\draw (-2,-3.4) -- (-4,-3.4);
 			\draw (-2,-3.4) -- (-1,-5.1);
 			\filldraw[draw=blue!80,fill=blue!20] (-4,0) -- (-3,0) -- (-2.33,0.56)--cycle;
 			\draw (-3,0) -- (-2,0) -- (-2.33,0.56)--cycle;
 			\draw (-2,0) -- (-1.5,0.85) -- (-2.33,0.56)--cycle;
 			\draw (-1.5,0.85) -- (-1,1.7) -- (-2.33,0.56)--cycle;
 			\draw (-2.5,0.85) -- (-1,1.7) -- (-2.33,0.56)--cycle;
 			\filldraw[fill=blue!20] (-2.5,0.85) -- (-4,0) -- (-2.33,0.56)--cycle;
 			\filldraw[fill=blue!20] (-2.5,0.85) -- (-4,0) -- (-2.33,1.7)--cycle;
 			\draw (-2.5,0.85) -- (-1,1.7) -- (-2.33,1.7)--cycle;
 			\draw (-1.5,2.55) -- (-1,1.7) -- (-2.33,1.7)--cycle;
 			\filldraw[fill=blue!20] (-1.5,2.55) -- (-2,3.4) -- (-2.33,1.7)--cycle;
 			\filldraw[fill=blue!20] (-3,1.7) -- (-4,0) -- (-2.33,1.7)--cycle;
 			\filldraw[fill=blue!20] (-3,1.7) -- (-2,3.4) -- (-2.33,1.7)--cycle;
 			\draw (-1,1.7) -- (0,1.7) -- (-0.67,2.27)--cycle;
 			\draw (1,1.7) -- (0,1.7) -- (-0.67,2.27)--cycle;
 			\draw (1,1.7) -- (-0.5,2.55) -- (-0.67,2.27)--cycle;
 			\filldraw[fill=blue!20] (-0.5,2.55) -- (-2,3.4) -- (-0.67,2.27)--cycle;
 			\filldraw[fill=blue!20] (-1.5,2.55) -- (-2,3.4) -- (-0.67,2.27)--cycle;
 			\draw (-0.5,2.55) -- (-1,1.7) -- (-0.67,2.27)--cycle;
 			\filldraw[fill=blue!20] (-2,3.4) -- (0,3.4) -- (0.33,2.83)--cycle;
 			\filldraw[fill=blue!20] (2,3.4) -- (0,3.4) -- (0.33,2.83)--cycle;
 			\filldraw[fill=blue!20] (2,3.4) -- (1.5,2.55) -- (0.33,2.83)--cycle;
 			\draw (1.5,2.55) -- (1,1.7) -- (0.33,2.83)--cycle;
 			\draw (-0.5,2.55) -- (1,1.7) -- (0.33,2.83)--cycle;
 			\filldraw[fill=blue!20] (-0.5,2.55) -- (-2,3.4) -- (0.33,2.83)--cycle;
 			\draw (1.5,2.55) -- (1,1.7) -- (1.67,1.7)--cycle;
 			\filldraw[fill=blue!20] (1.5,2.55) -- (2,3.4) -- (1.67,1.7)--cycle;
 			\filldraw[fill=blue!20] (2,1.7) -- (2,3.4) -- (1.67,1.7)--cycle;
 			\draw (2,0) -- (2,3.4) -- (1.67,1.7)--cycle;
 			\draw (2,0) -- (1.5,0.85) -- (1.67,1.7)--cycle;
 			\draw (1,1.7) -- (1.5,0.85) -- (1.67,1.7)--cycle;
 			\draw (2,0) -- (3,0) -- (2.67,1.13)--cycle;
 			\filldraw[fill=blue!20] (4,0) -- (3,0) -- (2.67,1.13)--cycle;
 			\filldraw[fill=blue!20] (4,0) -- (3,1.7) -- (2.67,1.13)--cycle;
 			\filldraw[fill=blue!20] (2,3.4) -- (3,1.7) -- (2.67,1.13)--cycle;
 			\filldraw[fill=blue!20] (2,3.4) -- (2,1.7) -- (2.67,1.13)--cycle;
 			\draw (2,0) -- (2,1.7) -- (2.67,1.13)--cycle;
 			\filldraw[fill=blue!20] (4,0) -- (3,0) -- (2.33,-0.56)--cycle;
 			\draw (3,0) -- (2,0) -- (2.33,-0.56)--cycle;
 			\draw (2,0) -- (1.5,-0.85) -- (2.33,-0.56)--cycle;
 			\draw (1.5,-0.85) -- (1,-1.7) -- (2.33,-0.56)--cycle;
 			\draw (2.5,-0.85) -- (1,-1.7) -- (2.33,-0.56)--cycle;
 			\filldraw[fill=blue!20] (2.5,-0.85) -- (4,0) -- (2.33,-0.56)--cycle;
 			\filldraw[fill=blue!20] (2.5,-0.85) -- (4,0) -- (2.33,-1.7)--cycle;
 			\draw (2.5,-0.85) -- (1,-1.7) -- (2.33,-1.7)--cycle;
 			\draw (1.5,-2.55) -- (1,-1.7) -- (2.33,-1.7)--cycle;
 			\filldraw[fill=blue!20] (1.5,-2.55) -- (2,-3.4) -- (2.33,-1.7)--cycle;
 			\filldraw[fill=blue!20] (3,-1.7) -- (4,0) -- (2.33,-1.7)--cycle;
 			\filldraw[fill=blue!20] (3,-1.7) -- (2,-3.4) -- (2.33,-1.7)--cycle;
 			\draw (1,-1.7) -- (0,-1.7) -- (0.67,-2.27)--cycle;
 			\draw (-1,-1.7) -- (0,-1.7) -- (0.67,-2.27)--cycle;
 			\draw (-1,-1.7) -- (0.5,-2.55) -- (0.67,-2.27)--cycle;
 			\filldraw[fill=blue!20] (0.5,-2.55) -- (2,-3.4) -- (0.67,-2.27)--cycle;
 			\filldraw[fill=blue!20] (1.5,-2.55) -- (2,-3.4) -- (0.67,-2.27)--cycle;
 			\draw (0.5,-2.55) -- (1,-1.7) -- (0.67,-2.27)--cycle;
 			\filldraw[fill=blue!20] (2,-3.4) -- (0,-3.4) -- (-0.33,-2.83)--cycle;
 			\filldraw[fill=blue!20] (-2,-3.4) -- (0,-3.4) -- (-0.33,-2.83)--cycle;
 			\filldraw[fill=blue!20] (-2,-3.4) -- (-1.5,-2.55) -- (-0.33,-2.83)--cycle;
 			\draw (-1.5,-2.55) -- (-1,-1.7) -- (-0.33,-2.83)--cycle;
 			\draw (0.5,-2.55) -- (-1,-1.7) -- (-0.33,-2.83)--cycle;
 			\filldraw[fill=blue!20] (0.5,-2.55) -- (2,-3.4) -- (-0.33,-2.83)--cycle;
 			\draw (-1.5,-2.55) -- (-1,-1.7) -- (-1.67,-1.7)--cycle;
 			\filldraw[fill=blue!20] (-1.5,-2.55) -- (-2,-3.4) -- (-1.67,-1.7)--cycle;
 			\filldraw[fill=blue!20] (-2,-1.7) -- (-2,-3.4) -- (-1.67,-1.7)--cycle;
 			\draw (-2,0) -- (-2,-3.4) -- (-1.67,-1.7)--cycle;
 			\draw (-2,0) -- (-1.5,-0.85) -- (-1.67,-1.7)--cycle;
 			\draw (-1,-1.7) -- (-1.5,-0.85) -- (-1.67,-1.7)--cycle;
 			\draw (-2,0) -- (-3,0) -- (-2.67,-1.13)--cycle;
 			\filldraw[fill=blue!20] (-4,0) -- (-3,0) -- (-2.67,-1.13)--cycle;
 			\filldraw[fill=blue!20] (-4,0) -- (-3,-1.7) -- (-2.67,-1.13)--cycle;
 			\filldraw[fill=blue!20] (-2,-3.4) -- (-3,-1.7) -- (-2.67,-1.13)--cycle;
 			\filldraw[fill=blue!20] (-2,-3.4) -- (-2,-1.7) -- (-2.67,-1.13)--cycle;
 			\draw (-2,0) -- (-2,-1.7) -- (-2.67,-1.13)--cycle;
 		\end{tikzpicture}
 		\\
 		\text{A picture of } \mathop{\cup}\limits_{\bar{\sigma} \in W_{\infty} \backslash \bar{e}(N \times S^1)}\bar{\sigma}^*
 	\end{gather*}
 
 	Before getting to the proof, let us introduce some notations first. Let $\mathring{W}_{\infty}$ be set of all the simplices that are in $W_{\infty}$ but not in $\bar{e}(N \times S^1)=\partial W_{\infty}$. It is an upper closed set in $\overline{M}$ and $\mathring{W}_{\infty}^-=W_0$. We make the following definitions:
 	\begin{equation*}
 		\begin{aligned}
 			&\mathring{d}_r(\hat{\tau},\hat{\sigma}): C_r(p\hat{p}\hat{\sigma}) \longrightarrow C_{r-1}(p\hat{p}\hat{\tau}) \\
 			&\mathring{d}_r(\hat{\tau},\hat{\sigma}):=\begin{cases} d_{C,r}(p\hat{p}\hat{\tau},p\hat{p}\hat{\sigma}) & \text{If} \ \hat{\sigma} \leq \hat{\tau} \ \text{and} \ \hat{p}\hat{\tau},\hat{p}\hat{\sigma} \in W_0 \\ \ \ \ \ \ \ \ 0 & \text{else}\\	\end{cases} \\
 			&\mathring{\psi}_u^r(\hat{\tau},\hat{\sigma}): C_{n-u-r}(p\hat{p}\hat{\sigma})^* \longrightarrow C_r(p\hat{p}\hat{\tau}) \\
 			&\mathring{\psi}_u^r(\hat{\tau},\hat{\sigma}):=\begin{cases} \sum\limits_{\substack{\hat{x} \in \hat{\tau} \cap \hat{\sigma} \\ |\hat{x}|=0}} \psi_u^r(p\hat{p}\hat{\tau},p\hat{p}\hat{x}) \Upsilon^{n-u-r}_{\hat{x},\hat{\sigma}} & \text{If} \  \hat{\tau} \cap \hat{\sigma} \neq \emptyset \ \text{and} \ \hat{p}\hat{\tau},\hat{p}\hat{\sigma} \in W_0 \\ \ \ \ \ \ \ \ \ \ \ \ \ \ \ 0 & \text{else}\end{cases}
 		\end{aligned}
 	\end{equation*}
	
	Now we begin constructing the prescribed quadratic chain complex in $\mathbb{F}_{\mathbb{N}}(M^h(\mathbb{Z}\Pi))$, it is given as follows:
	
	For every $r \in \mathbb{Z},b,b',u \in \mathbb{N}$:
	\begin{equation*}
		E_r(b)=\mathop{\oplus}\limits_{\hat{\sigma} \in \widehat{K}_b}C_r(p\hat{p}\hat{\sigma}) \text{ for } b>0
	\end{equation*}
	\begin{equation*}
		E_r(0)=\mathop{\oplus}\limits_{\hat{\sigma} \in \widehat{K}_0}C_r(p\hat{p}\hat{\sigma}) \oplus \big(\mathop{\oplus}\limits_{\sigma \notin W} C_r(\sigma) \otimes_{\mathbb{Z}}\mathbb{Z}\Pi\big)
	\end{equation*}
	\begin{equation*}
		d_E(b',b)=\mathop{\boxplus}\limits_{\hat{\sigma} \in \widehat{K}_b}\mathop{\oplus}\limits_{\hat{\tau} \in \widehat{K}_{b'}} \mathring{d}_{r}(\hat{\tau},\hat{\sigma}):E_r(b) \longrightarrow E_{r-1}(b')
	\end{equation*}
	\begin{equation*}
		\xi_{u}^r(b',b)=\mathop{\boxplus}\limits_{\hat{\sigma} \in \widehat{K}_b}\mathop{\oplus}\limits_{\hat{\tau} \in \widehat{K}_{b'}}\mathring{\psi}_u^r(\hat{\tau},\hat{\sigma}) :E_{n-u-r}(b) \longrightarrow E_r(b')
	\end{equation*}
	
	We will check step by step the properties that is stated in the theorem. The first is to check:
	
	(i) $E$ is chain complex in $\mathbb{F}_{\mathbb{N}}(M^h(\mathbb{Z}\Pi))$.
	
	To start with, let us check that $d_E$ is well defined. We need to check that for fixed $b \in \mathbb{N}$, the sets $\{b' \ | \ d_E(b',b) \neq 0\}$ and $\{b' \ | \ d_E(b,b') \neq 0\}$ are finite. Notice that $\mathring{d}_r(\hat{\tau},\hat{\sigma}) \neq 0$ implies $\hat{\sigma} \leq \hat{\tau}$. If $d_E(b',b) \neq 0$, then there is $\hat{\tau} \in \widehat{K}_{b'},\hat{\sigma} \in \widehat{K}_{b}$ with $\mathring{d}_r(\hat{\tau},\hat{\sigma}) \neq 0$ and so we have $\hat{\sigma} \leq \hat{\tau}$. Since we have $\hat{p}\hat{\tau} \in K_{b'} \subset K^{b'}$ by definition, we get that $\hat{p}\hat{\sigma} \in K^{b'}$ and there is vertex $v \leq \hat{p}\hat{\tau}$ with $d_{\infty}(v)=b'$. By (1) in lemma \ref{distance}, $\hat{p}\hat{\sigma} \notin K^{b'-2}$. Summarizing the above we get:
	\begin{equation*}
		d_E(b',b) \neq 0\Rightarrow b=b' \text{ or } b=b'-1
	\end{equation*}
	 
	From that one can easily see that $d_E$ is well defined.
	
	Then we need to check that $d_E^2=0$. Let $(\dot{C},\dot{\psi})$ be the infinite transfer of $(C,\psi)$ with respect to $p$. By Lemma \ref{infinitetransfer}, $(\dot{C},\dot{\psi})$ is a quadratic chain complex in $M^h(R)_*^{lf}(\overline{M})$. 
	
	Let $(\dot{C}(\widehat{W}_0),\dot{\psi}(\widehat{W}_0))$ be the partial assembly of $\dot{C}$ over $W_0$ with respect to $\hat{p}$. Denote $pr_{b'}: \dot{C}_*(\widehat{W}_0) \longrightarrow E_*(b')$ to be the composition of maps:
	\begin{equation*}
		\begin{tikzcd}
			\dot{C}_*(\widehat{W}_0)=\mathop{\oplus}\limits_{\hat{\sigma} \in \hat{p}^{-1}(W_0)}C(p\hat{p}\hat{\sigma}) \ar{rr}{\text{projection}} & & \mathop{\oplus}\limits_{\hat{\sigma} \in \hat{p}^{-1}(K_{b'} \cap W_0)}C(p\hat{p}\hat{\sigma}) \ar[hook]{r}{\subset} & E_*(b')
		\end{tikzcd}
	\end{equation*}
	
	 Then we claim that:

	 $d_E^2(b',b)$ is the same with $pr_{b'}d_{\dot{C}}(\widehat{W}_0)^2$ restricted on the direct summand $\mathop{\oplus}\limits_{\substack{\hat{\sigma} \in \widehat{K}_b \\ \hat{p}\hat{\sigma} \in W_0}}C_*(p\hat{p}\hat{\sigma}) \subset E_*(b)$.
	
	Combining the claim with Theorem \ref{partialassemblyfun}, we get $d_E^2(b',b)=0$. To prove the claim, choose $\hat{\sigma} \in \widehat{K}_{b}$ with $\hat{p}\hat{\sigma} \in W_0$ and choose $z \in C_r(p\hat{p}\hat{\sigma})$. Then:
	\begin{equation*}
		\begin{aligned}
			d_E^2(b',b)(z)&=\sum\limits_{b'' \in \mathbb{N}}d_E(b',b'')d_E(b'',b)(z) \\
			&=\sum\limits_{b'' \in \mathbb{N}}\sum\limits_{\substack{\hat{\kappa} \in \widehat{K}_{b''} \\ \hat{p}\hat{\kappa} \in W_0}}\mathop{\oplus}\limits_{\substack{\hat{\tau} \in \widehat{K}_{b'} \\ \hat{p}\hat{\tau} \in W_0}} \mathring{d}_{r-1}(\hat{\tau},\hat{\kappa})\mathring{d}_{r}(\hat{\kappa},\hat{\sigma})(z) \\
			& \quad (\text{Since for different } b'', K_{b''} \text{ is disjoint}) \\
			&=\mathop{\oplus}\limits_{\hat{p}\hat{\tau} \in W_0 \cap K_{b'}}\sum\limits_{\hat{p}\hat{\kappa} \in W_0}\mathring{d}_{r-1}(\hat{\tau},\hat{\kappa}) \mathring{d}_{r}(\hat{\kappa},\hat{\sigma})(z) \\
		\end{aligned}
	\end{equation*}

	Since $\hat{p}\hat{\sigma} \in W_0$, we have $\mathring{d}_{r}(\hat{\kappa},\hat{\sigma})=d_{\dot{C},r}(\hat{\kappa},\hat{\sigma})$ and $\mathring{d}_{r-1}(\hat{\tau},\hat{\kappa})=d_{\dot{C},r+1}(\hat{\tau},\hat{\kappa})$. Thus:
	\begin{equation*}
		\begin{aligned}
			d_E^2(b',b)(z)&=\mathop{\oplus}\limits_{\hat{p}\hat{\tau} \in K_{b'} \cap W_0}\sum\limits_{\hat{p}\hat{\kappa} \in W_0}d_{\dot{C},r-1}(\hat{\tau},\hat{\kappa}) d_{\dot{C},r}(\hat{\kappa},\hat{\sigma})(z) \\
			&=pr_{b'}\big(\mathop{\oplus}\limits_{\hat{p}\hat{\tau} \in W_0}\sum\limits_{\hat{p}\hat{\kappa} \in W_0}d_{\dot{C},r-1}(\hat{\tau},\hat{\kappa}) d_{\dot{C},r}(\hat{\kappa},\hat{\sigma})(z) \big) \\
			&=pr_{b'}d_{\dot{C}}(\widehat{W}_0)^2(z) \\
		\end{aligned}
	\end{equation*}
	
	Therefore, the claim is verified.
	
	(j) $\xi$ gives a quadratic structure on $E$.
	
	The proof is analogus to the proof above. To start with, using (1) in Lemma \ref{distance}, similar to (i) we can prove that for all $u \in \mathbb{N},r \in \mathbb{Z}$:
	\begin{equation*}
		\xi_u^r(b',b) \neq 0 \Rightarrow |b'-b| \leq 2
	\end{equation*}

	Therefore, $\xi_u^r$ is well defined. Then we need to check that the following equation holds for all $b,b',u \in \mathbb{N},r \in \mathbb{Z}$:
	\begin{equation}
		\label{quadcheck1}
		\chi_{u,r}^{b',b}=0:E_{r'}(b)^* \longrightarrow E_r(b')
	\end{equation}

	Where 
	\begin{equation*}
		\begin{aligned}
			\chi_{u,r}^{b',b}= \ &(d_C\xi_u^{r+1})(b',b)+(-1)^r(\xi_u^rd_C^*)(b',b)+(-1)^{n-u-1}\xi_{u+1}^r(b',b) \\
			&+(-1)^n(-1)^{rr'}\xi_{u+1}^{r'}(b,b')^*
		\end{aligned}
	\end{equation*}
	\begin{equation*}
		r'=n-r-u-1
	\end{equation*}

	Let $(\dot{C},\dot{\psi})$, $(\dot{C}(\widehat{W}_0),\dot{\psi}(\widehat{W}_0))$ and $pr_{b'}: \dot{C}_*(\widehat{W}_0) \longrightarrow E_*(b')$ be the same as in (i). Let $\Upsilon^*[\widehat{W}_0]$ be the map defined in Lemma \ref{Dualexchange} with $S=\mathring{W}_{\infty}$ and $\mathbf{p}=\hat{p}$. Let $p_{\mathring{W}_{\infty}}^{dual}$ and $i^{\mathring{W}_{\infty},dual}$ be the maps defined in Corollary \ref{partialquad}. Choose $\hat{\sigma} \in \widehat{K}_{b}$ with $\hat{p}\hat{\sigma} \in W_0$ and choose $z \in C_{r'}(p\hat{p}\hat{\sigma})^*$. Similar to the computations in (i), we can compute $\chi_{u,r}^{b',b}(z)$ term by term and we get:
	\begin{equation*}
			(d_C\xi_u^{r+1})(b',b)(z)=pr_{b'}d_{\dot{C}}(\widehat{W}_0)\dot{\psi}_u^{r+1}(\widehat{W}_0)p_{\mathring{W}_{\infty}}^{dual}\Upsilon^{r'}[\widehat{W}_0](z)
	\end{equation*}
	\begin{equation*}
		(\xi_u^rd_C^*)(b',b)(z)=pr_{b'}\dot{\psi}_u^{r}(\widehat{W}_0)p_{\mathring{W}_{\infty}}^{dual}\Upsilon^{r'+1}[\widehat{W}_0]d_{\dot{C}}(\widehat{W}_0)^{cd}(z)
	\end{equation*}
	\begin{equation*}
		\xi_{u+1}^r(b',b)(z)=pr_{b'}\dot{\psi}_{u+1}^{r}(\widehat{W}_0)p_{\mathring{W}_{\infty}}^{dual}\Upsilon^{r'}[\widehat{W}_0](z)
	\end{equation*}
	\begin{equation*}
		\xi_{u+1}^{r'}(b,b')^*(z)=pr_{b'}\Upsilon^{r}[\widehat{W}_0]^{cd}i^{\mathring{W}_{\infty},dual}\dot{\psi}_{u+1}^{r'}(\widehat{W}_0)^{cd}(z)
	\end{equation*}
	
	Since $W_0=\mathring{W}_{\infty}^-$ and $\mathring{W}_{\infty}$ is upper closed in $\overline{M}$, by Corollary \ref{partialquad}, we have that the following equation holds:
	\begin{equation}
		\label{quadcheck2}
		\begin{aligned}
			0= \ &d_{\dot{C}}(\widehat{W}_0)\dot{\psi}_u^{r+1}(\widehat{W}_0)^{ass}+(-1)^r\dot{\psi}_u^{r}(\widehat{W}_0)^{ass}d_{\dot{C}}(\widehat{W}_0)^{cd} \\
			&+(-1)^{n-u-1}\dot{\psi}_{u+1}^{r}(\widehat{W}_0)^{ass}+(-1)^{n+rr'}\dot{\psi}_{u+1}^{r'}(\widehat{W}_0)^{ass}_{cd} \\
		\end{aligned}
	\end{equation}
	
	Note that by definition in Corollary \ref{partialquad}, for any $u_0 \in \mathbb{N},r_0 \in \mathbb{Z}$, we have:
	\begin{equation*}
		\dot{\psi}_{u_0}^{r_0}(\widehat{W}_0)^{ass}=\dot{\psi}_{u_0}^{r_0}(\widehat{W}_0)p_{\mathring{W}_{\infty}}^{dual}\Upsilon^{n-r_0-u_0}[\widehat{W}_0]
	\end{equation*}
	\begin{equation*}
		\dot{\psi}_{u_0}^{r_0}(\widehat{W}_0)^{ass}_{cd}=\Upsilon^{n-r_0-u_0}[\widehat{W}_0]^{cd}i^{\mathring{W}_{\infty},dual}\dot{\psi}_{u_0}^{r_0}(\widehat{W}_0)^{cd}
	\end{equation*}
 	
 	Now composing the equation \ref{quadcheck2} with $pr_{b'}$ and using the two equations above, one can get the equation \ref{quadcheck1}. Thus we finished the proof of $\xi$ to be a quadratic structure on $E$.
 	
 	(k) $[(E,\xi)]=(D,\theta)$
 	
 	By definition, we see that $E_r=D_r$ for all $r \in \mathbb{Z}$. Therefore, we only have to check that:
 	
 	The set $\{(b',b) \ | \ d_E(b',b) \neq \bar{d}_D(b',b)\}$ is finite.
 	
 	The set $\{(b',b) \ | \ \xi(b',b) \neq \bar{\theta}(b',b)\}$ is finite.
 	
 	(k1) The set $\{(b',b) \ | \ d_E(b',b) \neq \bar{d}_D(b',b)\}$ is finite.
 	
 	We claim that for all $r \in \mathbb{Z}$, $b>1$ and $b'>1$, $d_E(b',b)=\bar{d}_D(b',b)$. If the claim holds, since $d_E,\bar{d}_D$ are morphisms in $\mathbb{F}_{\mathbb{N}}(M^h(\mathbb{Z}\Pi))$, it will follow that $\{(b',b) \ | \ d_E(b',b) \neq \bar{d}_D(b',b)\}$ is finite.
 	
 	To prove the claim, note first that for any vertice $v$ in $\bar{e}(N \times S^1 \otimes [0,1])$, it is either in $\bar{e}(N \times S^1)$ or connected to a vertex in $\bar{e}(N \times S^1)$. Therefore, we have $d_{\infty}(v) \leq 1$ and by definition we have $\bar{e}(N \times S^1) \otimes [0,1] \subset K^1$. For any $b_0>1$, we have:
 	\begin{equation*}
 		K_{b_0} \subset W_{\infty} \backslash K^1 \subset  W_{\infty} \backslash \bar{e}(N \times S^1 \otimes [0,1]) \subset W_0
 	\end{equation*}
 	
 	Therefore, we have $K_b,K_{b'} \subset W_0$. For any $\hat{\sigma} \in \widehat{K}_b, \hat{\tau} \in \widehat{K}_{b'}$, by definition we have that $\mathring{d}_{r}(\hat{\tau},\hat{\sigma})=d_r(\hat{\tau},\hat{\sigma})$. Thus:
 	\begin{equation*}
 		d_E(b',b)=\mathop{\boxplus}\limits_{\hat{\sigma} \in \widehat{K}_b}\mathop{\oplus}\limits_{\hat{\tau} \in \widehat{K}_{b'}} \mathring{d}_{r}(\hat{\tau},\hat{\sigma})=\mathop{\boxplus}\limits_{\hat{\sigma} \in \widehat{K}_b}\mathop{\oplus}\limits_{\hat{\tau} \in \widehat{K}_{b'}} d_{r}(\hat{\tau},\hat{\sigma})=\bar{d}_D(b',b)
 	\end{equation*}
 
 	(k2) The set $\{(b',b) \ | \ \xi(b',b) \neq \bar{\theta}(b',b)\}$ is finite.
 	
 	The proof is similar. For $b>1,b'>1$, we have $K_{b},K_{b'} \subset W_0$ and thus by definition, for any $u \in \mathbb{N},r \in \mathbb{Z}$, we have $\mathring{\psi}_{u}^r(\hat{\tau},\hat{\sigma})=\psi_{u}^r(\hat{\tau},\hat{\sigma})$ for any $\hat{\sigma} \in \widehat{K}_b, \hat{\tau} \in \widehat{K}_{b'}$. Thus:
 	\begin{equation*}
 		\xi_u^r(b',b)=\mathop{\boxplus}\limits_{\hat{\sigma} \in \widehat{K}_b}\mathop{\oplus}\limits_{\hat{\tau} \in \widehat{K}_{b'}} \mathring{\psi}_{u}^r(\hat{\tau},\hat{\sigma})=\mathop{\boxplus}\limits_{\hat{\sigma} \in \widehat{K}_b}\mathop{\oplus}\limits_{\hat{\tau} \in \widehat{K}_{b'}} \psi_{u}^r(\hat{\tau},\hat{\sigma})=\bar{\theta}_u^r(b',b)
 	\end{equation*}
	
	(l) $\partial(E,\xi)$ is cobordant to $(C_{res}^{uni},\psi_{res}^{uni})$
	
	We make some observations on the chain complex $E$ first. For any $r \in \mathbb{Z}$ and $b \in \mathbb{N}$, denote
	\begin{equation*}
		\underleftarrow{E}_r(b)=\mathop{\oplus}\limits_{\substack{\hat{\sigma} \in \widehat{K}_b \\ \hat{p}\hat{\sigma} \in W_0}}C_r(p\hat{p}\hat{\sigma})
	\end{equation*}
	\begin{equation*}
		\underrightarrow{E}_r(b)=\mathop{\oplus}\limits_{\substack{\hat{\sigma} \in \widehat{K}_b \\ \hat{p}\hat{\sigma} \notin W_0}}C_r(p\hat{p}\hat{\sigma}) \text{ if } b \neq 0
	\end{equation*}
	\begin{equation*}
		\underrightarrow{E}_r(0)=\mathop{\oplus}\limits_{\substack{\hat{\sigma} \in \widehat{K}_0 \\ \hat{p}\hat{\sigma} \notin W_0}}C_r(p\hat{p}\hat{\sigma}) \oplus \big(\mathop{\oplus}\limits_{\sigma \notin W} C_r(\sigma) \otimes_{\mathbb{Z}}\mathbb{Z}\Pi\big)
	\end{equation*}

	Then $E_r=\underleftarrow{E}_r \oplus \underrightarrow{E}_r$. Denote $d_{\underleftarrow{E},r}:\underleftarrow{E}_r \longrightarrow \underleftarrow{E}_{r-1}$ and $\underleftarrow{\xi}_{u}^r:\underleftarrow{E}_{n-u-r}^* \longrightarrow \underleftarrow{E}_{r}$ to be the restriction of $d_{E,r}$ and $\xi_{u}^r$, then by definition of $E$ and $\xi$ we have:
	\begin{equation}
		\label{decomE}
		(E,d_E)=(\underleftarrow{E},d_{\underleftarrow{E}}) \oplus (\underrightarrow{E},0), \  \xi_{u}^r=\underleftarrow{\xi}_{u}^r \oplus 0
	\end{equation}
	
	Thus $(\underleftarrow{E},\underleftarrow{\xi})$ is a quadratic chain complex in $\mathbb{F}_{\mathbb{N}}(M^h(\mathbb{Z}\Pi))$. We will represent it below as the algebraic Thom construction of some Poincare pair in $\mathbb{F}_{\mathbb{N}}(M^h(\mathbb{Z}\Pi))$.
	
	Let $E_{*}^{\partial},E_*^{all}$ be the following chain complexes in $\mathbb{F}_{\mathbb{N}}(M^h(\mathbb{Z}\Pi))$:
	
	For all $b,b' \in \mathbb{N},r \in \mathbb{Z}$:
	
	\begin{equation*}
		\begin{aligned}
			E^{\partial}_r(b)=\mathop{\oplus}\limits_{\substack{\hat{\sigma} \in \widehat{K}_b \\ \hat{p}\hat{\sigma} \in W_{\infty} \backslash W_0}}C_r(p\hat{p}\hat{\sigma}), \ d_{E^{\partial}}(b',b)=\mathop{\boxplus}\limits_{\substack{\hat{\sigma} \in \widehat{K}_b \\ \hat{p}\hat{\sigma} \in \mathring{W}_{\infty} \backslash W_0}}\mathop{\oplus}\limits_{\substack{\hat{\sigma}' \in \widehat{K}_{b'} \\ \hat{p}\hat{\sigma}' \in \mathring{W}_{\infty} \backslash W_0}}d_r(\hat{\tau},\hat{\sigma}) &\\
			E^{all}_r(b)=\mathop{\oplus}\limits_{\substack{\hat{\sigma} \in \widehat{K}_b \\ \hat{p}\hat{\sigma} \in \mathring{W}_{\infty}}}C_r(p\hat{p}\hat{\sigma}), \ d_{E^{all}}(b',b)= \mathop{\boxplus}\limits_{\substack{\hat{\sigma} \in \widehat{K}_b \\ \hat{p}\hat{\sigma} \in \mathring{W}_{\infty}}}\mathop{\oplus}\limits_{\substack{\hat{\sigma}' \in \widehat{K}_{b'} \\ \hat{p}\hat{\sigma}' \in \mathring{W}_{\infty}}}d_r(\hat{\tau},\hat{\sigma}) \qquad &\\
		\end{aligned}
	\end{equation*}

	Let $(\dot{C},\dot{\psi})$ be the infinite transfer of $(C,\psi)$ with respect to $p$. Similar to the argument in (i), we can prove that $d_{E^{\partial}},d_{E^{all}}$ are well-defined and the following equations hold:
	\begin{equation}
		\label{EqEpair}
		\begin{aligned}
			\mathop{\oplus}\limits_{b \in \mathbb{N}}E_r^{\partial}(b)=\dot{C}_r \big( \hat{p}^{-1}(\mathring{W}_{\infty} \backslash W_0) \big) \quad \ \ & \\
			\mathop{\boxplus}\limits_{b \in \mathbb{N}}\mathop{\oplus}\limits_{b' \in \mathbb{N}}d_{E^{\partial}}(b',b)=d_{\dot{C}} \big( \hat{p}^{-1}(\mathring{W}_{\infty} \backslash W_0) \big) & \\
			\mathop{\oplus}\limits_{b \in \mathbb{N}}E_r^{all}(b)=\dot{C}_r\big(\hat{p}^{-1}(\mathring{W}_{\infty}) \big) \qquad & \\
			\mathop{\boxplus}\limits_{b \in \mathbb{N}}\mathop{\oplus}\limits_{b' \in \mathbb{N}}d_{E^{all}}(b',b)=d_{\dot{C}}\big( \hat{p}^{-1}(\mathring{W}_{\infty}) \big) \ \  & \\
		\end{aligned}
	\end{equation}
	
	By Lemma \ref{partialassemblyfun}, we can deduce that they are chain complexes in $\mathbb{F}_{\mathbb{N}}(M^h(\mathbb{Z}\Pi))$. Let $i_E$ be the following morphism in $\mathbb{F}_{\mathbb{N}}(M^h(\mathbb{Z}\Pi))$:
	\begin{equation*}
		\begin{aligned}
			i_E(b',b):E_r^{\partial}(b) \longrightarrow E_r^{all}(b') &\\
			i_E(b',b)=\begin{cases} \text{inclusion} & \text{If } b=b' \\ \quad \ 0 & \text{else} \end{cases} &\\
		\end{aligned}
	\end{equation*}

	For all $u,b,b' \in \mathbb{N},r \in \mathbb{Z}$, let $\delta\psi_{u,r}^{ass},\psi_{u,r}^{ass}$ be the maps defined in Theorem \ref{thmpartialquadpair} with $K=M,L=N \times S^1,(X,\psi_X)=(C,\psi),p_K=p,\mathbf{p}=\hat{p}$. Then we have $\dot{C}^{\partial}=\dot{C} \big( \hat{p}^{-1}(\mathring{W}_{\infty} \backslash W_0) \big)$ and $\dot{C}^{all}=\dot{C} \big( \hat{p}^{-1}(\mathring{W}_{\infty}) \big)$. Let
	\begin{equation*}
		\begin{aligned}
			\delta\psi_{E,u}^r(b',b): E_{n-u-r}^{all}(b)^* \longrightarrow E_r^{all}(b') \\
			\psi_{E,u}^r(b',b): E_{n-u-r-1}^{\partial}(b)^* \longrightarrow E_r^{\partial}(b')
		\end{aligned}
	\end{equation*}
	
	be the corresponding component of $\delta\psi_{u,r}^{ass},\psi_{u,r}^{ass}$. By definition \ref{Defdetlapsiass+psiass} of $\delta\psi_{u,r}^{ass}$, \\$\psi_{u,r}^{ass}$ and Lemma \ref{DefT&varSigma}, we can make the same argument as in (i) to deduce that they are well defined. Then we claim that:
	
	$(\alpha)$ $(i_E:E^{\partial} \longrightarrow E^{all},(\psi_{E},\delta\psi_{E}))$ is a $n$-dimensional Poincare quadratic pair in $\mathbb{F}_{\mathbb{N}}(M^h(\mathbb{Z}\Pi))$.
	
	$(\beta)$ $(\underleftarrow{E},\underleftarrow{\xi})$ is the algebraic Thom construction of the Poincare pair in $(\alpha)$.

	To prove the claim, we will do the argument for $(\alpha)$ and $(\beta)$ seperately.
	
	Proof of $(\alpha)$: To start with, we check that the pair is $n$-dimensional quadratic in $\mathbb{F}_{\mathbb{N}}(M^h(\mathbb{Z}\Pi))$, that is, the following equation holds for all $b',b,u \in \mathbb{N},r \in \mathbb{Z}$:  
	\begin{equation}
		\label{Eq6.24}
		\begin{aligned}
			0=\ &d_{E^{all}}\delta\psi_{E,u}^{r+1}(b',b)-(-1)^{n-u}\delta\psi_{E,u}^r d_{(E^{all})^*}(b',b)+(-1)^{n-u-1}\delta\psi_{E,u+1}^r(b',b)  \\
			&+(-1)^{n+rr'}\delta\psi_{E,u+1}^{r'}(b,b')^* +(-1)^{n-1} i_E\psi_{E,u}^ri_E^*(b',b)\\
			& (r'=n-u-r-1)
		\end{aligned}
	\end{equation}
	
	Suming the maps over $b$ and $b'$, it is equivalent to:
	\begin{equation*}
		\begin{aligned}
			0=\ &\mathop{\boxplus}\limits_{b \in \mathbb{N}}\mathop{\oplus}\limits_{b' \in \mathbb{N}}d_{E^{all}}\delta\psi_{E,u}^{r+1}(b',b)-(-1)^{n-u}\mathop{\boxplus}\limits_{b \in \mathbb{N}}\mathop{\oplus}\limits_{b' \in \mathbb{N}}\delta\psi_{E,u}^r d_{(E^{all})^*}(b',b)\\
			&+(-1)^{n-u-1}\mathop{\boxplus}\limits_{b \in \mathbb{N}}\mathop{\oplus}\limits_{b' \in \mathbb{N}}\delta\psi_{E,u+1}^r(b',b)  +(-1)^{n+rr'}\mathop{\boxplus}\limits_{b \in \mathbb{N}}\mathop{\oplus}\limits_{b' \in \mathbb{N}}\delta\psi_{E,u+1}^{r'}(b,b')^* \\
			&+(-1)^{n-1} \mathop{\boxplus}\limits_{b \in \mathbb{N}}\mathop{\oplus}\limits_{b' \in \mathbb{N}}i_E\psi_{E,u}^ri_E^*(b',b)\\
		\end{aligned}
	\end{equation*}
	
	We can compute each term in the sum separately:
	\begin{equation*}
		\begin{aligned}
			\mathop{\boxplus}\limits_{b \in \mathbb{N}}\mathop{\oplus}\limits_{b' \in \mathbb{N}}d_{E^{all}}\delta\psi_{E,u}^{r+1}(b',b)&=\mathop{\boxplus}\limits_{b \in \mathbb{N}}\mathop{\oplus}\limits_{b' \in \mathbb{N}}\sum\limits_{b'' \in \mathbb{N}}d_{E^{all}}(b',b'')\delta\psi_{E,u}^{r+1}(b'',b) \\
			&=\big(\mathop{\boxplus}\limits_{b'' \in \mathbb{N}}\mathop{\oplus}\limits_{b' \in \mathbb{N}}d_{E^{all}}(b',b'')\big)\big(\mathop{\boxplus}\limits_{b \in \mathbb{N}}\mathop{\oplus}\limits_{b'' \in \mathbb{N}}\delta\psi_{E,u}^{r+1}(b'',b)\big) \\
			& \quad (\text{By } \ref{EqEpair} \text{ and definition of } \delta\psi_{E,u}^{r+1}(b'',b)) \\
			&=d_{\dot{C}^{all}}\delta\psi_{u,r+1}^{ass} \\
		\end{aligned}
	\end{equation*}
	
	Similarly we get:
	\begin{equation*}
		\mathop{\boxplus}\limits_{b \in \mathbb{N}}\mathop{\oplus}\limits_{b' \in \mathbb{N}}\delta\psi_{E,u}^r d_{(E^{all})^*}(b',b)=(-1)^{r'}\delta\psi_{u,r}^{ass}d_{\dot{C}^{all}}^{cd}
	\end{equation*}
	\begin{equation*}
		\mathop{\boxplus}\limits_{b \in \mathbb{N}}\mathop{\oplus}\limits_{b' \in \mathbb{N}}\delta\psi_{E,u+1}^r(b',b) =\delta\psi_{u+1,r}^{ass}
	\end{equation*}
	\begin{equation*}
		\mathop{\boxplus}\limits_{b \in \mathbb{N}}\mathop{\oplus}\limits_{b' \in \mathbb{N}}\delta\psi_{E,u+1}^{r'}(b,b')^*=(\delta\psi_{u+1,r'}^{ass})^{cd}
	\end{equation*}
	\begin{equation*}
		\mathop{\boxplus}\limits_{b \in \mathbb{N}}\mathop{\oplus}\limits_{b' \in \mathbb{N}}i_E\psi_{E,u}^ri_E^*(b',b)=i_{\dot{C}}\psi_{u,r}^{ass}i_{\dot{C}}^{cd}
	\end{equation*}
	
	Then it follows from Theorem \ref{thmpartialquadpair} that the equation \ref{Eq6.24} holds and thus it is a $n$-dimensional quadratic pair.
	
	Denote $\delta\varphi_r=\delta\psi_{E,0}^r+(-1)^{r(n-r)}(\delta\psi_{E,0}^{n-r})^*$. To prove that it is Poincare, we have to prove that the following morphism is a chain homotopy equivalence:
	\begin{equation*}
		(E_{n-r}^{all})^* \stackrel{\delta\varphi_r}{\longrightarrow} E_r^{all} \stackrel{Pr}{\longrightarrow} E_r^{all}/E_r^{\partial}
	\end{equation*}
	
	Where $Pr$ is the projection map.
	
	Denote $\delta\varphi_r^{sum}=	Pr \circ \big(\mathop{\boxplus}\limits_{b \in \mathbb{N}}\mathop{\oplus}\limits_{b' \in \mathbb{N}}\delta\varphi_r(b',b)\big)$. Let $Q_1,Q_2$ be two objects in $\mathbb{F}_{\mathbb{N}}(M^h(\mathbb{Z}\Pi))$, a $\mathbb{Z}\Pi$-morphism $f:\mathop{\oplus}\limits_{b \in \mathbb{N}}Q_1(b) \longrightarrow \mathop{\oplus}\limits_{b \in \mathbb{N}}Q_2(b)$ is called controlled, if there is a number $k \geq 0$, such that $f(Q_1(b)) \subset \mathop{\oplus}\limits_{\substack{b' \in \mathbb{N} \\ |b'-b| \leq k}}Q_2(b')$. Then the above statement can be deduced by definition from the following statement:
	
	Statement 1: There exists a chain map $\gamma_r:	\mathop{\oplus}\limits_{b \in \mathbb{N}}(E_{r}^{all}/E_r^{\partial})(b) \longrightarrow \mathop{\oplus}\limits_{b \in \mathbb{N}}(E_{n-r}^{all})^*(b)$ and homotopy equivalences $H_0:\delta\varphi^{sum} \circ \gamma \simeq Id$ and $H_1:\gamma \circ \delta\varphi^{sum} \simeq Id$, such that $\gamma_r,H_0,H_1$ are controlled.
	
	Now we begin to check that Statement 1 holds. Note that we have:
	\begin{equation*}
		\begin{aligned}
			\mathop{\oplus}\limits_{b \in \mathbb{N}}(E_{n-r}^{all})^*(b)=\mathop{\oplus}\limits_{b \in \mathbb{N}}E_{n-r}^{all}(b)^*=\dot{C}(\hat{p}^{-1}\big(\mathring{W}_{\infty})\big)^{cd} & \\
			\mathop{\oplus}\limits_{b \in \mathbb{N}}(E_{r}^{all}/E_r^{\partial})(b)=\dot{C}\big(\hat{p}^{-1}(W_0)\big) \qquad \quad \ \  & \\
		\end{aligned}
	\end{equation*}
	\begin{equation*}
		\begin{aligned}
			\mathop{\boxplus}\limits_{b \in \mathbb{N}}\mathop{\oplus}\limits_{b' \in \mathbb{N}}\delta\varphi_r(b',b)&=	\mathop{\boxplus}\limits_{b \in \mathbb{N}}\mathop{\oplus}\limits_{b' \in \mathbb{N}}\big(\delta\psi_{E,0}^r(b',b)+(-1)^{r(n-r)}\delta\psi_{E,0}^{n-r}(b,b')^*\big) \\
			&=\delta\psi_{0,r}^{ass}+(-1)^{r(n-r)} (\delta\psi_{0,n-r}^{ass})^{cd}\\
			& \quad (\text{By definition } \ref{Defdetlapsiass+psiass}) \\
			&=\dot{\psi}_{0}^r\big(\hat{p}^{-1}(\mathring{W}_{\infty})\big)\Upsilon^{n-r}[\hat{p}^{-1}(\mathring{W}_{\infty})] \\
			& \ +(-1)^{r(n-r)} \bigg(\dot{\psi}_{0}^{n-r}\big(\hat{p}^{-1}(\mathring{W}_{\infty})\big)\Upsilon^{r}[\hat{p}^{-1}(\mathring{W}_{\infty})]\bigg)^{cd}\\
		\end{aligned}
	\end{equation*}

	By Remark \ref{Dualexchangeupper}, we have:
	\begin{equation*}
		(-1)^{r(n-r)} \bigg(\dot{\psi}_{0}^{n-r}\big(\hat{p}^{-1}(\mathring{W}_{\infty})\big)\Upsilon^{r}[\hat{p}^{-1}(\mathring{W}_{\infty})]\bigg)^{cd}=T\dot{\psi}_0^r(\hat{p}^{-1}(\mathring{W}_{\infty})) \Upsilon^{n-r}[\hat{p}^{-1}(\mathring{W}_{\infty})]
	\end{equation*}
	
	Therefore, we have $\delta\varphi_r^{sum}= Pr \circ (1+T) \dot{\psi}_0^r\big(\hat{p}^{-1}(\mathring{W}_{\infty}) \big)\Upsilon^{n-r}[\hat{p}^{-1}(\mathring{W}_{\infty})]$. Since $\mathring{W}_{\infty}$ is upper closed and $W_0$ is a subcomplex, we can further deduce that:
	\begin{equation*}
		\delta\varphi_r^{sum}= (1+T) \dot{\psi}_0^r\big(\hat{p}^{-1}(W_0)\big) \circ Pr \circ \Upsilon^{n-r}[\hat{p}^{-1}(\mathring{W}_{\infty})]
	\end{equation*}

	Since $(\dot{C},\dot{\psi})$ is Poincare, we have that $(1+T)\dot{\psi}_0^r$ is a chain homotopy equivalence of chain complexes in $M^h(R)_*^{lf}(\overline{M})$. Since $W_0$ is a subcomplex, by Lemma \ref{partialassemblyfun}, $(1+T) \dot{\psi}_0^r\big(\hat{p}^{-1}(W_0)\big)$ induces a chain homotopy equivalence. Furthermore, similar to the proof in (i), the homotopy inverse and the corresponding chain homotopy equivalence are controlled.
	
	For the morphism $Pr \circ \Upsilon^{n-r}[\hat{p}^{-1}(\mathring{W}_{\infty})]:\dot{C}_{n-r}\big(\hat{p}^{-1}(\mathring{W}_{\infty})\big) \longrightarrow \dot{C}^{n-r}\big(\hat{p}^{-1}(W_0)\big)$, we make some identification on the chain complex $\dot{C}^{n-r}(\hat{p}^{-1}(W_0))$ first, we have:
	\begin{equation*}
		\begin{aligned}
			\dot{C}^{n-r}\big(\hat{p}^{-1}(W_0)\big)&=\mathop{\oplus}\limits_{\substack{\hat{\sigma} \in \widehat{W}_{\infty} \\ \hat{p}\hat{\sigma} \in W_0}}\dot{C}^{n-r}(\hat{p}\hat{\sigma}) \\
			&=\mathop{\oplus}\limits_{\substack{\hat{\sigma} \in \widehat{W}_{\infty} \\ \hat{p}\hat{\sigma} \in W_0}}\mathop{\oplus}\limits_{\tau \geq \hat{p}\hat{\sigma}}\dot{C}_{n-r-|\hat{\sigma}|}(\tau)^* \\
			&=\mathop{\oplus}\limits_{\substack{\hat{\sigma} \in \widehat{W}_{\infty} \\ \hat{p}\hat{\sigma} \in W_0}}\mathop{\oplus}\limits_{\hat{\tau} \geq \hat{\sigma}}\dot{C}_{n-r-|\hat{\sigma}|}(\hat{p}\hat{\tau})^* \\
			&=\mathop{\oplus}\limits_{\hat{\tau} \in \hat{p}^{-1}(\mathring{W}_{\infty})}\mathop{\oplus}\limits_{\substack{\hat{\sigma} \leq \hat{\tau} \\ \hat{p}\hat{\sigma} \in W_0}}\dot{C}_{n-r-|\hat{\sigma}|}(\hat{p}\hat{\tau})^* \\
			&=\mathop{\oplus}\limits_{\hat{\tau} \in \hat{p}^{-1}(\mathring{W}_{\infty})}\bigg(\dot{C}(\hat{p}\hat{\tau})^* \otimes_{\mathbb{Z}} \Delta^*\big(\hat{\tau} \cap \hat{p}^{-1}(W_0)\big)\bigg)_{n-r} \\
		\end{aligned}
	\end{equation*}

	Moreover, by viewing the right hand side as the partial assembly of chain complexes in $M^h(\mathbb{Z})_*^{lf}(\overline{M})$, this identification is an idenfication of chain complexes. Denote $0\mathbb{Z}$ to be the chain complex with only $\mathbb{Z}$ on dimension $0$ and all the others being zero. Denote $sum_{\hat{\tau}}^*:0\mathbb{Z} \longrightarrow \Delta^*(\hat{\tau} \cap \hat{p}^{-1}(W_0))$ to be the chain map given by the dual of the sum map on $0$-dimension. Then we have $Pr \circ \Upsilon^{n-r}[\hat{p}^{-1}(\mathring{W}_{\infty})]:\dot{C}_{n-r}(\hat{p}^{-1}(\mathring{W}_{\infty})) \longrightarrow \dot{C}^{n-r}(\hat{p}^{-1}(W_0))$ under the identification above is given by $\mathop{\oplus}\limits_{\hat{\tau} \in \hat{p}^{-1}(\mathring{W}_{\infty})} Id \otimes_{\mathbb{Z}} sum_{\hat{\tau}}^*$. Since every simplex $\hat{\tau} \in \hat{p}^{-1}(\mathring{W}_{\infty})$ intersects with $\hat{p}^{-1}(W_0)$ either itself or a codimension 1 face, we have that $sum_{\hat{\tau}}^*$ is a chain homotopy equivalence for all $\hat{\tau}$. Then by Proposition 9.11 in \cite{ranickiltheory} we have that $Pr \circ \Upsilon^{n-r}[\hat{p}^{-1}(\mathring{W}_{\infty})]$ is a chain homotopy equivalence. Moreover, since it is the assemble of local chain homotopy equivalences, following the similar arguments in (i), we can prove that the homotopy inverse and the corresponding homotopy equivalence are controlled.
	
	Now since the composition and finite linear combinations of controlled morphisms are still controlled. We get that Statement 1 holds. 
	
	Proof of $(\beta)$: Note first that for every $r \in \mathbb{Z}$, we have $E_r^{all}=E_r^{\partial} \oplus \underleftarrow{E}_r$. Moreover, if we denote $p_E:E_r^{all} \longrightarrow \underleftarrow{E}_r$ to be the projection map, then it is a chain map. Therefore, to prove $(\beta)$, it is equivalent to prove that for all $u \in \mathbb{N},r \in \mathbb{Z}$, we have $\underline{\xi}_u^r=p_E\delta\psi_{E,u}^rp_E^*$.
	
	Denote $p_E^{sum}=\mathop{\boxplus}\limits_{b \in \mathbb{N}}\mathop{\oplus}\limits_{b' \in \mathbb{N}}p_E(b',b)$, it is still the projection map. Moreover, we have:
	\begin{equation*}
		\begin{aligned}
			\mathop{\boxplus}\limits_{b \in \mathbb{N}}\mathop{\oplus}\limits_{b' \in \mathbb{N}}p_E\delta\psi_{E,u}^rp_E^*(b',b)&=p_E^{sum}\delta\psi_{u,r}^{ass}(p_E^{sum})^{cd} \\
			& \quad (\text{By definition } \ref{Defdetlapsiass+psiass}) \\
			&=p_E^{sum}\dot{\psi}_u^r\big(\hat{p}^{-1}(\mathring{W}_{\infty})\big)\Upsilon^{n-u-r}[\hat{p}^{-1}(\mathring{W}_{\infty})](p_E^{sum})^{cd} \\
			&=p_E^{sum}\dot{\psi}_u^r\big(\hat{p}^{-1}(\mathring{W}_{\infty})\big)\Upsilon^{n-u-r}[\hat{p}^{-1}(\mathring{W}_{\infty})^-] \\
		\end{aligned}
	\end{equation*}

	By definition of $\underline{\xi}_u^r$, we have:
	\begin{equation*}
		\begin{aligned}
			\mathop{\boxplus}\limits_{b \in \mathbb{N}}\mathop{\oplus}\limits_{b' \in \mathbb{N}}\underline{\xi}_u^r(b',b)&=\mathop{\boxplus}\limits_{\hat{p}\hat{\sigma} \in W_0}\mathop{\oplus}\limits_{\hat{p}\hat{\tau} \in W_0} \mathring{\psi}_u^r(\hat{\tau},\hat{\sigma}) \\
			&=\mathop{\boxplus}\limits_{\hat{p}\hat{\sigma} \in W_0}\mathop{\oplus}\limits_{\hat{p}\hat{\tau} \in W_0}\sum_{\substack{\hat{x} \in \hat{\sigma} \cap \hat{\tau} \\ |\hat{x}|=0}}\psi_u^r(p\hat{p}\hat{\tau},p\hat{p}\hat{x})\Upsilon_{\hat{x},\hat{\sigma}}^{n-u-r} \\
			&=\big(\mathop{\boxplus}\limits_{\hat{p}\hat{\sigma}' \in \mathring{W}_{\infty}}\mathop{\oplus}\limits_{\substack{\hat{p}\hat{\tau} \in W_0 \\ \hat{\tau} \geq \hat{\sigma}'}}\psi_u^r(p\hat{p}\hat{\tau},p\hat{p}\hat{\sigma}')\big) \circ \big(\mathop{\boxplus}\limits_{\hat{p}\hat{\sigma} \in W_0}\mathop{\oplus}\limits_{\substack{\hat{x} \leq \hat{\sigma} \\ |\hat{x}| \geq 0}}\Upsilon_{\hat{x},\hat{\sigma}}^{n-u-r}\big) \\
		\end{aligned}
	\end{equation*}

	By definition of $\Upsilon^{n-u-r}[\hat{p}^{-1}(\mathring{W}_{\infty})^-]$ in Lemma \ref{Dualexchange}, we have:
	\begin{equation*}
		\begin{aligned}
			\mathop{\boxplus}\limits_{b \in \mathbb{N}}\mathop{\oplus}\limits_{b' \in \mathbb{N}}\underline{\xi}_u^r(b',b)&=\big(\mathop{\boxplus}\limits_{\hat{p}\hat{\sigma}' \in \mathring{W}_{\infty}}\mathop{\oplus}\limits_{\substack{\hat{p}\hat{\tau} \in W_0 \\ \hat{\tau} \geq \hat{\sigma}'}}\dot{\psi}_u^r(\hat{p}\hat{\tau},\hat{p}\hat{\sigma}')\big) \circ \Upsilon^{n-u-r}[\hat{p}^{-1}(\mathring{W}_{\infty})^-] \\
			&=p_E^{sum}\big(\mathop{\boxplus}\limits_{\hat{p}\hat{\sigma}' \in \mathring{W}_{\infty}}\mathop{\oplus}\limits_{\hat{\tau} \geq \hat{\sigma}'}\dot{\psi}_u^r(\hat{p}\hat{\tau},\hat{p}\hat{\sigma}')\big) \circ \Upsilon^{n-u-r}[\hat{p}^{-1}(\mathring{W}_{\infty})^-] \\
			&=p_E^{sum}\dot{\psi}_u^r\big(\hat{p}^{-1}(\mathring{W}_{\infty})\big)\Upsilon^{n-u-r}[\hat{p}^{-1}(\mathring{W}_{\infty})^-] \\
		\end{aligned}
	\end{equation*}

	Therefore, we have $\mathop{\boxplus}\limits_{b \in \mathbb{N}}\mathop{\oplus}\limits_{b' \in \mathbb{N}}p_E\delta\psi_{E,u}^rp_E^*(b',b)=\mathop{\boxplus}\limits_{b \in \mathbb{N}}\mathop{\oplus}\limits_{b' \in \mathbb{N}}\underline{\xi}_u^r(b',b)$, from that we can deduce $p_E\delta\psi_{E,u}^rp_E^*=\underline{\xi}_u^r$.
	
	Now we return to the proof of (l), by the claim $(\alpha),(\beta)$ and the decomposition \ref{decomE}, we have:
	\begin{equation*}
		\partial(E,\xi)=\partial(\underleftarrow{E},\underleftarrow{\xi}) \oplus \partial(\underrightarrow{E},0)=(E^{\partial},\psi_E) \oplus \partial(\underrightarrow{E},0)
	\end{equation*}
	
	Since for $b>1$, we have $K_b \subset W_0$, thus $\underrightarrow{E}_*(b)=E_*^{\partial}(b)=0$. By (1) in Lemma \ref{thmTheta} and the fact that $\iota_{\infty}$ is fully faithful, there are quadratic chain complexes $(E',\psi_{E'}),(E'',\psi_{E''})$ in $M^h(\mathbb{Z}\Pi)$, such that:
	\begin{equation*}
		\iota_{\infty}(E',\psi_{E'}) \cong (\underrightarrow{E},0), \ \iota_{\infty}(E'',\psi_{E''}) \cong (E^{\partial},\psi_{E})
	\end{equation*}

	Thus $\partial(\underrightarrow{E},0)$ is homotopy equivalent to the boundary of a quadratic chain complex in $M^h(\mathbb{Z}\Pi)$, so it is null cobordant and we have that $\partial(E,\xi)$ is cobordant to $(E'',\psi_{E''})$.
	
	By the proof of Lemma \ref{thmTheta}, for any $r \in \mathbb{Z}, u \in \mathbb{N}$, we can get $E''_r=\mathop{\oplus}\limits_{b \in \mathbb{N}}E_r^{\partial}(b)$ and $\psi_{E'',u}^r=\mathop{\boxplus}\limits_{b \in \mathbb{N}}\mathop{\oplus}\limits_{b' \in \mathbb{N}}\psi_{E,u}^r(b',b)$. By equation \ref{EqEpair} and definition of $\psi_{E,u}^r(b',b)$ we have $(E'',\psi_{E''})=(\dot{C}^{\partial},\psi^{ass})$. We claim that $(\dot{C}^{\partial},\psi^{ass})$ is the same with $(C_{res}^{uni},\psi_{res}^{uni})$ and then $(l)$ follows from the claim.
	
	To prove the claim, recall first that by definition $(C_{res}^{uni},\psi_{res}^{uni})$ is the universal assembly of $(C,\psi)|_{N \times S^1}$. We have a description of $(C,\psi)|_{N \times S^1}$ as $(DL,\theta L)$ in Lemma \ref{surgeryres}. We can then make some computations:
	
	For all $r \in \mathbb{Z}$, we have:
	\begin{equation*}
		\begin{aligned}
			(C_{res}^{uni})_r&=\mathop{\oplus}\limits_{\hat{p}\hat{\sigma} \in \bar{e}(N \times S^1)}DL_r(p\hat{p}\hat{\sigma}) \\
			& \quad (\text{By the expression of } DL_r \text{ in Lemma } \ref{surgeryres}) \\
			&=\mathop{\oplus}\limits_{\hat{p}\hat{\sigma} \in \bar{e}(N \times S^1)}\mathop{\oplus}\limits_{s \in B_{p\hat{p}\hat{\sigma}}}C_r(s) \\
		\end{aligned}
	\end{equation*}
	
	Denote $U=N \times S^1 \otimes [0,1] \backslash N \times S^1 \otimes \partial [0,1]$. For any $s \in U$, denote $s_0$ be the intersection of $s$ with $(N \times S^1) \otimes 0$. For any $\hat{\sigma} \in \hat{p}^{-1}\bar{e}(N \times S^1)$, denote $A_{\hat{\sigma}}=\{\hat{s} \in \hat{p}^{-1}\bar{e}(U) \ | \ \hat{s} \geq \hat{\sigma}\}$ and $B_{\hat{\sigma}}=\{\hat{s} \in \hat{p}^{-1}\bar{e}(U) \ | \ \hat{s} \geq \hat{\sigma}, (p\hat{p}\hat{s})_0=p\hat{p}\hat{\sigma}\}$. By Remark \ref{RemB}, we have $B_{\sigma}=\{s \in U \ | \  s_0=\sigma\}$ with $\sigma \in N \times S^1$. Therefore, we can deduce that the following map gives a bijection for $\sigma=p\hat{p}\hat{\sigma}$:
	\begin{equation}
		\label{liftA}
		\hat{p}_*:A_{\hat{\sigma}} \longrightarrow A_{\sigma}, \ \hat{s} \mapsto p\hat{p}\hat{s}
	\end{equation}
	\begin{equation}
		\label{liftB}
		\hat{p}_*:B_{\hat{\sigma}} \longrightarrow B_{\sigma}, \ \hat{s} \mapsto p\hat{p}\hat{s}
	\end{equation}
	
	Therefore, we have $(C_{res}^{uni})_r=\mathop{\oplus}\limits_{\hat{p}\hat{\sigma} \in \bar{e}(N \times S^1)}\mathop{\oplus}\limits_{\hat{s} \in B_{\hat{\sigma}}}C_r(p\hat{p}\hat{s})$. Using the fact that $B_{\sigma}$ is disjoint for different $\sigma$ and their union of is $U$, it is straightforward to check that $B_{\hat{\sigma}}$ is disjoint for different $\hat{\sigma}$ and their union is $\hat{p}^{-1}\bar{e}(U)$. Therefore, we have $(C_{res}^{uni})_r=\mathop{\oplus}\limits_{\hat{p}\hat{s} \in \bar{e}(U)}C_r(p\hat{p}\hat{s})$.

	By definition we have $\dot{C}^{\partial}_r=\dot{C}_r(\hat{p}^{-1}\bar{e}(U))=\mathop{\oplus}\limits_{\hat{p}\hat{s} \in \bar{e}(U)}\dot{C}_r(\hat{p}\hat{s})=\mathop{\oplus}\limits_{\hat{p}\hat{s} \in \bar{e}(U)}C_r(p\hat{p}\hat{s})$. Thus we get $(C_{res}^{uni})_r=\dot{C}^{\partial}_r$.
	
	Next we will check that the differentials of the two chain complexes agree. For any $r \in \mathbb{Z}$, choose $\hat{s} \in \hat{p}^{-1}\bar{e}(U)$ and $z \in C_r(p\hat{p}\hat{s})$, let $s=p\hat{p}\hat{s} \in U$ and $s_0$ as above. Then $s \in B_{s_0}$ and there is a unique simplex $\hat{\sigma}_s \in \hat{p}^{-1}\bar{e}(N \times S^1)$, such that $p\hat{p}\hat{\sigma}_s=s_0$ and $\hat{\sigma}_s \leq \hat{s}$. Therefore, by the identification above, we have $z \in DL_r(p\hat{p}\hat{\sigma}_s)$ and thus:
	\begin{equation*}
		\begin{aligned}
			d_{C_{res}^{uni}}(z)&=\mathop{\oplus}\limits_{\substack{\hat{\tau} \geq \hat{\sigma}_s \\ \hat{p}\hat{\tau} \in \bar{e}(N \times S^1)}}d_{DL}(\hat{\tau},\hat{\sigma}_s)(z) \\
			& \quad (\text{By expression of } d_{DL} \text{ in Lemma } \ref{surgeryres}) \\
			&=\mathop{\oplus}\limits_{\substack{\hat{\tau} \geq \hat{\sigma}_s \\ \hat{p}\hat{\tau} \in \bar{e}(N \times S^1)}} \big( \mathop{\boxplus}\limits_{s'' \in B_{p\hat{p}\hat{\sigma}_s}} \mathop{\oplus}\limits_{s' \in B_{p\hat{p}\hat{\tau}}} d_{C}(s',s'') \big)(z) \\
			&=\mathop{\oplus}\limits_{\substack{\hat{\tau} \geq \hat{\sigma}_s \\ \hat{p}\hat{\tau} \in \bar{e}(N \times S^1)}}\mathop{\oplus}\limits_{s' \in B_{p\hat{p}\hat{\tau}}} d_{C}(s',s)(z) \\
			&=\mathop{\oplus}\limits_{\substack{\hat{\tau} \geq \hat{\sigma}_s \\ \hat{p}\hat{\tau} \in \bar{e}(N \times S^1)}}\mathop{\oplus}\limits_{\substack{s' \in B_{p\hat{p}\hat{\tau}} \\ s' \geq s}} d_{C}(s',s)(z) \\
		\end{aligned}
	\end{equation*}
	
	Note that for any simplex $s' \geq s$, there is a unique simplex $\hat{s}' \geq \hat{s}$, such that $p\hat{p}\hat{s}'=s'$. Moreover, since $\hat{\tau} \geq \hat{\sigma}_s \leq \hat{s} \leq \hat{s}'$, we have $\hat{\tau} \cap \hat{s}' \neq \emptyset$. Therefore, we get $s' \in B_{p\hat{p}\hat{\tau}} \Leftrightarrow \hat{s}' \in B_{\hat{\tau}}$. Denote $B'=\mathop{\amalg}\limits_{\substack{\hat{\tau} \geq \hat{\sigma}_s \\ \hat{p}\hat{\tau} \in \bar{e}(N \times S^1)}}B_{\hat{\tau}}$, then we have:
	\begin{equation*}
		\begin{aligned}
			d_{C_{res}^{uni}}(z)
			&=\mathop{\oplus}\limits_{\substack{\hat{\tau} \geq \hat{\sigma}_s \\ \hat{p}\hat{\tau} \in \bar{e}(N \times S^1)}}\mathop{\oplus}\limits_{\substack{\hat{s}' \in B_{\hat{\tau}} \\ \hat{s}' \geq \hat{s}}} d_{C}(p\hat{p}\hat{s}',p\hat{p}\hat{s})(z) \\
			&=\mathop{\oplus}\limits_{\substack{\hat{s}' \in B' \\ \hat{s}' \geq \hat{s}}} d_{C}(p\hat{p}\hat{s}',p\hat{p}\hat{s})(z) \\
		\end{aligned}
	\end{equation*}
	
	Now note that $U$ is upper closed and by our simplicial setting \ref{simplicial setting}, $\bar{e}$ is a simplicial embedding on $U$. Since by definition $\hat{p}\hat{s} \in \bar{e}(U)$, we have $\hat{p}\hat{s}' \in \bar{e}(U)$ for any $\hat{s}' \geq \hat{s}$. Let $\tau=(p\hat{p}\hat{s}')_0 \in N \times S^1$, then there is a unique $\hat{\tau} \leq \hat{s}'$, such that $p\hat{p}\hat{\tau}=\tau$. Then by definition we have $\hat{s}' \in B_{\hat{\tau}}$. Since $\hat{s}' \geq \hat{s}$, we have $p\hat{p}\hat{s}' \geq p\hat{p}\hat{s}$ and thus $\tau \geq s_0$. Since $\hat{\sigma}_s \leq \hat{s} \leq \hat{s}'$ and $\hat{\tau} \leq \hat{s}'$, combining with $\tau=p\hat{p}\hat{\tau} \geq p\hat{p}\hat{\sigma}_s=s_0$, we can conclude that $\hat{\tau} \geq \hat{\sigma}_s$. Therefore, we have $B_{\hat{\tau}} \subset B'$. In summary, we get $\hat{s}' \geq \hat{s}$ implies $\hat{s}' \in B'$, thus:
	\begin{equation*}
		d_{C_{res}^{uni}}(z)=\mathop{\oplus}\limits_{\hat{s}' \geq \hat{s}} d_{C}(p\hat{p}\hat{s}',p\hat{p}\hat{s})(z) \\
	\end{equation*}
	
	On the other hand, by definition we have $d_{\dot{C}^{\partial}}(z)=\mathop{\oplus}\limits_{\substack{\hat{s}' \geq \hat{s} \\ \hat{p}\hat{s'} \in \bar{e}(U)}}d_{\dot{C}}(\hat{p}\hat{s}',\hat{p}\hat{s})(z)=\mathop{\oplus}\limits_{\substack{\hat{s}' \geq \hat{s} \\ \hat{p}\hat{s'} \in \bar{e}(U)}}d_{C}(p\hat{p}\hat{s}',p\hat{p}\hat{s})(z)$. Since $U$ is upper closed and $\bar{e}$ is a simplicial embedding on $U$, we can conclude that $d_{C_{res}^{uni}}=d_{\dot{C}^{\partial}}$.
	
	Finally we have to check that the quadratic structure agrees. For any $u \in \mathbb{N},r \in \mathbb{Z}$, choose $\hat{s} \in \hat{p}^{-1}\bar{e}(U)$ and $z \in C_{n-u-r-1}(p\hat{p}\hat{s})^*$, let $s=p\hat{p}\hat{s} \in U$ and $s_0,\hat{\sigma}_s$ be as above. Then by the identification above we have $z \in DL_{n-u-r-1}(p\hat{p}\hat{\sigma}_s)^*$. Denote $V_0,V_1$ to be the set of all vertices of $s$ that are in $(N \times S^1) \otimes 0,(N \times S^1) \otimes 1$, respectively. Denote $\widehat{V}_0,\widehat{V}_1$ to be the set of all vertices of $\hat{s}$ that project to $V_0,V_1$. For any $\hat{\eta} \leq \hat{s}$, denote $\iota_{\hat{\eta}}$ to be the following inclusion map:
	\begin{equation*}
		\iota_{\hat{\eta}}: C_{n-u-r-1}(p\hat{p}\hat{s})^* \longrightarrow C^{n-u-r-1+|\hat{\eta}|}(p\hat{p}\hat{\eta})=\mathop{\oplus}\limits_{\hat{\kappa} \geq \hat{\eta}}C_{n-u-r-1}(p\hat{p}\hat{\kappa})^*
	\end{equation*}
	
	For any $\hat{v}_0 \in \widehat{V}_0$, we have:
	\begin{equation*}
		DL^{n-u-r-1}(p\hat{p}\hat{v}_0)=\mathop{\oplus}\limits_{\tau \in A_{p\hat{p}\hat{v}_0}}C_{n-u-r-1}(\tau)^*=\mathop{\oplus}\limits_{\hat{\tau} \in A_{\hat{v}_0}}C_{n-u-r-1}(p\hat{p}\hat{\tau})^* \ (\text{By } \ref{liftA})
	\end{equation*}

	By the definition of $A_{\hat{v}_0}$, we can embed $DL^{n-u-r}(p\hat{p}\hat{v}_0)$ as a direct summand into $C^{n-u-r}(p\hat{p}\hat{v}_0)=\mathop{\oplus}\limits_{\hat{\kappa} \geq \hat{v}_0}C_{n-u-r}(p\hat{p}\hat{\kappa})^*$. Denote $\iota_{\hat{v}_0}'$ to be the composition of $\iota_{\hat{v}_0}$ with projection onto $DL^{n-u-r-1}(p\hat{p}\hat{v}_0)$. Then we have:
	\begin{equation*}
		\begin{aligned}
			(\psi_{res}^{uni})_u^r(z)&=\theta L_u^r(\hat{p}^{-1}\bar{e}(N \times S^1))\big(\mathop{\oplus}\limits_{\hat{v}_0 \in \widehat{V}_0}\iota_{\hat{v}_0}'(z)\big) \\
			&=\sum\limits_{\hat{v}_0 \in \widehat{V}_0}\mathop{\oplus}\limits_{\substack{\hat{\tau} \geq \hat{v}_0 \\ \hat{p}\hat{\tau} \in \bar{e}(N \times S^1)}} \theta L_u^r(\hat{\tau},\hat{v}_0)\iota_{\hat{v}_0}'(z) \\
		\end{aligned}
	\end{equation*}

	By the definition in Lemma \ref{surgeryres}, $\theta L_u^r(\hat{\tau},\hat{v}_0)$ is $(-1)^{n+r+1}$ times the composition of $\mathop{\boxplus}\limits_{s'' \in B_{p\hat{p}\hat{v}_0}}\mathop{\oplus}\limits_{s' \in B_{p\hat{p}\hat{\tau}}}\psi_u^r(s',s'')$ with $\mho_{p\hat{p}\hat{v}_0}^r$. Note that we have $\iota_{\hat{v}_0'}(z) \in \mathop{\oplus}\limits_{\hat{\tau} \in A_{\hat{v}_0}}C_{n-u-r-1}(p\hat{p}\hat{\tau})^*$ and the element  lies in the direct summand $C_{n-u-r-1}(p\hat{p}\hat{s})^*$. By definition of $\mho_{p\hat{p}\hat{v}_0}^r$ in Lemma \ref{surgeryres} and the fact that $p\hat{p}:\widehat{V}_1 \longrightarrow V_1$ is a bijection, we have $\mho_{p\hat{p}\hat{v}_0}^r\iota_{\hat{v}_0}'(z)=\mathop{\oplus}\limits_{\hat{v}_1 \in \widehat{V}_1}\iota_{\hat{v}_0*\hat{v}_1}(z)$. Thus:
	\begin{equation}
		\label{Eq6.25}
		\begin{aligned}
			(\psi_{res}^{uni})_u^r(z)&=\sum\limits_{\hat{v}_0 \in \widehat{V}_0}\sum\limits_{\hat{v}_1 \in \widehat{V}_1}(-1)^{n+r+1} \mathop{\oplus}\limits_{\substack{\hat{\tau} \geq \hat{v}_0 \\ \hat{p}\hat{\tau} \in \bar{e}(N \times S^1)}} \mathop{\oplus}\limits_{s' \in B_{p\hat{p}\hat{\tau}}} \psi_u^r(s',p\hat{p}(\hat{v}_0*\hat{v}_1))\iota_{\hat{v}_0*\hat{v}_1}(z) \\
			&=(-1)^{n+r+1}\sum\limits_{\hat{v}_0 \in \widehat{V}_0}\sum\limits_{\hat{v}_1 \in \widehat{V}_1} \mathop{\oplus}\limits_{\substack{\hat{\tau} \geq \hat{v}_0 \\ \hat{p}\hat{\tau} \in \bar{e}(N \times S^1)}} \mathop{\oplus}\limits_{\substack{s' \in B_{p\hat{p}\hat{\tau}} \\ s' \geq p\hat{p}(\hat{v}_0*\hat{v}_1)}} \psi_u^r(s',p\hat{p}(\hat{v}_0*\hat{v}_1))\iota_{\hat{v}_0*\hat{v}_1}(z) \\
		\end{aligned}
	\end{equation} 
	
	It follows from the same reason as in the proof of $d_{C_{res}^{uni}}=d_{\dot{C}^{\partial}}$ that $\mathop{\oplus}\limits_{\substack{\hat{\tau} \geq \hat{v}_0 \\ \hat{p}\hat{\tau} \in \bar{e}(N \times S^1)}} \mathop{\oplus}\limits_{\substack{s' \in B_{p\hat{p}\hat{\tau}} \\ s' \geq p\hat{p}(\hat{v}_0*\hat{v}_1)}} \psi_u^r(s',p\hat{p}(\hat{v}_0*\hat{v}_1))=\mathop{\oplus}\limits_{\hat{s}' \geq \hat{v}_0*\hat{v}_1}\psi_u^r(p\hat{p}\hat{s}',p\hat{p}(\hat{v}_0*\hat{v}_1))$. Therefore:
	\begin{equation*}
		(\psi_{res}^{uni})_u^r(z)=(-1)^{n+r+1}\sum\limits_{\hat{v}_0 \in \widehat{V}_0}\sum\limits_{\hat{v}_1 \in \widehat{V}_1} \mathop{\oplus}\limits_{\hat{s}' \geq \hat{v}_0*\hat{v}_1}\psi_u^r(p\hat{p}\hat{s}',p\hat{p}(\hat{v}_0*\hat{v}_1))\iota_{\hat{v}_0*\hat{v}_1}(z) \\
	\end{equation*}
	
	By definition of $\psi_{u,r}^{ass}$ in  \ref{Defdetlapsiass+psiass}, we have 	$\psi_{u,r}^{ass}(z)=(-1)^u\dot{\psi}_{u}^r(\hat{p}^{-1}\bar{e}(U))\varOmega_{\widetilde{S}}^{n-u-r}(z)$. By Lemma \ref{DefT&varSigma}, we have $\varOmega_{\widetilde{S}}^{n-u-r}(z)=(-1)^{n-u-r-1}\sum\limits_{\substack{\hat{v}_0\in \widehat{V}_0 \\ \hat{v}_1 \in \widehat{V}_1}}\iota_{\hat{v}_0*\hat{v}_1}(z)$. Therefore:
	\begin{equation*}
		\psi_{u,r}^{ass}(z)=(-1)^{n-r-1}\sum\limits_{\substack{\hat{v}_0\in \widehat{V}_0 \\ \hat{v}_1 \in \widehat{V}_1}}\mathop{\oplus}\limits_{\substack{\hat{s}' \geq \hat{v}_0*\hat{v}_1 \\ \hat{p}\hat{s'} \in \bar{e}(U)}}\iota_{\hat{v}_0*\hat{v}_1}(z)
	\end{equation*}
	
	Since $U$ is upper closed and $\bar{e}$ is a simplicial embedding on $U$, combining with the fact that $\hat{p}(\hat{v}_0*\hat{v}_1) \in \bar{e}(U)$, we get:
	\begin{equation}
		\label{Eq6.26}
		\psi_{u,r}^{ass}(z)=(-1)^{n-r-1}\sum\limits_{\substack{\hat{v}_0\in \widehat{V}_0 \\ \hat{v}_1 \in \widehat{V}_1}}\mathop{\oplus}\limits_{\hat{s}' \geq \hat{v}_0*\hat{v}_1}\iota_{\hat{v}_0*\hat{v}_1}(z)
	\end{equation}

	Comparing the two results \ref{Eq6.25},\ref{Eq6.26}, we get $(\psi_{res}^{uni})_u^r=\psi_{u,r}^{ass}$, thus we prove that the quadratic structure agrees. Summarizing all the proofs above, we have proven that $(l)$ is true. We have checked all the statements in the Theorem and we finish the proof here.
\end{proof}

\begin{Corollary}
	\label{Cormain}
	If $\Gamma/\pi$ is an infinte set, then the following diagram is commutative:
	\begin{equation*}
		\begin{tikzcd}
			H_n(M;\underline{L}(\mathbb{Z})) \dar \rar & L_n(\mathbb{Z}\Gamma) \ar{dd}{\Theta\rho_*}\\
			H_n(\Sigma(N \times S^1)_+; \underline{L}(\mathbb{Z})) \dar{\cong} \\
			H_{n-1}(N \times S^1;\underline{L}(\mathbb{Z})) \rar & L_{n-1}^p(\mathbb{Z}\Pi) \\
		\end{tikzcd}
	\end{equation*}
\end{Corollary}

\begin{proof}
	It follows directly from Remark \ref{Rempartial} and the statement of Theorem \ref{mainproof}.
\end{proof}

\begin{Theorem}
	\label{surex}
	Let $M,N$ be as in Theorem \ref{L theory transfer}. Denote $P=N \times S^1$. Denote $\underline{L}(\mathbb{Z})\textlangle0\textrangle$ to be the $1$-connected cover of $\underline{L}(\mathbb{Z})$ and $N_{TOP}(M),N_{TOP}(P)$ to be the normal structure set of $M,P$. Then there is a commutative diagram:
	\begin{equation*}
		\begin{tikzcd}
			H_n(M;\underline{L}(\mathbb{Z})\textlangle0\textrangle) \dar{\cong}[swap]{t} \rar{(f_{M,N})_*} & H_{n-1}(P;\underline{L}(\mathbb{Z})\textlangle0\textrangle) \dar{\cong}[swap]{t} \\
			N_{TOP}(M) \rar & N_{TOP}(P) \\
		\end{tikzcd}
	\end{equation*}
	
	Here the morphism on the bottom is given as follows: let $x$ be an element in $N_{TOP}(M)$ and $(f,b):M' \longrightarrow M$ be a normal map representing it. After homotopy we can make $P$ transversal to $f$ and let $P'=f^{-1}(P)$, then $x$ is mapped to the equivalence class of $(f|_{P'},b|_{P'})$.
\end{Theorem}

\begin{proof}
	Recall from Proposition 18.3 in \cite{ranickiltheory} that $t$ is given by the following composition of isomorphisms:
	\begin{equation*}
		N_{TOP}(M) \stackrel{\cong}{\longrightarrow} [M,G/TOP] \stackrel{\cong}{\longrightarrow} H^0(M;\underline{L}(\mathbb{Z})\textlangle0\textrangle) \stackrel{\cap [M]_{L^{sym}}}{\longrightarrow} H_n(M;\underline{L}(\mathbb{Z})\textlangle0\textrangle)
	\end{equation*}

	Let $i_P:P \longrightarrow M$ be the inclusion map, to prove the theorem, it suffices to prove that each square in the following diagram commutes:
	\begin{equation}
		\label{comdia1}
		\begin{tikzcd}
			N_{TOP}(M) \rar \dar{\cong} & N_{TOP}(P) \dar{\cong} \\
			\mathop{[}M,G/TOP \mathop{]} \rar{i_P^*} \dar{\cong} & \mathop{[}N,G/TOP\mathop{]} \dar{\cong} \\
			H^0(M;\underline{L}(\mathbb{Z})\textlangle0\textrangle) \rar{i_P^*} \dar{\cap \ [M]_{L^{sym}}} & H^0(P;\underline{L}(\mathbb{Z})\textlangle0\textrangle) \dar{\cap \ [P]_{L^{sym}}} \\
			H_n(M;\underline{L}(\mathbb{Z})\textlangle0\textrangle) \rar{(f_{M,N})_*}& H_{n-1}(P;\underline{L}(\mathbb{Z})\textlangle0\textrangle) \\
		\end{tikzcd}
	\end{equation}
	
	Recall that the identification $N_{TOP}(M) \cong [M,G/TOP]$ is given by the "difference" of the bundle on the target manifold $M$ and the stable normal bundle of $M$. Since the pull back of the normal bundle of $M$ under $i_P$ is stably isomorphic to the normal bundle of $P$, we have that the upper square commutes.
	
	It follows from general theory of generalized cohomology theory that the middle square commutes.
	
	For every topological block bundle $E$ of dimension $k$, there is a Sullivan-Ranicki orientation for it. It is given by the push forward of the standard orientation $\widetilde{H}^k(Th(E);MSTOP)$ via a map $\sigma: MSTOP \longrightarrow \underline{L}^{sym}(\mathbb{Z})$ between spectra, see page 289 in \cite{Ranickitotal}, \cite{totalKMM} and \cite{Sullivanlocal} for details.
	
	Denote $\nu_P$ to be the normal bundle of $P=N \times S^1$ in $M$. Denote $D(\nu_P),S(\nu_P)$ to be the associated disk bundle and sphere bundle. Let $U_P \in \widetilde{H}^1\big(Th(\nu_P);\underline{L}^{sym}(\mathbb{Z})\big)$ be the Sullivan-Ranicki orientation of $\nu_P$. Let $j_P: M_+ \longrightarrow Th(\nu_P)$ be the quotient map. Then, by (d) in Proposition 16.16 in \cite{ranickiltheory}, we have $(j_P)_*[M]_{L^{sym}} \cap U_P=[N]_{L^{sym}}$. Therefore, the following diagram commutes:
	\begin{equation*}
		\begin{tikzcd}
			H^0(M;\underline{L}(\mathbb{Z})\textlangle0\textrangle) \ar{rr}{i_P^*} \dar{\cap \ [M]_{L^{sym}}} & & H^0(P;\underline{L}(\mathbb{Z})\textlangle0\textrangle) \dar{\cap \ [N]_{L^{sym}}} \\
			H_n(M;\underline{L}(\mathbb{Z})\textlangle0\textrangle) \rar{(j_P)_*} &  \widetilde{H}_n(Th(\nu_P);\underline{L}(\mathbb{Z})\textlangle0\textrangle) \rar{\cap U_P} & H_{n-1}(P;\underline{L}(\mathbb{Z})\textlangle0\textrangle) \\
		\end{tikzcd}
	\end{equation*}
	
	Moreover, since $\nu_P$ is trivial, we have $S(\nu_P)=P \times \{\pm 1\}$. Denote $i_{\pm}:P \longrightarrow S(\nu_P)$ to be the inclusions into $P \times \{\pm1\}$.  By definition of $f_{M,N}$, the following diagram commutes:
		\begin{equation*}
			\begin{tikzcd}
				H_n(M;\underline{L}(\mathbb{Z})\textlangle0\textrangle) \rar{(f_{M,N})_*}\dar{(j_P)_*}  & H_{n-1}(P;\underline{L}(\mathbb{Z})\textlangle0\textrangle) \dar{(i_-)_*} \\
				H_n(D(\nu_P),S(\nu_P);\underline{L}(\mathbb{Z})\textlangle0\textrangle) \rar{\partial} & H_{n-1}(S(\nu_P);\underline{L}(\mathbb{Z})\textlangle0\textrangle) \\
			\end{tikzcd}
		\end{equation*}
	Note that $H_{n-1}(S(\nu_P);\underline{L}(\mathbb{Z})\textlangle0\textrangle)$ is the direct summand of two copies of \\
	$H_{n-1}(P;\underline{L}(\mathbb{Z})\textlangle0\textrangle)$ and $(i_-)_*$ is the inclusion map into one direct summand. Combining with the two commutative diagrams above, we see that in order to prove that the lower square in the diagram \ref{comdia1} commutes, it suffices to prove that the following diagram commutes:
	\begin{equation*}
		\begin{tikzcd}
			H_n(D(\nu_P),S(\nu_P);\underline{L}(\mathbb{Z})\textlangle0\textrangle) \ar{r}{\cap U_P} \drar{\partial} &  H_{n-1}(P;\underline{L}(\mathbb{Z})\textlangle0\textrangle) \dar{(i_-)_*}\\
			& H_{n-1}(S(\nu_P);\underline{L}(\mathbb{Z})\textlangle0\textrangle) \\
		\end{tikzcd}
	\end{equation*}

	Since $\nu_P$ is trivial, we have that $(D(\nu_P),S(\nu_P)) \cong (P \times [-1,1],P \times \{\pm 1\})$ and $U_P$ is the pull back of the element $1 \in H^1([-1,1], \{\pm 1\}; \underline{L}^{sym}(\mathbb{Z}))=\pi_0\big(\underline{L}^{sym}(\mathbb{Z})\big)=\mathbb{Z}$ under the projection map $p_I:P \times [-1,1] \longrightarrow [-1,1]$. Then the commutativity of the diagram above follows from the following commutative diagram by the general properties of cap products and the fact that the class $1$ induces homotopically the identity map $\underline{L}(\mathbb{Z})\textlangle0\textrangle \longrightarrow \underline{L}(\mathbb{Z})\textlangle0\textrangle$.
	\begin{equation*}
		\begin{tikzcd}
			H_n(P \times [-1,1],P \times \{\pm1\};\underline{L}(\mathbb{Z})\textlangle0\textrangle) \rar{\cap U_P}\dar{\partial} & H_{n-1}(P \times [-1,1];\underline{L}(\mathbb{Z})\textlangle0\textrangle) \\
			H_{n-1}(P \times \{\pm 1\};\underline{L}(\mathbb{Z})\textlangle0\textrangle) \rar {\cap p_I^*(1,0)}& H_{n-1}(P \times \{\pm 1\};\underline{L}(\mathbb{Z})\textlangle0\textrangle) \uar \\
		\end{tikzcd}
	\end{equation*}
	
	Where $(1,0) \in H^0(\{\pm 1\};\underline{L}(\mathbb{Z})\textlangle0\textrangle)=\mathbb{Z} \oplus \mathbb{Z}$.
\end{proof}

Now we can prove Theorem \ref{L theory transfer}:
\begin{proof}[Proof of Theorem \ref{L theory transfer}]
	The proof is divided into two cases, depending on whether $\Gamma /\pi$ is finite.
	
	$(1)$ If $\Gamma/\pi$ is a finite set.
	
	On the one hand, by definition in \ref{transfer map}, $\rho$ is the trivial map. Therefore, we have $\rho_*(\sigma(f,b))=0$ and thus $\rho_{M,N}(\sigma(f,b))=0$.
	
	On the other hand, since $\Gamma/\pi$ is a finite set, we have that $\overline{M}$ is compact and thus $W_{\infty}$ is also compact. Since $\partial W_{\infty}=N \times S^1$, we have $\bar{e}_*\sigma^{\textlangle 0 \textrangle}(f|_{N' \times S^1},b|_{N' \times S^1}) $ \\ $=0$. By (5) in the geometric setting \ref{geometric setting}, we have $r_0\bar{e}=id$. Therefore, we get $\sigma^{\textlangle 0 \textrangle}(f|_{N' \times S^1},b|_{N' \times S^1})=0$. By Theorem \ref{split}, we have $\sigma^{\textlangle -1 \textrangle}(f|_{N'},b|_{N'})=0$. Therefore, we have $\rho_{M,N}(\sigma(f,b))=\sigma^{\textlangle -1 \textrangle}(f|_{N'},b|_{N'})$.
	
	$(2)$ If $\Gamma/\pi$ is an infinite set.
	
	By Corollary \ref{Cormain} and Theorem \ref{surex}, we have 
	\begin{equation*}
		\Theta\rho_*(\sigma(f,b))=\sigma^{\textlangle 0 \textrangle}(f|_{N' \times S^1},b|_{N' \times S^1})
	\end{equation*}

	By Theorem \ref{split} and the definition of $\rho_{M,N}$, we have:
	\begin{equation*}
		\rho_{M,N}(\sigma(f,b))=S\Theta\rho_*(\sigma(f,b))=S\big(\sigma^{\textlangle 0 \textrangle}(f|_{N' \times S^1},b|_{N' \times S^1})\big)=\sigma^{\textlangle -1 \textrangle}(f|_{N'},b|_{N'})
	\end{equation*}
	
	Summarizing the two cases above, we have proven the Theorem.
\end{proof}

	\section{Appendix}
	In this Appendix, we will give two descriptions stated in the previous sections, namely:

(1) A "Poincare pair" related to the construction in Corollary \ref{partialquad}.

(2) A description of the map $\delta:T^k(J,J') \longrightarrow T^{k+1}(J',J'')$ for some special ball complexes $J \subset J' \subset J''$.

We will make clear of the setup and goals in the subsections below. Before that, let us recall some notations in the previous sections first:

Let $l \in \mathbb{N}$, denote $\Sigma^l$ to be the simplicial complex with one $k$-simplex $\sigma^*$ for each $(l-k)$-simplex $\sigma$ in $\partial\Delta^{l+1}$, with $\sigma^* \leq \tau^*$ if and only if $\sigma \geq \tau$ in $\partial\Delta^{l+1}$. It inherits an order from the following simplicial isomorphism:
\begin{equation*}
	\Sigma^l \longrightarrow \partial\Delta^{l+1}; \ \sigma^* \mapsto \{0,1,2,...,l\} \backslash \sigma
\end{equation*}

For any subcomplex $V \subset \partial\Delta^{l+1}$, denote $\underline{V}=\mathop{\cup}\limits_{\sigma \notin V}\sigma^*$, it is a subcomplex of $\Sigma^l$. 

Now we state our setup:

(a) Let $[-1,1]$ be the simplicial complex with three zero dimensional simplices $-1,0,1$ and two one dimensional simplices $[0,1],[-1,0]$.

(b) Let $L \subset K$ be a pair of finite ordered geometric simplicial complexes with a decomposition $K=K_0 \cup _{L \times \{\pm 1\}} L \otimes [-1,1]$.

(c) Choose $l \in \mathbb{N}$ sufficiently large, such that  $K$ can be embeded simplicially and order-preservingly in $\partial \Delta^{l+1}$.

(d) Denote $L_0=K_0,L_1=K_0 \cup L,L_2=K_0 \cup L \otimes [0,1]$. We have subcomplexes $\underline{L}_2 \subset \underline{L}_1 \subset \underline{L}_0 \subset \Sigma^l$.

(e) For every simplex $\sigma \in L=L_1 \backslash L_0$, denote $A_{\sigma}=\{s \in L_2 \backslash L_1 | \ s>\sigma\}$. Then by definition $\tau>\sigma$ implies $A_{\tau} \subset A_{\sigma}$. Denote $B_{\sigma}=A_{\sigma} \backslash \mathop{\cup}\limits_{\substack{\tau > \sigma \\ \tau \in L}} A_{\tau}$. We have that $B_{\sigma}$ are disjoint and their union is $L_2 \backslash L_1=L \otimes [0,1] \backslash L \otimes \partial [0,1]$.


(f) For any $s \in A_{\sigma}$, since $s \in L_2 \backslash L_1=L \otimes [0,1] \backslash L \otimes \partial[0,1]$, there are unique simplices in $L \otimes 0,L \otimes 1$ that linearly spans $s$. Denote them by $(s \cap L)_0, (s \cap L)_1$ respectively. For any simplex $t_0 \leq (s \cap L)_0,t_1 \leq (s \cap L)_1$, denote $t_0 * t_1$ to be linear span of them. It is a face of $s$.

\begin{Remark}
	\label{RemAB}
	It is easy to see that for $s \in L_2 \backslash L_1$, we have $s \in A_{\sigma} \Leftrightarrow (s \cap L)_0 \geq \sigma$ and $s \in B_{\sigma} \Leftrightarrow (s \cap L)_0 =\sigma$.
\end{Remark}

\subsection{"Poincare pair" construction}
\

We make clear of what Remark \ref{partialquadpair} means here, given a ring $R$ with involution and a cover $p_K:\overline{K} \longrightarrow K$. Suppose that the embedding $L \otimes [-1,1] \subset K$ lifts to an embedding $L \otimes [-1,1] \subset \overline{K}$ and that there is a decomposition $\overline{K}=\overline{K}' \cup_{L \otimes 1} L \otimes [-1,1] \cup_{L \otimes 1} \overline{K}''$. Denote $\overline{K}_+=\overline{K}' \cup_{L \otimes 1} L \otimes [0,1]$ and suppose that there is a Galois covering $\mathbf{p}:\widetilde{K}_+ \longrightarrow \overline{K}_+$ with transformation group $G_0$. Let $n \in \mathbb{Z}$ and $(X,\psi_X)$ be a $n$-dimensional Poincare quadratic chain complex in $M^h(R)_*(K)$. Denote $(C,\psi_C)$ to be the infinite transfer of $(X,\psi_X)$ with respect to $p_K$. Let $S=\overline{K}_+ \backslash L$, we will prove the following theorem, which is the detailed description of Remark \ref{partialquadpair}:

\begin{Theorem}
	\
	\label{thmpartialquadpair}
	
	Let $C^{\partial},C^{all}$ be the following chain complexes in $M^{f}(RG_0)$:
	
	For every $r \in \mathbb{Z}$:
	\begin{equation*}
		C^{\partial}_r=\mathop{\oplus}\limits_{\mathbf{p}\tilde{\sigma}\in S \backslash S^-}C_r(\mathbf{p}\tilde{\sigma})
	\end{equation*}
	\begin{equation*}
		C^{all}_r=\mathop{\oplus}\limits_{\mathbf{p}\tilde{\sigma}\in S}C_r(\mathbf{p}\tilde{\sigma})
	\end{equation*}
	
	Denote $i_C$ to be the inclusion map on every dimension, it is an assembled map. Then:
	
	$(1)$ $i_C$ is a chain map.
	
	$(2)$ For all $u \in \mathbb{N},r \in \mathbb{Z}$, there are morphisms
	\begin{equation*}
		\delta\psi^{ass}_{u,r}:C^{all,cd}_{n-u-r}=\mathop{\oplus}\limits_{\mathbf{p}\tilde{\sigma}\in S}C_{n-u-r}(\mathbf{p}\tilde{\sigma})^* \longrightarrow C^{all}_r=\mathop{\oplus}\limits_{\mathbf{p}\tilde{\sigma}\in S}C_r(\mathbf{p}\tilde{\sigma})
	\end{equation*}
	\begin{equation*}
		\psi^{ass}_{u,r}:C^{\partial,cd}_{n-u-r-1}=\mathop{\oplus}\limits_{\mathbf{p}\tilde{\sigma}\in S \backslash S^-}C_{n-u-r-1}(\mathbf{p}\tilde{\sigma})^* \longrightarrow C^{\partial}_r=\mathop{\oplus}\limits_{\mathbf{p}\tilde{\sigma}\in S \backslash S^-}C_r(\mathbf{p}\tilde{\sigma})
	\end{equation*}
	
	such that the following equation holds:
	\begin{equation*}
		\begin{aligned}
			0= \ & d_{C^{all},r}\delta\psi^{ass}_{u,r+1} + (-1)^r\delta\psi^{ass}_{u,r}d_{C^{all},r+1}^{cd}+ (-1)^{n-1-u}\delta\psi_{u+1,r}^{ass} \\
			& +(-1)^{n+r'r}(\delta\psi_{u+1,r'}^{ass})^{cd}+(-1)^{n-1}i_{C,r}\psi_{u,r}^{ass}i_{C,r'}^{cd}
		\end{aligned}
	\end{equation*}
	
	Where $r'=n-1-u-r$.
	
	$(3)$ Let $p_S,\Upsilon^*[\widetilde{S}^-],i^S$ be the maps defined in Lemma \ref{Dualexchange}, then:
	\begin{equation*}
		p_S\psi_u^r(\widetilde{S})\Upsilon^{n-u-r}[\widetilde{S}^-]=p_S\delta\psi_{u,r}^{ass}i^S:C_{n-u-r}(\widetilde{S}^-)^{cd} \longrightarrow C_r(\widetilde{S}^-)
	\end{equation*}
\end{Theorem}

Before proving the theorem, let us do some preparations:

\begin{Lemma}
	\
	\label{DefT&varSigma}
	
	Let $\tilde{\sigma} \in \widetilde{S \backslash S^-}=\mathbf{p}^{-1}(S \backslash S^-)$ and set $\sigma=\mathbf{p}\tilde{\sigma}$. Since $\sigma \in S \backslash S^-=L \otimes [0,1] \backslash L \otimes \partial [0,1]$, let $V_0$ and $V_1$ be the non-empty set of vertices of $(\sigma \cap L)_0$ and $(\sigma \cap L)_1$, respectively. For $i=0,1$, let $\widetilde{V}_i$ be the set of vertices $\tilde{v} \in \tilde{\sigma}$ such that $\mathbf{p}\tilde{v} \in V_i$. For any $\tilde{\sigma}' \leq \tilde{\sigma}$, denote $\iota_{\tilde{\sigma}'}$ to be the following inclusion map:
	\begin{equation*}
		\iota_{\tilde{\sigma}'}:C_*(\mathbf{p}\tilde{\sigma})^* \longrightarrow C^{*+|\sigma'|}(\mathbf{p}\tilde{\sigma}')=\mathop{\oplus}\limits_{\kappa \geq \mathbf{p}\tilde{\sigma}'}C_*(\kappa)^*
	\end{equation*}
	
	For any $r \in \mathbb{Z}$, let $\Upsilon^r[\widetilde{S}]:C_r(\widetilde{S})^{cd} \longrightarrow C^r(\widetilde{S})$ be the map defined in Remark \ref{Dualexchangeupper}. Define $T_{\widetilde{S}}^r=d_{C^{-*}}(\widetilde{S})\Upsilon^{r-1}[\widetilde{S}]-(-1)^{r-1}\Upsilon^{r}[\widetilde{S}]d_C(\widetilde{S})^{cd}$. Let $\varOmega^r_{\widetilde{S}}:C_{r-1}(\widetilde{S \backslash S^-})^{cd} \longrightarrow C^r(\widetilde{S \backslash S^-})$ be the composition of the following maps:
	\begin{equation}
		\label{DefvarOmega}
		\begin{tikzcd}
			C_{r-1}(\widetilde{S \backslash S^-})^{cd} \rar[hook] & C_{r-1}(\widetilde{S})^{cd} \rar{T_{\widetilde{S}}^r} & C^r(\widetilde{S}) \ar{rr}{\text{projection}} & & C^r(\widetilde{S \backslash S^-})
		\end{tikzcd}
	\end{equation}

	Then for any $r \in \mathbb{Z}$ and $z \in C_{r-1}(\mathbf{p}\tilde{\sigma})^*$, we have $T_{\widetilde{S}}^r(z) \in C^r(\widetilde{S \backslash S^-})$ and:
	\begin{equation*}
		\varOmega_{\widetilde{S}}^r(z)=(-1)^{r-1}\sum\limits_{\substack{\tilde{v}_0\in \widetilde{V}_0 \\ \tilde{v}_1 \in \widetilde{V}_1}}\iota_{\tilde{v}_0*\tilde{v}_1}(z)
	\end{equation*}
\end{Lemma}

\begin{proof}
	By definition, we have $\Upsilon^{r-1}[\widetilde{S}](z)=\mathop{\oplus}\limits_{\tilde{v}_1 \in \widetilde{V}_1}\iota_{\tilde{v}_1}(z)$. Therefore, we have:
	\begin{equation*}
		d_{C^{-*}}(\widetilde{S})\Upsilon^{r-1}[\widetilde{S}](z)=\sum\limits_{\tilde{v}_1 \in \widetilde{V}_1}d_{C^{-*}}(\widetilde{S})\iota_{\tilde{v}_1}(z)=\sum\limits_{\tilde{v}_1 \in \widetilde{V}_1}\mathop{\oplus}\limits_{\tilde{\tau} \in \widetilde{S}}d_{C^{-*}}(\tilde{\tau},\tilde{v}_1)\iota_{\tilde{v}_1}(z)
	\end{equation*}
	
	By definition, we have $d_{C^{-*}}(\tilde{\tau},\tilde{v}_1) \neq 0 \Rightarrow \tilde{\tau}=\tilde{v}_1$ or $\tilde{\tau} \in \widetilde{K}_+^*(\tilde{v}_1)$
	
	If $\tilde{\tau} =\tilde{v}_1$, then by definition we have:
	\begin{equation*}
		d_{C^{-*}}(\tilde{\tau},\tilde{v}_1)\iota_{\tilde{v}_1}(z)=(-1)^{r-1}\sum\limits_{\substack{\tilde{v}_1 \leq \tilde{\kappa} \leq \tilde{\sigma} \\ \tilde{\kappa} \in \widetilde{S}}}\iota_{\tilde{v}_1}d_{C}(\tilde{\sigma},\tilde{\kappa})^*(z)
	\end{equation*}
	
	If $\tilde{\tau} \in \widetilde{K}_+^*(\tilde{v}_1)$ and $\tilde{\tau} \in \widetilde{S}$, then by definition we have:
	\begin{equation*}
		d_{C^{-*}}(\tilde{\tau},\tilde{v}_1)\iota_{\tilde{v}_1}(z)=\begin{cases} (-1)^{r-1}(-1)^{n_{\tilde{v}_1}^{\tilde{\tau}}}\iota_{\tilde{\tau}}(z) & \text{If } \tilde{\tau} \leq \tilde{\sigma} \\ \qquad \qquad 0 & \text{else} \end{cases}
	\end{equation*}
	
	If $\tilde{\tau} \in \widetilde{K}_+^*(\tilde{v}_1) \cap \widetilde{S}$ and $\tilde{\tau} \leq \tilde{\sigma}$, then either of the following statement holds:
	
	(a) There is a vertex $\tilde{v}_1' \in  \widetilde{V}_1$ that is different from $\tilde{v}_1$, such that $\tilde{\tau}=\tilde{v}_1*\tilde{v}_1'$.
	
	(b) There is a vertex $\tilde{v}_0 \in \widetilde{V}_0$, such that $\tilde{\tau}=\tilde{v}_0*\tilde{v}_1$.
	
	Therefore, we have:
	\begin{equation*}
		\begin{aligned}
			d_{C^{-*}}(\widetilde{S})\Upsilon^{r-1}[\widetilde{S}](z)=(-1)^{r-1}\big( & \mathop{\oplus}\limits_{\tilde{v}_1 \in \widetilde{V}_1}\sum\limits_{\substack{\tilde{v}_1 \leq \tilde{\kappa} \leq \tilde{\sigma} \\ \tilde{\kappa} \in \widetilde{S}}}\iota_{\tilde{v}_1}d_{C}(\tilde{\sigma},\tilde{\kappa})^*(z) \\
			&+\sum\limits_{\tilde{v}_1 \in \widetilde{V}_1}\sum\limits_{\substack{\tilde{v}_1' \in \widetilde{V}_1 \\ \tilde{v}_1' \neq \tilde{v}_1}}(-1)^{n_{\tilde{v}_1}^{\tilde{v}_1*\tilde{v}_1'}}\iota_{\tilde{v}_1*\tilde{v}_1'}(z) \\
			&+\sum\limits_{\tilde{v}_1 \in \widetilde{V}_1}\sum\limits_{\tilde{v}_0\in \widetilde{V}_0}(-1)^{n_{\tilde{v}_1}^{\tilde{v}_0*\tilde{v}_1}}\iota_{\tilde{v}_0*\tilde{v}_1}(z)\big)
		\end{aligned}
	\end{equation*}
	
	Now since for any $\tilde{v}_1 \neq \tilde{v}_1' \in \widetilde{V}_1$, we have $\tilde{v}_1*\tilde{v}_1'=\tilde{v}_1'*\tilde{v}_1$ and $n_{\tilde{v}_1}^{\tilde{v}_1*\tilde{v}_1'}=1-n_{\tilde{v}_1'}^{\tilde{v}_1'*\tilde{v}_1}$. Thus we have:
	\begin{equation*}
		\sum\limits_{\tilde{v}_1 \in \widetilde{V}_1}\sum\limits_{\substack{\tilde{v}_1' \in \widetilde{V}_1 \\ \tilde{v}_1' \neq \tilde{v}_1}}(-1)^{n_{\tilde{v}_1}^{\tilde{v}_1*\tilde{v}_1'}}\iota_{\tilde{v}_1*\tilde{v}_1'}(z)=0
	\end{equation*}
	
	For any $\tilde{v}_0 \in \widetilde{V}_0, \tilde{v}_1 \in \widetilde{V}_1$, we have $n_{\tilde{v}_1}^{\tilde{v}_0*\tilde{v}_1}=0$. Therefore, we get:
	\begin{equation}
		\label{Eq7.1}
		d_{C^{-*}}(\widetilde{S})\Upsilon^{r-1}[\widetilde{S}](z)=(-1)^{r-1}\big( \mathop{\oplus}\limits_{\tilde{v}_1 \in \widetilde{V}_1}\sum\limits_{\substack{\tilde{v}_1 \leq \tilde{\kappa} \leq \tilde{\sigma} \\ \tilde{\kappa} \in \widetilde{S}}}\iota_{\tilde{v}_1}d_{C}(\tilde{\sigma},\tilde{\kappa})^*(z)+\sum\limits_{\substack{\tilde{v}_0\in \widetilde{V}_0 \\ \tilde{v}_1 \in \widetilde{V}_1}}\iota_{\tilde{v}_0*\tilde{v}_1}(z)\big)
	\end{equation}
	
	On the other hand, we have $d_C(\widetilde{S})^{cd}(z)=\mathop{\oplus}\limits_{\tilde{\kappa} \in \widetilde{S}}d_C(\tilde{\sigma},\tilde{\kappa})^*(z)$. Since $d_C(\tilde{\sigma},\tilde{\kappa}) \neq 0 \Rightarrow \tilde{\kappa} \leq \tilde{\sigma}$, we have $d_C(\widetilde{S})^{cd}(z)=\mathop{\oplus}\limits_{\substack{\tilde{\kappa} \leq \tilde{\sigma} \\ \tilde{\kappa} \in \widetilde{S}}}d_C(\tilde{\sigma},\tilde{\kappa})^*(z)$. Then:
	\begin{equation*}
		\begin{aligned}
			\Upsilon^r[\widetilde{S}]d_C(\widetilde{S})^{cd}(z)&=\sum\limits_{\substack{\tilde{\kappa} \leq \tilde{\sigma} \\ \tilde{\kappa} \in \widetilde{S}}}\Upsilon^r[\widetilde{S}]d_C(\tilde{\sigma},\tilde{\kappa})^*(z) \\
			&=\sum\limits_{\substack{\tilde{\kappa} \leq \tilde{\sigma} \\ \tilde{\kappa} \in \widetilde{S}}}\mathop{\oplus}\limits_{\substack{\widetilde{S} \ni \tilde{v}_1 \leq \tilde{\kappa} \\ |\tilde{v}_1|=0}}\iota_{\tilde{v}_1}d_C(\tilde{\sigma},\tilde{\kappa})^*(z) \\
		\end{aligned}
	\end{equation*}
	
	Note that $|\tilde{v}_1|=0$ and $\widetilde{S} \ni \tilde{v}_1 \leq \tilde{\kappa} \leq \tilde{\sigma}$ implies $\tilde{v}_1 \in \widetilde{V}_1$. Therefore, we have:
	\begin{equation}
		\label{Eq7.2}
		\begin{aligned}
			\Upsilon^r[\widetilde{S}]d_C(\widetilde{S})^{cd}(z)
			&=\mathop{\oplus}\limits_{\tilde{v}_1 \in \widetilde{V}_1}\sum\limits_{\substack{\tilde{v}_1 \leq \tilde{\kappa} \leq \tilde{\sigma} \\ \tilde{\kappa} \in \widetilde{S}}}\iota_{\tilde{v}_1}d_C(\tilde{\sigma},\tilde{\kappa})^*(z) \\
		\end{aligned}
	\end{equation}
	
	Combining equation \ref{Eq7.1} and \ref{Eq7.2}, we have:
	\begin{equation*}
		\begin{aligned}
			T_{\widetilde{S}}^r(z)&=\big( d_{C^{-*}}(\widetilde{S})\Upsilon^{r-1}[\widetilde{S}]-(-1)^{r-1}\Upsilon^r[\widetilde{S}]d_C(\widetilde{S})^{cd} \big) (z) \\
			&=(-1)^{r-1}\sum\limits_{\substack{\tilde{v}_0\in \widetilde{V}_0 \\ \tilde{v}_1 \in \widetilde{V}_1}}\iota_{\tilde{v}_0*\tilde{v}_1}(z) \\
		\end{aligned}
	\end{equation*}
	
	Since $S \backslash S^-=L \otimes [0,1] \backslash L \otimes \partial [0,1]$, for all $\tilde{v}_0 \in \widetilde{V}_0, \tilde{v}_1 \in \widetilde{V}_1$, we have $\tilde{v}_0 * \tilde{v}_1 \in \widetilde{S \backslash S^-}$. Thus $T_{\widetilde{S}}^r(z) \in C_r(\widetilde{S \backslash S^-})$ and by definition \ref{DefvarOmega}, we have:
	\begin{equation*}
		\varOmega_{\widetilde{S}}^r(z)=(-1)^{r-1}\sum\limits_{\substack{\tilde{v}_0\in \widetilde{V}_0 \\ \tilde{v}_1 \in \widetilde{V}_1}}\iota_{\tilde{v}_0*\tilde{v}_1}(z)
	\end{equation*}
\end{proof}

\begin{Lemma}
	\label{vanT}
	Let $r \in \mathbb{Z}$ and $T_{\widetilde{S}}^r$ be the map defined in Lemma \ref{DefT&varSigma}, then $T_{\widetilde{S}}^r$ restricts to zero morphism on $C_{r-1}(\widetilde{S}^-)^{cd}$.
\end{Lemma}

\begin{proof}
	Choose any $\tilde{\sigma} \in \widetilde{S}^-$ and $z \in C_{r-1}(\mathbf{p}\tilde{\sigma})^*$, denote $\widetilde{V}$ to be the set of all vertices of $\tilde{\sigma}$. Similar to the compuation in the proof of Lemma \ref{DefT&varSigma}, we can get:
	\begin{equation*}
		\begin{aligned}
			d_{C^{-*}}(\widetilde{S})\Upsilon^{r-1}[\widetilde{S}](z)=(-1)^{r-1}\big( & \mathop{\oplus}\limits_{\tilde{v} \in \widetilde{V}}\sum\limits_{\substack{\tilde{v} \leq \tilde{\kappa} \leq \tilde{\sigma} \\ \tilde{\kappa} \in \widetilde{S}}}\iota_{\tilde{v}_1}d_{C}(\tilde{\sigma},\tilde{\kappa})^*(z) \\
			&+\sum\limits_{\tilde{v} \in \widetilde{V}}\sum\limits_{\substack{\tilde{v}' \in \widetilde{V} \\ \tilde{v}' \neq \tilde{v}}}(-1)^{n_{\tilde{v}}^{\tilde{v}*\tilde{v}'}}\iota_{\tilde{v}*\tilde{v}'}(z) \big) \\
		\end{aligned}
	\end{equation*}
	\begin{equation*}
		\Upsilon^{r}[\widetilde{S}]d_{C}(\widetilde{S})^{cd}(z)= \mathop{\oplus}\limits_{\tilde{v} \in \widetilde{V}}\sum\limits_{\substack{\tilde{v} \leq \tilde{\kappa} \leq \tilde{\sigma} \\ \tilde{\kappa} \in \widetilde{S}}}\iota_{\tilde{v}_1}d_{C}(\tilde{\sigma},\tilde{\kappa})^*(z)
	\end{equation*}

	Since for any $\tilde{v} \neq \tilde{v}' \in \widetilde{V}$, we have $\tilde{v}*\tilde{v}'=\tilde{v}'*\tilde{v}$ and $n_{\tilde{v}}^{\tilde{v}*\tilde{v}'}=1-n_{\tilde{v}'}^{\tilde{v}'*\tilde{v}}$, we can get:
	\begin{equation*}
		T_{\widetilde{S}}^r(z)=d_{C^{-*}}(\widetilde{S})\Upsilon^{r-1}[\widetilde{S}](z)-(-1)^{r-1}\Upsilon^{r}[\widetilde{S}]d_{C}(\widetilde{S})^{cd}(z)=0
	\end{equation*}
\end{proof}

Now we can prove Theorem \ref{thmpartialquadpair}:

\begin{proof}[Proof of Theorem \ref{thmpartialquadpair}]
	\
	
	(1) Since
	\begin{equation*}
		\begin{aligned}
			d_{C^{all}}i_C&=\mathop{\boxplus}\limits_{\mathbf{p}\tilde{\sigma} \in S}\mathop{\oplus}\limits_{\mathbf{p}\tilde{\tau} \in S}d_C(\mathbf{p}\tilde{\tau},\mathbf{p}\tilde{\sigma})i_C \\
			&=\mathop{\boxplus}\limits_{\mathbf{p}\tilde{\sigma} \in S \backslash S^-}\mathop{\oplus}\limits_{\mathbf{p}\tilde{\tau} \in S}d_C(\mathbf{p}\tilde{\tau},\mathbf{p}\tilde{\sigma}) \\
			&=\mathop{\boxplus}\limits_{\mathbf{p}\tilde{\sigma} \in S \backslash S^-}\mathop{\oplus}\limits_{\substack{\mathbf{p}\tilde{\tau} \in S \\ \tilde{\tau} \geq \tilde{\sigma}}}d_C(\mathbf{p}\tilde{\tau},\mathbf{p}\tilde{\sigma}) \\
			&=i_C\mathop{\boxplus}\limits_{\mathbf{p}\tilde{\sigma} \in S \backslash S^-}\mathop{\oplus}\limits_{\substack{\mathbf{p}\tilde{\tau} \in S \backslash S^- \\ \tilde{\tau} \geq \tilde{\sigma}}}d_C(\mathbf{p}\tilde{\tau},\mathbf{p}\tilde{\sigma}) \\
			&=i_Cd_{C^{\partial}}
		\end{aligned}
	\end{equation*}

	We have that $i_C$ is a chain map.
	
	(2) We will construct the map $\delta\psi_{u,r}^{ass},\psi_{u,r}^{ass}$ as follows:
	
	For $r \in \mathbb{Z}$, let $\Upsilon^r[\widetilde{S}]:C_r(\widetilde{S})^* \longrightarrow C^r(\widetilde{S})$ be the map defined in Remark \ref{Dualexchangeupper} and let $T_{\widetilde{S}}^r,\varOmega_{\widetilde{S}}^r$ be the maps defined in Lemma \ref{DefT&varSigma}. We define
	\begin{equation}
		\label{Defdetlapsiass+psiass}
		\begin{aligned}
			& \quad \delta\psi_{u,r}^{ass}=\psi_{C,u}^r(\widetilde{S})\Upsilon^{n-u-r}[\widetilde{S}] \\
			& \psi_{u,r}^{ass}=(-1)^u\psi_{C,u}^r(\widetilde{S \backslash S^-})\varOmega_{\widetilde{S}}^{n-u-r} \\
		\end{aligned}
	\end{equation}
	
	Since $(C,\psi_C)$ is a quadratic chain complex, we have:
	\begin{equation}
		\label{quad10}
		\begin{aligned}
			&d_C\psi_{C,u}^{r+1}-(-1)^{n+u}\psi_{C,u}^rd_{C^{-*}}+(-1)^{n-u-1}\psi_{C,u+1}^r \\
			&+(-1)^{n}(T\psi)_{C,u+1}^r=0: C^{n-u-1-r} \longrightarrow C_r \\
		\end{aligned}
	\end{equation}

	Since $S$ is upper closed, we can take the partial assmebly over $S$ with respect to the covering $\mathbf{p}$. By Lemma \ref{partialassemblyfun}, the equation still holds. Compose the equation further with $\Upsilon^{n-u-r-1}[\widetilde{S}]$ on the right, we can compute each term separately:
	\begin{equation*}
		d_C(\widetilde{S})\psi_{C,u}^{r+1}(\widetilde{S})\Upsilon^{n-u-r-1}[\widetilde{S}]=d_{C^{all}}\delta\psi_{u,r+1}^{ass}
	\end{equation*}
	\begin{equation*}
		\begin{aligned}
			\psi_{C,u}^r(\widetilde{S})d_{C^{-*}}(\widetilde{S})\Upsilon^{n-u-r-1}[\widetilde{S}]&=(-1)^{r'}\psi_{C,u}^r(\widetilde{S})\Upsilon^{n-u-r}[\widetilde{S}]d_{C}(\widetilde{S})^{cd}+\psi_{C,u}^r(\widetilde{S})T_{\widetilde{S}}^{n-u-r} \\
			&=(-1)^{r'}\delta\psi_{u,r}^{ass}d_{C^{all},r+1}^{cd}+\psi_{C,u}^r(\widetilde{S})T_{\widetilde{S}}^{n-u-r}
		\end{aligned}
	\end{equation*}
	\begin{equation*}
		\psi_{C,u+1}^r(\widetilde{S})\Upsilon^{n-u-r-1}[\widetilde{S}]=\delta\psi_{u+1,r}^{ass}
	\end{equation*}
	\begin{equation*}
		\begin{aligned}
			(T\psi)_{C,u+1}^r(\widetilde{S})\Upsilon^{n-u-r-1}[\widetilde{S}]&= (-1)^{rr'}\Upsilon^{r}[\widetilde{S}]^{cd} \psi_{C,u+1}^{r'}(\widetilde{S})^{cd} \quad (\text{By Remark } \ref{Dualexchangeupper}) \\
			&=(\delta\psi_{u+1,r'}^{ass})^{cd}
		\end{aligned}
	\end{equation*}

	Substituing the equations above in equation \ref{quad10} and compare it with the eqaution stated in the Theorem, it only leaves us to prove:
	\begin{equation}
		\label{Eq7.3}
		\psi_{C,u}^r(\widetilde{S})T_{\widetilde{S}}^{n-u-r}=i_{C,r}\psi_{C,u}^r(\widetilde{S \backslash S^-})\varOmega_{\widetilde{S}}^{n-u-r}i_{C,r'}^{cd}
	\end{equation}

	Denote $i^-$ to be the inclusion map $C^*(\widetilde{S \backslash S^-}) \longrightarrow C^*(\widetilde{S})$. By definition \ref{DefvarOmega} of $\varOmega_{\widetilde{S}}^*$, Lemma \ref{DefT&varSigma} and Lemma \ref{vanT}, we have $T_{\widetilde{S}}^{n-u-r}=i^-\varOmega_{\widetilde{S}}^{n-u-r}i_{C,r'}^{cd}$. Moreover, we have:
	\begin{equation*}
		\begin{aligned}
			\psi_{C,u}^r(\widetilde{S})i^-&=\mathop{\boxplus}\limits_{\mathbf{p}\tilde{\sigma} \in S \backslash S^-}\mathop{\oplus}\limits_{\mathbf{p}\tilde{\tau} \in S} \psi_{C,u}^r(\tilde{\tau},\tilde{\sigma}) \\
			&=\mathop{\boxplus}\limits_{\mathbf{p}\tilde{\sigma} \in S \backslash S^-}\mathop{\oplus}\limits_{\substack{\mathbf{p}\tilde{\tau} \in S \\ \tilde{\tau} \geq \tilde{\sigma}}} \psi_{C,u}^r(\tilde{\tau},\tilde{\sigma}) \\
			&=i_{C,r}\mathop{\boxplus}\limits_{\mathbf{p}\tilde{\sigma} \in S \backslash S^-}\mathop{\oplus}\limits_{\substack{\mathbf{p}\tilde{\tau} \in S \backslash S^- \\ \tilde{\tau} \geq \tilde{\sigma}}} \psi_{C,u}^r(\tilde{\tau},\tilde{\sigma}) \\
			&=i_{C,r}\psi_{C,u}^r(\widetilde{S \backslash S^-}) \\
		\end{aligned}
	\end{equation*}

	Therefore, we have:
	\begin{equation*}
			\psi_{C,u}^r(\widetilde{S})T_{\widetilde{S}}^{n-u-r}=\psi_{C,u}^r(\widetilde{S})i^-\varOmega_{\widetilde{S}}^{n-u-r}i_{C,r'}^{cd}=i_{C,r}\psi_{C,u}^r(\widetilde{S \backslash S^-})\varOmega_{\widetilde{S}}^{n-u-r}i_{C,r'}^{cd} 
	\end{equation*}
	
	Which is the same as equation \ref{Eq7.3}. Therefore, we have completed our proof.
\end{proof}

\subsection{Description of the map $\delta:T^k(\underline{L_1},\underline{L_2}) \longrightarrow T^{k+1}(\underline{L_0},\underline{L_1})$}
\

In this subsection, we will give a detailed description of the map $\delta:T^k(\underline{L_1},\underline{L_2}) \longrightarrow T^{k+1}(\underline{L_0},\underline{L_1})$. We begin with some preparations.

\begin{Lemma}
	For any $\sigma \in L$, $A_{\sigma}$ is upper closed in $L_2$.
\end{Lemma}

\begin{proof}
	
	Let $\sigma_1 \in A_{\sigma}$ and $\sigma_1 \leq \sigma_2 \in L_2$, we need to prove that $\sigma_2 \in A_{\sigma}$. By definition of $A_{\sigma}$, we have $\sigma_1 \in L_2 \backslash L_1$ and $\sigma_1 > \sigma$. Since $\sigma_1 \leq \sigma_2$, we have $\sigma_2 \notin L_1$ and $\sigma_2>\sigma$. Combining with $\sigma_2 \in L_2$, we get $\sigma_2 \in A_{\sigma}$.
\end{proof}

\begin{Lemma}
	\
	\label{Apppartialass}
	
	Let $k \in \mathbb{Z}$ and $(D,\theta)$ be a $k$-dimensional quadratic chain complex in $M^h(R)_*(L_2)$. For any $\sigma \in L$ and any $r \in \mathbb{Z},u \in \mathbb{N}$, define the maps $\theta_{\sigma,u}^r,T\theta_{\sigma,u}^r$ as follows:
	\begin{equation*}
		\theta_{\sigma,u}^r= \mathop{\boxplus}\limits_{s \in B_{\sigma}}\mathop{\oplus}\limits_{s' \in A_{\sigma}}\theta_u^r(s',s): \mathop{\oplus}\limits_{s \in B_{\sigma}} D^{k-u-r}(s) 
		\longrightarrow \mathop{\oplus}\limits_{s' \in A_{\sigma}} D_{r}(s')
	\end{equation*}
	\begin{equation*}
	T\theta_{\sigma,u}^r= \mathop{\boxplus}\limits_{s \in B_{\sigma}}\mathop{\oplus}\limits_{s' \in A_{\sigma}}(T\theta)_u^r(s',s):\mathop{\oplus}\limits_{s \in B_{\sigma}} D^{k-u-r}(s) 
	\longrightarrow \mathop{\oplus}\limits_{s' \in A_{\sigma}} D_{r}(s')
	\end{equation*}
	
	Then there is a chain homotopy equivalence $\mho_{\sigma}^r:\mathop{\oplus}\limits_{s \in A_{\sigma}} D_{r-|\sigma|-1}(s)^* \longrightarrow \mathop{\oplus}\limits_{s \in B_{\sigma}} D^r(s)$ $(r \in \mathbb{Z})$, such that the following equation holds:
	\begin{equation*}
		T\theta_{\sigma,u}^r \mho_{\sigma}^{k-u-r}-(-1)^{\mathbf{r}}(\mho_{\sigma}^{r+1+|\sigma|})^*(\theta_{\sigma,u}^{r'})^*=0: \mathop{\oplus}\limits_{s \in A_{\sigma}} D_{r'}(s)^* \longrightarrow \mathop{\oplus}\limits_{s \in A_{\sigma}}D_{r}(s)
	\end{equation*}
	\begin{equation*}
		(r'=k-u-r-|\sigma|-1,\mathbf{r}=rr'+(|\sigma|+1)(r+r'))
	\end{equation*}
	
\end{Lemma}

\begin{proof}
	For any $r \in \mathbb{Z}$, we have:
	\begin{equation*}
		\begin{aligned}
			\mathop{\oplus}\limits_{s \in B_{\sigma}} D^r(s)=\mathop{\oplus}\limits_{s \in B_{\sigma}} [D_{r-|s|}][s]^*&=\mathop{\oplus}\limits_{s \in B_{\sigma}}\mathop{\oplus}\limits_{\mathfrak{s} \geq s} D_{r-|s|}(\mathfrak{s})^* \\
			&=\mathop{\oplus}\limits_{\mathfrak{s} \in A_{\sigma}}\mathop{\oplus}\limits_{\substack{s \in B_{\sigma} \\ s \leq \mathfrak{s}}} D_{r-|s|}(\mathfrak{s})^* \\
		\end{aligned}
	\end{equation*}

	For any $\mathfrak{s} \in A_{\sigma}$, denote $\mathfrak{s}_0=(\mathfrak{s} \cap L)_0,\mathfrak{s}_1=(\mathfrak{s} \cap L)_1$. Denote $K_{\sigma,\mathfrak{s}}=\sigma * \mathfrak{s}_1$ and $\partial_1K_{\sigma,\mathfrak{s}}=\sigma \cup (\mathop{\cup}\limits_{\sigma_0<\sigma} \sigma_0*\mathfrak{s}_1)$, then we claim that:
	\begin{equation*}
		s \in B_{\sigma} \text{ and } s \leq \mathfrak{s} \Leftrightarrow s \in K_{\sigma,\mathfrak{s}} \backslash \partial_1K_{\sigma,\mathfrak{s}} 
	\end{equation*}

	To prove the claim, note that since $\mathfrak{s}=\mathfrak{s}_0 * \mathfrak{s}_1$, by Remark \ref{RemAB}, we have:
	
	\begin{equation*}
		s \in B_{\sigma} \text{ and } s \leq \mathfrak{s} \Leftrightarrow s=\sigma * \mathfrak{s}' \text{ for some } \mathfrak{s}' \leq \mathfrak{s}_1
	\end{equation*}
	
	Then it is straightforward to see that the right hand side is equivalent to $s \in K_{\sigma,\mathfrak{s}} \backslash \partial_1K_{\sigma,\mathfrak{s}}$ and thus the claim holds.
	
	Now by the claim above, we can further write that:
	\begin{equation*}
		\mathop{\oplus}\limits_{s \in B_{\sigma}} D^r(s)=\mathop{\oplus}\limits_{\mathfrak{s} \in A_{\sigma}} (\Delta(K_{\sigma,\mathfrak{s}},\partial_1K_{\sigma,\mathfrak{s}}) \otimes_{\mathbb{Z}} D(\mathfrak{s}))_{r}^*
	\end{equation*}

	Let $(|\sigma|+1)\mathbb{Z}$ be the chain complex with only one $\mathbb{Z}$ on dimension $(|\sigma|+1)$, then the map of taking the sum gives a homotopy equivalence:
	\begin{equation*}
		\Delta_*(K_{\sigma,\mathfrak{s}},\partial_1K_{\sigma,\mathfrak{s}}) \longrightarrow (|\sigma|+1)\mathbb{Z}
	\end{equation*} 
	
	Taking its dual and tensoring with $D(\mathfrak{s})^*$ gives $\mho_{\sigma}$, more precisely, it is given by the following construction:
	
	Choose $\mathfrak{s} \in A_{\sigma}$ and let $V_1$ be the set of the vertices of $\mathfrak{s}_1$. For every $v \in V_1$, let $\sigma_v=v * \sigma \leq \mathfrak{s}$ and $\iota_{v}$ to be the following inclusion map:
	\begin{equation*}
		\iota_{v}:D_{r-|\sigma|-1}(\mathfrak{s})^* \longrightarrow D^r(\sigma_v)= \mathop{\oplus}\limits_{\kappa \geq \sigma_v}D_{r-|\sigma|-1}(\kappa)^*
	\end{equation*}
	
	Since $(\sigma_v \cap L)_0=\sigma$, we have $\sigma_v \in B_{\sigma}$. Then $\mho_{\sigma}^r$ restricted on the $D_{r-|\sigma|-1}(\mathfrak{s})^*$ component is given by:
	\begin{equation*}
		\mho_{\sigma}^r|_{D_{r-|\sigma|-1}(\mathfrak{s})^*}: D_{r-|\sigma|-1}(\mathfrak{s})^* \longrightarrow \mathop{\oplus}\limits_{s \in B_{\sigma}}D^r(s)
	\end{equation*}
	\begin{equation}
		\label{expressmho}
		\mho_{\sigma}^r|_{D_{r-|\sigma|-1}(\mathfrak{s})^*}(z_{\mathfrak{s}})=\mathop{\oplus}\limits_{v \in V_1}\iota_{v}(z_{\mathfrak{s}}) \in \mathop{\oplus}\limits_{v \in V_1} D^r(\sigma_v) \subset \mathop{\oplus}\limits_{s \in B_{\sigma}}D^r(s)
	\end{equation}

	For the equation argument, we can compute it in components. Fix $\tau,\varsigma \in A_{\sigma}$ and choose $z_{\tau} \in D_{r'}(\tau)^*,w_{\varsigma} \in D_{r}(\varsigma)^*$, it suffices to prove that the following holds:
	\begin{equation*}
		\textlangle T\theta_{\sigma,u}^r\mho_{\sigma}^{k-u-r}(z_{\tau}), w_{\varsigma}\textrangle=(-1)^{\mathbf{r}}\textlangle z_{\tau}, \theta_{\sigma,u}^{r'} \mho_{\sigma}^{r+|\sigma|+1}(w_{\varsigma})\textrangle
	\end{equation*}
	 
	The proof is the same with the proof of Lemma \ref{Dualexchange}, it will follow from the following commutative diagram $(v \in V_1)$:
	\begin{small}
		\begin{equation*}
			\begin{tikzcd}
				(TD)_{-k+u+r}(\sigma_v) \ar{rrr}{(-1)^{(r-|\sigma_v|)(r'-|\sigma_v|)}(T\theta)_u^r(\varsigma,\sigma_v)}\dar{=} & & & D_r(\varsigma) \\ 
				T(D_{r'})_{-|\sigma_v|}(\sigma_v) \ar{rr} \dar{=} & & T(T(D_{r})_{-|\sigma_v|})_{-|\sigma_v|}(\sigma_v) \rar{\subset}\dar{=} & (T^2D_{r})_0(\sigma_v) \uar{\mathfrak{D}(D_{r})_0}\dar{=} \\
				\mathop{\oplus}\limits_{\kappa \geq \sigma_v}D_{r'}(\kappa)^* \ar{rr}{\mathop{\boxplus}\limits_{\kappa \geq \sigma_v}\theta_u^{r'}(\kappa,\sigma_v)^*} & & \mathop{\oplus}\limits_{\kappa \geq \sigma_v}D_r(\kappa) \rar{=} & \mathop{\oplus}\limits_{\kappa \geq \sigma_v}D_r(\kappa) \ar[bend right=80,swap]{uu}{(-1)^{|\sigma_v|}p_{\sigma_v,\varsigma}}
			\end{tikzcd}
		\end{equation*}
	\end{small}
\end{proof}

\begin{Lemma}
	\
	\label{Appquadpair}
	
	Let $(D,\theta)$ be a $k$-dimensional Poincare quadratic chain complex in $M^h(R)_*(L_2)$. Then for any $\sigma \in L$, the following construction gives a $(k-|\sigma|-1)$-dimensional Poincare quadratic pair in $M^h(R)$:
	\begin{equation*}
		\big(i[\sigma]: \mathop{\oplus}\limits_{s \in A_{\sigma} \backslash B_{\sigma}} D_*(s) \longrightarrow \mathop{\oplus}\limits_{s \in A_{\sigma}} D_*(s),(\psi[\sigma],\delta\psi[\sigma])\big)
	\end{equation*}
	
	Where for all $r \in \mathbb{Z}$:
	
	$(1)$ $i_r[\sigma]:\mathop{\oplus}\limits_{s \in A_{\sigma} \backslash B_{\sigma}} D_r(s) \longrightarrow \mathop{\oplus}\limits_{s \in A_{\sigma}} D_r(s)$ is the inclusion map.
	
	$(2)$ For any $u \in \mathbb{N}$, let $\theta_{\sigma,u}^r= \mathop{\boxplus}\limits_{s \in B_{\sigma}}\mathop{\oplus}\limits_{s' \in A_{\sigma}}\theta_u^r(s',s)$, then $\delta\psi_u^r[\sigma]$ is given by $(-1)^{(|\sigma|+1)r}$ times the following composition of maps:
	\begin{equation*}
		\begin{tikzcd}
			\mathop{\oplus}\limits_{s \in A_{\sigma}} D_{k-|\sigma|-1-u-r}(s)^* \ar{r}{\mho_{\sigma}^{k-u-r}} & \mathop{\oplus}\limits_{s \in B_{\sigma}} D^{k-u-r}(s)^* \ar{r}{\theta_{\sigma,u}^r} & \mathop{\oplus}\limits_{s \in A_{\sigma}} D_{r}(s)
		\end{tikzcd}
	\end{equation*}

	$(3)$ For any $\tau \in L^*(\sigma)$, since $\sigma < \tau$, we have $A_{\tau} \subset A_{\sigma} \backslash B_{\sigma}$. Denote $I_{\tau,r},I^{\tau,r}$ to be inclusion map and  $P_{\tau,r},P^{\tau,r}$ to be the projection map:
	\begin{equation*}
		I_{\tau,r}:\mathop{\oplus}\limits_{s \in A_{\tau}} D_{r}(s) \longrightarrow \mathop{\oplus}\limits_{s \in A_{\sigma} \backslash B_{\sigma}} D_{r}(s), \ I^{\tau,r}:\mathop{\oplus}\limits_{s \in A_{\tau}} D_{r}(s)^* \longrightarrow \mathop{\oplus}\limits_{s \in A_{\sigma} \backslash B_{\sigma}} D_{r}(s)^*
	\end{equation*}
	\begin{equation*}
		P_{\tau,r}:\mathop{\oplus}\limits_{s \in A_{\sigma} \backslash B_{\sigma}} D_{r}(s) \longrightarrow \mathop{\oplus}\limits_{s \in A_{\tau}} D_{r}(s), \ P^{\tau,r}:\mathop{\oplus}\limits_{s \in A_{\sigma} \backslash B_{\sigma}} D_{r}(s)^* \longrightarrow \mathop{\oplus}\limits_{s \in A_{\tau}} D_{r}(s)^*
	\end{equation*}
	
	Then $\psi_u^r[\sigma]$ is given by:
	\begin{equation*}
		\psi_u^r[\sigma]: \mathop{\oplus}\limits_{s \in A_{\sigma} \backslash B_{\sigma}} D_{k-u-r-|\sigma|-2}(s)^* \longrightarrow \mathop{\oplus}\limits_{s \in A_{\sigma} \backslash B_{\sigma}} D_r(s)
	\end{equation*}
	\begin{equation*}
		\psi_u^r[\sigma]=(-1)^{k+1}\sum\limits_{\tau \in L^*(\sigma)}(-1)^{n_{\sigma}^{\tau}} I_{\tau,r}\delta\psi_u^r[\tau]P^{\tau,k-u-r-|\sigma|-2}
	\end{equation*}
\end{Lemma}

\begin{proof}
	Note that $A_{\sigma}$ is upper closed and $A_{\sigma} \backslash B_{\sigma}=\mathop{\cup}\limits_{\sigma'>\sigma}A_{\sigma'}$ is a union of upper closed sets, and therefore is upper closed. By Lemma \ref{partialassemblyfun}, the partial assembly over them is a functor. In particular, if we take the partial assembly with respect to the trivial covering, we get that $\mathop{\oplus}\limits_{s \in A_{\sigma} \backslash B_{\sigma}} D_*(s)$ and $\mathop{\oplus}\limits_{s \in A_{\sigma}} D_*(s)$ are chain complexes. Then we will check step by step the conditions for a Poincare quadratic pair:
	
	(1) $i[\sigma]$ is a chain map.
	
	For any $\mathop{\oplus}\limits_{s \in A_{\sigma} \backslash B_{\sigma}}z_{s} \in \mathop{\oplus}\limits_{s \in A_{\sigma} \backslash B_{\sigma}}D_r(s)$, we have:
	
	\begin{equation*}	
		\begin{aligned}
		d_D(A_{\sigma})i_r[\sigma](\mathop{\oplus}\limits_{s \in A_{\sigma} \backslash B_{\sigma}}z_{s})&=d_D(A_{\sigma})(\mathop{\oplus}\limits_{s \in A_{\sigma} \backslash B_{\sigma}}z_{s}) \\
		&=\sum\limits_{s \in A_{\sigma} \backslash B_{\sigma}}\mathop{\oplus}\limits_{s' \in A_{\sigma}}d_D(s',s)(z_{s})	\\		
		&=\sum\limits_{s \in A_{\sigma} \backslash B_{\sigma}}\mathop{\oplus}\limits_{\substack{s \leq s' \\ s' \in A_{\sigma}}}d_D(s',s)(z_{s})\\
		&=\sum\limits_{s \in A_{\sigma} \backslash B_{\sigma}}\mathop{\oplus}\limits_{\substack{s \leq s' \\ s' \in A_{\sigma} \backslash B_{\sigma}}}d_D(s',s)(z_{s})\\
		&=\sum\limits_{s \in A_{\sigma} \backslash B_{\sigma}}\mathop{\oplus}\limits_{s' \in A_{\sigma} \backslash B_{\sigma}}d_D(s',s)(z_{s})\\
		&=i_r[\sigma]d_D(A_{\sigma} \backslash B_{\sigma})(\mathop{\oplus}\limits_{s \in A_{\sigma} \backslash B_{\sigma}}z_{s}) \\
		\end{aligned}
	\end{equation*}
	
	Therefore, we have $d_D(A_{\sigma})i_r[\sigma]=i_r[\sigma]d_D(A_{\sigma} \backslash B_{\sigma})$, proving that $i_r[\sigma]$ is a chain map.
	
	(2) $(\delta\psi_u^r[\sigma],\psi_u^r[\sigma])_{u \in \mathbb{N}}^{r \in \mathbb{Z}}$ gives a quadratic structure on the pair.
	
	To prove this, one needs to check that for all $u \in \mathbb{N},r \in \mathbb{Z}$, the following equation holds:
	\begin{equation}
		\label{Appcheck1}
		\chi_{\sigma}^{r',r}=0:\mathop{\oplus}\limits_{s \in A_{\sigma}}D_{r'}(s)^* \longrightarrow \mathop{\oplus}\limits_{s \in A_{\sigma}}D_{r}(s)
	\end{equation}

	Where $n'=k-|\sigma|-2,r'=n'-u-r$ and
	\begin{equation}
		\label{Appcheck2}
		\begin{aligned}
			\chi_{\sigma}^{r',r}= \ & d_D(A_{\sigma})\delta\psi_u^{r+1}[\sigma] + (-1)^r\delta\psi_u^{r}[\sigma]d_D(A_{\sigma})^*+ (-1)^{n'-u}\delta\psi_{u+1}^r[\sigma] \\
			& +(-1)^{n'+1+r'r}(\delta\psi_{u+1}^{r'}[\sigma])^*+(-1)^{n'}i_r[\sigma]\psi_u^r[\sigma]i_{r'}[\sigma]^*
		\end{aligned}
	\end{equation}
	
	Since $\theta$ gives a $k$-dimensional quadratic structure on $M^h(R)_*(L_2)$, we have that the following equation holds:
	\begin{equation*}
	0=d_D\theta_u^{r+1}-(-1)^{k-u}\theta_u^rd_{D^{-*}}+(-1)^{k-u-1}\theta_{u+1}^r+(-1)^{k}(T\theta)_{u+1}^r: D^{k-u-1-r} \longrightarrow D_r
	\end{equation*}
	
	Since $A_{\sigma}$ is upper closed in $L_2$, we can take the partial assmebly over $A_{\sigma}$ with respect to the trivial covering and the equation still holds. Let $i_{\sigma}$ to be the inclusion morphism: $i_{\sigma}:\mathop{\oplus}\limits_{s \in B_{\sigma}} D^*(s) \longrightarrow \mathop{\oplus}\limits_{s \in A_{\sigma}} D^*(s)$. We can further compose the equation on the right with $i_{\sigma}\mho_{\sigma}^{k-u-1-r}$ and we get:
	\begin{equation}
		\label{Eq7.4}
		\begin{aligned}
			0=\ &d_D(A_{\sigma})\theta_u^{r+1}(A_{\sigma})i_{\sigma}\mho_{\sigma}^{k-u-1-r}-(-1)^{k-u}\theta_u^r(A_{\sigma})d_{D^{-*}}(A_{\sigma})i_{\sigma}\mho_{\sigma}^{k-u-1-r} \\
			&+(-1)^{k-u-1}\theta_{u+1}^r(A_{\sigma})i_{\sigma}\mho_{\sigma}^{k-u-1-r}+(-1)^{k}(T\theta)_{u+1}^r(A_{\sigma})i_{\sigma}\mho_{\sigma}^{k-u-1-r}
		\end{aligned}
	\end{equation}
	
	 We will compute all the terms on the right hand side of the equation seperately:
	
	Note first that by definition of $\theta_{\sigma,u}^r$ and the partial assembly, we have $\theta_u^{r}(A_{\sigma})i_{\sigma}=\theta_{\sigma,u}^r$ for all $r \in \mathbb{Z}, u \in \mathbb{N}$, similarly for $T\theta$. Thus we have:
	
	\begin{equation*}
		\begin{aligned}
			d_D(A_{\sigma})\theta_u^{r+1}(A_{\sigma})i_{\sigma}\mho_{\sigma}^{k-u-1-r} &=	d_D(A_{\sigma})\theta_{\sigma,u}^{r+1}\mho_{\sigma}^{k-u-1-r} \\
			&=(-1)^{(|\sigma|+1)(r+1)}d_D(A_{\sigma})\delta\psi_u^{r+1}[\sigma] \\
		\end{aligned}
	\end{equation*}
	\begin{equation*}
		\begin{aligned}
			(-1)^{k-u-1}\theta_{u+1}^r(A_{\sigma})i_{\sigma}\mho_{\sigma}^{k-u-1-r}&=(-1)^{k-u-1}\theta_{\sigma,u+1}^r\mho_{\sigma}^{k-u-1-r} \\
			&=(-1)^{k-u-1}(-1)^{(|\sigma|+1)r}\delta\psi_{u+1}^r[\sigma]\\
		\end{aligned}
	\end{equation*}
	\begin{equation*}
		\begin{aligned}
			(-1)^{k}(T\theta)_{u+1}^r(A_{\sigma})i_{\sigma}\mho_{\sigma}^{k-u-1-r}&=	(-1)^{k}T\theta_{\sigma,u+1}^r\mho_{\sigma}^{k-u-1-r} \\
			& \quad (\text{By Lemma } \ref{Apppartialass}) \\
			&=(-1)^{k+\mathbf{r}}(\mho_{\sigma}^{r+|\sigma|+1})^{*}(\theta_{\sigma,u+1}^{r'})^* \\
			&=(-1)^{k+\mathbf{r}}(-1)^{(|\sigma|+1)r'}(\delta\psi_{u+1}^{r'}[\sigma])^*
		\end{aligned}
	\end{equation*}
	\begin{equation*}
		(\mathbf{r}=rr'+(|\sigma|+1)(r+r'))
	\end{equation*}

	Substuting the equations above into equation \ref{Eq7.4} and rescaling by $(-1)^{(|\sigma|+1)(r+1)}$, we get:
	\begin{equation*}
		\begin{aligned}
			0= \ &d_D(A_{\sigma})\delta\psi_u^{r+1}[\sigma]- (-1)^{k-u}(-1)^{(|\sigma|+1)(r+1)}\theta_u^{r}(A_{\sigma})d_{D^{-*}}(A_{\sigma})i_{\sigma}\mho_{\sigma}^{k-u-1-r} \\
			&+(-1)^{n'-u}\delta\psi_{u+1}^r[\sigma]+(-1)^{n'+1+r'r}(\delta\psi_{u+1}^{r'}[\sigma])^* \\
		\end{aligned}
	\end{equation*}

	Let $r_1=k-u+(|\sigma|+1)(r+1)+1$. Comparing the equation above with \ref{Appcheck1} and \ref{Appcheck2}, it leaves us to check:
	\begin{equation}
		\label{Eq7.5}
		(-1)^{r_1}\theta_u^{r}(A_{\sigma})d_{D^{-*}}(A_{\sigma})i_{\sigma}\mho_{\sigma}^{k-u-1-r}=(-1)^r\delta\psi_u^{r}[\sigma]d_D(A_{\sigma})^*+(-1)^{n'}i_r[\sigma]\psi_u^r[\sigma]i_{r'}[\sigma]^*
	\end{equation}

	Since
	\begin{equation*}
		\begin{aligned}
			\delta\psi_u^{r}[\sigma]d_D(A_{\sigma})^*&=(-1)^{(|\sigma|+1)r}\theta_{u,\sigma}^r\mho_{\sigma}^{k-u-r}d_D(A_{\sigma})^* \\
			& \quad (\text{Since } \mho_{\sigma} \text{ is a chain map}) \\
			&=(-1)^{(|\sigma|+1)r}(-1)^{r'}\theta_{u,\sigma}^rd_{D^{-*}}(B_{\sigma})\mho_{\sigma}^{k-u-r-1} \\
			&=(-1)^{(|\sigma|+1)r}(-1)^{r'}\theta_u^{r}(A_{\sigma})i_{\sigma}d_{D^{-*}}(B_{\sigma})\mho_{\sigma}^{k-u-r-1} \\
		\end{aligned}
	\end{equation*}
 	
 	If we denote $t_{\sigma}=d_{D^{-*}}(A_{\sigma})i_{\sigma}-i_{\sigma}d_{D^{-*}}(B_{\sigma})$, equation \ref{Eq7.5} is equivalent to the following one:
 	\begin{equation}
 		\label{Appcheck3}
 		(-1)^{r_1}\theta_u^{r}(A_{\sigma})t_{\sigma}\mho_{\sigma}^{k-u-r-1}=(-1)^{n'}i_r[\sigma]\psi_u^r[\sigma]i_{r'}[\sigma]^*
 	\end{equation}
 
 	Let $r_0=k-u-r$ and $i_{r_0}[\sigma]^{dual},I_{\tau,r_0}^{dual}$ be the following inclusion maps:
 	\begin{equation*}
 		i_{r_0}[\sigma]^{dual}: \mathop{\oplus}\limits_{s \in A_{\sigma} \backslash B_{\sigma}}D^{r_0}(s) \longrightarrow \mathop{\oplus}\limits_{s \in A_{\sigma}}D^{r_0}(s)
 	\end{equation*}
 	\begin{equation*}
 		I_{\tau,r_0}^{dual}: \mathop{\oplus}\limits_{s \in A_{\tau}}D^{r_0}(s) \longrightarrow \mathop{\oplus}\limits_{s \in A_{\sigma} \backslash B_{\sigma}}D^{r_0}(s)
 	\end{equation*}
 	
 	We claim that:
 	\begin{equation}
 		\label{Appcheck4}
 		t_{\sigma}\mho_{\sigma}^{r_0-1}=(-1)^{r_0-1}\sum\limits_{\tau \in L^*(\sigma)}(-1)^{n_{\sigma}^{\tau}}i_{r_0}[\sigma]^{dual}I_{\tau,r_0}^{dual}i_{\tau}\mho_{\tau}^{r_0}P^{\tau,r'}i_{r'}[\sigma]^*
 	\end{equation}
 
	To prove the claim, we compute the terms on both side in components. Choose any $s \in A_{\sigma}$ and $z_s \in D_{r'}(s)^*$, denote $s_0=(s \cap L)_0,s_1=(s \cap L)_1$ and let $V_1$ be the set of all the vertices of $s_1$. For every $v \in V_1$, let $\sigma_v=\sigma*v \leq s$. For every $\eta \leq s$, denote $\iota_{\eta}$ to be the inclusion map:
	\begin{equation}
		\label{Defiotaeta}
		\iota_{\eta}: D_{r'}(s)^* \longrightarrow D^{r_0-2-|\sigma|+|\eta|}(\eta)= \mathop{\oplus}\limits_{\kappa \geq \eta} D_{r'}(\kappa)^*
	\end{equation}
	
	Denote $\iota_v=\iota_{\sigma_v}$ for $v \in V_1$. By the definition \ref{expressmho} of $\mho_{\sigma}$, we have $\mho_{\sigma}^{r_0-1}(z_s)=\mathop{\oplus}\limits_{v \in V_1}\iota_v(z_s)$, thus:
	\begin{equation}
		\label{Eq7.6}
		\begin{aligned}
			t_{\sigma}\mho_{\sigma}^{r_0-1}(z_s)&=\big(d_{D^{-*}}(A_{\sigma})i_{\sigma}-i_{\sigma}d_{D^{-*}}(B_{\sigma})\big)\big(\mathop{\oplus}\limits_{v \in V_1}\iota_v(z_s)\big) \\
			&=\sum\limits_{v \in V_1}\bigg(d_{D^{-*}}(A_{\sigma})\iota_v(z_s)-d_{D^{-*}}(B_{\sigma})\iota_v(z_s)\bigg) \\
			&=\sum\limits_{v \in V_1}\bigg(\mathop{\oplus}\limits_{\substack{s' \in A_{\sigma}}}d_{D^{-*}}(s',\sigma_v)\iota_v(z_s)-\mathop{\oplus}\limits_{\substack{s' \in B_{\sigma}}}d_{D^{-*}}(s',\sigma_v)\iota_v(z_s)\bigg) \\
			&=\sum\limits_{v \in V_1}\mathop{\oplus}\limits_{\substack{s' \in A_{\sigma} \backslash B_{\sigma}}}d_{D^{-*}}(s',\sigma_v)\iota_v(z_s) \\
		\end{aligned}
	\end{equation}

	For every $v \in V_1,s' \in A_{\sigma} \backslash B_{\sigma}$, by the definition of $d_{D^{-*}}$, we have:
	\begin{equation*}
		d_{D^{-*}}(s',\sigma_v) \neq 0 \Rightarrow \sigma_v \leq s' \text{ and } |s'|=|\sigma_v|+1
	\end{equation*}

	If $d_{D^{-*}}(s',\sigma_v) \neq 0$, then it is given by $(-1)^{r_0-1+n_{\sigma_v}^{s'}}$ times the following projection map:
	\begin{equation*}
		D^{r_0-1}(\sigma_v)=\mathop{\oplus}\limits_{\kappa \geq \sigma_v}D_{r'}(\kappa)^* \longrightarrow D^{r_0}(s')=\mathop{\oplus}\limits_{\kappa \geq s'}D_{r'}(\kappa)^*
	\end{equation*}
	
	Let $R_{\sigma,s}=\{\eta \in A_{\sigma} \ | \ \eta \leq s, |(\eta \cap L)_0|=|\sigma|+1,|(\eta \cap L)_1|=0\}$ and $\eta_0=(\eta \cap L)_0,\eta_1=(\eta \cap L)_1$ for $\eta \in R_{\sigma,s}$. Then for every $\eta \in R_{\sigma,s}$, we have $\eta_0 \in K^*(\sigma)$ and thus $\eta \in A_{\sigma} \backslash B_{\sigma}$. Moreover $\eta_1 \in V_1$. By the definition of $d_{D^{-*}}(s',\sigma_v)$, it is straightforward to check the following statement:
	\begin{equation*}
		d_{D^{-*}}(s',\sigma_v)\iota_v(z_s)=\begin{cases} (-1)^{r_0-1+n_{\sigma}^{s_0'}}\iota_{s'}(z_s)& \text{If } s' \in R_{\sigma,s} \text{ and } v=s'_1 \\ (-1)^{r_0-1+n_{\sigma_v}^{s'}}\iota_{s'}(z_s) & \text{If } s'=\sigma*(v*v') \text{ for some vertex } \\
		& v' \in V_1 \text{ with } v' \neq v \\ \qquad \quad \ \  0 & \text{else} \\ \end{cases}
	\end{equation*}
	
	Substituting the equations above into equation \ref{Eq7.6}, we get:
	\begin{equation}
		\label{Eq7.7}
		\begin{aligned}
			t_{\sigma}\mho_{\sigma}^{r_0-1}(z_s)	&=\sum\limits_{v \in V_1}\mathop{\oplus}\limits_{\substack{s' \in A_{\sigma} \backslash B_{\sigma}}}d_{D^{-*}}(s',\sigma_v)\iota_v(z_s) \\
			&=\mathop{\oplus}\limits_{s' \in R_{\sigma,s}}(-1)^{r_0-1+n_{\sigma}^{s_0'}}\iota_{s'}(z_s)+\sum\limits_{v \in V_1}\sum\limits_{v' \in V_1}(-1)^{r_0-1+n_{\sigma_v}^{\sigma *(v*v')}}\iota_{\sigma *(v*v')}(z_s) \\
		\end{aligned}
	\end{equation}
	
	Since $|n_{\sigma_v}^{\sigma*(v*v')}-n_{\sigma_{v'}}^{\sigma*(v*v')}|=1$, the second term of equation \ref{Eq7.7} vanishes and thus we have $t_{\sigma}\mho_{\sigma}^{r_0-1}(z_s)=\mathop{\oplus}\limits_{s' \in R_{\sigma,s}}(-1)^{r_0-1+n_{\sigma}^{s_0'}}\iota_{s'}(z_s)$. Let $L^*(\sigma,s)=\{\sigma' \in L^*(\sigma) \ | \ \sigma' \leq s\}$, then there is a bijection:
	\begin{equation*}
		R_{\sigma,s} \longrightarrow L^*(\sigma,s) \times V_1; \eta \mapsto (\eta_0,\eta_1)
	\end{equation*}
	
	The inverse is given by $(\eta_0,\eta_1) \mapsto \eta_0*\eta_1$. Therefore, we have:
	\begin{equation}
		\label{Appcheck5}
		t_{\sigma}\mho_{\sigma}^{r_0-1}(z_s)=\mathop{\oplus}\limits_{\substack{s_0' \in L^*(\sigma,s) \\ s_1' \in V_1}}(-1)^{r_0-1+n_{\sigma}^{s_0'}}\iota_{s_0'*s_1'}(z_s)
	\end{equation}

	We will now compute the right hand side term of equation \ref{Appcheck4}. Notice that $i[\sigma]^*$ is the projection map and we have that for $\tau \in L^*(\sigma)$, $s \in A_{\tau} \Leftrightarrow \tau \in L^*(\sigma,s)$. Therefore, we have:
	\begin{equation}
		\label{Appcheck6}
		\begin{aligned}
			& \quad \; \sum\limits_{\tau \in L^*(\sigma)}(-1)^{n_{\sigma}^{\tau}}i_{r_0}[\sigma]^{dual}I_{\tau,r_0}^{dual}i_{\tau}\mho_{\tau}^{r_0}P^{\tau,r'}i_{r'}[\sigma]^*(z_s) \\ 
			&=\sum\limits_{\tau \in L^*(\sigma,s)}(-1)^{n_{\sigma}^{\tau}}i_{r_0}[\sigma]^{dual}I_{\tau,r_0}^{dual}i_{\tau}\mho_{\tau}^{r_0}(z_s) \\
			& \quad (\text{By definition } \ref{expressmho} \text{ of } \mho_{\tau}) \\
			&=\sum\limits_{\tau \in L^*(\sigma,s)}(-1)^{n_{\sigma}^{\tau}}i_{r_0}[\sigma]^{dual}I_{\tau,r_0}^{dual}i_{\tau}\big(\mathop{\oplus}\limits_{v \in V_1}\iota_{v*\tau}(z_s)\big) \\
			&=\mathop{\oplus}\limits_{\substack{v \in V_1 \\ \tau \in L^*(\sigma,s)}}(-1)^{n_{\sigma}^{\tau}}\iota_{\tau*v}(z_s) \\
		\end{aligned}
	\end{equation}

	Comparing \ref{Appcheck5} and \ref{Appcheck6}, we have proven our claim \ref{Appcheck4}.
	
	Now we begin to check that equation \ref{Appcheck3} holds:
	\begin{equation*}
		\begin{aligned}
			&\quad \; (-1)^{r_1}\theta_u^{r}(A_{\sigma})t_{\sigma}\mho_{\sigma}^{-k-u-r-1}\\
			& \quad (\text{By claim } \ref{Appcheck4}) \\
			&=(-1)^{r_1}(-1)^{r_0-1}\theta_u^{r}(A_{\sigma})\sum\limits_{\tau \in L^*(\sigma)}(-1)^{n_{\sigma}^{\tau}}i_{r_0}[\sigma]^{dual}I_{\tau,r_0}^{dual}i_{\tau}\mho_{\tau}^{r_0}P^{\tau,r'}i_{r'}[\sigma]^* \\
			&=(-1)^{|\sigma|r+|\sigma|+1}\sum\limits_{\tau \in L^*(\sigma)}(-1)^{n_{\sigma}^{\tau}}\theta_u^{r}(A_{\sigma})i_{r_0}[\sigma]^{dual}I_{\tau,r_0}^{dual}i_{\tau}\mho_{\tau}^{r_0}P^{\tau,r'}i_{r'}[\sigma]^* \\
		\end{aligned}
	\end{equation*}
	
	For any $s \in A_{\sigma} \backslash B_{\sigma}$ and $z_s \in D^{r_0}(s)$, we have:
	\begin{equation*}
		\begin{aligned}
			\theta_u^{r}(A_{\sigma})i_{r_0}[\sigma]^{dual}(z_s)&=\mathop{\oplus}\limits_{\substack{s' \in A_{\sigma} \\ s' \geq s}}\theta_u^r(s',s)(z_s) \\
			& (\text{Since } A_{\sigma} \backslash B_{\sigma} \text{ is upper closed}) \\
			&=\mathop{\oplus}\limits_{\substack{s' \in A_{\sigma} \backslash B_{\sigma}\\ s' \geq s}}\theta_u^r(s',s)(z_s) \\
			&=i_r[\sigma]\theta_u^r(A_{\sigma} \backslash B_{\sigma})(z_s) \\
		\end{aligned}
	\end{equation*}

	Thus $\theta_u^{r}(A_{\sigma})i_{r_0}[\sigma]^{dual}=i_r[\sigma]\theta_u^r(A_{\sigma} \backslash B_{\sigma})$ and similarly $\theta_u^r(A_{\sigma} \backslash B_{\sigma})I_{\tau,r_0}^{dual}=I_{\tau,r}\theta_u^r(A_{\tau})$. We can continue our computation:
		\begin{equation*}
		\begin{aligned}
			&\quad \; (-1)^{r_1}\theta_u^{r}(A_{\sigma})t_{\sigma}\mho_{\sigma}^{-k-u-r-1}\\
			&=(-1)^{|\sigma|r+|\sigma|+1}\sum\limits_{\tau \in L^*(\sigma)}(-1)^{n_{\sigma}^{\tau}}\theta_u^{r}(A_{\sigma})i_{r_0}[\sigma]^{dual}I_{\tau,r_0}^{dual}i_{\tau}\mho_{\tau}^{r_0}P^{\tau,r'}i_{r'}[\sigma]^* \\
			&=(-1)^{|\sigma|r+|\sigma|+1}\sum\limits_{\tau \in L^*(\sigma)}(-1)^{n_{\sigma}^{\tau}}i_{r}[\sigma]I_{\tau,r}\theta_u^{r}(A_{\tau})i_{\tau}\mho_{\tau}^{r_0}P^{\tau,r'}i_{r'}[\sigma]^* \\
			&=(-1)^{|\sigma|+1}\sum\limits_{\tau \in L^*(\sigma)}(-1)^{n_{\sigma}^{\tau}}i_{r}[\sigma]I_{\tau,r}\delta\psi_u^{r}[\tau]P^{\tau,r'}i_{r'}[\sigma]^* \\
			&=(-1)^{|\sigma|-k}i_{r}[\sigma]\psi_u^{r}[\sigma]i_{r'}[\sigma]^* \\
			& \quad (\text{Since } n'=|\sigma|-k-2) \\	&=(-1)^{n'}i_{r}[\sigma]\psi_u^{r}[\sigma]i_{r'}[\sigma]^* \\
		\end{aligned}
	\end{equation*}
	
	We have proven that equation \ref{Appcheck3} holds and thus $(\delta\psi_u^r[\sigma],\psi_u^r[\sigma])_{u \in \mathbb{N}}^{r \in \mathbb{Z}}$ gives a quadratic structure on the pair.
	
	(3) The quadratic pair is Poincare.
	
	For any $r \in \mathbb{Z}$, let $r''=k-|\sigma|-1-r$ and $\phi^r[\sigma]=\delta\psi_0^r[\sigma]+(-1)^{rr''}\delta\psi_0^{r''}[\sigma]^*$. By the definition of Poincare, it suffices to show that the following chain map is a homotopy equivalence:
	\begin{equation}
		\label{Eq7.8}
		\begin{tikzcd}
			\mathop{\oplus}\limits_{s \in A_{\sigma}} D_{r''}(s)^* \ar{r}{\phi^r[\sigma]} & \mathop{\oplus}\limits_{s \in A_{\sigma}}D_{r}(s) \ar{rr}{\text{projection}} & & \mathop{\oplus}\limits_{s \in B_{\sigma}} D_{r}(s)
		\end{tikzcd}
	\end{equation}
	
	We have:
	\begin{equation*}
		\begin{aligned}
			\phi^r[\sigma]&=\delta\psi_0^r[\sigma]+(-1)^{rr''}\delta\psi_0^{r''}[\sigma]^* \\
			&=(-1)^{(|\sigma|+1)r}\theta_{\sigma,0}^r\mho_{\sigma}^{k-r}+(-1)^{(|\sigma|+1)r''+rr''}(\theta_{\sigma,0}^{r''}\mho_{\sigma}^{k-r''})^* \\
			& \quad (\text{By Lemma } \ref{Apppartialass}) \\
			&=(-1)^{(|\sigma|+1)r}\theta_{\sigma,0}^r\mho_{\sigma}^{k-r}+(-1)^{(|\sigma|+1)r}T\theta_{\sigma,0}^{r}\mho_{\sigma}^{k-r} \\
			&=(-1)^{(|\sigma|+1)r}\theta_{0}^r(A_{\sigma})i_{\sigma}\mho_{\sigma}^{k-r}+(-1)^{(|\sigma|+1)r}T\theta_{0}^{r}(A_{\sigma})i_{\sigma}\mho_{\sigma}^{k-r} \\
			&=(-1)^{(|\sigma|+1)r}(1+T)\theta_{0}^r(A_{\sigma})i_{\sigma}\mho_{\sigma}^{k-r} \\
		\end{aligned}
	\end{equation*}

	Therefore, the composition of the two maps in \ref{Eq7.8}  is $(-1)^{(|\sigma|+1)r}(1+T)\theta_{0}^r(B_{\sigma})\mho_{\sigma}^{k-r}$. Now by Lemma \ref{Apppartialass}, $\mho_{\sigma}$ is a chain homotopy equivalence. Since $(D,\theta)$ is Poincare, $(1+T)\theta_0$ is a homotopy equivalence of chain complexes in $M^h(R)_*(L_2)$. Since $B_{\sigma}=A_{\sigma} \cap (A_{\sigma} \backslash B_{\sigma})^c$ is the intersection of an upper closed set with a subcomplex, Lemma \ref{partialassemblyfun} implies that partial assembly over it is a functor. Thus $(1+T)\theta_0(B_{\sigma})$ is a chain homotopy equivalence. In summary, we have shown that the pair is Poincare.
\end{proof}

\begin{Remark}
	\label{AppRemquadpair}
	\
	
	Assume that $(D,\theta)$ is a $k$-dimensional Poincare quadratic chain complex in $M^h(R)_*(K')$ for some finite ordered simplicial complex $K'$. Replacing the set $A_{\sigma}$ to $G_{\sigma}=\{\tau \in K' \ | \ \tau \geq \sigma\}$ and $B_{\sigma}$ to $\{\sigma\}=G_{\sigma} \backslash \mathop{\cup}\limits_{\sigma'>\sigma}G_{\sigma'}$ with $\sigma \in K'$, we can similarly get the following is a $(k-|\sigma|)$-dimensional Poincare pair:
	
	\begin{equation*}
		\big(i[\sigma]: \mathop{\oplus}\limits_{\tau>\sigma} D_*(\tau) \longrightarrow \mathop{\oplus}\limits_{\tau \geq \sigma} D_*(\tau),(\{\theta\}[\sigma],\delta\{\theta\}[\sigma])\big)
	\end{equation*}
	
	Where for all $r \in \mathbb{Z}$:
	
	$(1)$ $i_r[\sigma]:\mathop{\oplus}\limits_{\tau>\sigma} D_*(\tau) \longrightarrow \mathop{\oplus}\limits_{\tau \geq \sigma} D_*(\tau)$ is the inclusion map.
	
	$(2)$ For any $u \in \mathbb{N}$, define $\delta\{\theta\}_u^r[\sigma]=(-1)^{|\sigma|r}\mathop{\oplus}\limits_{\tau \geq \sigma}\theta_u^r(\tau,\sigma)$.
	
	$(3)$ For any	$\tau \in K'^*(\sigma)$, since $\sigma < \tau$, we have $G_{\tau} \subset G_{\sigma} \backslash \{\sigma\}$. Denote $I_{\tau,r},I^{\tau,r}$ to be inclusion map and  $P_{\tau,r},P^{\tau,r}$ to be the projection map as follows:
	\begin{equation*}
		I_{\tau,r}:\mathop{\oplus}\limits_{\kappa \geq \tau} D_{r}(\kappa) \longrightarrow \mathop{\oplus}\limits_{\kappa>\sigma} D_{r}(\kappa), \ I^{\tau,r}:\mathop{\oplus}\limits_{\kappa \geq \tau} D_{r}(\kappa)^* \longrightarrow \mathop{\oplus}\limits_{\kappa>\sigma} D_{r}(\kappa)^*
	\end{equation*}
	\begin{equation*}
		P_{\tau,r}:\mathop{\oplus}\limits_{\kappa>\sigma} D_{r}(\kappa) \longrightarrow \mathop{\oplus}\limits_{\kappa \geq \tau} D_{r}(\kappa), \ P^{\tau,r}:\mathop{\oplus}\limits_{\kappa >\sigma} D_{r}(\kappa)^* \longrightarrow \mathop{\oplus}\limits_{\kappa \geq \tau} D_{r}(\kappa)^*
	\end{equation*}
	
	Then $\{\theta\}_u^r[\sigma]$ is given by:
	\begin{equation*}
		\{\theta\}_u^r[\sigma]: \mathop{\oplus}\limits_{\kappa > \sigma} D_{k-|\sigma|-u-r-1}(\kappa)^* \longrightarrow \mathop{\oplus}\limits_{\kappa>\sigma} D_r(\kappa)
	\end{equation*}
	\begin{equation*}
		\{\theta\}_u^r[\sigma]=(-1)^{k+1}\sum\limits_{\tau \in K'^*(\sigma)}(-1)^{n_{\sigma}^{\tau}} I_{\tau,r}\delta\{\theta\}_u^r[\tau]P^{\tau,k-|\sigma|-u-r-1}
	\end{equation*}
	
	This is the explict form of the Poincare pair given in Proposition 8.4 in \cite{ranickiltheory}.
	
	There is also a local dual version of it. Let $(D,\theta)$ be a $k$-dimensional Poincare quadratic chain complex in $M^h(R)^*(K')$ for some finite ordered simplicial complex $K'$. The following construction gives a $(k+|\sigma|)$-dimensional Poincare pair:
	
	\begin{equation*}
		\big(i[\sigma]: \mathop{\oplus}\limits_{\tau<\sigma} D_*(\tau) \longrightarrow \mathop{\oplus}\limits_{\tau \leq \sigma} D_*(\tau),(\{\theta\}[\sigma],\delta\{\theta\}[\sigma])\big)
	\end{equation*}
	
	Where for all $r \in \mathbb{Z}$:
	
	$(1)$ $i_r[\sigma]:\mathop{\oplus}\limits_{\tau<\sigma} D_*(\tau) \longrightarrow \mathop{\oplus}\limits_{\tau \leq \sigma} D_*(\tau)$ is the inclusion map.
	
	$(2)$ For any $u \in \mathbb{N}$, define $\delta\{\theta\}_u^r[\sigma]=(-1)^{|\sigma|r}\mathop{\oplus}\limits_{\tau \leq \sigma}\theta_u^r(\tau,\sigma)$.
	
	$(3)$ For any	$\tau \in K'_*(\sigma)$, denote $I_{\tau,r},I^{\tau,r}$ to be inclusion map and $P_{\tau,r},P^{\tau,r}$ to be the projection map as follows:
	\begin{equation*}
		I_{\tau,r}:\mathop{\oplus}\limits_{\kappa \leq \tau} D_{r}(\kappa) \longrightarrow \mathop{\oplus}\limits_{\kappa<\sigma} D_{r}(\kappa), \ I^{\tau,r}:\mathop{\oplus}\limits_{\kappa \leq \tau} D_{r}(\kappa)^* \longrightarrow \mathop{\oplus}\limits_{\kappa<\sigma} D_{r}(\kappa)^*
	\end{equation*}
	\begin{equation*}
		P_{\tau,r}:\mathop{\oplus}\limits_{\kappa<\sigma} D_{r}(\kappa) \longrightarrow \mathop{\oplus}\limits_{\kappa \leq \tau} D_{r}(\kappa), \ P^{\tau,r}:\mathop{\oplus}\limits_{\kappa <\sigma} D_{r}(\kappa)^* \longrightarrow \mathop{\oplus}\limits_{\kappa \leq \tau} D_{r}(\kappa)^*
	\end{equation*}
	
	Then $\{\theta\}_u^r[\sigma]$ is given by:
	\begin{equation*}
		\{\theta\}_u^r[\sigma]: \mathop{\oplus}\limits_{\kappa < \sigma} D_{k+|\sigma|-u-r-1}(\kappa)^* \longrightarrow \mathop{\oplus}\limits_{\kappa<\sigma} D_r(\kappa)
	\end{equation*}
	\begin{equation*}
		\{\theta\}_u^r[\sigma]=(-1)^{k+1}\sum\limits_{\tau \in K'_*(\sigma)}(-1)^{n_{\sigma}^{\tau}} I_{\tau,r}\delta\{\theta\}_u^r[\tau]P^{\tau,k+|\sigma|-u-r-1}
	\end{equation*}
\end{Remark}

The following theorems give us a description of the quadratic structure on a cylinder:

\begin{Theorem}[Definition 15.73 in  \cite{luecksurgery}]
	\
	\label{tensorstr}
	
	Let $R$ be a ring with involution. Let $C$ be a finite chain complex in $M^h(R)$ and $D$ be a finite chain complex in $M^h(\mathbb{Z})$. Then there is a natural chain map:
	\begin{equation*}
		- \otimes - : W_{\%}(C) \otimes_{\mathbb{Z}} W^{\%}(D) \longrightarrow W_{\%}(C \otimes_{\mathbb{Z}} D)
	\end{equation*}
	
	Furthermore, let $0\mathbb{Z}$ be the chain complex with only $\mathbb{Z}$ on $0$ dimension. Let $\nu \in \big(W^{\%}(0\mathbb{Z})\big)_0=\mathbb{Z}$ be the element given by $1 \in \mathbb{Z}$. Then $- \otimes \nu$ is the identity map.
	
	The chain map is given as follows:
	
	Let $W^{-*}$ be the following chain complex in $M^h(\mathbb{Z}[\mathbb{Z}_2])$: 
	
	For $s \in \mathbb{N}$,
	\begin{equation*}
		W^{-s}=\mathbb{Z}[\mathbb{Z}_2]=\mathbb{Z}\{1_s,T_s\}, d=1+(-1)^{s+1}T: W^{-s} \longrightarrow W^{-s-1}
	\end{equation*}

	Then we have $W_{\%}(C)=Hom_{\mathbb{Z}[\mathbb{Z}_2]}(W^{-*},(C \otimes_R C))$. Given $\psi \in W_{\%}(C),\phi \in W^{\%}(D)$, $\psi \otimes \phi \in W_{\%}(C \otimes D)$ is the composition of the following maps:
	\begin{equation*}
		W^{-*} \stackrel{\Delta}{\longrightarrow} W^{-*} \otimes_{\mathbb{Z}} W_* \stackrel{\psi \widehat{\otimes}_{\mathbb{Z}}\phi}{\longrightarrow} (C \otimes_R C) \otimes_{\mathbb{Z}} (D \otimes_{\mathbb{Z}} D) \stackrel{sw}{\longrightarrow} (C \otimes_{\mathbb{Z}} D) \otimes_{R} (C \otimes_{\mathbb{Z}} D)
	\end{equation*}

	Where $\Delta$ is the following chain map:
	\begin{equation*}
		\Delta: W^{-s} \longrightarrow (W^{-*} \otimes_{\mathbb{Z}} W_*)^{-s}, \Delta(1_{-s})=\sum\limits_{r=1}^{+\infty} T_{-s-r}^r \otimes 1_r
	\end{equation*}
	
	$\psi \widehat{\otimes}_{\mathbb{Z}}\phi$ is the map given by:
	\begin{equation*}
		(\psi \widehat{\otimes}_{\mathbb{Z}}\phi) (x \otimes_{\mathbb{Z}} y)=(-1)^{|\psi||y|+|x||y|}\psi(x) \otimes_{\mathbb{Z}} \phi(y), \ x \in W^{-*}, y \in W_*
	\end{equation*}

	and $sw$ is the following chain map:
	\begin{equation*}
		\begin{aligned}
			sw:(C_p \otimes C_q) \otimes (D_r \otimes D_t) \longrightarrow (C_p \otimes D_r) \otimes (C_q \otimes D_t) \\
			sw\big((x \otimes y) \otimes (z \otimes w)\big)=(-1)^{qr}(x \otimes z) \otimes (y \otimes w)
		\end{aligned}
	\end{equation*}
	
\end{Theorem}


\begin{Theorem}[Lemma 15.78 and Remark 15.82 in \cite{luecksurgery}]
	\
	\label{Isymstr}
	
 	Let $I$ be the cellular chain complex of the 1-simplex $\Delta^1$. Let $\mathbf{i}_0,\mathbf{i}_1:0\mathbb{Z} \longrightarrow I$ be the inclusions of two ends. Then there is an element $\omega_I \in \big(W^{\%}(I)\big)_1$, such that $d_{W^{\%}(I)}\omega_I=\mathbf{i}_1^{\%}\nu-\mathbf{i}_0^{\%}\nu$. 
\end{Theorem}

Now we can state our main theorem in this subsection:

\begin{Theorem}
	\
	\label{computationdelta}
	
	Let $k \in \mathbb{Z}$ and $g:(\underline{L_1},\underline{L_2}) \longrightarrow (\mathbb{L}_{-k}(M^h(R)),*)$ be a $\Delta$-set map. Let $x \in T^k(\underline{L_1},\underline{L_2})$ be the image of $g$ under the identification in Theorem \ref{Lcohomologywithadcord}. Then there is an element $y \in T^{k+1}(\underline{L_0} \times \Delta^1,\underline{L_2} \times \Delta^1 \cup \underline{L_0} \times 1 \cup \underline{L_1} \times 0)$ that maps to $\mathcal{K}^*x$ under the restriction map. Furthermore, we can describe $y$ as the cobordism class of a $(\underline{L_0} \times \Delta^1,\underline{L_2} \times \Delta^1 \cup \underline{L_0} \times 1 \cup \underline{L_1} \times 0)$-ad $F$ of degree $k+1$, given as follows:
	
	Let $(\widecheck{D},\widecheck{\theta})$ be the image of $g$ under the identification in Theorem \ref{Lhomology}. It is a $(-k)$-dimensional Poincare quadratic chain complex in $M^h(R)^*(\underline{L_1},\underline{L_2})$. Let $(D,\theta)$ be the local dual of $(\widecheck{D},\widecheck{\theta})$ in the sense of Theorem \ref{localdual}. For $r \in \mathbb{Z},u \in \mathbb{N},\sigma \in L$, let $\delta\psi_u^r[\sigma]$ be the map given in Lemma \ref{Appquadpair} and let $r_{\sigma}=J^{all}_{\sigma}+\frac{|\sigma|(|\sigma|-1)}{2}$.
	
	The order on $\Sigma^l$ and the standard order on $\Delta^1$ with $0<1$ give standard orientations on the closed cells of $\underline{L_0} \times \Delta^1$. For any oriented closed cell $(\zeta,o)$ of $\underline{L_0} \times \Delta^1$, let $sgn(o)=1$ if $o$ agrees with the standard one and $-1$ otherwise. Then the functor $F$, defined by $F(\zeta,o)=(C_{\zeta},\psi_{\zeta,o})$, is given as follows:
	
	On objects, the functor $F$ is given by:
	\begin{equation*}
		\begin{aligned}
			& \quad \quad \text{For } \sigma \in L_2 \backslash L_1: \\
			& \quad \quad (C_{\sigma^* \times \Delta^1})_r= [D_r][\sigma] \\
			& \quad \quad d_{C_{\sigma^* \times \Delta^1}}=[d_{D}][\sigma]:(C_{\sigma^* \times \Delta^1})_r \longrightarrow (C_{\sigma^* \times \Delta^1})_{r-1} \\
			& \quad \quad \psi_{\sigma^* \times \Delta^1,o}^{u,r}: C^{|\sigma^* \times \Delta^1|-u-deg F-r}_{\sigma^* \times \Delta^1}=[D_{|\sigma^*|-u-k-r}][\sigma]^* \longrightarrow (C_{\sigma^* \times \Delta^1})_r=[D_r][\sigma]\\
			& \quad \quad \psi_{\sigma^* \times \Delta^1,o}^{u,r}=(-1)^{r_{\sigma}}(-1)^{|\sigma|r}sgn(o)\mathop{\oplus}\limits_{\tau \geq \sigma} \theta_u^r(\tau,\sigma) \\
		\end{aligned}
	\end{equation*}
	\begin{equation*}
		\begin{aligned}
			& \quad \quad \text{For } \sigma \in L=L_1 \backslash L_0: \\
			& \quad \quad (C_{\sigma^* \times 0})_r=\mathop{\oplus}\limits_{s \in A_{\sigma}}D_r(s) \\
			& \quad \quad d_{C_{\sigma^* \times 0}}=\mathop{\boxplus}\limits_{s \in A_{\sigma}}\mathop{\oplus}\limits_{s' \in A_{\sigma}}d_D(s',s):(C_{\sigma^* \times 0})_r \longrightarrow (C_{\sigma^* \times 0})_{r-1}\\
			& \quad \quad \psi^{u,r}_{\sigma^* \times 0,o}: C_{\sigma^* \times 0}^{|\sigma^* \times 0|-u-deg F-r}=\mathop{\oplus}\limits_{s \in A_{\sigma}}D_{|\sigma^*|-u-1-k-r}(s)^* \longrightarrow (C_{\sigma^* \times 0})_r=\mathop{\oplus}\limits_{s' \in A_{\sigma}}D_r(s') \\
			& \quad \quad \psi^{u,r}_{\sigma^* \times 0,o}=(-1)^{r_{\sigma}}sgn(o)\delta\psi_u^r[\sigma] \\
		\end{aligned}
	\end{equation*}
	\begin{equation*}
		\begin{aligned}
			& \quad \quad (C_{\sigma^* \times \Delta^1})_r =\bigg(\big(\mathop{\oplus}\limits_{s \in A_{\sigma}}D_*(s)\big) \otimes I\bigg)_r \\
			& \quad \quad \qquad \qquad \ \, =\big(\mathop{\oplus}\limits_{s \in A_{\sigma}}D_r(s)\big) \oplus \big(\mathop{\oplus}\limits_{s \in A_{\sigma}}D_r(s)\big) \oplus \big(\mathop{\oplus}\limits_{s \in A_{\sigma}}D_{r-1}(s)\big) \\
			& \quad \quad d_{C_{\sigma^* \times \Delta^1}}=\begin{bmatrix} d_{C_{\sigma^* \times 0}} & 0 & -1 \\ 0 & d_{C_{\sigma^* \times 0}} & 1 \\ 0 & 0 & -d_{C_{\sigma^* \times 0}} \end{bmatrix} : (C_{\sigma^* \times \Delta^1})_r \longrightarrow (C_{\sigma^* \times \Delta^1})_{r-1} \\
			& \quad \quad \psi^{u,r}_{\sigma^* \times \Delta^1,o}: C_{\sigma^* \times \Delta^1}^{|\sigma^* \times \Delta^1|-u-degF-r} \longrightarrow (C_{\sigma \otimes \Delta^1})_r \\
			& \quad \quad \psi^{u,r}_{\sigma \otimes \Delta^1,o}=(-1)^{r_{\sigma}+1}sgn(o)(\delta\psi[\sigma] \otimes \omega_{I})_u^r \\
			& \quad \quad \text{For all other closed cells } \zeta: (C_{\zeta},\psi_{\zeta,o})=\emptyset_{|\zeta|-k-1}. \\
			& \quad \quad \text{For all } l \in \mathbb{Z}: F(\emptyset_l)=\emptyset_{l-k-1}. \\
		\end{aligned}
	\end{equation*}
	
	On morphisms, the functor $F$ is given by:
	
	Let $\zeta_1 \leq \zeta_2$, then:
	
	$(a)$ If $\zeta_1=\tau^* \times \Delta^1, \zeta_2=\sigma^* \times \Delta^1$, with $\tau,\sigma \in L$.
	
	We have $\tau \geq \sigma$ and thus $A_{\tau} \subset A_{\sigma}$, $F((\zeta_1,o_1)-(\zeta_2,o_2))$ is then given by the inclusion map:
	\begin{equation*}
		\mathop{\oplus}\limits_{s \in A_{\tau}}D_*(s) \otimes I \longrightarrow \mathop{\oplus}\limits_{s \in A_{\sigma}}D_*(s) \otimes I
	\end{equation*}
	
	$(b)$ If $\zeta_1=\tau^* \times \Delta^1,\zeta_2=\sigma^* \times \Delta^1$, with $\tau \notin L_1, \sigma \in L$
	
	Since $\tau \geq \sigma$ and $\tau \notin L_1$, for any $\kappa \geq \tau$ with $\kappa \in L_2$, we have that $\kappa \in L_2 \backslash L_1$ and $\kappa>\sigma$. Therefore, we have $\kappa \in A_{\sigma}$. 
	
	Since $g|_{\underline{L_2}}$ is the constant map to the $0$ chain complex, by the construction in Theorem \ref{Lhomology}, we have that $D_*(s)=0$ for all $s \notin L_2$. Then $F((\zeta_1,o_1)-(\zeta_2,o_2))$ is given by the composition of the following inclusion maps:
	\begin{equation*}
		C_{\tau^* \times \Delta^1}=\mathop{\oplus}\limits_{\kappa \geq \tau}D_*(\kappa) \stackrel{inclusion}{\longrightarrow} \mathop{\oplus}\limits_{s \in A_{\sigma}}D_*(s) \stackrel{\mathbf{i}_0}{\longrightarrow} C_{\sigma^* \times \Delta^1}=\mathop{\oplus}\limits_{s \in A_{\sigma}}D_*(s) \otimes I
	\end{equation*}
	
	$(c)$ If $\zeta_1=\tau^* \times \Delta^1$, $\zeta_2=\sigma^* \times \Delta^1$, with $\tau \notin L_1,\sigma \notin L_1$. 
	
	We have $\tau \geq \sigma$ and then $F((\zeta_1,o_1)-(\zeta_2,o_2))$ is given by the inclusion map:
	\begin{equation*}
		C_{\tau^* \times \Delta^1}=[D_*][\tau] \longrightarrow C_{\sigma^* \times \Delta^1}=[D_*][\sigma]
	\end{equation*}

	$(d)$ If $\zeta_1=\tau^* \times 0,\zeta_2=\sigma^* \times \Delta^1$, with $\tau,\sigma \in L$.
	
	We have $\tau \geq \sigma$ and thus $A_{\tau} \subset A_{\sigma}$. $F((\zeta_1,o_1)-(\zeta_2,o_2))$ is then given by the composition of the following inclusion maps:
	\begin{equation*}
		C_{\tau^* \times 1}=\mathop{\oplus}\limits_{s \in A_{\tau}}D_*(s) \stackrel{inclusion}{\longrightarrow} \mathop{\oplus}\limits_{s \in A_{\sigma}}D_*(s) \stackrel{\mathbf{i}_1}{\longrightarrow} C_{\sigma^* \times \Delta^1}=\mathop{\oplus}\limits_{s \in A_{\sigma}}D_*(s) \otimes I
	\end{equation*}
	
	$(e)$ If $\zeta_1=\tau^* \times 0,\zeta_2=\sigma^* \times 0$, with $\tau,\sigma \in L$, 
	
	We have $\tau \geq \sigma$ and thus $A_{\tau} \subset A_{\sigma}$, $F((\zeta_1,o_1)-(\zeta_2,o_2))$ is then given by the inclusion map:
	\begin{equation*}
		\mathop{\oplus}\limits_{s \in A_{\tau}}D_*(s) \longrightarrow \mathop{\oplus}\limits_{s \in A_{\sigma}}D_*(s)
	\end{equation*}
	
	(f) If all the cases above do not happen, then define $F((\zeta_1,o_1)-(\zeta_2,o_2))$ to be zero. (In fact, in this case the domain of $F((\zeta_1,o_1)-(\zeta_2,o_2))$ will always be zero.)
\end{Theorem}

\begin{proof}
	To start with, we need to prove that $F$ is a $(\underline{L_0} \times \Delta^1,\underline{L_2} \times \Delta^1 \cup \underline{L_0} \times 1 \cup \underline{L_1} \times 0)$-ad of degree $k+1$. We will check step by step the conditions listed in Definition \ref{Defad}:
	
	(1) $F$ is a $(k+1)$-morphism from $Cell(\underline{L_0} \times I)$ to $\mathfrak{A}_{R}$.
	
	Since $F$ maps morphisms in $Cell(\underline{L_0} \times I)$ to the corresponding inclusions of chain complexes, it is clear that $F$ is a functor. Since $sgn(-o)=-sgn(o)$, we have $(C_{\zeta},\psi_{\zeta,-o})=(C_{\zeta},-\psi_{\zeta,o})$, thus $F$ commutes with involution. By definition, it is easy to see that $F$ commutes with $\emptyset$ and decrease the dimension by $k+1$. Therefore, $F$ is a $(k+1)$-morphism from $Cell(\underline{L_0} \times I)$ to $\mathfrak{A}_{R}$.
	
	(2) $F$ is a pre $(\underline{L_0} \times I,\underline{L_2} \times I \cup \underline{L_0} \times 1 \cup \underline{L_1} \times 0)$-ad.
	
	It only leaves us to check that $F|_{\underline{L_2} \times I \cup \underline{L_0} \times 1 \cup \underline{L_1} \times 0}$ is the trivial ad. Equivalently, we have to check that for every $\zeta \in \underline{L_2} \times I \cup \underline{L_0} \times 1 \cup \underline{L_1} \times 0$, $F(\zeta)=\emptyset_{|\zeta|-k-1}$. This follows from the definition of $F$.
	
	
	(3) $F$ is balanced.
	
	It follows from definition that $F((\zeta_1,o_1)-(\zeta_2,o_2))$ is independent of $o_1,o_2$. By the definition of balance structure on $Cell(\underline{L_0} \times I)$, $F$ is balanced.
	
	(4) $F$ is closed.
	
	For every closed cell $\zeta \in \underline{L_0} \times I$, denote $o_{std}^{\zeta}$ to be its standard orientation. When the cell is clear from context, we simply write $o_{std}$ for $o_{std}^{\zeta}$. By definition, we have to check that for every closed cell $\zeta$, the following map is a chain map:
	\begin{equation*}
		\Psi_{\zeta}: cl(\zeta) \longrightarrow W \otimes_{\mathbb{Z}[\mathbb{Z}_2]}(C^t_{\zeta} \otimes C_{\zeta})
	\end{equation*}
	\begin{equation*}
		\Psi_{\zeta}(\textlangle \zeta',o' \textrangle)=F((\zeta,o_{std}^{\zeta})-(\zeta',o'))_*\psi_{\zeta',o'} \ \ \ (\zeta' \leq \zeta)
	\end{equation*}
	
	Note that since $F$ is balanced, for all $\zeta' \leq \zeta$, we have the following commutative diagram:
	\begin{equation*}
			\begin{tikzcd}
			cl(\zeta) \rar{\Psi_{\zeta}} & W \otimes_{\mathbb{Z}[\mathbb{Z}_2]}(C^t_{\zeta} \otimes C_{\zeta})\\
			cl(\zeta') \uar{\text{inclusion}}\rar{\Psi_{\zeta'}} & W \otimes_{\mathbb{Z}[\mathbb{Z}_2]}(C^t_{\zeta'} \otimes C_{\zeta'}) \uar[swap]{F((\zeta,o_{std}^{\zeta})-(\zeta',o_{std}^{\zeta'}))_*}\\
		\end{tikzcd}
	\end{equation*}

	Since $F((\zeta,o_{std}^{\zeta})-(\zeta',o_{std}^{\zeta'}))_*$ is a chain map, we only need to check that $\Psi_{\zeta}$ is a chain map in top dimension, that is, to check that the following equation holds:
	\begin{equation}
		\label{Appcheck7}
		\partial\Psi_{\zeta}(\textlangle\zeta,o_{std}\textrangle)=\Psi_{\zeta}(\partial\textlangle\zeta,o_{std}\textrangle)
	\end{equation}
	We divide the proof into several cases:
	
	(a) $\zeta \in \underline{L_0} \times 1 \cup \underline{L_1} \times 0$
	
	By definition, we have $C_{\zeta}=0$ and $\psi_{\zeta,o_{std}}=0$, therefore:
	\begin{equation*}
		W \otimes_{\mathbb{Z}[\mathbb{Z}_2]} (C_{\zeta}^t \otimes C_{\zeta})=0
	\end{equation*}
	
	Thus equation \ref{Appcheck7} holds since both side is $0$.
	
	(b) $\zeta=\sigma^* \times \Delta^1$ for some $\sigma \notin L_1$
	
	Note that by definition we have:
	\begin{equation*}
		\Psi_{\zeta}(\textlangle\zeta,o_{std}\textrangle)=\psi_{\sigma^* \times \Delta^1,o_{std}}
	\end{equation*}
	\begin{equation*}
		\partial\textlangle\zeta,o_{std}\textrangle=\textlangle\sigma^* \times 1,o_{std}\textrangle-\textlangle\sigma^* \times 0,o_{std}\textrangle-\sum\limits_{\tau \in K^*(\sigma)}(-1)^{n_{\tau^*}^{\sigma^*}}\textlangle\tau^* \times \Delta^1,o_{std}\textrangle
	\end{equation*}
	
	Since by definition we have $\psi_{\sigma^* \times 0,o_{std}}=\psi_{\sigma^* \times 1,o_{std}}=0$, the equation \ref{Appcheck7} is equivalent to the following one:
	\begin{equation}
		\label{Appcheck8}
		\partial\psi_{\sigma^* \times \Delta^1,o_{std}}=-\sum\limits_{\tau \in K^*(\sigma)}(-1)^{n_{\tau^*}^{\sigma^*}}\Psi_{\zeta}(\textlangle \tau^* \times \Delta^1,o_{std} \textrangle)
	\end{equation}
	
	By Remark \ref{AppRemquadpair}, we have that $(\{\theta\}[\sigma],\delta\{\theta\}[\sigma])$ gives the structure of a $(l-k-|\sigma|)$-dimensional Poincare pair. Since $\psi_{\sigma^* \times \Delta^1,o_{std}}^{u,r}=(-1)^{r_{\sigma}}\delta\{\theta\}_u^r[\sigma]$, let $r'=|\sigma^*|-1-k-u-r$, we can deduce:
	\begin{equation*}
		\begin{aligned}
			\partial\psi_{\sigma^* \times \Delta^1,o_{std}}^{u,r}&=(-1)^{r_{\sigma}}\partial\delta\{\theta\}_u^r[\sigma]\\
			&=(-1)^{r_{\sigma}+l-k-|\sigma|}i_r[\sigma]\{\theta\}_u^r[\sigma]i_{r'}[\sigma]^* \\
			&=(-1)^{r_{\sigma}+|\sigma|+1}\sum\limits_{\tau \in K^*(\sigma)}(-1)^{n_{\sigma}^{\tau}} i_r[\sigma]I_{\tau,r}\delta\{\theta\}_u^r[\tau]P^{\tau,r'}i_{r'}[\sigma]^* \\
		\end{aligned}
	\end{equation*}

	Since $\sigma \notin L_1$, for $\tau \in K^*(\sigma)$, we have that $\sigma \leq \tau$ and thus $\tau \notin L_1$. Therefore, by definition we have $\psi_{\tau^* \times \Delta^1,o_{std}}^{u,r}=(-1)^{r_{\tau}}\delta\{\theta\}_u^r[\tau]$. By definition of $r_{\sigma}$ and Lemma \ref{Dualincidnum}, we have:
	\begin{equation}
		\label{eqrsigma}
		r_{\sigma}-r_{\tau}=J_{\sigma}^{all}-J_{\tau}^{all}-|\sigma|=n_{\sigma}^{\tau}+n_{\tau^*}^{\sigma^*}-|\sigma|
	\end{equation}
	
	Thus we have:
	\begin{equation}
		\label{Appcheck9}
		\partial\psi_{\sigma^* \times \Delta^1,o_{std}}^{u,r}=-\sum\limits_{\tau \in K^*(\sigma)}(-1)^{n_{\tau^*}^{\sigma^*}} i_r[\sigma]I_{\tau,r}\psi_{\tau^* \times \Delta^1,o_{std}}^{u,r}P^{\tau,r'}i_{r'}[\sigma]^* 
	\end{equation}

	Note that $P^{\tau,r'}=I_{\tau,r'}^*$ and for every $r_0 \in \mathbb{Z}$, we have that $i_{r_0}[\sigma]I_{\tau,r_0}$ is the chain map $F((\sigma^* \times \Delta^1,o_{std})-(\tau^* \times \Delta^1,o_{std}))$ on dimension $r_0$. Thus we have:
	\begin{equation*}
		\Psi_{\zeta}(\textlangle \tau^* \times \Delta^1 ,o_{std}\textrangle)_u^r=i_r[\sigma]I_{\tau,r}\psi_{\tau^* \times \Delta^1,o_{std}}^{u,r}P^{\tau,r'}i_{r'}[\sigma]^*
	\end{equation*}
	
	Comparing equation \ref{Appcheck8} and \ref{Appcheck9}, we get that equation \ref{Appcheck8} holds.
	
	(c) $\zeta=\sigma^* \times 0$ for some $\sigma \in L=L_1 \backslash L_0$
	
	Note that by definition we have: 
	\begin{equation*}
		\Psi_{\zeta}(\textlangle\zeta,o_{std}\textrangle)=\psi_{\sigma^* \times 0,o_{std}}, \ \partial\textlangle\zeta,o_{std}\textrangle=\sum\limits_{\tau \in K^*(\sigma)}(-1)^{n_{\tau^*}^{\sigma^*}}\textlangle\tau^* \times 0,o_{std}\textrangle
	\end{equation*}

	Since $\psi_{\tau^* \times 0,o_{std}}=0$ for $\tau \notin L$, the equation \ref{Appcheck7} is equivalent to the following one:
	\begin{equation}
		\label{Appcheck10}
		\partial\psi_{\sigma^* \times 0,o_{std}}=\sum\limits_{\tau \in L^*(\sigma)} (-1)^{n_{\tau^*}^{\sigma^*}}\Psi_{\zeta}(\textlangle \tau^* \times 0, o_{std} \textrangle)
	\end{equation}

	By Lemma \ref{Appquadpair}, we have that $(\psi[\sigma],\delta\psi[\sigma])$ gives the structure of a $(l-k-|\sigma|-1)$-dimensional Poincare pair. Since  $\psi_{\sigma^* \times 1,o_{std}}^{u,r}=(-1)^{r_{\sigma}}\delta\psi_u^r[\sigma]$, let $r'=|\sigma^*|-2-k-u-r$, we can deduce:
	\begin{equation*}
		\begin{aligned}
			\partial\psi_{\sigma^* \times 0,o_{std}}^{u,r}&=(-1)^{r_{\sigma}}\partial\delta\psi_u^r[\sigma]\\
			&=(-1)^{r_{\sigma}}(-1)^{l-k-|\sigma|-1}i_r[\sigma]\psi_u^r[\sigma]i_{r'}[\sigma]^* \\
			&=(-1)^{r_{\sigma}+|\sigma|}\sum\limits_{\tau \in L^*(\sigma)}(-1)^{n_{\sigma}^{\tau}} i_r[\sigma]I_{\tau,r}\delta\psi_u^r[\tau]P^{\tau,r'}i_{r'}[\sigma]^* \\
		\end{aligned}
	\end{equation*}
	
	For $\tau \in L^*(\sigma)$, by definition we have $\psi_{\tau^* \times 0,o_{std}}^{u,r}=(-1)^{r_{\tau}}\delta\psi_u^r[\tau]$. Combining with equation \ref{eqrsigma} we have:
	\begin{equation}
		\label{Appcheck11}
		\partial\psi_{\sigma^* \times 0,o_{std}}^{u,r}=\sum\limits_{\tau \in L^*(\sigma)}(-1)^{n_{\tau^*}^{\sigma^*}} i_r[\sigma]I_{\tau,r}\psi_{\tau^* \times 0,o_{std}}^{u,r}P^{\tau,r'}i_{r'}[\sigma]^*
	\end{equation}
	
	Note that $P^{\tau,r'}=I_{\tau,r'}^*$ and for every $r_0 \in \mathbb{Z}$, we have that $i_{r_0}[\sigma]I_{\tau,r_0}$ is the chain map $F((\sigma^* \times 0,o_{std})-(\tau^* \times 0,o_{std}))$ on dimension $r_0$. Thus we have:
	\begin{equation*}
		\Psi_{\zeta}(\textlangle \tau^* \times 0,o_{std} \textrangle)_u^r=i_r[\sigma]I_{\tau,r}\psi_{\tau^* \times 0,o_{std}}^{u,r}P^{\tau,r'}i_{r'}[\sigma]^*
	\end{equation*}
	
	Comparing equation \ref{Appcheck10} and \ref{Appcheck11}, we get that equation \ref{Appcheck10} holds.
	
	(d) $\zeta=\sigma^* \times \Delta^1$ for some $\sigma \in L=L_1 \backslash L_0$
	
	Note that by definition we have:
	\begin{equation*}
		\Psi_{\zeta}(\textlangle\zeta,o_{std}\textrangle)=\psi_{\sigma^* \times \Delta^1,o_{std}}
	\end{equation*}
	\begin{equation*}
		\partial\textlangle\zeta,o_{std}\textrangle=\textlangle \sigma^* \times 1,o_{std} \textrangle-\textlangle \sigma^* \times 0,o_{std} \textrangle -\sum\limits_{\tau \in K^*(\sigma)}(-1)^{n_{\tau^*}^{\sigma^*}}\textlangle\tau^* \times \Delta^1,o_{std}\textrangle
	\end{equation*}

	Since $\psi_{\sigma^* \times 1,o_{std}}=0$, the equation \ref{Appcheck7} is equivalent to the following one:
	\begin{equation}
		\label{Appcheck12}
		\partial\psi_{\sigma^* \times \Delta^1,o_{std}}=-\Psi_{\zeta}(\textlangle \sigma^* \times 0,o_{std} \textrangle)-\sum\limits_{\tau \in K^*(\sigma)} (-1)^{n_{\tau^*}^{\sigma^*}}\Psi_{\zeta}(\textlangle \tau^* \times \Delta^1, o_{std} \textrangle)
	\end{equation}
	
	Since $\psi_{\sigma^* \times \Delta^1,o_{std}}=(-1)^{r_{\sigma}+1}\delta\psi[\sigma] \otimes \omega_{I}$, by Theorem \ref{tensorstr}, we can deduce:
	\begin{equation}
		\label{Appcheck20}
		\begin{aligned}
			\partial\psi_{\sigma^* \times \Delta^1,o_{std}}&=(-1)^{r_{\sigma}+1}\partial(\delta\psi[\sigma] \otimes \omega_I)\\
			&=(-1)^{r_{\sigma}+1}(\delta\psi[\sigma] \otimes \partial\omega_I-\partial\delta\psi[\sigma] \otimes \omega_I) \\
			&=(-1)^{r_{\sigma}+1}\big(\delta\psi[\sigma] \otimes ({\mathbf{i}_1}_*(\nu)-{\mathbf{i}_0}_*(\nu))-\partial\delta\psi[\sigma] \otimes \omega_I\big) \\
		\end{aligned}
	\end{equation}

	Since $\otimes:W_{\%}(C) \otimes W^{\%}(D) \longrightarrow W_{\%}(C \otimes D)$ is natural and $\nu$ is the unit of the tensor product, we get:
	\begin{equation}
		\label{Appcheck21}
		\delta\psi[\sigma] \otimes {\mathbf{i}_0}_*(\nu)={\mathbf{i}_0}_*\delta\psi[\sigma], \ \delta\psi[\sigma] \otimes {\mathbf{i}_1}_*(\nu)={\mathbf{i}_1}_*\delta\psi[\sigma],
	\end{equation}

	For every $\tau \in L^*(\sigma)$, let $F^{\sigma\tau}=F((\sigma^* \times 1,o_{std})-(\tau^* \times 1,o_{std}))$. Since $\sigma \in L$, by the proof in (c), we have the following equivalent form of equation \ref{Appcheck10}:
	\begin{equation}
		\label{Appcheck22}
		\begin{aligned}
			\partial\delta\psi[\sigma]&=(-1)^{|\sigma|}\sum\limits_{\tau \in L^*(\sigma)}(-1)^{n_{\sigma}^{\tau}}F^{\sigma\tau}_*\delta\psi[\tau] \\
		\end{aligned}
	\end{equation}

	Since $F((\sigma^* \times \Delta^1,o_{std})-(\tau^* \times \Delta^1,o_{std}))=F^{\sigma\tau} \otimes Id$, we get:
	\begin{equation}
		\label{Appcheck23}
		\Psi_{\zeta}(\textlangle \tau^* \times \Delta^1, o_{std} \textrangle)=(-1)^{r_{\tau}+1}(F^{\sigma\tau}_*\delta\psi_u^r[\tau])\otimes \omega_I
	\end{equation}

	Substituting equation \ref{Appcheck21}, \ref{Appcheck22} and \ref{Appcheck23} into equation \ref{Appcheck20}, we have that the following equation holds:
	\begin{equation}
		\label{Appcheck13}
		\begin{aligned}
			\partial\psi_{\sigma^* \times \Delta^1,o_{std}}=&(-1)^{r_{\sigma}+1}{\mathbf{i}_1}_*\delta\psi[\sigma]-(-1)^{r_{\sigma}+1}{\mathbf{i}_0}_*\delta\psi[\sigma]\\
			&-\sum\limits_{\tau \in L^*(\sigma)}(-1)^{n_{\sigma}^{\tau}+r_{\sigma}+r_{\tau}+|\sigma|}\Psi_{\zeta}(\textlangle \tau^* \times \Delta^1, o_{std} \textrangle)
		\end{aligned}
	\end{equation}

	By definition, we have $F((\sigma^* \times 0,o_{std})-(\sigma^*\times \Delta^1,o_{std}))=\mathbf{i}_1$ and $\psi_{\sigma^* \times 0,o_{std}}=(-1)^{r_{\sigma}}\delta\psi[\sigma]$. Therefore, we have:
	\begin{equation*}
		(-1)^{r_{\sigma}}{\mathbf{i}_1}_*\delta\psi[\sigma]=\Psi_{\zeta}(\textlangle \sigma^* \times 0,o_{std} \textrangle)
	\end{equation*}

	Comparing equation \ref{Appcheck12} and \ref{Appcheck13} together with \ref{eqrsigma}, it only leaves us to prove:
	\begin{equation*}
		\sum\limits_{\substack{\tau \in K^*(\sigma) \\ \tau \notin L}} (-1)^{n_{\tau^*}^{\sigma^*}}\Psi_{\zeta}(\textlangle \tau^* \times \Delta^1, o_{std} \textrangle)=(-1)^{r_{\sigma}+1}{\mathbf{i}_0}_*\delta\psi[\sigma]
	\end{equation*}

	For any $\tau \in K^*(\sigma)$ with $\tau \notin L$, since $\tau \geq \sigma$ and $\sigma \notin L_0$, we have $\tau \notin L_0$ and thus $\tau \notin L_0 \cup L= L_1$. Since $\psi_{\tau^* \times \Delta^1,o_{std}}=0$ for $\tau \notin L_2$, the equation above is equivalent to:
	\begin{equation*}
		\sum\limits_{\substack{\tau \in K^*(\sigma) \\ \tau \in L_2 \backslash L_1}} (-1)^{n_{\tau^*}^{\sigma^*}}\Psi_{\zeta}(\textlangle \tau^* \times \Delta^1, o_{std} \textrangle)=(-1)^{r_{\sigma}+1}{\mathbf{i}_0}_*\delta\psi[\sigma]
	\end{equation*}
	
	For any $\tau \in K^*(\sigma)$ with $\tau \in L_2 \backslash L_1=L \otimes [0,1] \backslash L \otimes \partial [0,1]$, we have that $\tau$ is the span of $\sigma$ with a vertice in $L \times \{1\}$. By the definition of order in $L \otimes [0,1]$, we have $n_{\sigma}^{\tau}=|\sigma|+1$. By Lemma \ref{Dualincidnum}, we have $n_{\tau^*}^{\sigma^*}=J_{\sigma}^{all}-J_{\tau}^{all}-n_{\sigma}^{\tau}=r_{\sigma}-r_{\tau}-1$. Therefore, the equation above is equivalent to:
	\begin{equation}
		\label{Appcheck14}
		\sum\limits_{\substack{\tau \in K^*(\sigma) \\ \tau \in L_2 \backslash L_1}}(-1)^{r_{\tau}}\Psi_{\zeta}(\textlangle \tau^* \times \Delta^1, o_{std} \textrangle)={\mathbf{i}_0}_*\delta\psi[\sigma]
	\end{equation}
	
	Now we begin to compute the term on the left hand side. For any $\tau \in K^*(\sigma)$ with $\tau \in L_2 \backslash L_1$, denote $\mathfrak{I}^{\tau\sigma}$ to be the following inclusion map:
	\begin{equation*}
		\mathfrak{I}^{\tau\sigma}:C_{\tau^* \times \Delta^1}=\mathop{\oplus}\limits_{\kappa \geq \tau}D_*(\kappa) \longrightarrow \mathop{\oplus}\limits_{s \in A_{\sigma}}D_*(s)
	\end{equation*}
	
	Then by definition of $F((\tau^* \times \Delta^1,o_{std})-(\sigma^* \times \Delta^1,o_{std}))$, we have:
	\begin{equation*}
		\sum\limits_{\substack{\tau \in K^*(\sigma) \\ \tau \in L_2 \backslash L_1}}(-1)^{r_{\tau}}\Psi_{\zeta}(\textlangle \tau^* \times \Delta^1, o_{std} \textrangle)=\sum\limits_{\substack{\tau \in K^*(\sigma) \\ \tau \in L_2 \backslash L_1}}(-1)^{r_{\tau}}{\mathbf{i}_0}_* \mathfrak{I}^{\tau\sigma}_*\psi_{\tau^* \times \Delta^1,o_{std}}
	\end{equation*}

	Thus it suffices to prove that:
	\begin{equation}
		\label{Appcheck15}
		\sum\limits_{\substack{\tau \in K^*(\sigma) \\ \tau \in L_2 \backslash L_1}}(-1)^{r_{\tau}} \mathfrak{I}^{\tau\sigma}_*\psi_{\tau^* \times \Delta^1,o_{std}}=\delta\psi[\sigma]
	\end{equation}
	
	We will check it by computing the maps in components. For any $u \in \mathbb{N},r \in \mathbb{Z}$, let $r'=|\sigma^*|-1-k-r-u$. Choose $s \in A_{\sigma}$ and $z_s \in D_{r'}(s)^*$, then by definition we have:
	\begin{equation*}
			\sum\limits_{\substack{\tau \in K^*(\sigma) \\ \tau \in L_2 \backslash L_1}}(-1)^{r_{\tau}} \mathfrak{I}^{\tau\sigma}_*\psi_{\tau^* \times \Delta^1,o_{std}}^{u,r}(z_s)=\sum\limits_{\substack{\tau \in K^*(\sigma) \\ \tau \in L_2 \backslash L_1}} \mathfrak{I}^{\tau\sigma}\delta\{\theta\}_u^r[\tau](\mathfrak{I}^{\tau\sigma})^*(z_s)
	\end{equation*}
	
	Since $\mathfrak{I}^{\tau\sigma}$ is the inclusion map, its dual $(\mathfrak{I}^{\tau\sigma})^*$ is the projection map. Let $K^*(\sigma,s)=\{\tau \in K^*(\sigma) \ | \ \tau \leq s\}$, we have:
	\begin{equation*}
		\begin{aligned}
			\sum\limits_{\substack{\tau \in K^*(\sigma) \\ \tau \in L_2 \backslash L_1}}(-1)^{r_{\tau}} \mathfrak{I}^{\tau\sigma}_*\psi_{\tau^* \times \Delta^1,o_{std}}^{u,r}(z_s)&=\sum\limits_{\substack{\tau \in K^*(\sigma,s) \\ \tau \in L_2 \backslash L_1}} \mathfrak{I}^{\tau\sigma}\delta\{\theta\}_u^r[\tau](z_s) \\
			&=\sum\limits_{\substack{\tau \in K^*(\sigma,s) \\ \tau \in L_2 \backslash L_1}} (-1)^{|\tau|r}\mathop{\oplus}\limits_{\kappa \geq \tau} \theta_u^r(\kappa,\tau)(z_s) \\
		\end{aligned}
	\end{equation*}
	
	Let $s_1=(s \cap L)_1$ and $V_1$ be the set of all vertices in $s_1$. For any $v \in V_1$, denote $\sigma_v=v*\sigma \leq s$. For any $\tau \in K^*(\sigma,s)$ with $\tau \in L_2 \backslash L_1=L \otimes [0,1] \backslash L \otimes \partial [0,1]$, we have that $\tau$ is the span of $\sigma$ with a vertex in $V_1$ and the correspondance is a bijection. Thus:
	\begin{equation}
			\label{Appcheck16}
			\sum\limits_{\substack{\tau \in K^*(\sigma) \\ \tau \in L_2 \backslash L_1}} (-1)^{r_{\tau}}\mathfrak{I}^{\tau\sigma}_*\psi_{\tau^* \times \Delta^1,o_{std}}^{u,r}(z_s)=\sum\limits_{v \in V_1} (-1)^{(|\sigma|+1)r}\mathop{\oplus}\limits_{\kappa \geq \sigma_v} \theta_u^r(\kappa,\sigma_v)(z_s)
	\end{equation}

	Now by the expression of $\mho_{\tau}$ given in \ref{expressmho} and the expression of $\delta\psi[\sigma]$ in Lemma \ref{Appquadpair}, we have:
	\begin{equation*}
		\begin{aligned}
			\delta\psi_u^r[\sigma](z_s)&=(-1)^{(|\sigma|+1)r}\big(\mathop{\boxplus}\limits_{s' \in B_{\sigma}}\mathop{\oplus}\limits_{s'' \in A_{\sigma}}\theta_u^r(s'',s')\big) \big( \mathop{\oplus}\limits_{v \in V_1}\iota_{\sigma_v}(z_s) \big) \\
			&=(-1)^{(|\sigma|+1)r}\sum\limits_{v \in V_1}\big(\mathop{\boxplus}\limits_{s' \in B_{\sigma}}\mathop{\oplus}\limits_{s'' \in A_{\sigma}}\theta_u^r(s'',s')\big)\iota_{\sigma_v}(z_s)\\
			&=(-1)^{(|\sigma|+1)r}\sum\limits_{v \in V_1}\mathop{\oplus}\limits_{s'' \in A_{\sigma}}\theta_u^r(s'',\sigma_v)(z_s)\\
		\end{aligned}
	\end{equation*}

	Note that $\theta_u^r(s'',\sigma_v)=0$ if $\sigma_v$ is not a face of $s''$. Since $A_{\sigma}$ is upper closed, we have $s'' \geq \sigma_v$ implies $s'' \in A_{\sigma}$. Therefore, we have:
	\begin{equation}
		\label{Appcheck17}
		\delta\psi_u^r[\sigma](z_s)=(-1)^{(|\sigma|+1)r}\sum\limits_{v \in V_1}\mathop{\oplus}\limits_{s'' \geq \sigma_v}\theta_u^r(s'',\sigma_v)(z_s)
	\end{equation}
	
	Comparing equation \ref{Appcheck16} and \ref{Appcheck17}, we get that equation \ref{Appcheck15} holds.

	(5) The associated functor $C$ is well-behaved.
	
	By the definition of $F$, it is straightforward to see that the associated functor $C$ maps each morphism to a cofibration.
	
	For every closed cell $\zeta$ in $\underline{L_0} \times I$, we need to check that the map
	\begin{equation}
		\label{Appcheck24}
		\mathop{colim}\limits_{\zeta' \subsetneq \zeta}F((\zeta',o_{std})-(\zeta,o_{std})):\mathop{colim}\limits_{\zeta' \subsetneq \zeta} C_{\zeta'} \longrightarrow C_{\zeta}
	\end{equation}
	is a cofibration. 
	
	We make the following claim:
	\begin{Claim}
		\label{Appclaim1}
		For every $r \in \mathbb{Z}$ and closed cell $\zeta \in \underline{L_0} \times I$, there is an $R$ module $E_{\zeta}^r$, such that:
		
		$(1)$ For every $r \in \mathbb{Z}$ and  closed cell $\zeta \in \underline{L_0} \times I$, we have $(C_{\zeta})_r=\mathop{\oplus}\limits_{\zeta_0 \subset \zeta}E_{\zeta_0}^r$.
		
		$(2)$ For every $r \in \mathbb{Z}$ and closed cells $\zeta,\zeta' \in \underline{L_0} \times I$ with $\zeta' \subsetneq \zeta$, $F((\zeta',o_{std})-(\zeta,o_{std}))$ is the inclusion map $	\mathop{\oplus}\limits_{\zeta_0 \subset \zeta'}E_{\zeta_0}^r \hookrightarrow \mathop{\oplus}\limits_{\zeta_0 \subset \zeta}E_{\zeta_0}^r$.
	\end{Claim}

	Assuming that the claim holds, it is straightforward to see that the map in \ref{Appcheck24} is the inclusion map $\mathop{\oplus}\limits_{\zeta_0 \subsetneq \zeta}E_{\zeta_0}^r \hookrightarrow \mathop{\oplus}\limits_{\zeta_0 \subset \zeta}E_{\zeta_0}^r$, which is a fibration in the sense of Definition \ref{Deffib}.
	
	To prove the claim, we first write down the $R$ module $E_{\zeta}^r$ and then check the statements in the claim.
	
	For every $r \in \mathbb{Z}$ and closed cell $\zeta \in \underline{L_0} \times I$, we define $E_{\zeta}^r$ as follows:
	
	(a) If $\zeta \in \underline{L_0} \times 1 \cup \underline{L_1} \times 0$, define $E_{\zeta}^r=0$.
	
	(b) If $\zeta=\sigma^* \times \Delta^1$ for some simplex $\sigma \in L_2 \backslash L_1$, define $E_{\zeta}^r=D_r(\sigma)$.
	
	(c) If $\zeta=\sigma^* \times 0$ for some simplex $\sigma \in L$, define $E_{\zeta}^r=\mathop{\oplus}\limits_{s \in B_{\sigma}}D_r(s)$.
	
	(d) If $\zeta=\sigma^* \times \Delta^1$ for some simplex $\sigma \in L$, define $E_{\zeta}^r=\mathop{\oplus}\limits_{s \in B_{\sigma}}D_{r-1}(s)$.
	
	We begin to check the statement (1) of Claim \ref{Appclaim1}. The proof is divided into four cases:
	
	(a) $\zeta \in \underline{L_0} \times 1 \cup \underline{L_1} \times 0$
	
	By definition, $E_{\zeta_0}^r=0$ for any $\zeta_0 \subset \zeta$ and $(C_{\zeta})_r=0$. Thus the statement holds trivially.
	
	(b) $\zeta=\sigma^* \times \Delta^1$ for some simplex $\sigma \in L_2 \backslash L_1$
	
	By definition, for $\zeta_0 \subset \zeta$, we have that $E_{\zeta_0}^r=0$ unless $\zeta_0=\tau^* \times \Delta^1$ for some $\tau \geq \sigma$ with $\tau \in L_2 \backslash L_1$; in that case $E_{\zeta_0}^r=D_r(\tau)$. Therefore, we have:
	\begin{equation*}
		\mathop{\oplus}\limits_{\zeta_0 \subset \zeta}E_{\zeta_0}^r=\mathop{\oplus}\limits_{\substack{\tau \geq \sigma \\ \tau \in L_2 \backslash L_1}}D_r(\tau)
	\end{equation*}
	
	Note that $D_r(s)=0$ for all $s \notin L_2$. Furthermore, since $\sigma \notin L_1$, $\tau \geq \sigma$ implies $\tau \notin L_1$. Therefore, we have:
	\begin{equation*}
		\mathop{\oplus}\limits_{\zeta_0 \subset \zeta}E_{\zeta_0}^r=\mathop{\oplus}\limits_{\tau \geq \sigma}D_r(\tau)=[D_r][\sigma]=(C_{\sigma^* \times \Delta^1})_r
	\end{equation*}

	Therefore, the statement holds.
	
	(c) $\zeta=\sigma^* \times 0$ for some simplex $\sigma \in L$
	
	By definition, for $\zeta_0 \subset \zeta$, we have that $E_{\zeta_0}=0$ unless $\zeta_0=\tau^* \times 0$ for some $\tau \geq \sigma$ with $\tau \in L$; in that case $E_{\zeta_0}^r=\mathop{\oplus}\limits_{s \in B_{\tau}}D_r(s)$. Therefore, we have:
	\begin{equation*}
		\mathop{\oplus}\limits_{\zeta_0 \subset \zeta}E_{\zeta_0}^r=\mathop{\oplus}\limits_{\tau \geq \sigma}\mathop{\oplus}\limits_{s \in B_{\tau}}D_r(s)
	\end{equation*}

	Since the sets $B_{\tau}$ are disjoint for different $\tau$ and $\mathop{\cup}\limits_{\tau \geq \sigma}B_{\tau}=A_{\sigma}$, we have:
	\begin{equation*}
		\mathop{\oplus}\limits_{\zeta_0 \subset \zeta}E_{\zeta_0}^r=\mathop{\oplus}\limits_{s \in A_{\sigma}}D_r(s)=(C_{\sigma^* \times 0})_r
	\end{equation*}

	Therefore, the statement holds.
	
	(d) $\zeta=\sigma^* \times \Delta^1$ with $\sigma \in L$
	
	By definition, for any $r \in \mathbb{Z}$, we have:
	\begin{equation*}
		(C_{\sigma^* \times \Delta^1})_r=\big(\mathop{\oplus}\limits_{s \in A_{\sigma}}D_r(s)\big) \oplus \big(\mathop{\oplus}\limits_{s \in A_{\sigma}}D_r(s)\big) \oplus \big(\mathop{\oplus}\limits_{s \in A_{\sigma}}D_{r-1}(s)\big)
	\end{equation*}
	
	For $\zeta_0 \subset \zeta$, there are in general three different cases:
	
	($\alpha$) $\zeta_0=\tau^* \times 0$ for some $\tau \geq \sigma$. In this case, we have $E_{\zeta_0}^r=0$ unless $\tau \in L$. When $\tau \in L$, we have $E_{\zeta_0}^r=\mathop{\oplus}
	\limits_{s \in B_{\tau}}D_r(s)$.
	
	($\beta$) $\zeta_0=\tau^* \times 1$ for some $\tau \geq \sigma$. In this case, we have $E_{\zeta_0}^r=0$.
	
	($\gamma$) $\zeta_0=\tau^* \times \Delta^1$ for some $\tau \geq \sigma$. In this case, we have $E_{\zeta_0}^r=0$ unless $\tau \in L_2 \backslash L_0$. When $\tau \in L_2 \backslash L_0$, we have $E_{\zeta_0}^r=\begin{cases} \quad \  D_r(\tau) & \text{If } \tau \in L_2 \backslash L_1 \\ \mathop{\oplus}\limits_{s \in B_{\tau}} D_{r-1}(s) & \text{If } \tau \in L_1 \backslash L_0 \\
	\end{cases}$.
	
	Therefore, we have:
	\begin{equation*}
		\mathop{\oplus}\limits_{\zeta_0 \subset \zeta}E_{\zeta_0}^r=\big(\mathop{\oplus}\limits_{\tau \geq \sigma}\mathop{\oplus}\limits_{s \in B_{\tau}}D_r(s)\big) \oplus \big(\mathop{\oplus}\limits_{\substack{\tau \geq \sigma \\ \tau \in L_2 \backslash L_1}} D_r(\tau )\big) \oplus \big(\mathop{\oplus}\limits_{\tau \geq \sigma}\mathop{\oplus}\limits_{s \in B_{\tau}}D_{r-1}(s)\big)
	\end{equation*}
	
	Since the sets $B_{\tau}$ are disjoint for different $\tau$ and $\mathop{\cup}\limits_{\tau \geq \sigma}B_{\tau}=A_{\sigma}$, we have $\mathop{\oplus}\limits_{\tau \geq \sigma}\mathop{\oplus}\limits_{s \in B_{\tau}}D_*(s)=\mathop{\oplus}\limits_{s \in A_{\sigma}}D_*(s)$. Furthermore, by definition we have $A_{\sigma}=\{\tau \in L_2 \backslash L_1 \ | \ \tau > \sigma\}$. Note that since $\sigma \in L_1$, we have $A_{\sigma}=\{\tau \in L_2 \backslash L_1 \ | \ \tau \geq \sigma\}$. Thus we have:
	\begin{equation*}
		\mathop{\oplus}\limits_{\zeta_0 \subset \zeta}E_{\zeta_0}^r=\big(\mathop{\oplus}\limits_{s \in A_{\sigma}}D_r(s)\big) \oplus \big(\mathop{\oplus}\limits_{\tau \in A_{\sigma}} D_r(\tau)\big) \oplus \big(\mathop{\oplus}\limits_{s \in A_{\sigma}}D_{r-1}(s)\big)=(C_{\sigma^* \times \Delta^1})_r
	\end{equation*}

	Therefore, the statement holds.
	
	Combining the arguments in the four cases, we see that the statement (1) in Claim \ref{Appclaim1} holds.
	
	For statement (2) in Claim \ref{Appclaim1}, since by definition, $F((\zeta',o_{std})-(\zeta,o_{std}))$ is the inclusion map of the corresponding modules, the statement follows directly.
	
	Therefore, we have proven the Claim \ref{Appclaim1} and thus $C$ is well-behaved. 

	(6) The condition 1(b) in Definition \ref{Defad} holds for every closed cell $\zeta \in \underline{L_0} \times I$.
	
	Let $p_{\zeta}:(C_{\zeta})_* \longrightarrow (C_{\zeta}/C_{\partial\zeta})_*$ be the projection map. By the definition of dual functor $T$ on chain complexes of $R$ modules, it is equivalent to check that the following chain map is a chain homotopy equivalence:
	\begin{equation}
		\label{Appcheck25}
		\psi_{\zeta,rel}^r:(C_{\zeta})_{|\zeta|-k-1-r}^* \longrightarrow (C_{\zeta}/C_{\partial\zeta})_r
	\end{equation}
	
	Where $\psi_{\zeta,rel}^r=p_{\zeta}\big(\psi_{\zeta,o_{std}}^{0,r}+(-1)^{r(|\zeta|-k-1-r)}(\psi_{\zeta,o_{std}}^{0,|\zeta|-k-1-r})^*\big)$.
	
	The proof is divided into four cases:
	
	(a) If $\zeta \in \underline{L_0} \times 1 \cup \underline{L_1} \times 0$.
	
	By definition, $C_{\zeta}$ is the zero chain complex. Therefore, $\psi_{\zeta,rel}^r$ is a chain homotopy equivalence, since the chain complexes on both sides are zero.
	
	(b) If $\zeta=\sigma^* \times \Delta^1$ for some simplex $\sigma \in L_2 \backslash L_1$.
	
	By the Claim \ref{Appclaim1} and the definition of $E_{\zeta}$, for every $r \in \mathbb{Z}$, we have:
	\begin{equation}
		\label{Appcheck26}
		(C_{\zeta})_r=\mathop{\oplus}\limits_{\tau \geq \sigma}D_r(\tau), (C_{\partial\zeta})_r=\mathop{\oplus}\limits_{\tau>\sigma}D_r(\tau)
	\end{equation}

	Moreover, let $\delta\{\theta\}_0^r[\sigma]$ be the map defined in Remark \ref{AppRemquadpair}. By definition, we have:
	\begin{equation}
		\label{Appcheck27}
		\begin{aligned}
			\psi_{\zeta,rel}^r&=p_{\zeta}\big(\psi_{\zeta,o_{std}}^{0,r}+(-1)^{r(|\zeta|-k-1-r)}(\psi_{\zeta,o_{std}}^{0,|\zeta|-k-1-r})^*\big) \\
			&=p_{\zeta}\big((-1)^{r_{\sigma}}\delta\{\theta\}_0^r[\sigma]+(-1)^{r_{\sigma}}(-1)^{r(l-|\sigma|-k-r)}(\delta\{\theta\}_0^{l-|\sigma|-k-r}[\sigma])^*\big)
		\end{aligned}
	\end{equation}
	
	Let $\delta\{\theta\}_{rel}^r[\sigma]=\delta\{\theta\}_0^r[\sigma]+(-1)^{r(l-|\sigma|-k-r)}(\delta\{\theta\}_0^{l-|\sigma|-k-r}[\sigma])^*$. Since $(D,\theta)$ is a $(l-k)$-dimensional Poincare quadratic chain complex, by Remark \ref{AppRemquadpair}, the following chain map is a chain homotopy equivalence:
	\begin{equation*}
		\begin{tikzcd}
			\mathop{\oplus}\limits_{\tau \geq \sigma}D_{l-k-|\sigma|-r}(\tau)^* \ar{rr}{\delta\{\theta\}_{rel}^r[\sigma]} & & \mathop{\oplus}\limits_{\tau \geq \sigma}D_{r}(\tau) \ar{rr}{\text{projection}} & & D_r(\sigma)
		\end{tikzcd}
	\end{equation*}

	By equation \ref{Appcheck26} and \ref{Appcheck27}, we can see that the chain map above is $(-1)^{r_{\sigma}} \\ \psi_{\zeta,rel}^r$. Therefore, $\psi_{\zeta,rel}$ is a chain homotopy equivalence.
	
	(c) If $\zeta=\sigma^* \times 0$ for some simplex $\sigma \in L$
	
	By the Claim \ref{Appclaim1} and the definition of $E_{\zeta}$, for every $r \in \mathbb{Z}$, we have:
	\begin{equation}
		\label{Appcheck28}
		(C_{\zeta})_r=\mathop{\oplus}\limits_{s \in A_{\sigma}}D_r(s), (C_{\partial\zeta})_r=\mathop{\oplus}\limits_{s \in A_{\sigma} \backslash B_{\sigma}}D_r(s)
	\end{equation}
	
	Moreover, let $\delta\psi_0^r[\sigma]$ be the map defined in Lemma \ref{Appquadpair}. By definition, we have:
	\begin{equation}
		\label{Appcheck29}
		\begin{aligned}
			\psi_{\zeta,rel}^r&=p_{\zeta}\big(\psi_{\zeta,o_{std}}^{0,r}+(-1)^{r(|\zeta|-k-1-r)}(\psi_{\zeta,o_{std}}^{0,|\zeta|-k-1-r})^*\big) \\
			&=p_{\zeta}\big((-1)^{r_{\sigma}}\delta\psi_0^r[\sigma]+(-1)^{r_{\sigma}}(-1)^{r(l-|\sigma|-k-1-r)}(\delta\psi_0^{l-|\sigma|-k-1-r}[\sigma])^*\big)
		\end{aligned}
	\end{equation}
	
	Let $\delta\psi_{rel}^r[\sigma]=\delta\psi_0^r[\sigma]+(-1)^{r(l-|\sigma|-k-1-r)}(\delta\psi_0^{l-|\sigma|-k-1-r}[\sigma])^*$. Since $(D,\theta)$ is a $(l-k)$-dimensional Poincare quadratic chain complex, by Lemma \ref{Appquadpair}, the following chain map is a chain homotopy equivalence:
	\begin{equation*}
		\begin{tikzcd}
			\mathop{\oplus}\limits_{s \in A_{\sigma}}D_{l-k-|\sigma|-1-r}(s)^* \ar{rr}{\delta\psi_{rel}^r[\sigma]} & & \mathop{\oplus}\limits_{s \in A_{\sigma}}D_{r}(s) \ar{rr}{\text{projection}} & & \mathop{\oplus}\limits_{s \in B_{\sigma}}D_r(s)
		\end{tikzcd}
	\end{equation*}
	
	By equation \ref{Appcheck28} and \ref{Appcheck29}, we can see that the chain map above is $(-1)^{r_{\sigma}} \\ \psi_{\zeta,rel}^r$. Therefore, $\psi_{\zeta,rel}$ is a chain homotopy equivalence.
	
	(d) If $\zeta=\sigma^* \times \Delta^1$ for some simplex $\sigma \in L$
	
	By the Claim \ref{Appclaim1} and the definition of $E_{\zeta}$, for every $r \in \mathbb{Z}$, we have:
	\begin{equation}
		\label{Appcheck30}
		\begin{aligned}
			(C_{\zeta})_r=\big(\mathop{\oplus}\limits_{s \in A_{\sigma}}D_r(s)\big) \oplus \big(\mathop{\oplus}\limits_{s \in A_{\sigma}}D_r(s)\big) \oplus \big(\mathop{\oplus}\limits_{s \in A_{\sigma}}D_{r-1}(s)\big) \\ (C_{\partial\zeta})_r=\big(\mathop{\oplus}\limits_{s \in A_{\sigma}}D_r(s)\big) \oplus \big(\mathop{\oplus}\limits_{s \in A_{\sigma}}D_r(s)\big) \oplus \big(\mathop{\oplus}\limits_{s \in A_{\sigma} \backslash B_{\sigma}}D_{r-1}(s)\big) \\
		\end{aligned}
	\end{equation}
	
	Moreover, let $\delta\psi_0^r[\sigma]$ be the map defined in Lemma \ref{Appquadpair}. By definition, we have:
	\begin{equation}
		\label{Appcheck31}
		\begin{aligned}
			\psi_{\zeta,rel}^r&=p_{\zeta}\big(\psi_{\zeta,o_{std}}^{0,r}+(-1)^{r(|\zeta|-k-1-r)}(\psi_{\zeta,o_{std}}^{0,|\zeta|-k-1-r})^*\big) \\
			&=(-1)^{r_{\sigma}+1}p_{\zeta}\bigg(\big(\delta\psi[\sigma] \otimes \omega_I\big)_0^r+(-1)^{r(l-|\sigma|-k-r)}\big((\delta\psi[\sigma] \otimes \omega_I)_0^{l-|\sigma|-k-r}\big)^*\bigg)
		\end{aligned}
	\end{equation}

	Let $r'=l-|\sigma|-k-r$. By Theorem \ref{tensorstr}, we have:
	\begin{equation}
		\label{Appcheck32}
		\big(\delta\psi[\sigma] \otimes \omega_I\big)_0^r=(-1)^r\delta\psi_0^r[\sigma] \otimes (\omega_I)_0^0+\delta\psi_0^{r-1}[\sigma] \otimes (\omega_I)_0^1+(-1)^{r'-1}T\delta\psi_1^{r-1}[\sigma] \otimes (\omega_I)_1^1
	\end{equation}

	Let $\sigma_{01}$ be the $1$-simplex in $\Delta^1$ and $\sigma_0,\sigma_1$ be the two $0$-simplices in $\Delta^1$. By Remark 15.82 in \cite{luecksurgery}, a choice of $\omega_I$ is as follows:
	\begin{equation}
		\label{Appcheck33}
		\begin{aligned}
			(\omega_I)_0^0: I^1 \longrightarrow I_0, (\omega_I)_0^0(\sigma_{01}^*)=\sigma_0 \qquad \quad & \\
			(\omega_I)_0^1:I^0 \longrightarrow I_1, (\omega_I)_0^1(\sigma_0^*)=0, (\omega_I)_0^1(\sigma_1^*)=\sigma_{01} & \\
			(\omega_I)_1^1:I^1 \longrightarrow I_1, (\omega_I)_1^1(\sigma_{01}^*)=-\sigma_{01} \qquad \ \ & \\
		\end{aligned}
	\end{equation}
	
	Let $F_r=\mathop{\oplus}\limits_{s \in A_{\sigma}}D_r(s)$. Denote the differential of the chain complex $F_*$ to be $d_F$. Choose any element $(x,y,z) \in (C_{\zeta})_{r'}^*=\big(F_r \otimes \mathbb{Z}[\sigma_0]\big) \oplus \big(F_r \otimes \mathbb{Z}[\sigma_1]\big) \oplus \big(F_{r-1} \otimes \mathbb{Z}[\sigma_{01}]\big)$, by equation \ref{Appcheck32} and \ref{Appcheck33}, we have:
	\begin{equation}
		\label{Appcheck34}
		\begin{aligned}
			& \quad \ \big(\delta\psi[\sigma] \otimes \omega_I\big)_0^r(x,y,z)  \\
			&=\big((-1)^r\delta\psi_0^r[\sigma](z),0,\delta\psi_0^{r-1}[\sigma](y)+(-1)^{r'}T\delta\psi_1^{r-1}[\sigma](z)\big) \\
			&=\big((-1)^r\delta\psi_0^r[\sigma](z),0,\delta\psi_0^{r-1}[\sigma](y)+(-1)^{r'+(r'-1)(r-1)}(\delta\psi_1^{r'-1}[\sigma])^*(z)\big) 
		\end{aligned}
	\end{equation}

	From that we can also compute the dual of $\big(\delta\psi[\sigma] \otimes \omega_I\big)_0^{r'}$, it is given as follows:
	\begin{equation}
		\label{Appcheck35}
		\begin{aligned}
		&\quad \; \big((\delta\psi[\sigma] \otimes \omega_I)_0^{r'}\big)^*(x,y,z) \\
		&=\big(0,(\delta\psi_0^{r'-1}[\sigma])^*(z),(-1)^{r'}(\delta\psi_0^{r'}[\sigma])^*(x)+(-1)^{r+(r-1)(r'-1)}\delta\psi_1^{r-1}[\sigma](z)\big)
		\end{aligned}
	\end{equation}
	
	Let $pr_3^r: \mathop{\oplus}\limits_{s \in A_{\sigma}}D_r(s) \longrightarrow \mathop{\oplus}\limits_{s \in B_{\sigma}}D_r(s)$ be the projection map. 
	
	By definition of $p_{\zeta}$, combined with equation \ref{Appcheck31}, \ref{Appcheck34} and \ref{Appcheck35}, we have:
	\begin{equation}
		\label{Appcheck36}
		\begin{aligned}
			\psi_{\zeta,rel}^r(x,y,z)= \ & pr_3^{r-1} \bigg( \delta\psi_0^{r-1}[\sigma](y)+(-1)^{r'+(r'-1)(r-1)}(\delta\psi_1^{r'-1}[\sigma])^*(z) \\
			&+(-1)^{r'+rr'}(\delta\psi_0^{r'}[\sigma])^*(x)+(-1)^{r+(r-1)(r'-1)+rr'}\delta\psi_1^{r-1}[\sigma](z)\bigg)\\
		\end{aligned}
	\end{equation}
	
	By Lemma \ref{Appquadpair}, $(\delta\psi[\sigma],\psi[\sigma])$ gives the quadratic structure of a pair, we have:
	\begin{equation}
		\label{Appcheck37}
		\begin{aligned}
			0= \ & d_F\delta\psi_0^r[\sigma]+(-1)^{r-1}\delta\psi_0^{r-1}[\sigma]d_F^*+(-1)^{r+r'-2}\delta\psi_1^{r-1}[\sigma] \\
			&+(-1)^{r+r'-1+(r-1)(r'-1)}(\delta\psi_1^{r'-1}[\sigma])^*+(-1)^{r+r'}i_{r-1}[\sigma] \psi_{0}^{r-1}[\sigma] i_{r'-1}[\sigma]^* \\
		\end{aligned}
	\end{equation}

	By definition, we have $pr_3^{r-1} \circ i_{r-1}[\sigma]=0$. Substituting equation \ref{Appcheck37} into equation \ref{Appcheck36}, we get:
	\begin{equation}
		\label{Appcheck38}
		\begin{aligned}
			\psi_{\zeta,rel}^r(x,y,z)= \ & pr_3^{r-1} \bigg( \delta\psi_0^{r-1}[\sigma](y)+(-1)^{r}d_F\delta\psi_0^{r}[\sigma](z) \\
			&+(-1)^{r'+rr'}(\delta\psi_0^{r'}[\sigma])^*(x)-\delta\psi_0^{r-1}[\sigma]d_F^*(z)\bigg)\\
		\end{aligned}
	\end{equation}

	Let $\psi_{\zeta,sign}^r=(-1)^{\frac{(r-1)(r-2)}{2}}\psi_{\zeta,rel}^r$ and $\delta\psi_{rel}^r[\sigma]=\delta\psi_0^r[\sigma]+(-1)^{r(r'-1)}\delta\psi_0^{r'-1}[\sigma]^*$. Let $\Psi_{sign}^r:(C_{\zeta})_{r'}^* \longrightarrow (C_{\zeta}/C_{\partial\zeta})_{r-1}$ be the map given as follows:
	\begin{equation}
		\label{Appcheck39}
		\Psi_{sign}^r(x,y,z)=(-1)^{\frac{(r-1)(r-2)}{2}}pr_3^{r-1}\delta\psi_{rel}^{r-1}[\sigma](x)
	\end{equation}

	We claim that $\psi_{\zeta,sign}$ is chain homotopic to $\Psi_{sign}$. To prove the claim, we write the chain homotopy and check that it is the chain homotopy between the two chain maps.
	
	For every $r \in \mathbb{Z}$, let $B^r: (C_{\zeta})_{l-k-|\sigma|-r}^* \longrightarrow (C_{\zeta}/C_{\partial\zeta})_r$ be the map given by $B^r(x,y,z)=(-1)^{\frac{r(r-1)}{2}}pr_3^{r}\delta\psi_0^{r}[\sigma](z)$. Notice that by definition, we have $d_{C_{\zeta}/C_{\partial\zeta}}=-d_F$. Therefore, we have:
	\begin{equation*}
		d_{C_{\zeta}/C_{\partial\zeta}}B^r(x,y,z)=(-1)^{\frac{r(r-1)}{2}+1}d_Fpr_3^{r}\delta\psi_0^r[\sigma](z)=(-1)^{\frac{r(r-1)}{2}+1}pr_3^{r-1}d_F\delta\psi_0^r[\sigma](z)
	\end{equation*}
	\begin{equation*}
		B^{r-1}d_{C_{\zeta}}(x,y,z)=(-1)^{\frac{(r-1)(r-2)}{2}}pr_3^{r-1}\delta\psi_0^r[\sigma](-d^*_Fz-x+y)
	\end{equation*}
	
	Combining with equation \ref{Appcheck38} and \ref{Appcheck39}, we get:
	\begin{equation*}
		\begin{aligned}
			(d_{C_{\zeta}/C_{\partial\zeta}}B^r+B^{r-1}d_{C_{\zeta}})(x,y,z)&=(-1)^{\frac{(r-1)(r-2)}{2}}\psi_{\zeta,rel}^r(x,y,z)-\Psi_{sign}^r(x,y,z) \\
			&=\psi_{\zeta,sign}^r(x,y,z)-\Psi_{sign}^r(x,y,z) \\
		\end{aligned}
	\end{equation*}
	
	Therefore, the claim holds. Since $(D,\theta)$ is a $(l-k)$-dimensional Poincare quadratic chain complex, by Lemma \ref{Appquadpair}, $pr_3^r\delta\psi_{rel}^r[\sigma]$ is a chain homotopy equivalence. By equation \ref{Appcheck39}, $\Psi_{sign}^r$ is the composition of the following maps:
	\begin{equation*}
		\begin{tikzcd}
			(C_{\zeta})_{r'}^*=F_{r'}^* \oplus F_{r'}^* \oplus F_{r'-1}^* \ar{rr}{\text{projection}} & & F_{r'}^* \ar{rrrr}{(-1)^{\frac{(r-1)(r-2)}{2}}pr_3^{r-1}\delta\psi_{rel}^{r-1}[\sigma]} & & & & (C_{\zeta}/C_{\partial\zeta})_{r-1}
		\end{tikzcd}
	\end{equation*}
	
	Since $I$ is a contractible chain complex, the first map in the diagram above is also a chain homotopy equivalence. Therefore, $\Psi_{sign}$ is a chain homotopy equivalence. Since $\psi_{\zeta,sign}$ is chain homotopic to $\Psi_{sign}$, $\psi_{\zeta,sign}$ is also a chain homotopy equivalence. Note that $\psi_{\zeta,sign}^r=(-1)^{\frac{(r-1)(r-2)}{2}}\psi_{\zeta,rel}^r$. Thus $\psi_{\zeta,rel}$ is also a chain homotopy equivalence, which is exactly what we want to prove.
	
	Summarizing all the proofs above, we have checked all the conditions listed in Definition \ref{Defad}.
	
	Then we need to check that the restriction of $F$ to $(\underline{L_1} \times \Delta^1,\underline{L_2} \times \Delta^1)$, denoted by $F_{res}$, represents $\mathcal{K}^*x$.
	
	By Theorem \ref{Lcohomologywithadcord} and the definition of $\widecheck{D}$, $x$ is represented by the following $(\underline{L_1},\underline{L_2})$-ad $F'$, denoted by $F'(\xi,o)=(C_{\xi}',\varphi_{\xi,o})$:
	
	On objects, $F'$ is given by:
	
	For any simplex $\sigma \in L_2 \backslash L_1$: 
	\begin{equation}
		\label{Appcheck40}
		\begin{aligned}
			&(C_{\sigma^*}')_r=[\widecheck{D}_r][\sigma^*] \\
			&d_{C_{\sigma^*}'}=[d_{\widecheck{D}}][\sigma^*] \\
			&\varphi_{\sigma^*,o}^{u,r}=(-1)^{\frac{|\sigma^*|(|\sigma^*|-1)}{2}}(-1)^{|\sigma^*|r}sgn(o)\mathop{\oplus}\limits_{\tau^* \leq \sigma^*} \widecheck{\theta}_u^r(\tau^*,\sigma^*) \\
		\end{aligned}
	\end{equation}
	
	For any simplex $\sigma \notin L_2$:
	\begin{equation*}
			(C_{\sigma^*}',\varphi_{\sigma^*,o})=\emptyset_{|\sigma^*|-k}
	\end{equation*}

	On morphisms, $F'$ is given by:
	
	For simplex $\sigma,\tau \in L_2 \backslash L_1$ with $\tau \geq \sigma$, $F'((\tau^*,o)-(\sigma^*,o'))$ is given by the following inclusion map:
	\begin{equation*}
		[\widecheck{D}_r][\tau^*]=\mathop{\oplus}\limits_{\kappa^* \leq \tau^*}\widecheck{D}_r(\kappa^*) \hookrightarrow [\widecheck{D}_r][\sigma^*]=\mathop{\oplus}\limits_{\kappa^* \leq \sigma^*}\widecheck{D}_r(\kappa^*)
	\end{equation*}

	Else, we define $F'((\tau^*,o)-(\sigma^*,o'))$ to be $0$.
	
	By Definition \ref{DefK} of $\mathcal{K}$ and Definition 3.7 in \cite{LauresBallL}, $\mathcal{K}^*x$ is represented by the following $(\underline{L_1} \times \Delta^1,\underline{L_2} \times \Delta^1)$-ad $F''$:
	\begin{equation*}
		F''(\xi \times \Delta^1,o \times o_{std})=(C_{\xi}',(-1)^{k+|\xi|}\varphi_{\xi,o})
	\end{equation*}
	
	By equation \ref{Eqlocaldualchaincpx} and \ref{Eqlocaldualboundarymap}, for $\sigma \in L_2 \backslash L_1$, we have:
	\begin{equation*}
		\begin{aligned}
			[\widecheck{D}_r][\sigma^*]=\mathop{\oplus}\limits_{\substack{\tau^* \leq \sigma^* \\ \tau^* \in \underline{L_1} \backslash \underline{L_2} }}\widecheck{D}_r(\tau^*) =\mathop{\oplus}\limits_{\substack{\tau \geq \sigma \\ \tau \in L_2 \backslash L_1}}D_r(\tau)=[D_r][\sigma]
		\end{aligned}
	\end{equation*}
	\begin{equation*}
		[d_{\widecheck{D}}][\sigma^*]=\mathop{\boxplus}\limits_{\substack{\tau^* \leq \sigma^* \\ \tau^* \in \underline{L_1} \backslash \underline{L_2} }}\mathop{\oplus}\limits_{\substack{\tau'^* \leq \sigma^* \\ \tau'^* \in \underline{L_1} \backslash \underline{L_2} }}d_{\widecheck{D}}(\tau'^*,\tau^*)=\mathop{\boxplus}\limits_{\substack{\tau \geq \sigma \\ \tau \in L_2 \backslash L_1}}\mathop{\oplus}\limits_{\substack{\tau' \geq \sigma \\ \tau' \in L_2 \backslash L_1}}d_{D}(\tau',\tau)=[d_D][\sigma]
	\end{equation*}
	\begin{equation*}
		\text{(Here we use the fact that $D_r(\tau)=0$ for $\tau \notin L_2$.)}
	\end{equation*}
	
	Thus we have $C'_{\xi}=C_{\xi \times \Delta^1}$. Moreover, from the definition we see that $F_{res}$ and $F''$ agree on morphisms. Both $F_{res}$ and $F''$ are $0$ on objects of the form $(\sigma^* \times \Delta^1 ,o)$ with $\sigma \notin L_2$. Hnece it only leaves us to check that the following equation holds for all $u \in \mathbb{N},r \in \mathbb{Z},\sigma \in L_2 \backslash L_1$:
	\begin{equation}
		\label{Appcheck41}
		(-1)^{k+|\sigma^*|}\varphi_{\sigma^*,o}^{u,r}=\psi_{\sigma^* \times \Delta^1,o}^{u,r}
	\end{equation}
	
	Now we begin to compute the term on the left hand side, by the definition in \ref{Appcheck40}, we have:
	\begin{equation*}
		\varphi_{\sigma^*,o}^{u,r}=(-1)^{\frac{|\sigma^*|(|\sigma^*|-1)}{2}}(-1)^{|\sigma^*|r}sgn(o)\mathop{\oplus}\limits_{\tau^* \leq \sigma^*} \widecheck{\theta}_u^r(\tau^*,\sigma^*)
	\end{equation*}

	By equation \ref{Eqlocaldualquadraticstr}, we have:
	\begin{equation*}
		\widecheck{\theta}_u^r(\tau^*,\sigma^*)=(-1)^{l-k+J_{\sigma}^{all}+l|\sigma|+lr+\frac{l(l-1)}{2}}\theta_u^r(\tau,\sigma)
	\end{equation*}
	
	Therefore, we have:
	\begin{equation}
		\label{Appcheck42}
		\begin{aligned}
			(-1)^{k+|\sigma^*|}\varphi_{\sigma^*,o}^{u,r}= \ &(-1)^{k+|\sigma^*|}(-1)^{\frac{|\sigma^*|(|\sigma^*|-1)}{2}}(-1)^{|\sigma^*|r}(-1)^{l-k+J_{\sigma}^{all}+l|\sigma|+lr+\frac{l(l-1)}{2}} \\
			&sgn(o)\mathop{\oplus}\limits_{\tau \geq \sigma} \theta_u^r(\tau,\sigma) \\
		\end{aligned}
	\end{equation}
	
	By definition we have:
	\begin{equation}
		\label{Appcheck43}
		\psi_{\sigma^* \times \Delta^1,o}^{u,r}=(-1)^{\frac{|\sigma|(|\sigma|-1)}{2}+J_{\sigma}^{all}}(-1)^{|\sigma|r}sgn(o)\mathop{\oplus}\limits_{\tau \geq \sigma} \theta_u^r(\tau,\sigma)
	\end{equation}

	Comparing equation \ref{Appcheck42} and \ref{Appcheck43} with \ref{Appcheck41}, it only leaves us to prove that:
	\begin{equation*}
		\begin{aligned}
			&\quad \; (-1)^{k+|\sigma^*|}(-1)^{\frac{|\sigma^*|(|\sigma^*|-1)}{2}}(-1)^{|\sigma^*|r}(-1)^{l-k+J_{\sigma}^{all}+l|\sigma|+lr+\frac{l(l-1)}{2}} \\
			&=(-1)^{\frac{|\sigma|(|\sigma|-1)}{2}+J_{\sigma}^{all}}(-1)^{|\sigma|r}
		\end{aligned}
	\end{equation*}
	
	Which can be done by direct computations.
\end{proof}

An important Corollary is:

\begin{Corollary}
	\label{computationpartial}
	Let $n \in \mathbb{Z}$ and $R$ be a ring with involution. Let $x \in H_n(L_2,L_1,\underline{L}(R))$. Let $g:(\underline{L_1},\underline{L_2}) \longrightarrow (\mathbb{L}_{n-l}(M^h(R)),*)$ be the $\Delta$-set map representing $x$ and let $(D,\theta)$ be the $n$-dimensional Poincare quadratic chain complex in $M^h(R)_*(L_2)$ obtained from $g$ as in Theorem \ref{computationdelta}. Define $(DL,\theta L)$ to be the following $(n-1)$-dimensional quasi quadratic chain complex in $M^h(R)_*(L)$:
	
	For $\sigma,\tau \in L$ and $u \in \mathbb{N},r \in \mathbb{Z}$,
	\begin{equation*}
		DL_r(\sigma)=\mathop{\oplus}\limits_{s \in B_{\sigma}}D_r(s)
	\end{equation*}
	\begin{equation*}
		d_{DL,r}(\tau,\sigma)=\mathop{\boxplus}\limits_{s \in B_{\sigma}}\mathop{\oplus}\limits_{s' \in B_{\tau}}d_{D,r}(s',s):DL_r(\sigma) \longrightarrow DL_{r-1}(\tau)
	\end{equation*}
	\begin{equation*}
		\begin{aligned}
			\theta L_u^r(\tau,\sigma):DL^{n-1-u-r}(\sigma)=\mathop{\oplus}\limits_{s \in A_{\sigma}}D_{n-1-u-r-|\sigma|}(s)^* \longrightarrow DL_r(\tau)=\mathop{\oplus}\limits_{s \in B_{\tau}}D_r(s) \\
		\end{aligned}
	\end{equation*}
	\begin{equation*}
		\theta L_u^r(\tau,\sigma)=(-1)^{n+|\sigma|+r+1}\big( \mathop{\boxplus}\limits_{s \in B_{\sigma}}\mathop{\oplus}\limits_{s' \in B_{\tau}}\theta_u^r(s',s) \big) \circ \mho_{\sigma}^{n-u-r}
	\end{equation*}
	
	Then:
	
	$(1)$ $(DL,\theta L)$ is Poincare quadratic.
	
	$(2)$ $\partial x \in H_{n-1}(L_1,L_0,\underline{L}(\mathbb{Z}))$ is given by $(DL,\theta L)$ via the isomorphisms:
	\begin{equation*}
		L_{n-1}(M^h(R)_*(L)) \cong H_{n-1}(L,\underline{L}(\mathbb{Z})) \cong H_{n-1}(L_1,L_0,\underline{L}(\mathbb{Z}))
	\end{equation*}
\end{Corollary}

Before we prove the Corollary, we prove the following Lemma first:

\begin{Lemma}
	\label{rearannge}
	The following construction gives a functor:
	\begin{equation*}
		A^{part}:M^h(R)_*(L_2) \longrightarrow M^h(R)_*(L), \ A^{part}C'(\sigma)=C'(B_{\sigma}) \text{ for any } \sigma \in L
	\end{equation*}
	\begin{equation*}
		A^{part}h(\sigma,\tau)=\mathop{\boxplus}\limits_{s \in B_{\sigma}}\mathop{\oplus}\limits_{s' \in B_{\tau}}h(s',s) \text{ for any morphism } h:C' \longrightarrow C''
	\end{equation*}
\end{Lemma}
\begin{proof}
	Notice first that if $A^{part}h(\tau,\sigma) \neq 0$, then there are simplices $s \in B_{\sigma}$ and $s' \in B_{\tau}$, such that $h(s',s) \neq 0$. Since $h$ is a morphism in $M^h(R)_*(L_2)$, $h(s',s) \neq 0$ implies $s \leq s'$. Since $s \in B_{\sigma}, s' \in B_{\tau}$, we have $(s \cap L)_0=\sigma,(s' \cap L)_0=\tau$, combining with $s \leq s'$ we can deduce that $\sigma \leq \tau$. In summary, we have proved that $A^{part}h(\tau,\sigma)=0$ unless $\sigma \leq \tau$. By definition, $A^{part}h$ is thererfore a morphism in $M^h(R)_*(L)$. The functor is well defined.
		
	It is easy to check that $A^{part}$ maps the identity morphism to the identity morphism. Hence it only leaves us to check that the following equation holds for all morphisms $h_1,h_2$ and all simplices $\sigma,\tau \in L$ with $\sigma \leq \tau$: 
	\begin{equation}
		\label{Appcheck18}
		(A^{part}h_1 \circ A^{part}h_2)(\tau,\sigma)=\big( A^{part}(h_1 \circ h_2) \big)(\tau,\sigma)
	\end{equation}
	
	For any $s \in B_{\sigma}, r \in \mathbb{Z}$ and $z_s \in C'_r(s)$, we have:
	\begin{equation*}
		\begin{aligned}
			\big( A^{part}(h_1 \circ h_2) \big)(\tau,\sigma)(z_s)&=\mathop{\oplus}\limits_{s' \in B_{\tau}} (h_1 \circ h_2)(s',s)(z_s) \\
			&=\mathop{\oplus}\limits_{s' \in B_{\tau}} \sum \limits_{s \leq s'' \leq s'}h_1(s',s'') h_2(s'',s)(z_s) \\
		\end{aligned}
	\end{equation*}
	
	For any $s \in B_{\sigma},s' \in B_{\tau}$ and $s \leq s'' \leq s'$, denote $s''_0=(s'' \cap L)_0$. By Remark \ref{RemAB}, we have $\sigma \leq s''_0 \leq \tau$. Since the sets $B_{\kappa}$ are disjoint for different $\kappa$, we get:
	\begin{equation}
		\label{Appcheck44}
		\begin{aligned}
			\big( A^{part}(h_1 \circ h_2) \big)(\tau,\sigma)(z_s)
			&=\mathop{\oplus}\limits_{s' \in B_{\tau}} \sum \limits_{\sigma \leq \kappa \leq \tau}\sum \limits_{\substack{s'' \in B_{\kappa} \\ s \leq s'' \leq s'}}h_1(s',s'') h_2(s'',s)(z_s) \\
			&=\mathop{\oplus}\limits_{s' \in B_{\tau}} \sum \limits_{\sigma \leq \kappa \leq \tau}\sum \limits_{s'' \in B_{\kappa}}h_1(s',s'') h_2(s'',s)(z_s) \\
		\end{aligned}
	\end{equation}

	Now we have
	\begin{equation}
		\label{Appcheck45}
		\begin{aligned}
			(A^{part}h_1 \circ A^{part}h_2)(\tau,\sigma)(z_s)&=\sum \limits_{\sigma \leq \kappa \leq \tau}A^{part}h_1(\tau,\kappa) A^{part}h_2(\kappa,\sigma)(z_s) \\
			&=\sum \limits_{\sigma \leq \kappa \leq \tau}A^{part}h_1(\tau,\kappa)  \big(\mathop{\oplus} \limits_{s'' \in B_{\kappa}}h_2(s'',s)(z_s)\big)\\
			&=\sum \limits_{\sigma \leq \kappa \leq \tau}\sum\limits_{s'' \in B_{\kappa}}A^{part}h_1(\tau,\kappa)  h_2(s'',s)(z_s)\\
			&=\sum \limits_{\sigma \leq \kappa \leq \tau}\sum\limits_{s'' \in B_{\kappa}}\mathop{\oplus}\limits_{s' \in B_{\tau}}h_1(s',s'')h_2(s'',s)(z_s)\\
		\end{aligned}
	\end{equation}
	
	Comparing equation \ref{Appcheck44} and \ref{Appcheck45}, we get that the equation \ref{Appcheck18} holds. Thus $A^{part}$ is a functor.
\end{proof}

Now we turn back to the proof of Corollary \ref{computationpartial}:
\begin{proof}[Proof of Corollary \ref{computationpartial}]
	\
	
	We will check the statements step by step. 
	
	(1) $DL$ is a chain complex in $M^h(R)_*(L)$.
	
	For any $r \in \mathbb{Z}$, notice that $d_{DL,r}=A^{part}d_{D,r}$. By definition, $d_{DL,r}$ is a morphism in $M^h(R)_*(L)$.
	
	Then it remains to check that $d_{DL,r-1} \circ d_{DL,r}=0$. This follows from the fact that $D$ itself is a chain complex together with Lemma \ref{rearannge}.
	
	(2) For any $u \in \mathbb{N},r \in \mathbb{Z}$, $\theta L_u^r$ is a morphism in $M^h(R)_*(L)$.
	
	Note that by definition, $\theta L_u^r(\tau,\sigma) \neq 0$ implies that $A^{part}(\theta_u^r)(\tau,\sigma)=\mathop{\boxplus}\limits_{s \in B_{\sigma}}\mathop{\oplus}\limits_{s' \in B_{\tau}}\theta_u^r(s',s) \neq 0$. Since $\theta_u^r$ is a morphism in $M^h(R)_*(L_2)$, it follows from the definition of $A^{part}$ that $\theta L_u^r$ is a morphism in $M^h(R)_*(L)$.
	
	(3) $(DL,\theta L)$ is a quadratic chain complex.
	
	For any $u \in \mathbb{N}, r \in \mathbb{Z}$, let $n'=n-1-u-r$. We need to check that the following equation holds for any pair of simplices $\sigma,\tau \in L$ with $\sigma \leq \tau$:
	\begin{equation}
		\label{Appcheck19}
		\begin{aligned}
			0= \ & (d_{DL,r+1} \circ \theta L_u^{r+1})(\tau,\sigma)-(-1)^{n-1-u}(\theta L_u^r \circ d_{DL^{-*},n'-1})(\tau,\sigma) \\
			& +(-1)^{n-u-2} \theta L_{u+1}^r(\tau,\sigma)+(-1)^{n-1}T\theta L_{u+1}^r(\tau,\sigma) \\
		\end{aligned}
	\end{equation}
	
	Let $r'=n-1-u-r-|\sigma|$. Choose any $s \in A_{\sigma}$ and any $z \in D_{r'-1}(s)^*$. Consider its image under each term on the right hand side of equation \ref{Appcheck19}. We will compute them separately. 
	
	We introduce some notations first. Let $s_0=(s \cap L)_0,s_1=(s \cap L)_1$. Denote $V_1$ to be set of the vertices of $s_1$. For any vertex $v \in V_1$, denote $\sigma_v=v*\sigma$. For any $\eta \leq s$, let $\iota_{\eta}$ be the map defined in \ref{Defiotaeta}.
	
	Then we can compute the first term:
	\begin{equation*}
		\begin{aligned}
			&\quad (d_{DL,r+1} \circ \theta L_u^{r+1})(\tau,\sigma)(z) \\
			&=\sum\limits_{\sigma \leq \kappa \leq \tau}d_{DL,r+1}(\tau,\kappa) \theta L_u^{r+1}(\kappa,\sigma)(z) \\
			& \quad (\text{By definition of } \theta L \text{ in Corollary } \ref{computationpartial}) \\
			&=(-1)^{n+|\sigma|+r+2}\sum\limits_{\sigma \leq \kappa \leq \tau}d_{DL,r+1}(\tau,\kappa) \big( \mathop{\boxplus}\limits_{s' \in B_{\sigma}}\mathop{\oplus}\limits_{s'' \in B_{\kappa}}\theta_u^{r+1}(s'',s') \big)\mho_{\sigma}^{n'}(z) \\
		\end{aligned}
	\end{equation*}
	
	By definition of $d_{DL}$ in Corollary \ref{computationpartial} and definition of $A^{part}$ in Lemma \ref{rearannge}, we have:
	\begin{equation}
		\label{Appcheck46}
		\begin{aligned}
			&\quad \; (d_{DL,r+1} \circ \theta L_u^{r+1})(\tau,\sigma)(z) \\
			&=(-1)^{n+|\sigma|+r+2}\sum\limits_{\sigma \leq \kappa \leq \tau}A^{part}d_{D,r+1}(\tau,\kappa) A^{part}\theta_u^{r+1} (\kappa,\sigma)\mho_{\sigma}^{n'}(z) \\
			& \quad (\text{By Lemma } \ref{rearannge}) \\
			&=(-1)^{n+|\sigma|+r+2}A^{part}(d_{D,r+1} \circ \theta_u^{r+1})(\tau,\sigma)\mho_{\sigma}^{n'}(z)
		\end{aligned}
	\end{equation}
	
	For the second term, we claim that:
	\begin{equation}
		\label{Appcheck47}
		(\theta L_u^r \circ d_{DL^{-*},n'-1})(\tau,\sigma)(z)=(-1)^{n+|\sigma|+r+1}A^{part}( \theta_u^{r} \circ d_{D^{-*},n'})(\tau,\sigma)\mho_{\sigma}^{n'}(z)
	\end{equation}
	
	To prove the claim, we compute the term on the left hand side first, we have:
	\begin{equation*}
		(\theta L_u^r \circ d_{DL^{-*},n'-1})(\tau,\sigma)(z)=\sum\limits_{\sigma \leq \kappa \leq \tau}\theta L_u^r(\tau,\kappa) d_{DL^{-*},n'-1}(\kappa,\sigma)(z)
	\end{equation*}
	
	Let $L^*(\sigma,s)=\{\kappa \in L^*(\sigma) \ | \ \kappa \leq s\}$. By definition, we have:
	\begin{equation*}
		d_{DL^{-*},n'-1}(\kappa,\sigma)(z)=\begin{cases} \mathop{\oplus}\limits_{s' \in A_{\kappa}}(-1)^{n'-1+|\sigma|}d_D(s,s')^*(z) & \text{If } \kappa=\sigma \\ \qquad \quad \ (-1)^{n'-1+n_{\sigma}^{\kappa}}z & \text{If } \kappa \in L^*(\sigma,s) \\ \qquad \qquad \quad \ \ \ 0 & \text{else} \end{cases}
	\end{equation*}

	Thus:
	\begin{equation}
		\label{Appcheck48}
		\begin{aligned}
			(\theta L_u^r \circ d_{DL^{-*},n'-1})(\tau,\sigma)(z)= \ &\theta L_u^r(\tau,\sigma) \big( \mathop{\oplus}\limits_{s' \in A_{\kappa}}(-1)^{n'-1+|\sigma|}d_D(s,s')^*(z) \big) \\
			&+\sum\limits_{\substack{\kappa \in L^*(\sigma,s) \\ \kappa \leq \tau}}(-1)^{n'-1+n_{\sigma}^{\kappa}}\theta L_u^r(\tau,\kappa)(z)
		\end{aligned}
	\end{equation}

	Then we compute the term the right hand side of equation \ref{Appcheck47}. By equation \ref{expressmho}, we have $\mho_{\sigma}^{n'}(z)=\mathop{\oplus}\limits_{v \in V_1}\iota_{\sigma_v}(z)$. Therefore, we have:
	\begin{equation}
		\label{Appcheck49}
		\begin{aligned}
			& \quad \ \; A^{part}( \theta_u^{r} \circ d_{D^{-*},n'})(\tau,\sigma)\mho_{\sigma}^{n'}(z) \\
			&=\sum\limits_{v \in V_1}A^{part}( \theta_u^{r} \circ d_{D^{-*},n'})(\tau,\sigma)\iota_{\sigma_v}(z) \\
			&=\sum\limits_{v \in V_1}\sum\limits_{\sigma \leq \kappa \leq \tau}A^{part}\theta_u^{r}(\tau,\kappa)A^{part}d_{D^{-*},n'}(\kappa,\sigma)\iota_{\sigma_v}(z) \\
		\end{aligned}
	\end{equation}
	
	By definition of $A^{part}$ in Lemma \ref{rearannge}, we have:
	\begin{equation}
		\label{Appcheck50}
		A^{part}d_{D^{-*},n'}(\kappa,\sigma)\iota_{\sigma_v}(z)=\mathop{\oplus}\limits_{s' \in B_{\kappa}}d_{D^{-*},n'}(s',\sigma_v)\iota_{\sigma_v}(z)
	\end{equation}

	By definition of $d_{D^{-*}}$, we have:
	\begin{equation*}
		d_{D^{-*},n'}(s',\sigma_v)\iota_{\sigma_v}(z)=\begin{cases} \mathop{\oplus}\limits_{s'' \geq s'}(-1)^{n'+|\sigma_v|}d_D(s,s'')^*z & \text{If } s'=\sigma_v \\ \qquad \; (-1)^{n'+n_{\sigma_v}^{s'}}\iota_{s'}(z) & \text{If } s' \in L_2^*(\sigma_v) \text{ and } s' \leq s\\ \qquad \qquad \quad \ 0 & \text{else} \end{cases}
	\end{equation*}

	Note that if $s'=\sigma_v \in B_{\kappa}$, then $\kappa=(s' \cap L)_0=\sigma$. If $s' \in L_2^*(\sigma_v)$ and $s' \leq s,s' \in B_{\kappa}$, then there are two possibilities:
	
	(1) $s'=\kappa * v$ with $\kappa \in L^*(\sigma,s)$.
	
	(2) $s'=\sigma_v*v'$ with $V_1 \ni v' \neq v$, in this case $\kappa=(s' \cap L)_0=\sigma$.
	
	Thus if $\kappa=\sigma$, we have:
	\begin{equation}
		\label{Appcheck51}
		\begin{aligned}
			& \quad \; \sum\limits_{v \in V_1}\mathop{\oplus}\limits_{s' \in B_{\kappa}}d_{D^{-*},n'}(s',\sigma_v)\iota_{\sigma_v}(z) \\
			&=\sum\limits_{v \in V_1}\mathop{\oplus}\limits_{s'' \geq \sigma_v}(-1)^{n'+|\sigma_v|}d_D(s,s'')^*(z)+\sum\limits_{v \in V_1}\sum\limits_{\substack{v' \in V_1 \\ v' \neq v}}(-1)^{n'+n_{\sigma_v}^{\sigma_v *v'}}\iota_{\sigma_v *v'}(z)
		\end{aligned}
	\end{equation}

	Since $\sigma_v*v'=\sigma_{v'}*v$ and $|n_{\sigma_v}^{\sigma_v *v'}-n_{\sigma_{v'}}^{\sigma_v' *v}|=1$, we get that:
	\begin{equation}
		\label{Appcheck52}
		\sum\limits_{v \in V_1}\sum\limits_{\substack{v' \in V_1 \\ v' \neq v}}(-1)^{n'+n_{\sigma_v}^{\sigma_v *v'}}\iota_{\sigma_v *v'}(z)=0
	\end{equation}
	
	If $\kappa \in L^*(\sigma,s)$ with $\kappa \leq \tau$, then $n_{\sigma_v}^{s'}=n_{\sigma_v}^{\kappa *v}=n_{\sigma}^{\kappa}$, thus:
	
	\begin{equation}
		\label{Appcheck53}
		\begin{aligned}
			\sum\limits_{v \in V_1}\mathop{\oplus}\limits_{s' \in B_{\kappa}}d_{D^{-*},n'}(s',\sigma_v)\iota_{\sigma_v}(z)=\sum\limits_{v \in V_1}(-1)^{n'+n_{\sigma}^{\kappa}}\iota_{\kappa * v}(z) \\
		\end{aligned}
	\end{equation}

	Substituting the equation \ref{Appcheck50}, \ref{Appcheck51}, \ref{Appcheck52} and \ref{Appcheck53} into equation \ref{Appcheck49}, we get:
	
	\begin{equation}
		\label{Appcheck54}
		\begin{aligned}
			& \quad \ \ \; A^{part}( \theta_u^{r} \circ d_{D^{-*},n'})(\tau,\sigma)\mho_{\sigma}^{n'}(z) \\
			&=\sum\limits_{\sigma \leq \kappa \leq \tau}A^{part}\theta_u^{r}(\tau,\kappa)\sum\limits_{v \in V_1}A^{part}d_{D^{-*},n'-l}(\kappa,\sigma)\iota_{\sigma_v}(z) \\
			&=A^{part}\theta_u^{r}(\tau,\sigma)\sum\limits_{v \in V_1}\mathop{\oplus}\limits_{s'' \geq \sigma_v}(-1)^{n'+|\sigma_v|}d_D(s,s'')^*(z) \\
			& \quad \; +\sum\limits_{\substack{\kappa \in L^*(\sigma,s) \\ \kappa \leq \tau}}A^{part}\theta_u^{r}(\tau,\kappa)\sum\limits_{v \in V_1}(-1)^{n'+n_{\sigma}^{\kappa}}\iota_{\kappa * v}(z) \\
		\end{aligned}
	\end{equation}

	Note that:
	\begin{equation}
		\label{Appcheck55}
		\begin{aligned}
			&\quad \ \; A^{part}\theta_u^{r}(\tau,\sigma)\sum\limits_{v \in V_1}\mathop{\oplus}\limits_{s'' \geq \sigma_v}(-1)^{n'+|\sigma_v|}d_D(s,s'')^*(z) \\
			& \quad (\text{By definiton } \ref{expressmho} \text{ of } \mho_{\sigma}) \\
			&=A^{part}\theta_u^{r}(\tau,\sigma)\mho_{\sigma}^{n'} \big( \mathop{\oplus}\limits_{s'' \geq \sigma}(-1)^{n'+|\sigma|+1}d_D(s,s'')^*(z) \big) \\
			& \quad (\text{By definition of } \theta L \text{ in Corollary } \ref{computationpartial}) \\
			&=(-1)^{n+|\sigma|+r+1}\theta L_u^r(\tau,\sigma) \big( \mathop{\oplus}\limits_{s' \in A_{\kappa}}(-1)^{n'-1+|\sigma|}d_D(s,s')^*(z) \big)
		\end{aligned}
	\end{equation}

	Fix $\kappa \in L^*(\sigma,s)$ with $\kappa \leq \tau$, $\kappa *v$ are different simplices for different vertex $v$. By definition \ref{expressmho} of $\mho_{\kappa}$, we have:
	\begin{equation*}
		\sum\limits_{v \in V_1}(-1)^{n'+n_{\sigma}^{\kappa}}\iota_{\kappa * v}(z)=\mathop{\oplus}\limits_{v \in V_1}(-1)^{n'+n_{\sigma}^{\kappa}}\iota_{\kappa * v}(z)=(-1)^{n'+n_{\sigma}^{\kappa}}\mho_{\kappa}^{n'+1}(z)
	\end{equation*}

	Thus:
	\begin{equation}
		\label{Appcheck56}
		\begin{aligned}
			&\quad \; \sum\limits_{\substack{\kappa \in L^*(\sigma,s) \\ \kappa \leq \tau}}A^{part}\theta_u^{r}(\tau,\kappa)\sum\limits_{v \in V_1}(-1)^{n'+n_{\sigma}^{\kappa}}\iota_{\kappa * v}(z) \\
			&=\sum\limits_{\substack{\kappa \in L^*(\sigma,s) \\ \kappa \leq \tau}}(-1)^{n'+n_{\sigma}^{\kappa}}A^{part}\theta_u^{r}(\tau,\kappa)\mho_{\kappa}^{n'+1}(z) \\
			& \quad (\text{By definition of } \theta L \text{ in Corollary } \ref{computationpartial}) \\
			&=\sum\limits_{\substack{\kappa \in L^*(\sigma,s) \\ \kappa \leq \tau}}(-1)^{n'+n_{\sigma}^{\kappa}+n+|\kappa|+r+1}\theta L_u^r(\tau,\kappa)(z) \\
			&=(-1)^{n+|\sigma|+r+1}\sum\limits_{\substack{\kappa \in L^*(\sigma,s) \\ \kappa \leq \tau}}(-1)^{n'-1+n_{\sigma}^{\kappa}}\theta L_u^r(\tau,\kappa)(z)
		\end{aligned}
	\end{equation}

	Comparing equation \ref{Appcheck54}, \ref{Appcheck55} and \ref{Appcheck56} with equation \ref{Appcheck48},  we get that equation \ref{Appcheck47} holds and the claim is therefore ture. 
	
	For the third term, by definition of $\theta L$ and $A^{part}$, we have:
	\begin{equation}
		\label{Appcheck57}
		\begin{aligned}
			\theta L_{u+1}^r(\tau,\sigma)(z)&=(-1)^{n+|\sigma|+r+1}\big( \mathop{\boxplus}\limits_{s \in B_{\sigma}}\mathop{\oplus}\limits_{s' \in B_{\kappa}}\theta_{u+1}^{r}(s',s) \big)\mho_{\sigma}^{n'}(z) \\
			&=(-1)^{n+|\sigma|+r+1} A^{part}\theta_{u+1}^{r} (\tau,\sigma)\mho_{\sigma}^{n'}(z) \\
		\end{aligned}
	\end{equation}
	
	For the last term, by the commutative diagram \ref{Tfdiagram}, for any $\kappa \geq \sigma$, we can see that $T\theta L_{u+1}^r(\tau,\sigma)$ restricted on $D_{r'-1}(B_{\kappa})^*$ is the same with $(-1)^{(r'-1)r} \cdot$ \\
	$(-1)^{|\sigma|(r'-1+r)}$ times the following composition of morphisms:
	\begin{equation*}
		\begin{tikzcd}
			D_{r'-1}(B_{\kappa})^* \ar{rr}{\theta L_{u+1}^{r'-1}(\kappa,\sigma)^*} & & D_{r}(A_{\sigma}) \ar{rr}{\text{projection}}& & D_{r}(B_{\tau}) \\
		\end{tikzcd}
	\end{equation*}
	
	Denote $s_0=(s \cap L)_0$, then we have $s \in B_{s_0}$. Since
	\begin{equation*}
		\begin{aligned}
			\theta L_{u+1}^{r'-1}(s_0,\sigma)^*(z)&=(-1)^{n+|\sigma|+r'}(\mho_{\sigma}^{r+|\sigma|+1})^* \big( \mathop{\oplus}\limits_{s' \in B_{\sigma}} \theta_{u+1}^{r'-1}(s,s')^* (z)\big) \\
			&=(-1)^{n+|\sigma|+r'}(\mho_{\sigma}^{r+|\sigma|+1})^*(\theta_{\sigma,u+1}^{r'-1})^*(z) \\
			& \quad (\text{By Lemma } \ref{Apppartialass}) \\
			&=(-1)^{n+|\sigma|+r'}(-1)^{r(r'-1)+(|\sigma|+1)(r+r'-1)}T\theta_{\sigma,u+1}^r\mho_{\sigma}^{n'}(z) \\
		\end{aligned}
	\end{equation*}
	
	Therefore, we have:
	\begin{equation}
		\label{Appcheck58}
		\begin{aligned}
			T\theta L_{u+1}^r(\tau,\sigma)(z)&=(-1)^{n+|\sigma|+r+1}\big(\mathop{\boxplus}\limits_{s \in B_{\sigma}}\mathop{\oplus}\limits_{s' \in B_{\tau}} T\theta_{u+1}^r(s',s) \big) \big(\mho_{\sigma}^{n'}(z)\big) \\
			&=(-1)^{n+|\sigma|+r+1} A^{part}T\theta_{u+1}^r(\tau,\sigma)\mho_{\sigma}^{n'}(z) \\
		\end{aligned}
	\end{equation}

	By equation \ref{Appcheck46}, \ref{Appcheck47}, \ref{Appcheck57} and \ref{Appcheck58}, we can see that the equation \ref{Appcheck19} is equivalent to the following one:
	\begin{equation*}
		\begin{aligned}
			0= \ & A^{part}(d_{D,r+1} \circ \theta_u^{r+1})(\tau,\sigma) \circ \mho_{\sigma}^{n'}-(-1)^{n-u}A^{part}(\theta_u^r \circ d_{D^{-*},n'})(\tau,\sigma) \circ \mho_{\sigma}^{n'} \\
			& +(-1)^{n-u-1} (A^{part}\theta_{u+1}^r)(\tau,\sigma) \circ \mho_{\sigma}^{n'}+(-1)^{n}(A^{part}T\theta_{u+1}^r)(\tau,\sigma) \circ \mho_{\sigma}^{n'}\\
		\end{aligned}
	\end{equation*}

	Now it follows from Lemma \ref{rearannge} and the fact that $(D,\theta)$ is quadratic that the equation \ref{Appcheck19} holds.
	
	(4) $(DL,\theta L)$ is Poincare.
	
	By definition, we need to check that $(1+T)(\theta L)_0$ is a homotopy equivalence. By Proposition 4.7 in \cite{ranickiltheory}, it suffices to check that for every simplex $\sigma \in L$, the chain map $(1+T)(\theta L)_0(\sigma,\sigma)$ is a chain homotopy equivalence. We begin by computing this morphism.
	
	Note that by equation \ref{Appcheck58}, we have that for every $r \in \mathbb{Z}$, 
	\begin{equation*}
		T\theta L_0^r(\sigma,\sigma)=(-1)^{n+|\sigma|+r+1}A^{part}T\theta_0^r(\sigma,\sigma) \mho_{\sigma}^{n-r}
	\end{equation*}
	
	Therefore:
	\begin{equation*}
		(1+T)\theta L_0^r(\sigma,\sigma)=(-1)^{n+|\sigma|+r+1}(A^{part}\theta_0^r(\sigma,\sigma)+A^{part}T\theta_0^r(\sigma,\sigma)) \circ \mho_{\sigma}^{n-r}
	\end{equation*}

	Since $A^{part}$ is a functor and $(1+T)\theta_0^r$ is a chain homotopy equivalence, by Propsition 4.7 in \cite{ranickiltheory}, we have that $A^{part}\theta_0^*(\sigma,\sigma)+A^{part}T\theta_0^*(\sigma,\sigma)$ is a chain homotopy equivalence. By Lemma \ref{Apppartialass}, $\mho_{\sigma}^{*}$ is a chain homotopy equivalence. Combining the two arguments we conclude that $(1+T)\theta L_0(\sigma,\sigma)$ is a chain homotopy equivalence. Therefore, $(DL,\theta L)$ is Poincare.
	
	(5) $\partial x \in H_{n-1}(L_1,L_0;\underline{L}(R))$ is given by $(DL,\theta L)$ via the isomorphism given in the statement of the Corollary \ref{computationpartial}.
	
	Let $y \in T^{l-n}(\underline{L_1},\underline{L_2})$ be the element corresponding to the map $g$. By Theorem \ref{Lcohomologywithadcord} and Theorem \ref{Lhomologyandco}, the following diagram commutes:
	\begin{equation*}
		\begin{tikzcd}
			H_n(L_2,L_1;\underline{L}(R)) \rar{\cong}\dar{\partial} & H^{l-n}(\underline{L_1},\underline{L_2};\underline{L}(R))\dar{\delta} \rar{\cong} & T^{l-n}(\underline{L_1},\underline{L_2}) \dar{\delta}\\
			H_{n-1}(L_1,L_0;\underline{L}(R)) \rar{\cong} & H^{l-n+1}(\underline{L_0},\underline{L_1};\underline{L}(R)) \rar{\cong} & T^{l-n+1}(\underline{L_0},\underline{L_1}) \\
		\end{tikzcd}
	\end{equation*}
	
	By the definition of the horizontal isomorphisms, the element $x$ is sent to $y$. Therefore, $\partial x$ is sent to $\delta y$ under the horizontal isomorphisms. Moreover, we have the following commutative diagram:
	\begin{equation*}
		\begin{tikzcd}
			H_{n-1}(L_1,L_0;\underline{L}(R)) \rar{\cong} & H^{l-n+1}(\underline{L_0},\underline{L_1};\underline{L}(R)) \\
			H_{n-1}(L;\underline{L}(R)) \uar{\cong}\rar{\cong} & H^{l-n+1}(\Sigma^l,\underline{L};\underline{L}(R \uar{\cong}[swap]{res}))
		\end{tikzcd}
	\end{equation*}

	Let $(\widecheck{DL},\widecheck{\theta L})$ be the local dual of $(DL,\theta L)$. Since $\partial x$ is identified with $\delta y$, it suffices to prove that $(\widecheck{DL},\widecheck{\theta L})$ maps to $\delta y$ under the following composition of morphisms:
	\begin{equation*}
		H^{l-n+1}(\Sigma^l,\underline{L};\underline{L}(R)) \stackrel{res}{\longrightarrow} H^{l-n+1}(\underline{L_0},\underline{L_1};\underline{L}(R)) \longrightarrow T^{l-n+1}(\underline{L_0},\underline{L_1})
	\end{equation*}
	
	By the explict formula in Theorem \ref{Lhomology}, we have that $(\widecheck{DL},\widecheck{\theta L})$ is mapped to the following Poincare quadratic chain complex $(D',\theta')$ in $M^h(R)^*(\underline{L_0},\underline{L_1})$ under the identification in Theorem \ref{Lhomology}:
	
	For all $u \in \mathbb{N},r \in \mathbb{Z},\sigma,\tau \notin L_0$:
	\begin{equation*}
		D'_r(\sigma^*)=\widecheck{DL}_r(\sigma^*), \ (\theta')_u^r(\tau^*,\sigma^*)=\widecheck{\theta L}_u^r(\tau^*,\sigma^*)
	\end{equation*}

	By the explict formula in Theorem \ref{Lcohomologywithadcord}, it maps to the following $(\underline{L_0},\underline{L}_1)$-ad $F'$:
	
	For all $\sigma^*,\tau^* \in \underline{L_0}$ with $\tau^* \leq \sigma^*$, denote $r_{\sigma^*}'=\frac{|\sigma^*|(|\sigma^*|-1)}{2}$, then:
	\begin{equation*}
		F'(\sigma^*,o)=\begin{cases} \big([D_r'][\sigma^*],(-1)^{r_{\sigma^*}'}(-1)^{|\sigma^*|r}sgn(o) \mathop{\oplus}\limits_{\tau^* \leq \sigma^*}(\theta')_u^r(\tau^*,\sigma^*)\big) & \text{If } \sigma \in L_1 \backslash L_0 \\
		\qquad \qquad \qquad \qquad \qquad (0,0) & \text{else}\end{cases}
	\end{equation*}
	\begin{equation*}
		F'((\tau^*,o'),(\sigma^*,o))=\begin{cases} \text{Inclusion} & \text{If } \sigma,\tau \in L_1 \backslash L_0 \\ \quad \ \ 0 & \text{else} \end{cases}
	\end{equation*}

	Now we begin to compute $\delta y$. Denote $L'=\underline{L_2} \times I \cup \underline{L_0} \times 1 \cup \underline{L_1} \times 0$. By Theorem \ref{deltaofT}, the map $\delta:T^{l-n}(\underline{L_1},\underline{L_2}) \longrightarrow T^{l-n+1}(\underline{L_0},\underline{L_1})$ is given by the negative of the composite of the following morphisms:
	\begin{equation*}
		\begin{tikzcd}
			T^{l-n}(\underline{L_1},\underline{L_2}) \rar{\mathcal{K}^*} &   T^{l-n+1}(\underline{L_1} \times \Delta^1,\underline{L_2} \times \Delta^1 \cup \underline{L_1} \times \partial\Delta^1) \\
			T^{l-n+1}(\underline{L_0} \times I,L') \urar{\cong}[swap]{res} \rar{res} & T^{l-n+1}(\underline{L_0} \times 0, \underline{L_1} \times 0) \\
		\end{tikzcd}
	\end{equation*}
	
	Given the same notation as in Theorem \ref{computationdelta}, by Theorem \ref{computationdelta}, $\delta y$ is given by the cobordism class of the following $(\underline{L_0},\underline{L_1})$-ad $F''$:
	
	For all $\sigma^*,\tau^* \in \underline{L_0}$ with $\tau^* \leq \sigma^*$:
	\begin{equation*}
		F''(\sigma^*,o)=\begin{cases} (C_{\sigma^* \times 0},-\psi_{\sigma^{*} \times 0,o}) & \text{If } \sigma \in L_1 \backslash L_0 \\ \qquad \ (0,0) & \text{else} \\ \end{cases}
	\end{equation*}
	\begin{equation*}
		F''((\tau^*,o'),(\sigma^*,o))=\begin{cases} \text{Inclusion} & \text{If } \sigma,\tau \in L_1 \backslash L_0 \\ \quad \ \ 0 & \text{else} \end{cases}
	\end{equation*}

	Thus it only leaves us to check that the two functors $F'$ and $F''$ agree. Note first that if $F''(\sigma^*,o)=F'(\sigma^*,o)$, then automatically $F''((\tau^*,o'),(\sigma^*,o))=F'((\tau^*,o'),(\sigma^*,o))$. Therefore, it suffices to prove that for all $u \in \mathbb{N},r \in \mathbb{Z},\sigma \in L_1 \backslash L_0$, we have :
	\begin{equation*}
		(C_{\sigma^* \times 0})_r=[D_r'][\sigma^*]
	\end{equation*}
	\begin{equation*}
		\psi_{\sigma^* \times 0,o}^{u,r}=(-1)^{r_{\sigma^*}'+1}(-1)^{|\sigma^*|r}sgn(o) \mathop{\oplus}\limits_{\tau^* \leq \sigma^*}(\theta')_u^r(\tau^*,\sigma^*)
	\end{equation*}
	
	Let us prove the first equation. By the definition in Theorem \ref{computationdelta}, we have that $(C_{\sigma^* \times 0})_r=\mathop{\oplus}\limits_{s \in A_{\sigma}}D_r(s)$. By the definitions of $DL$ and $A_{\sigma},B_{\sigma}$, we have $\mathop{\oplus}\limits_{s \in A_{\sigma}}D_r(s)=\mathop{\oplus}\limits_{\tau \geq \sigma}DL_r(\tau)$. Now, using the definitions of $D'$ and local dual in Theoerm \ref{localdual}, we have $[D_r'][\sigma^*]=\mathop{\oplus}\limits_{\tau^* \leq \sigma^*}\widecheck{DL}_r(\tau^*)=\mathop{\oplus}\limits_{\tau \geq \sigma}DL_r(\tau)=(C_{\sigma^* \times 0})_r$.
	
	For the second equation, by the definition in Theorem \ref{computationdelta}, we have that $\psi_{\sigma^* \times 0,o}^{u,r}=(-1)^{r_{\sigma}}sgn(o)\delta\psi_u^r[\sigma]$. By the definition of $\delta\psi_u^r[\sigma]$ in Lemma \ref{Appquadpair}, we can further write:
	\begin{equation*}
		\begin{aligned}
			\psi_{\sigma^* \times 0,o}^{u,r}&=(-1)^{r_{\sigma}}(-1)^{(|\sigma|+1)r}sgn(o)\mathop{\boxplus}\limits_{s \in B_{\sigma}}\mathop{\oplus}\limits_{s' \in A_{\sigma}}\theta_u^r(s',s)\mho_{\sigma}^{n-u-r} \\
			& \quad (\text{By defintion of } \theta L_u^r,A_{\sigma},B_{\sigma}) \\
			&=(-1)^{r_{\sigma}}(-1)^{(|\sigma|+1)r}(-1)^{n+|\sigma|+r+1}sgn(o)\mathop{\oplus}\limits_{\tau \geq \sigma}\theta L_u^r(\tau,\sigma) \\
			& \quad (\text{By definition of local dual and } \theta') \\
			&=(-1)^{n+\frac{(l-1)l}{2}+J_{\sigma}^{all}+l|\sigma|+lr}(-1)^{r_{\sigma}}(-1)^{(|\sigma|+1)r}(-1)^{n+|\sigma|+r+1} \\
			& \quad \ sgn(o)\mathop{\oplus}\limits_{\tau^* \leq \sigma^*} (\theta')_u^r(\tau^*,\sigma^*) \\
		\end{aligned}
	\end{equation*}
	
	Therefore, it suffices to prove that:
	\begin{equation}
		\label{Appcheck59}
		(-1)^{\frac{(l-1)l}{2}+J_{\sigma}^{all}+l|\sigma|+lr+r_{\sigma}+(|\sigma|+1)r+|\sigma|+r}=(-1)^{r_{\sigma^*}'+|\sigma^*|r}
	\end{equation}
	
	We have:
	\begin{equation*}
		\begin{aligned}
			&\quad \; \frac{(l-1)l}{2}+J_{\sigma}^{all}+l|\sigma|+lr+r_{\sigma}+(|\sigma|+1)r+|\sigma|+r \\
			&\equiv \frac{(l-1)l}{2}+J_{\sigma}^{all}+l|\sigma|+(l-|\sigma|)r+r_{\sigma}+|\sigma| \mod 2 \\
			&\equiv \frac{(l-1)l}{2}+l|\sigma|+|\sigma^*|r+\frac{|\sigma|(|\sigma|-1)}{2}+|\sigma| \mod 2 \\
		\end{aligned}
	\end{equation*}

	and
	
	\begin{equation*}
		\begin{aligned}
			r_{\sigma^*}'+|\sigma^*|r &=\frac{|\sigma^*|(|\sigma^*|-1)}{2}+|\sigma^*|r \\
			&=\frac{(l-|\sigma|)(l-|\sigma|-1)}{2}+|\sigma^*|r \\
			&=\frac{l(l-1)}{2}-l|\sigma|+\frac{-|\sigma|+|\sigma|^2}{2}+|\sigma|+|\sigma^*|r \\
		\end{aligned}
	\end{equation*}

	Therefore, the equation \ref{Appcheck59} holds. Thus we can finish the proof of the Corollary \ref{computationpartial}.
\end{proof}

	\bibliographystyle{abbrv}
	\bibliography{refer}
\end{document}